%% file: 2nd-flip-arxiv.tex
\documentclass[11pt, final]{amsart}
\usepackage[notcite, notref]{showkeys}

\usepackage{epic, eepic, amsfonts, latexsym, amssymb, graphicx,
multicol, mathrsfs, color, amscd, verbatim, paralist,
xspace, url, euscript, stmaryrd,  amsmath, enumitem,
bbold, verbatim}
\usepackage{multirow}

\usepackage[all]{xy}

\newif\ifprint 
\printfalse
\ifprint
\definecolor{linkred}{rgb}{0,0,0} 
\definecolor{linkblue}{rgb}{0,0,0} 
\else
\definecolor{linkred}{rgb}{0.7,0.2,0.2} 
\definecolor{linkblue}{rgb}{0,0.2,0.6}
\fi

\usepackage{xr-hyper}
\PassOptionsToPackage{obeyspaces}{url} 
\usepackage[
pdfauthor={AFSW},pdftitle={2nd flip}, pagebackref,
colorlinks,
linkcolor=linkblue, 
citecolor=linkred, 
urlcolor=linkred, 
]{hyperref}

\usepackage[svgnames]{xcolor}
\usepackage{tikz}
\usepgflibrary{decorations.text}
\usetikzlibrary{decorations.text}
\pgfdeclarelayer{edgelayer}
\pgfdeclarelayer{nodelayer}
\pgfsetlayers{edgelayer,nodelayer,main}

\tikzstyle{none}=[inner sep=0pt]
\definecolor{hexcolor0xfefdfd}{rgb}{0.996,0.992,0.992}

\newlength{\baseunit} 
\newcount{\numlines} 
\newtheorem{thm}{Theorem}[section]

\newtheorem{lem}[thm]{Lemma}
\newtheorem{prop}[thm]{Proposition}

\newtheorem{theorem}[thm]{Theorem}

\newtheorem*{theorem*}{Theorem}
\newtheorem*{mainresult}{Main Theorem}

\newtheorem{lemma}[thm]{Lemma}
\newtheorem{corollary}[thm]{Corollary}
\newtheorem{proposition}[thm]{Proposition}

\theoremstyle{definition}
\newtheorem{defn}[thm]{Definition}

\newtheorem{example}[thm]{Example}
\newtheorem{definition}[thm]{Definition}
\theoremstyle{remark}
\newtheorem{remark}[thm]{Remark}

\newtheorem*{remark*}{Remark}

\numberwithin{equation}{section}

\DeclareRobustCommand{\gobblefive}[5]{}
\newcommand*{\SkipTocEntry}{\addtocontents{toc}{\gobblefive}}

\newcommand\cA{\mathcal{A}}  \newcommand\cC{\mathcal{C}} \newcommand\cE{\mathcal{E}} \newcommand\cH{\mathcal{H}}\newcommand\cK{\mathcal{K}}\newcommand\cL{\mathcal{L}}\newcommand\cM{\mathcal{M}}\newcommand\cN{\mathcal{N}}\newcommand\cQ{\mathcal{Q}}\newcommand\cR{\mathcal{R}}\newcommand\cS{\mathcal{S}}\newcommand\cT{\mathcal{T}}\newcommand\cU{\mathcal{U}}\newcommand\cV{\mathcal{V}}\newcommand\cW{\mathcal{W}}\newcommand\cX{\mathcal{X}}\newcommand\cY{\mathcal{Y}}\newcommand\cZ{\mathcal{Z}}

\newcommand\cEC{\mathcal{EC}}
\newcommand\cWC{\mathcal{WC}}

\newcommand\cO{\mathcal{O}}

\renewcommand\AA{\mathbb{A}}\newcommand\CC{\mathbb{C}}\newcommand\GG{\mathbb{G}}\newcommand\HH{\mathbb{H}}\newcommand\PP{\mathbb{P}}\newcommand\QQ{\mathbb{Q}}
\newcommand\ZZ{\mathbb{Z}}

  \newcommand\fm{\mathfrak{m}}

\newcommand{\bs}{{\bf s}}
\newcommand{\br}{{\bf r}}

\newcommand\id{\mathrm{id}}

\newcommand\RS{R}

\newcommand{\gM}{\overline{\mathbb{M}}}

\newcommand\hookarr{\hookrightarrow}
\newcommand{\xarr}{\xrightarrow}
\newcommand\im{\operatorname{im}}

\renewcommand{\setminus}{\smallsetminus}

\newcommand{\red}{_{\mathrm{red}}}

\renewcommand{\ss}{\operatorname{ss}}

\newcommand{\rrarrows}{\rightrightarrows}

\newcommand{\Pic}{\operatorname{Pic}}

\newcommand{\Proj}{\operatorname{Proj}}

\newcommand{\Sym}{\operatorname{Sym}}

\newcommand{\oh}{\cO}

\DeclareMathOperator{\Spec}{Spec}

\newcommand{\tensor} {\otimes}
\newcommand{\wt} {\widetilde}
\renewcommand{\tilde}{\widetilde}

\newcommand{\Spf}{\operatorname{Spf}}

\newcommand{\iso}{\stackrel{\sim}{\to}}

\renewcommand{\bar}{\overline}

\newcommand{\GL}{\operatorname{GL}}
\newcommand{\PGL}{\operatorname{PGL}}

\newcommand{\dual}{\vee}
\renewcommand{\hat}{\widehat}

\renewcommand{\AA}{{\mathbb A}}
\renewcommand{\emptyset}{\varnothing}
\renewcommand{\leq}{\leqslant}
\renewcommand{\geq}{\geqslant}
\renewcommand{\tilde}{\widetilde}
\renewcommand{\hat}{\widehat}
\renewcommand{\bar}{\overline}

\def\col{\colon\thinspace} 
\def\Cu{\operatorname{Cu}}
\def\Tn{\operatorname{Tn}_1}
\def\Tnt{\operatorname{Tn}_2}

\newcommand{\sar}[1][]{{\ar@<-0.5ex>@{^{(}->}[#1]}}
\newcommand{\sarl}[1][]{{\ar@<0.5ex>@{_{(}->}[#1]}}

\newcommand{\arr}{\rightarrow}

\def\irr{\mathrm{irr}}
\def\red{\mathrm{red}}
\def\odd{\mathrm{odd}}

\newcommand{\eps}{{\varepsilon}}

\DeclareMathOperator{\Cr}{Cr}

\DeclareMathOperator{\TT}{T}
\renewcommand{\HH}{\mathrm{H}}
\DeclareMathOperator{\Aut}{Aut}
\DeclareMathOperator{\Def}{Def}

\newtheorem*{ques*}{Question}

\def\C{\mathcal{C}}

\def\E{\mathcal{E}}

\def\F{\mathscr{F}}

\def\H{\mathcal{H}}

\def\id{\text{\rm id}}
\def\im{\text{im\,}}

\def\pn{\{p_i\}_{i=1}^{n}}

\def\qm{\{q_i\}_{i=1}^{m}}

\def\K{\mathcal{K}}

\def\M{\overline{M}}

\def\SM{\overline{\mathcal{M}}}

\def\O{\mathscr{O}}

\def\B{\mathcal{B}}
\def\T{\mathcal{T}}
\def\W{\mathcal{W}}
\def\B{\mathcal{B}}

\def\P{\mathbb{P}}
\def\Q{\mathbb{Q}}
\def\X{\mathcal{X}}

\def\U{\mathcal{U}}
\def\V{\mathcal{V}}
\def\Y{\mathcal{Y}}
\def\Z{\mathcal{Z}}

\def\sigman{\{\sigma_{i}\}_{i=1}^{n}}

\def\Spec{\text{\rm Spec\,}}
\def\Spf{\text{\rm Spf\,}}
\def\Proj{\text{\rm Proj\,}}

\def\Pic{\text{\rm Pic\,}}

\DeclareMathOperator{\Sprout}{Sprout}

\newcommand\ra{\rightarrow}
\newcommand\dra{\dashrightarrow}

\newcommand\ind{\operatorname{index}}

\def\firstf{\!\text{{\smaller $\hspace{.01cm}\frac{9}{11}$}}}
\def\secondf{\!\text{{\smaller $\hspace{.01cm}\frac{7}{10}$}}}
\def\thirdf{\!\text{{\smaller $\hspace{.03cm} \frac{2}{3}$}}}

\def\first{\text{{\smaller $9/11$}}}
\def\second{\text{{\smaller $7/10$}}}
\def\third{\text{{\smaller $2/3$}}}
\def\one{\text{{\smaller $1$}}}

\def\pe{\!+\!\small{\epsilon}}
\def\me{\!-\!\small{\epsilon}}

\def\firstfme{\!\text{{\smaller $\hspace{.01cm}\frac{9}{11}\!-\! \epsilon$}}}
\def\secondfme{\!\text{{\smaller $\hspace{.01cm}\frac{7}{10}\!-\! \epsilon$}}}
\def\thirdfme{\!\text{{\smaller $\hspace{.03cm}\frac{2}{3}\!-\! \epsilon$}}}

\def\firstme{\text{{\smaller $9/11\!-\! \epsilon$}}}
\def\secondme{\text{{\smaller $7/10\!-\! \epsilon$}}}
\def\thirdme{\text{{\smaller $2/3\!-\! \epsilon$}}}

\def\firstpe{\text{{\smaller $9/11\!+\! \epsilon$}}}

\makeatletter
\ifcase \@ptsize \relax
  \newcommand{\miniscule}{\@setfontsize\miniscule{4}{5}}
\or
  \newcommand{\miniscule}{\@setfontsize\miniscule{5}{6}}
\or
  \newcommand{\miniscule}{\@setfontsize\miniscule{5}{6}}
\fi
\makeatother

\newcommand{\epf}{\qed \vspace{+10pt}}

\newcommand{\gitq}{/\hspace{-0.25pc}/}
\newcommand\Mg[1]{\overline{\mathcal{M}}_{#1}}
\newcommand\co{\colon \thinspace} 

\begin{document}
\title[Log Minimal Model Program for $\M_{g,n}$: The second flip]
{Log Minimal Model Program for the moduli space of stable curves: The second flip}
\author[Alper]{Jarod Alper}
\author[Fedorchuk]{Maksym Fedorchuk}
\author[Smyth]{David Ishii Smyth}
\author[van der Wyck]{Frederick van der Wyck}

\address[Alper]{Mathematical Sciences Institute\\
Australian National University\\
Canberra, ACT 0200, Australia}
\email{\href{mailto:jarod.alper@anu.edu.au}{jarod.alper@anu.edu.au}}

\address[Fedorchuk]{Department of Mathematics\\
Boston College \\
Carney Hall 324\\
140 Commonwealth Avenue\\
Chestnut Hill, MA 02467}
\email{\href{mailto:maksym.fedorchuk@bc.edu}{maksym.fedorchuk@bc.edu}}

\address[Smyth]{Mathematical Sciences Institute\\
Australian National University\\
Canberra, ACT 0200, Australia}
\email{\href{mailto:david.smyth@anu.edu.au}{david.smyth@anu.edu.au}}

\address[van der Wyck]{Goldman Sachs International \\
120 Fleet Street \\
London EC4A 2BE}
\email{\href{mailto:frederick.vanderwyck@gmail.com}{frederick.vanderwyck@gmail.com}}

\begin{abstract}
We prove an existence theorem for good moduli spaces, 
and use it to construct the second flip in the log minimal model program for $\M_{g}$. 
In fact, our methods give a uniform self-contained
construction of the first three steps of the log minimal model program for $\M_{g}$ and $\M_{g,n}$.
\end{abstract}

\maketitle

\linespread{0.6}
\setcounter{tocdepth}{2}
\tableofcontents

\linespread{1.07}
\input{2nd-flip-Intro-submit-v9}

\input{2nd-flip-Section2-submit-v9}

\input{2nd-flip-Section3-submit-v9}

\input{2nd-flip-Section4-submit-v9}
\input{2nd-flip-Section5-submit-v9}

\input{2nd-flip-Appendix-submit-v9}

\newcommand{\netarxiv}[1]{\href{http://arxiv.org/abs/#1}{{\sffamily{\texttt{arXiv:#1}}}}}

\renewcommand{\bibliofont}{\normalfont\small} 
\bibliography{references-flip}{}
\bibliographystyle{amsalpha}


\end{document}

%% file: 2nd-flip-Intro-submit-v9.tex
\section{Introduction} \label{S:intro}

\noindent
In an effort to understand the canonical model of $\M_{g}$, 
Hassett and Keel introduced the log minimal model program (LMMP) for $\M_{g}$. 
For any $\alpha \in \Q \cap [0,1]$ such that $K_{\SM_{g}}+\alpha\delta$ is big, Hassett defined
\begin{equation}\label{E:log-canonical-models}
\M_{g}(\alpha):=\Proj \bigoplus_{m \geq0} \HH^0(\SM_{g}, \lfloor m(K_{\SM_{g}}+\alpha\delta) \rfloor),
\end{equation}
and asked whether the spaces $\M_{g}(\alpha)$ admit a modular interpretation \cite{hassett_genus2}. 
In \cite{HH1,HH2}, Hassett and Hyeon carried out the first two steps of this program by showing that:
\[
\M_{g}(\alpha)=\begin{cases}
\M_{g} & \text{  if }\alpha \in (\first, \one]\\
\M_{g}^{\, ps}&  \text{  if } \alpha \in (\second,\first]\\
\M_{g}^{\, c} & \text{  if } \alpha =\second\\
\M_{g}^{\, h}&  \text{  if } \alpha \in (\thirdme, \second)
\end{cases}
\]
where $\M_g^{\, ps}$, $\M_{g}^{\, c}$, and $\M_{g}^{\, h}$ are the moduli spaces of pseudostable (see  \cite{schubert}), 
c-semistable, and h-semistable curves (see \cite{HH2}), respectively. 
Additional steps of the LMMP for $\M_{g}$ are known when $g \leq 5$ 
\cite{hassett_genus2,hyeon-lee_genus3,hyeon-lee_genus4,Fed4,CMJL1,CMJL2,FSgenus5}. 
In these works, new projective moduli spaces of curves are constructed using Geometric Invariant Theory (GIT). 
Indeed, one of the most appealing features of the Hassett-Keel program is the way that it ties together 
different compactifications of $M_{g}$ obtained by varying the parameters implicit in 
Gieseker and Mumford's classical GIT construction of $\M_g$ \cite{git,gieseker}. 
We refer the reader to \cite{morrison-git} for a detailed discussion of these modified GIT constructions.

In this paper, we develop new techniques for constructing moduli spaces without GIT and apply them to construct the third step of the LMMP, 
a flip replacing Weierstrass genus $2$ tails by ramphoid cusps.
In fact, we give a uniform construction of the first three steps of the LMMP for $\M_{g}$, as 
well as an analogous program for $\M_{g,n}$. To motivate our approach, let us recall the three-step procedure used to construct $\M_{g}$  and establish its projectivity intrinsically:
\begin{enumerate}
\item Prove that the functor of stable curves is a proper Deligne-Mumford stack $\SM_{g}$ \cite{deligne-mumford}.
\item Use the Keel-Mori theorem to show that  $\SM_{g}$ has a coarse moduli space $\SM_{g} \rightarrow \M_{g}$ \cite{keel-mori}.
\item Prove that some line bundle on $\SM_{g}$ descends to an ample line bundle on $\M_{g}$ \cite{kollar-projectivity,cornalba}.
\end{enumerate}

This is now the standard procedure for constructing projective moduli spaces in algebraic geometry. 
It is indispensable in cases where a global quotient 
presentation for the relevant moduli problem is not available, or where the GIT stability analysis is intractable, and there are good reasons to 
expect both these issues to arise in further stages of the LMMP for $\M_{g}$.  Unfortunately, this procedure cannot be used to construct the  log 
canonical models $\M_{g}(\alpha)$ because potential moduli stacks $\SM_{g}(\alpha)$ may include curves with infinite automorphism groups. In 
other words, the stacks $\SM_{g}(\alpha)$ may be non-separated and therefore may not possess a Keel-Mori coarse moduli space. 
The correct fix is to replace the notion of a coarse moduli space by a good moduli space, as defined and developed 
by Alper \cite{alper_good, alper_local, alper_quotient, alper_adequate}.

In this paper, we prove a general existence theorem for good moduli spaces of non-separated
algebraic stacks (Theorem \ref{T:general-existence}) that can be viewed as a generalization of the Keel-Mori theorem \cite{keel-mori}.
This allows us to carry out a modified version of the standard three-step procedure 
in order to construct moduli interpretations for the log canonical models\footnote{Note that the natural divisor for scaling in the pointed case is 
$K_{\SM_{g,n}}+\alpha\delta+(1-\alpha)\psi=13\lambda-(2-\alpha)(\delta-\psi)$ rather than $K_{\SM_{g,n}}+\alpha \delta$; 
see \cite[p.1845]{smyth_elliptic2} for a discussion of this point.}  
\begin{equation}\label{E:log-canonical-models-2}
\M_{g,n}(\alpha):=\Proj \bigoplus_{m \geq0} \HH^0(\SM_{g,n}, \lfloor m(K_{\SM_{g,n}}+\alpha\delta+(1-\alpha)\psi) \rfloor).
\end{equation}
Specifically, for all $\alpha>\thirdme$, where $0 < \epsilon \ll 1$, we
\begin{enumerate}
\item Construct an algebraic stack $\SM_{g,n}(\alpha)$ of $\alpha$-stable curves (Theorem \ref{T:algebraicity}).
\item Construct a good moduli space $\SM_{g,n}(\alpha) \rightarrow \gM_{g,n}(\alpha)$ (Theorem \ref{T:Existence}).
\item Show that $K_{\SM_{g,n}(\alpha)}+\alpha\delta+(1-\alpha)\psi$ on $\SM_{g,n}(\alpha)$ 
descends to an ample line bundle on $\gM_{g,n}(\alpha)$, 
and conclude that $\gM_{g,n}(\alpha)\simeq \M_{g,n}(\alpha)$ (Theorem \ref{T:Projectivity}).
\end{enumerate}
In sum, we obtain the following result.
\begin{mainresult}
There exists a diagram
\begin{equation*}
\xymatrix{
\SM_{g,n} \ar@{^(->}[r]_{i_1^+} \ar[d]	& \SM_{g,n}(\firstf) \ar[dd]_{\phi_{1}}  & \SM_{g,n}(\firstfme) \ar@{^(->}[r]_{i_2^+}  \ar[d]_{\phi_1^-}  \ar@{_(->}[l]	^{i_1^-} & \SM_{g,n}(\secondf) \ar[dd]_{\phi_2} & \SM_{g,n}(\secondfme) \ar@{_(->}[l]^{i_2^-}  \ar@{^(->}[r]_{i_3^+}  \ar[d]_{\phi_2^-}	& \SM_{g,n}(\thirdf) \ar[dd]_{\phi_3} & \SM_{g,n}(\thirdfme) \ar@{_(->}[l]^{i_3^-}  \ar[d]_{\phi_3^-}\\
\M_{g,n}\ar[rd]^{j_1^+}	& 			&\M_{g,n}(\firstfme) \ar[rd]^{j_2^+} \ar[ld]_{j_1^-}	& 			& \M_{g,n}(\secondfme)\ar[rd]^{j_3^+} \ar[ld]_{j_2^-}& 			& \M_{g,n}(\thirdfme)\ar[ld]_{j_3^-} \\
& \M_{g,n}(\firstf)&	        & \M_{g,n}(\secondf)&		& \M_{g,n}(\thirdf)& 
}
\end{equation*}
where:
\begin{enumerate}
\item $\SM_{g,n}(\alpha)$ is the moduli stack of $\alpha$-stable curves, and for $c=1,2,3$:
\item $i_{c}^+$ and $i_{c}^-$ are open immersions of algebraic stacks.
\item The morphisms $\phi_c$ and $\phi_c^-$ are good moduli spaces.
\item The morphisms $j_{c}^+$ and $j_{c}^-$ are projective morphisms induced by $i_c^+$ and $i_c^-$, respectively. 
\end{enumerate}
When $n=0$, the above diagram
constitutes the steps of the log minimal model program for $\M_{g}$. In particular, $j_1^{+}$ is the first contraction, 
$j_1^{-}$ is an isomorphism, $(j_2^+, j_2^-)$ is the first flip, and $(j_3^+, j_3^-)$ is the second flip.
\end{mainresult}

The parameter $\alpha$ passes through three critical values, namely $\alpha_1=\first, \alpha_2=\second$, and $\alpha_3=\third$.  
In the open intervals $(\first, \one), (\second, \first), (\third, \second)$ and $(\thirdme, \third)$, 
the definition of $\alpha$-stability does not change, and 
consequently neither do $\SM_{g,n}(\alpha)$ or $\M_{g,n}(\alpha)$. 

The theorem is degenerate in several special cases:  
For $(g,n)=(1,1)$, $(1,2)$, $(2,0)$, the divisor $K_{\SM_{g,n}}+\alpha\delta+(1-\alpha)\psi$
hits the edge of the effective cone at $\first, \second$, and $\second$, respectively, 
and hence the diagram should be taken to terminate at these 
critical values. Furthermore, when $g=1$ and $n \geq 3$, or $(g,n)=(3,0), (3,1)$, 
$\alpha$-stability does not change at the threshold value 
$\alpha_3=\third$, so the morphisms $(i_3^+, i_3^-)$ and $(j_3^+, j_3^-)$ are isomorphisms. 
Finally, for $(g,n)=(2,1)$, $j_3^+$ is a divisorial contraction and $j_3^-$ is an isomorphism. 

\begin{remark*}
 As mentioned above, when $n=0$ and $\alpha>\secondme$, these spaces have been constructed using GIT. 
 In these cases, our definition of $\alpha$-stability agrees with the GIT semistability notions studied 
 in the work of Schubert, Hassett, Hyeon, and Morrison \cite{schubert, HH1, HH2, hyeon-morrison}.
\end{remark*}

The key observation underlying our proof of the main theorem is that at each critical value $\alpha_c \in \{\first, \second, \third\}$, the inclusions
\begin{equation*}
\SM_{g,n}(\alpha_c \pe) \hookrightarrow \SM_{g,n}(\alpha_c) \hookleftarrow \SM_{g,n}(\alpha_c \me)
\end{equation*}
can be locally modeled by an intrinsic variation of GIT problem (Theorem \ref{theorem-etale-VGIT}). 
It is this feature of the geometry which enables us to verify the hypotheses of Theorem \ref{T:general-existence}. 
We axiomatize this connection between local VGIT and the existence of good moduli spaces in Theorem \ref{T:vgit-existence}. 
In short, Theorem \ref{T:vgit-existence} says that if $\X$ is an algebraic stack with a pair of open immersions 
$\X^{+} \hookrightarrow \X \hookleftarrow \X^{-}$ which can be locally modeled by a VGIT problem, 
and if the open substack $\X^+$ and the two closed substacks $\X \setminus \X^{-}$ and $\X \setminus \X^{+}$ 
each admit good moduli spaces, then $\X$ admits a good moduli space. 
This paves the way for an inductive construction of good moduli spaces for the stacks $\SM_{g,n}(\alpha)$.

Let us conclude by briefly describing the geometry of the second flip. 
At $\alpha_3 = \third$, the locus of curves with a genus 2 Weierstrass tail 
(i.e., a genus $2$ subcurve nodally attached to the rest of the curve at a Weierstrass point),
or more generally a Weierstrass chain (see Definition \ref{defn-chains}), is flipped to the locus of curves with a ramphoid cusp ($y^2=x^5$).  
See Figure \ref{F:second-flip}.  
The fibers of $j_3^{+}$ correspond to varying moduli of Weierstrass chains, while the fibers of $j_3^-$ correspond to 
varying moduli of ramphoid cuspidal crimpings.
\begin{figure}[!thb] 
\begin{center}
\begin{tikzpicture}[scale=.85]
	\begin{pgfonlayer}{nodelayer}
		\coordinate [style=black] (0) at (-2.5, 1.75) {};
		\coordinate [style=black] (1) at (-2.75, -1.5) {};
		\coordinate [style=black] (2) at (-3.5, 1) {};
		\coordinate [style=black] (3) at (-1.5, 1) {};
		\coordinate [style=black] (4) at (-3.75, 2.25) {};
		\coordinate [style=black] (5) at (-3.25, 2.25) {};
		\coordinate [style=black] (6) at (-3.75, -2) {};
		\coordinate [style=black] (7) at (-3.25, -2) {};
		\coordinate [style=black] (8) at (-1.5+.3, 2.25) {};
		\coordinate [style=black] (9) at (-1+.3, 2.25) {};
		\coordinate [style=black] (10) at (-1.5+.3, -2) {};
		\coordinate [style=black] (11) at (-1+.3, -2) {};
		\coordinate [style=black] (12) at (1+1, 2.25) {};
		\coordinate [style=black] (13) at (1.5+1, 2.25) {};
		\coordinate [style=black] (14) at (1.5+1, -2) {};
		\coordinate [style=black] (15) at (1+1, -2) {};
		\coordinate [style=black] (16) at (3.25+1, 2.25) {};
		\coordinate [style=black] (17) at (3.75+1, 2.25) {};
		\coordinate [style=black] (18) at (3.25+1, -2) {};
		\coordinate [style=black] (19) at (3.75+1, -2) {};
		\coordinate [style=black] (20) at (2.75+1, 2) {};
		\coordinate [style=black] (21) at (2.25+1, 1.25) {};
		\coordinate [style=black] (22) at (2.5+1, -1.75) {};
		
		\coordinate [style=black] (a) at (-1-1.5, 1.25-1+.25) {};
		\coordinate [style=black] (b) at (0-1.5, 0.75-1) {};
		\coordinate [style=black] (c) at (-0.5-1.5-.25, 1.75-1) {};
		\coordinate [style=black] (d) at (-1-1.5, 1.75-1) {};
		\node (23) at (-1.5,1.2)  {\small{$g$=2}};		
		\node (23') at (-1.60,-.4)  {\tiny{Weierstrass}};
		\node (23'') at (-1.60,-.66)  {\tiny{point}};

		\node (24) at (-4.5,.1)  {\Large{$j_3^+$}};
		\node (25) at (0,.1)  {\Large{$=$}};
		\node (23) at (1.8+1,.9)  {\small{$y^2$=$x^5$}};
		\node (24) at (1.25,.1)  {\Large{$j_3^-$}};

	\end{pgfonlayer}
	\begin{pgfonlayer}{edgelayer}
		\draw [style=very thick, in=30, out=210] (0) to (1);
		\draw [style=very thick, in=210, out=30, looseness=1.25] (2) to (3);
		\draw [style=thin] (4) to (6);
		\draw [style=thin] (7) to (6);
		\draw [style=thin] (4) to (5);
		\draw [style=thin] (8) to (9);
		\draw [style=thin] (10) to (11);
		\draw [style=thin] (11) to (9);
		\draw [style=thin] (12) to (15);
		\draw [style=thin] (15) to (14);
		\draw [style=thin] (12) to (13);
		\draw [style=thin] (16) to (17);
		\draw [style=thin] (18) to (19);
		\draw [style=thin] (17) to (19);
		\draw [style=very thick, bend right=300, looseness=1.25] (20) to (21);
		\draw [style=very thick, in=0, out=45, looseness=0.75] (22) to (21);
		\draw [style=thin]  (d) to (b);
		\draw [style=thin] (d) to (a);
		\draw [style=thin]  (c) to (d);
	\end{pgfonlayer}
\end{tikzpicture}

\end{center}
\vspace{-.2cm}
\caption{Curves with a nodally attached genus $2$ Weierstrass tail are flipped to curves with a ramphoid cuspidal ($y^2=x^5$) singularity.} \label{F:second-flip}
\end{figure}
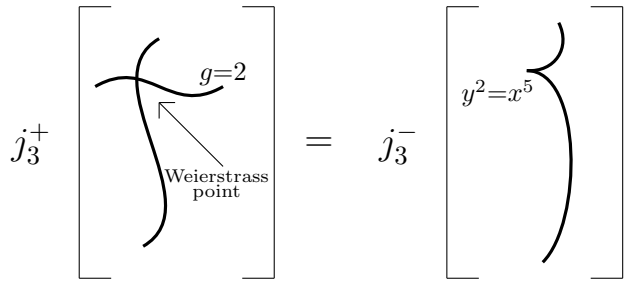
Moreover, if $(K, p)$ is a fixed curve of genus $g-2$, all curves obtained by attaching a Weierstrass genus 2 tail at $p$ or imposing a ramphoid 
cusp at $p$ are identified in $\M_{g,n}(\third)$.  This can be seen on the level of stacks since, in $\Mg{g,n}(\third)$, all such curves admit an 
isotrivial specialization to the curve $C_0$, obtained by attaching a rational ramphoid cuspidal tail to $K$ at $p$. See Figure \ref{F:isotrivial}. 

\begin{figure}[thb]
\begin{center}
\begin{tikzpicture}[scale=1]
	\begin{pgfonlayer}{nodelayer}
		\coordinate [style=black] (a) at (-2.5-2.5, 1.75) {};
		\coordinate [style=black] (b) at (-2.75-2.5, -1.5) {};
		\coordinate [style=black] (c) at (-3.5-2.5, 1) {};
		\coordinate [style=black] (d) at (-1.5-2.5, 1) {};
		\coordinate [style=black] (a0) at (-1-1.5-2.5, 1.25-1+.25) {};
		\coordinate [style=black] (b0) at (0-1.5-2.5, 0.75-1) {};
		\coordinate [style=black] (c0) at (-0.5-1.5-.25-2.5, 1.75-1) {};
		\coordinate [style=black] (d0) at (-1-1.5-2.5, 1.75-1) {};		

		\node  at (-1.5-2-.5,1.2)  {\small{$g$=2}};		
		\node  at (-1.65-2-.5,-.4)  {\tiny{Weierstrass}};
		\node  at (-1.65-2-.5,-.66)  {\tiny{point}};

		\coordinate [style=black] (a') at (-2.5+2.7, 1.75) {};
		\coordinate [style=black] (b') at (-2.75+2.7, -1.5) {};
		\coordinate [style=black] (c') at (-3.5+2.7, 1) {};
		\coordinate [style=black] (d') at (-1.5+2.7, 1) {};
		\node (23) at (1.5,1.75)  {\small{$y^2$=$x^5$}};
		\node (24) at (0,-1.8)  {$C_0$};

		\coordinate [style=black] (e) at (2.75+3.25-.5, 2) {};
		\coordinate [style=black] (f) at (2.25+3.25-.5, 1.25) {};
		\coordinate [style=black] (g) at (2.5+3.25-.5, -1.75) {};
		\node  at (4.4,1.2)  {\small{$y^2$=$x^5$}};

		\coordinate [style=black] (x) at (-.8, 1) {};
		\coordinate [style=black] (y) at (-.8+2.5, 1) {};
		\coordinate [style=black] (z) at (-.8+1.75, 1) {};

		\coordinate [style=black] (0) at (-3, 0) {};
		\coordinate [style=black] (1) at (-2.5, 0) {};
		\coordinate [style=black] (2) at (-2, 0) {};
		\coordinate [style=black] (3) at (-1.5, 0) {};
		\coordinate [style=black] (4) at (-1.25, 0) {};
		\coordinate [style=black] (5) at (-1.5, 0.25) {};
		\coordinate [style=black] (6) at (-1.5, -0.25) {};
		
		\coordinate [style=black] (0') at (3, 0) {};
		\coordinate [style=black] (1') at (2.5, 0) {};
		\coordinate [style=black] (2') at (2, 0) {};
		\coordinate [style=black] (3') at (1.5, 0) {};
		\coordinate [style=black] (4') at (1.25, 0) {};
		\coordinate [style=black] (5') at (1.5, 0.25) {};
		\coordinate [style=black] (6') at (1.5, -0.25) {};

		\node  at (-1.5+3.7,1.1)  {\small{$g$=0}};

	\end{pgfonlayer}
	\begin{pgfonlayer}{edgelayer}
		\draw [style=very thick, in=30, out=210] (a) to (b);
		\draw [style=very thick, in=210, out=30, looseness=1.25] (c) to (d);

		\draw [style=very thick, in=30, out=210] (a') to (b');
		\draw [style=very thick,in=225, out=45, looseness=1] (x) to (z);
		\draw [style=very thick,in=150, out=30, looseness=3.25] (z) to (y);
		
		\draw [style=very thick, bend right=300, looseness=1.25] (e) to (f);
		\draw [style=very thick, in=0, out=45, looseness=0.75] (g) to (f);
		
		\draw [bend left, looseness=1.50] (0) to (1);
		\draw [bend right, looseness=1.50] (1) to (2);
		\draw [bend left, looseness=1.5] (2) to (3);
		\draw [bend right] (3) to (4);
		\draw (6) to (4);
		\draw (4) to (5);

		\draw [bend left, looseness=1.50] (0') to (1');
		\draw [bend right, looseness=1.50] (1') to (2');
		\draw [bend left, looseness=1.5] (2') to (3');
		\draw [bend right] (3') to (4');
		\draw (6') to (4');
		\draw (4') to (5');
		
		\draw [style=thin]  (d0) to (b0);
		\draw [style=thin] (d0) to (a0);
		\draw [style=thin]  (c0) to (d0);

	\end{pgfonlayer}
\end{tikzpicture}
\end{center}
\vspace{-.5cm}
\caption{The curve $C_0$ is the nodal union of a genus $g-2$ curve $K$ and a rational ramphoid cuspidal tail.  All curves obtained by either attaching a Weierstrass genus 2 tail to $K$ at $p$, or imposing a ramphoid cusp on $K$ at $p$, isotrivially specialize to $C_0$. 
Observe that $\Aut(C_0)$ is not finite.} \label{F:isotrivial}
\end{figure}
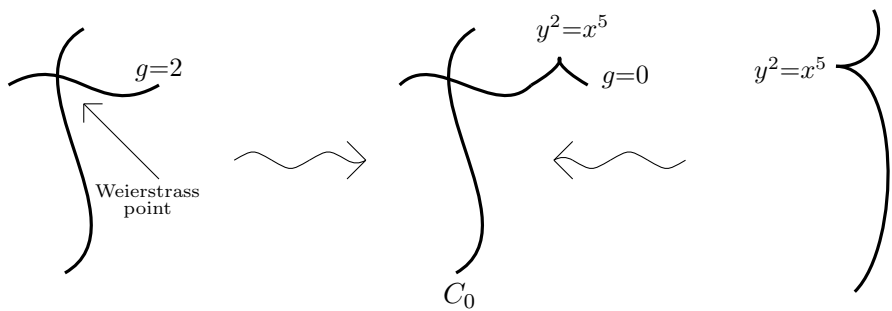

\subsection*{Outline of the paper}

In Section \ref{S:alpha-stability}, we define the notion of $\alpha$-stability for $n$-pointed curves 
and prove that they are deformation open conditions. 
We conclude that $\SM_{g,n}(\alpha)$, the stack of $n$-pointed $\alpha$-stable curves of genus $g$, is algebraic. 
We also characterize the closed 
points of $\SM_{g,n}(\alpha_c)$ for each critical value $\alpha_c$. 
In Section \ref{section-local-vgit}, we develop the machinery of local quotient presentations and local 
variation of GIT, and compute the VGIT chambers associated to closed points in $\SM_{g,n}(\alpha_c)$ 
for each critical value $\alpha_c$. In particular, we show that the inclusions 
$\SM_{g,n}(\alpha_c\pe) \hookrightarrow \SM_{g,n}(\alpha) \hookleftarrow \SM_{g,n}(\alpha_c \me)$ are cut out 
by these chambers.  
In Section \ref{S:existence}, we prove three existence theorems for good moduli spaces, 
and apply these to give an inductive proof that the stacks 
$\SM_{g,n}(\alpha)$ admit good moduli spaces. 
In Section \ref{S:projectivity}, we give a direct proof that the line bundle 
$K_{\SM_{g,n}(\alpha_c + \epsilon)}+\alpha_c\delta+(1-\alpha_c)\psi$ is nef on $\SM_{g,n}(\alpha_c \pe)$ for each critical value $\alpha_c$, 
and use this to show that the good moduli spaces of $\SM_{g,n}(\alpha)$ are the corresponding log canonical models $\M_{g,n}(\alpha)$.

\SkipTocEntry\subsection*{Notation}

We work over a fixed algebraically closed field  $\CC$ of characteristic zero.
An $n$-pointed curve $(C, \{p_i\}_{i=1}^{n})$ is a connected, reduced, proper $1$-dimensional $\CC$-scheme 
$C$ with $n$ distinct smooth marked points $p_i \in C$.  A curve $C$ has an $A_k$-singularity at  
$p \in C$ if $\widehat{\cO}_{C,p}  \simeq \CC[[x,y]]/(y^2-x^{k+1})$.
An $A_1$- (resp., $A_2$-, $A_3$-, $A_4$-) singularity is also called a \emph{node} (resp., \emph{cusp, tacnode, ramphoid cusp}).

Line bundles and divisors, such as $\lambda$, $\delta$, $K$, and $\psi$, 
on the stack of pointed curves with at-worst $A$-singularities, 
are discussed in \S \ref{S:main-result-positivity}.

We use the notation $\Delta = \Spec R$ and $\Delta^* = \Spec K$, where $R$ is
a discrete valuation ring with fraction field $K$; we set $0$, $\eta$
and $\bar{\eta}$ to be the closed point, the generic point and the
geometric generic point respectively of $\Delta$. We say that a flat family $\C \ra \Delta$ is an \emph{isotrivial specialization}
if $\C \times_{\Delta} \Delta^* \ra \Delta^*$ is isotrivial.

\subsection*{Acknowledgments}

We thank Brendan Hassett and Ian Morrison for their enthusiastic and long-standing support of this project.  
In particular, we are grateful to Ian Morrison for detailed comments and suggestions on the earlier version of this paper.
We also thank Joe Harris, David Hyeon, Johan de Jong,  Se\'{a}n Keel, and Ravi Vakil for many useful conversations and suggestions.

This project originated from the stimulating environment provided by MSRI's Algebraic Geometry program in the spring of 2009.  
During the preparation of this paper, the first author was partially supported by an NSF Postdoctoral Research Fellowship grant 0802921, 
short-term research grants provided by Harvard University and the IHES, and the Australian Research Council grant DE140101519.
The second author was partially supported by NSF grant DMS-1259226 and the Australian National University MSRVP fund. 
The third author was partially supported by NSF grant DMS-0901095 and the Australian Research Council grant DE140100259.  
The third and fourth author were also partially supported by a Columbia University short-term visitor grant. 
A portion of this work was revised when the second author visited the Max-Planck-Institute for Mathematics in Bonn.

%% file: 2nd-flip-Section2-submit-v9.tex
\section{$\alpha$-stability}
\label{S:alpha-stability}

In this section, we define $\alpha$-stability (Definition \ref{defn-stability}) and show that it is an open condition. 
We conclude that $\SM_{g,n}(\alpha)$, the stack of $n$-pointed $\alpha$-stable curves of genus $g$, 
is an algebraic stack of finite type over $\CC$ (see Theorem \ref{T:algebraicity}). 
We also give a complete description of the closed points of $\SM_{g,n}(\alpha_c)$ 
for $\alpha_c \in \{\third, \second, \first\}$ (Theorem \ref{T:ClosedCurves}).

\subsection{Definition of $\alpha$-stability}
\label{S:definition-stability}
The basic idea is to modify Deligne-Mumford stability by designating certain curve singularities as `stable,' and certain subcurves as `unstable.' 
We begin by defining the unstable subcurves associated to the first three steps of the log MMP for $\Mg{g,n}$.

\begin{defn}[Tails and Bridges]  \label{defn-stable-tails}
\hfill
\begin{enumerate}
\item
An \emph{elliptic tail} is a $1$-pointed curve $(E,q)$ of
arithmetic genus $1$ which admits a finite, surjective, degree $2$
map $\phi\co E \rightarrow \P^{1}$ ramified at $q$.
\item
An \emph{elliptic bridge} is a $2$-pointed curve
$(E,q_1,q_2)$ of arithmetic genus $1$ which admits a finite,
surjective, degree $2$ map $\phi\co E \rightarrow \P^{1}$ such that
$\phi^{-1}(\{\infty\})=\{q_1+q_2\}$.  
\item
A \emph{Weierstrass genus $2$ tail} (or simply \emph{Weierstrass tail}) is a $1$-pointed curve $(E,q)$ of
arithmetic genus $2$ which admits a finite, surjective, degree $2$
map $\phi\co E \rightarrow \P^{1}$ ramified at $q$.
\end{enumerate}
We use the term {\it $\alpha_c$-tail} to mean an elliptic tail if $\alpha_c = \first$, an elliptic bridge if $\alpha_c = \second$,
and a Weierstrass tail if $\alpha_c = \third$.
\end{defn}

\vspace{-.7cm}
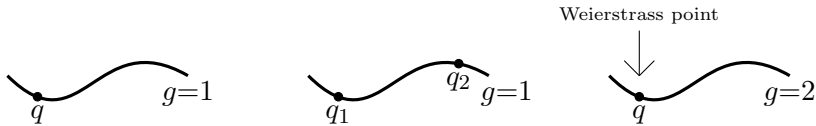
\begin{figure}[hbt]
\begin{center}
\begin{tikzpicture}[scale=.8]
\tikzstyle{blackfill}=[circle,fill=black,draw=Black]
	\begin{pgfonlayer}{nodelayer}
		\coordinate [style=black] (0) at (-1.5-5, 0) {};
		\coordinate [style=black] (1) at (1.5-5, 0) {};
		\filldraw [black] (-1-5,-.35) circle (2pt);
		
		\coordinate [style=black] (a) at (-1.5, 0) {};
		\coordinate [style=black] (b) at (1.5, 0) {};
		\filldraw [black] (-1,-.35) circle (2pt);
		\filldraw [black] (1,.2) circle (2pt);

		\coordinate [style=black] (0') at (-1.5+5, 0) {};
		\coordinate [style=black] (1') at (1.5+5, 0) {};
		\filldraw [black] (-1+5,-.35) circle (2pt);
		
		\coordinate [style=black] (x) at (4, 0) {};
		\coordinate [style=black] (y1) at (4, .75) {};
		\coordinate [style=black] (y2) at (3.75, .25) {};
		\coordinate [style=black] (y3) at (4.25, .25) {};

		\node at (-6,-.65)  {$q$};
		\node  at (-3.5,-.3)  {$g$=1};

		\node  at (-1,-.65)  {$q_1$};
		\node  at (1,-.1)  {$q_2$};
		\node  at (1.8,-.3)  {$g$=1};

		\node at (4,-.65)  {$q$};
		\node  at (6.5,-.3)  {$g$=2};
		
		\node at (4, 1.0) {\tiny{Weierstrass point}};
		\node at (4, 1) {}; 

		\end{pgfonlayer}
	\begin{pgfonlayer}{edgelayer}
		\draw [style=very thick, in=150, out=-45, looseness=1.50] (0) to (1);
		\draw [style=very thick, in=150, out=-45, looseness=1.50] (0') to (1');	
		\draw [style=very thick, in=150, out=-45, looseness=1.50] (a) to (b);
		\draw (x) to (y1);
		\draw (x) to (y2);
		\draw (x) to (y3);

	\end{pgfonlayer}
\end{tikzpicture}
\end{center}
\vspace{-.5cm}
\caption{An elliptic tail, elliptic bridge, and Weierstrass tail.}\label{F:tails}
\end{figure}

\begin{remark*}
If $(E,q)$ is an elliptic or Weierstrass tail, then $E$ is irreducible. If $(E,q_1,q_2)$ is an elliptic bridge, 
then $E$ is irreducible or $E$ is a union of  two smooth rational curves.
\end{remark*}

Unfortunately, we cannot describe our $\alpha$-stability conditions purely in terms of tails and bridges. As seen in \cite{HH2}, one extra layer of 
combinatorial description is needed, and this is encapsulated in our definition of \emph{chains}.

\begin{definition} [Chains] \label{defn-chains}
An \emph{elliptic chain of length $r$} is a $2$-pointed curve $(E, p_1, p_2)$ which admits a finite, surjective morphism
\[
\gamma \co  \coprod_{i=1}^{r}(E_i,q_{2i-1},q_{2i}) \to (E, p_1, p_2)
\]
such that:
\begin{enumerate}
\item $(E_i,q_{2i-1},q_{2i})$ is an elliptic bridge for $i=1, \ldots, r$.
\item $\gamma$ is an isomorphism when restricted to $E_i \setminus \{q_{2i-1},q_{2i}\}$ for $i=1, \ldots, r$.
\item $\gamma(q_{2i})=\gamma(q_{2i+1})$ is an $A_{3}$-singularity for $i=1, \ldots, r-1$.
\item $\gamma(q_1)=p_1$ and $\gamma(q_{2r})=p_{2}$.
\end{enumerate}
A \emph{Weierstrass chain of length $r$} is a $1$-pointed curve $(E, p)$ which admits a finite, surjective morphism
\[
\gamma \co  \coprod_{i=1}^{r-1}(E_i,q_{2i-1},q_{2i}) \coprod (E_r, q_{2r-1}) \to (E, p)
\]
such that:
\begin{enumerate}
\item $(E_i,q_{2i-1},q_{2i})$ is an elliptic bridge for $i=1, \ldots, r-1$, and $(E_r, q_{2r-1})$ is a Weierstrass tail.
\item $\gamma$ is an isomorphism when restricted to $E_i \setminus \{q_{2i-1},q_{2i}\}$ (for $i=1, \ldots, r-1$) and $E_r \setminus \{q_{2r-1}\}$.
\item $\gamma(q_{2i})=\gamma(q_{2i+1})$ is an $A_{3}$-singularity for $i=1, \ldots, r-1$.
\item $\gamma(q_1)=p$.
\end{enumerate}
An elliptic (resp., Weierstrass) chain of length $1$ is simply an elliptic bridge (resp., Weierstrass tail).
\end{definition}

\begin{figure}[bht]
\begin{center}
\begin{tikzpicture}[scale=.55]
\tikzstyle{blackfill}=[circle,fill=black,draw=Black]
	\begin{pgfonlayer}{nodelayer}
		\coordinate [style=black] (0) at (-5.5-1.7, 0) {};
		\coordinate [style=black] (1) at (-4-1.7, 1.25) {};
		\coordinate [style=black] (2) at (-2.25-1.7, 1.25) {};
		\coordinate [style=black] (3) at (-1-1.7, 1.25) {};
		\coordinate [style=black] (4) at (1.5-1.7+.1, 1.25) {};
		\coordinate [style=black] (5) at (3-1.7+.1, 1.25) {};
		\coordinate [style=black] (6) at (4.75-1.7+.19, 1.25) {};
		\coordinate [style=black] (7) at (6.25-1.7+.19, 0) {};
		\node at (-8,0.7)  {(A)};

		\filldraw [black] (-6.8,-0) circle (3pt);
		\node  at (-6.8,-.45)  {$p_1$};

		\filldraw [black] (4.2,-0) circle (3pt);
		\node  at (4.2,-.45)  {$p_2$};
		
		\node  at (3.3,1.65)  {$1$};
		\node  at (-0.1,1.65)  {$1$};
		\node  at (-2.8,1.65)  {$1$};
		\node  at (-5.8,1.65)  {$1$};
		
		\coordinate [style=black] (B0) at (-5.5-1.7+15.001, 0) {};
		\coordinate [style=black] (B1) at (-4-1.7+15.001, 1.25) {};
		\coordinate [style=black] (B2) at (-2.25-1.7+15.001, 1.25) {};
		\coordinate [style=black] (B3) at (-1-1.7+15.001, 1.25) {};
		\coordinate [style=black] (B4) at (1.5-1.7+.1+15.001, 1.25) {};
		\coordinate [style=black] (B5) at (3-1.7+.1+15.001, 1.25) {};
		\coordinate [style=black] (B6) at (4.75-1.7+.29+15.001, 1.1) {};
		\coordinate [style=black] (B7) at (6.25-1.7+.19+15.001+1, 0) {};
		\node at (-8+15.001,0.7)  {(B)};

		\filldraw [black] (-6.8+15.001,-0) circle (3pt);
		\node  at (-6.8+15.001,-.45)  {$p$};

		\node  at (3.3+15.001,1.6)  {$2$};
		\node  at (-0.1+15.001,1.65)  {$1$};
		\node  at (-2.8+15.001,1.65)  {$1$};
		\node  at (-5.8+15.001,1.65)  {$1$};

		\coordinate [style=black] (x) at (17.5+1, .7) {};
		\coordinate [style=black] (y1) at (19+1, .75+.8) {};
		\coordinate [style=black] (y2) at (17.5+.15+1, .7+.4) {};
		\coordinate [style=black] (y3) at (17.5+.5+1, .7)  {};

		\node at (19+1, 2.4) {\tiny{Weierstrass point}};
		\node at (19+1, 2) {}; 

		\end{pgfonlayer}
	\begin{pgfonlayer}{edgelayer}
		\draw (x) to (y1);
		\draw (x) to (y2);
		\draw (x) to (y3);

		\draw [style=very thick, in=-15, out=0, looseness=1.75] (0) to (1);
		\draw [style=very thick, in=-15, out=195, looseness=10.5] (2) to (3);
		\draw [style=very thick, in=-15, out=195, looseness=9] (4) to (5);
		\draw [style=very thick, in=180, out=210, looseness=2.00] (6) to (7);
		
		\draw [style=very thick, in=-15, out=0, looseness=1.75] (B0) to (B1);
		\draw [style=very thick, in=-15, out=195, looseness=10.5] (B2) to (B3);
		\draw [style=very thick, in=-15, out=195, looseness=9] (B4) to (B5);
		\draw [style=very thick, in=180, out=210, looseness=2.00] (B6) to (B7);
	\end{pgfonlayer}
\end{tikzpicture}
\end{center}
\vspace{-.5cm}
\caption{Curve (A) (resp., (B)) is an elliptic (resp., Weierstrass) chain of length 4.}\label{F:elliptic-chain}
\end{figure}
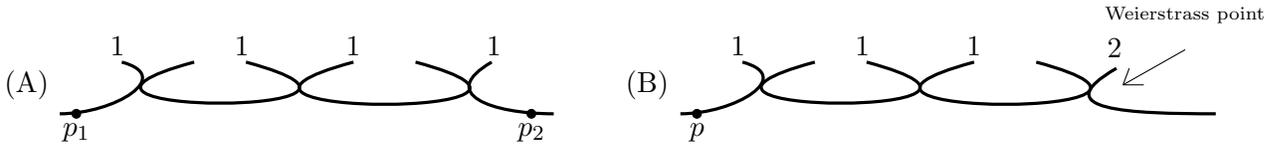

When describing tails and chains as subcurves, 
it is important to specify the singularities along which the tail or chain is 
attached. This motivates the following pair of definitions.

\begin{definition}[Gluing morphism]
A \emph{gluing morphism} $\gamma \co (E, \qm) \rightarrow (C, \pn)$ between two pointed curves 
is a finite morphism $E \rightarrow C$, which is an open immersion when restricted to $E-\{q_1, \ldots, q_m\}$. 
We do not require the points $\{\gamma(q_i)\}_{i=1}^{m}$ to be distinct. 
\end{definition}
\begin{definition}[Tails and Chains with Attaching Data]\label{D:Attaching} Let $(C, \pn)$ be an $n$-pointed curve. We say that
$(C, \pn)$ has
\begin{enumerate}
\item  \emph{$A_k$-attached elliptic tail}
if there is a gluing morphism $\gamma \co (E, q) \to (C, \pn)$ such that
\begin{enumerate}
\item $(E,q)$ is an elliptic tail.
\item $\gamma(q)$ is an $A_k$-singularity of $C$, or $k=1$ and $\gamma(q)$ is  a marked point.
\end{enumerate}
\item \emph{$A_{k_1}/A_{k_2}$-attached elliptic chain}
if there is a gluing morphism $\gamma \co (E, q_1, q_2) \to (C, \pn)$ such that
\begin{enumerate}
\item $(E,q_1, q_2)$ is an elliptic chain.
\item $\gamma(q_i)$ is an $A_{k_i}$-singularity of $C$, or $k_i=1$ and $\gamma(q_i)$ is  a marked point $(i=1,2)$.
\end{enumerate}
\item \emph{$A_k$-attached Weierstrass chain}
if there is a gluing morphism $\gamma \co (E, q) \to (C, \pn)$ such that
\begin{enumerate}
\item $(E,q)$ is a Weierstrass chain.
\item $\gamma(q)$ is an $A_k$-singularity of $C$, or $k=1$ and $\gamma(q)$ is  a marked point.
\end{enumerate}
\end{enumerate}
Note that this definition entails an essential, systematic abuse of notation: 
when we say that a curve has an $A_1$-attached tail or chain, 
we always allow the $A_1$-attachment points to be marked points. 
\end{definition}

We can now define $\alpha$-stability.
\begin{definition}[$\alpha$-stability] \label{defn-stability}
For $\alpha \in (\third \me, 1]$, we say that an $n$-pointed curve $(C, \pn)$ is \emph{$\alpha$-stable}  
if $\omega_{C}(\Sigma_{i=1}^{n}p_i)$ is ample and:

\begin{enumerate}
\item[] For $\alpha\in (\first,\one)$: $C$ has only $A_1$-singularities.
\item[] For $\alpha=\first$: $C$ has only $A_1,A_2$-singularities.
\item[] For $\alpha\in (\second, \first)$: $C$ has only $A_1,A_2$-singularities, and does not contain:
\begin{itemize}
\item $A_1$-attached elliptic tails.
\end{itemize}
\item[] For $\alpha=\second$:  $C$ has only $A_1, A_2, A_3$-singularities, and does not contain:
\begin{itemize}
	\item $A_1, A_3$-attached elliptic tails.
\end{itemize}
\item[] For $\alpha\in (\third, \second)$: $C$ has only $A_1,A_2,A_3$-singularities, and does not contain:
\begin{itemize}
	\item $A_1, A_3$-attached elliptic tails,
	\item $A_1/A_1$-attached elliptic chains.
\end{itemize}
\item[] For $\alpha=\third$: $C$ has only $A_1, A_2, A_3, A_4$-singularities, and does not contain:
\begin{itemize}
	\item $A_1, A_3, A_4$-attached elliptic tails, 
	\item $A_1/A_1, A_1/A_4, A_4/A_4$-attached elliptic chains.
\end{itemize}
\item[] For $\alpha\in (\thirdme, \third)$: $C$ has only $A_1,A_2,A_3,A_4$-singularities, and does not contain:
\begin{itemize}
	\item $A_1, A_3, A_4$-attached elliptic tails, 
	\item $A_1/A_1, A_1/A_4, A_4/A_4$-attached elliptic chains,
	\item $A_1$-attached Weierstrass chains.
	\end{itemize}
\end{enumerate}

A family of $\alpha$-stable curves is a flat and proper family whose geometric fibers are $\alpha$-stable. 
We let $\SM_{g,n}(\alpha)$ denote the stack of $n$-pointed $\alpha$-stable curves of arithmetic genus $g$.
\end{definition}
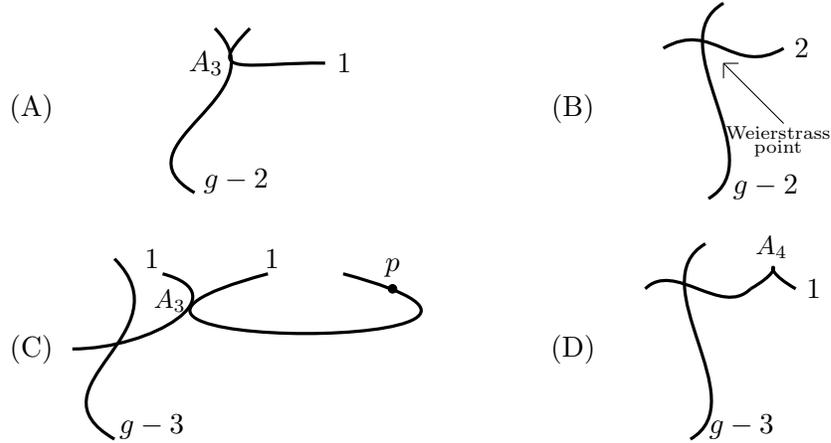
\begin{figure}[hbt]
\begin{center}
\begin{tikzpicture}[scale=0.8]
\tikzstyle{blackfill}=[circle,fill=black,draw=Black]
	\begin{pgfonlayer}{nodelayer}
	
		\node [style=black] (a0) at (-2-4-3.12-1, 1.5) {};
		\node [style=black] (a1) at (-2-4-3.12-1, -1.5) {};
		\node [style=black] (a2) at (-1-.08-4-3.12-1, 1.5) {};
		\node [style=black] (a3) at (.25-.08-4-3.12-1, 0.75) {};
		\node at (-13,0)  {(A)};
		\node at (-11+.9,.75)  {${\small A_3}$};
		\node at (-8+.2,.75)  {$1$};
		\node at (-11+1.4,-1.2)  {$g-2$};

		\coordinate [style=black] (b0) at (-2.5+2-1, 1.75) {};
		\coordinate [style=black] (b1) at (-2.75+2-1, -1.5) {};
		\coordinate [style=black] (b2) at (-3.5+2-1, 1) {};
		\coordinate [style=black] (b3) at (-1.5+2-1, 1) {};
		
		\coordinate [style=black] (a) at (-1-1.5+2-1, 1.25-1+.25) {};
		\coordinate [style=black] (b) at (0-1.5+2-1, 0.75-1) {};
		\coordinate [style=black] (c) at (-0.5-1.5-.25+2-1, 1.75-1) {};
		\coordinate [style=black] (d) at (-1-1.5+2-1, 1.75-1) {};
		
		\node at (-1.65+1,-.4)  {\tiny{  Weierstrass}};
		\node  at (-1.65+1,-.66)  {\tiny{  point}};
		\node at (-4,0)  {(B)};
		\node at (-.2,1)  {$2$};
		\node at (-4+3+.2,-1.3)  {$g-2$};

		\coordinate [style=black] (c0) at (-5.5-1.7-4.12-1, 0-4) {};
		\coordinate [style=black] (c1) at (-4-1.7-4.12-1, 1.25-4) {};
		\coordinate [style=black] (c2) at (-2.25-1.7-4.12-1, 1.25-4) {};
		\coordinate [style=black] (c3) at (-1-1.7-4.12-1, 1.25-4) {};
		\coordinate [style=black] (c4) at (-2-4-2.12-3.5, 1.5-4) {};
		\coordinate [style=black] (c5) at (-2-4-2.12-3.5, -1.5-4) {};

		\node at (-13,-4)  {(C)};
		\node at (-12+1.3,-3.2)  {{\small $A_3$}};
		\node at (-11,-5.3)  {$g-3$};
		\node at (-11,-2.5)  {$1$};
		\node at (-11+2,-2.5)  {$1$};

		\filldraw [black] (-8+1,-3.35+.35) circle (2pt);
		\node at (-8+1,-3.35+.7)  {$p$};
		
		\coordinate [style=black] (d1) at (-2.5+1.7-1, 1.75-4) {};
		\coordinate [style=black] (d2) at (-2.75+1.7-1, -1.5-4) {};
		\coordinate [style=black] (d3) at (-1.8-1, 1-4) {};
		\coordinate [style=black] (d4) at (-.8+1.5-1, 1-4) {};
		\coordinate [style=black] (d5) at (-.8+0.75-1, 1-4) {};

		\node at (-4,-4)  {(D)};
		\node at (0,1-4)  {$1$};
		\node at (-.7,1-4+.7)  {{\small $A_4$}};
		\node at (-1.2,-5.3)  {$g-3$};
	\end{pgfonlayer}
	\begin{pgfonlayer}{edgelayer}
		\draw [style=very thick, in=150, out=-45, looseness=1.50] (a0) to (a1);
		\draw [style=very thick, in=179, out=225, looseness=2.25] (a2) to (a3);
	
		\draw [style=very thick, in=30, out=210] (b0) to (b1);
		\draw [style=very thick, in=210, out=30, looseness=1.25] (b2) to (b3);
		
		\draw [style=thin]  (d) to (b);
		\draw [style=thin] (d) to (a);
		\draw [style=thin]  (c) to (d);
		
		\draw [style=very thick, in=-15, out=0, looseness=1.75] (c0) to (c1);
		\draw [style=very thick, in=-15, out=195, looseness=10.5] (c2) to (c3);
		\draw [style=very thick, in=150, out=-45, looseness=1.50] (c4) to (c5);

		\draw [style=very thick, in=30, out=210] (d1) to (d2);
		\draw [style=very thick,in=225, out=45, looseness=1] (d3) to (d5);
		\draw [style=very thick,in=150, out=30, looseness=3.25] (d5) to (d4);
	\end{pgfonlayer}
\end{tikzpicture}
\end{center}
\caption{Curve (A) has an $A_3$-attached elliptic tail; it is never $\alpha$-stable.  Curve (B) has an $A_1$-attached Weierstrass tail; 
it is $\alpha$-stable for $\alpha \ge \third$.  Curve (C) has an $A_1/A_1$-attached elliptic chain of length $2$; 
it is $\alpha$-stable for $\alpha \ge \second$.  Curve (D) has an $A_1/A_4$-attached elliptic bridge; it is never $\alpha$-stable.}
\label{F:alpha-curves}
\end{figure}

\begin{remark*}
Our definition of an elliptic chain is similar, but not identical to, the definition of an open tacnodal elliptic chain appearing in \cite[Definition 2.4]{HH2}. 
Whereas open tacnodal elliptic chains are built out of arbitrary curves of arithmetic genus one, our elliptic chains are built out of elliptic bridges. 
Nevertheless, it is easy to see that our definition of $(\secondme)$-stability agrees with the definition of h-semistability in \cite[Definition 2.7]{HH2}.
\end{remark*}
		
It will be useful to have a uniform way of referring to the singularities allowed and the subcurves excluded at each stage of the LMMP. 
Thus, for any $\alpha \in (\thirdme, \one]$, we use the term \emph{$\alpha$-stable singularity} to refer to any allowed singularity at the given
value of $\alpha$. For example, a $\secondf$-stable singularity is a node, cusp, or tacnode. Similarly, we use the term 
\emph{$\alpha$-unstable subcurve} to refer to any 
excluded subcurve at the given value of $\alpha$. 
For example, a $\secondf$-unstable subcurve is simply an $A_1$ or $A_3$-attached elliptic tail. 
With this terminology, we may say that a curve is $\alpha$-stable if it has only $\alpha$-stable singularities and no $\alpha$-unstable subcurves. 
Furthermore, if $\alpha_c \in \{\third,\second,\first\}$ is a critical value, we use the term \emph{$\alpha_c$-critical singularity} to refer to the 
newly-allowed singularity at $\alpha=\alpha_c$ and \emph{$\alpha_c$-critical subcurve} to refer to the newly disallowed subcurves at 
$\alpha=\alpha_c - \epsilon$. 
Thus, a $\secondf$-critical singularity is a tacnode, and a $\secondf$-critical subcurve is an elliptic chain with $A_1/A_1$-attaching.

Before plunging into the deformation theory and combinatorics of $\alpha$-stable curves necessary 
to prove Theorem \ref{T:algebraicity} and carry out the VGIT analysis in Section \ref{section-local-vgit},
we take a moment to contemplate on the features of $\alpha$-stability that underlie our arguments
and to give some intuition behind the items of Definition \ref{defn-stability}. The following 
are the properties of $\alpha$-stability that are desired and that we prove to be true for all $\alpha \in (\thirdme, 1]$:
\begin{enumerate}
\item $\alpha$-stability is deformation open.
\item The stack $\bar{\cM}_{g,n}(\alpha)$ of all $\alpha$-stable curves has a good moduli space, and
\item The line bundle $K+\alpha\delta+(1-\alpha)\psi$ on $\bar{\cM}_{g,n}(\alpha)$ descends to an ample line bundle on the good moduli space. 
\end{enumerate}

We will verify (1) in Proposition \ref{P:Openness} (see also Definition \ref{defn-loci}) and 
deduce Theorem \ref{T:algebraicity}. Note that we disallow $A_3$-attached elliptic tails at $\alpha=\second$,
so that $A_1/A_1$-attached elliptic bridges form a closed locus in $\Mg{g,n}(7/10)$.

Existence of good moduli space in (2) requires that the automorphism of every \emph{closed} 
$\alpha$-stable curve 
is reductive. We verify this necessary condition in Proposition \ref{P:reductive-stabilizer},
and turn around to use it as an ingredient in the \emph{proof of existence} for the good moduli space (Corollary \ref{C:local-quotient}
and Theorem \ref{T:vgit-existence}). 

Statement (3) implies that the action of the stabilizer on the fiber of the line bundle $K+\alpha\delta+(1-\alpha)\psi$ at every point is trivial.
As explained in \cite{afs}, this condition places strong restrictions on what curves with $\GG_m$-action
can be $\alpha$-stable. 
For example, at $\alpha=\second$, the fact that a nodally attached $A_{3/2}$-atom (i.e., the tacnodal union of a smooth rational curve with 
a cuspidal rational curve) is disallowed by character considerations provides different heuristics 
for why we disallow $A_3$-attached elliptic tails. 
At $\alpha=\third$,
the fact that a nodally attached $A_{3/4}$-atom (i.e., the tacnodal union of a smooth rational curve with a ramphoid cuspidal rational curve) is disallowed by character considerations 
explains why we disallow $A_1/A_4$-attached elliptic chains (as this $A_{3/4}$-atom is an $A_1/A_4$-attached elliptic bridge). 
\begin{prop}\label{P:reductive-stabilizer}
$\Aut(C,\pn)^\circ$ is a torus for every $\alpha$-stable curve $(C,\pn)$. Consequently, $\Aut(C,\pn)$ is reductive.
\end{prop}
\begin{proof} Automorphisms in $\Aut(C,\pn)^\circ$ do not permute irreducible 
components. Every geometric genus $1$ irreducible component has at least one special (singular or marked) point, 
and every geometric genus $0$ irreducible component has at least two special points. It follows that the only 
irreducible components with a positive dimensional automorphism group are rational curves with two special points,
whose automorphism group is $\GG_m$. The claim follows. 
\end{proof}

\begin{remark*}
We should note that our proof of Proposition \ref{P:reductive-stabilizer} uses features of $\alpha$-stability 
that hold only for $\alpha > 2/3-\epsilon$. We expect that for lower values of $\alpha$, the yet-to-be-defined,
$\alpha$-stability will allow for $\alpha$-stable curves with non-reductive stabilizers. For example, a
curve with an $A_5$-attached $\PP^1$ can have $\GG_a$ as its automorphism group. However, 
we believe that for a correct definition of $\alpha$-stability, it will hold to be true that the stabilizers of 
all \emph{closed} points will be reductive. 
\end{remark*}

\subsection{Deformation openness}  
\label{S:deformation-openness}
Our first main result is the following theorem.
\begin{theorem}\label{T:algebraicity}
For $\alpha \in (\thirdme, 1]$, the stack $\bar{\cM}_{g,n}(\alpha)$ of $\alpha$-stable curves is algebraic and of finite type over $\Spec \CC$. Furthermore, for each critical value $\alpha_c \in \{\third, \second, \first\}$, we have open immersions:
\[
\bar{\cM}_{g,n}(\alpha_c+\epsilon) \hookrightarrow \bar{\cM}_{g,n}(\alpha_c) \hookleftarrow \bar{\cM}_{g,n}(\alpha_c-\epsilon).
\]
\end{theorem}

Let $\U_{g,n}$  be the stack of flat, proper families  of curves $(\pi\co \C \rightarrow T, \sigman)$, 
where the sections $\sigman$ are distinct and lie in the smooth locus of $\pi$, 
$\omega_{\C/T}(\Sigma_{i=1}^{n} \sigma_i)$ is relatively ample, and the 
geometric fibers 
of $\pi$ are $n$-pointed curves of arithmetic genus $g$ with only $A$-singularities.  Since $\U_{g,n}$ parameterizes canonically polarized 
curves, $\U_{g,n}$ is algebraic and finite type over 
$\mathbb{C}$. Let $\U_{g,n}(A_\ell) \subset \U_{g,n}$ be the open substack parameterizing curves with at worst $A_1, \ldots, A_\ell$ 
singularities. 
We will show that each $\SM_{g,n}(\alpha)$ can be obtained from a suitable $\U_{g,n}(A_\ell)$ by excising a finite collection of closed substacks.
As a result, we obtain a proof of Theorem \ref{T:algebraicity}.
 
\begin{defn} \label{defn-loci}
We let $\T^{A_k}, \B^{A_{k_1}/A_{k_2}}, \W^{A_k}$ denote the following constructible subsets of $\U_{g,n}$:
\begin{align*}
\T^{A_k}:=& \text{ Locus of curves containing an $A_{k}$-attached elliptic tail}.\\
\B^{A_{k_1}/A_{k_2}}:=& \text{ Locus of curves containing an $A_{k_1}/A_{k_2}$-attached elliptic chain}.\\
\W^{A_k}:=& \text{ Locus of curves containing an $A_{k}$-attached Weierstrass chain}.
\end{align*}
\end{defn}
With this  notation, we can describe our stability conditions (set-theoretically) as follows:
\begin{align*}
\SM_{g,n}(\firstpe)&=\U_{g,n}(A_1)\vspace{0.6cm}\\
 \SM_{g,n}(\first)&=\U_{g,n}(A_2)\\
  \SM_{g,n}(\firstme)&=\SM_{g,n}(\first)-\T^{A_1}\vspace{0.6cm}\\
 \SM_{g,n}(\second)&=\U_{g,n}(A_3)-\bigcup_{i \in \{1,3\}} \T^{A_i}\\
 \SM_{g,n}(\secondme)&=\SM_{g,n}(\second)- \B^{A_1/A_1}\vspace{0.6cm}\\
 \SM_{g,n}(\third)&=\U_{g,n}(A_4)-\bigcup_{i \in \{1,3,4\}} \T^{A_i}- \bigcup_{i,j \in \{1,4\}} \B^{A_i/A_j}\\
 \SM_{g,n}(\thirdme)&=\SM_{g,n}(\third)-\W^{A_1}
\end{align*}
Here, when we write $\SM_{g,n}(\first)-\T^{A_1}$, we mean of course $\SM_{g,n}(\first)-\bigl(\T^{A_1} \cap \SM_{g,n}(\first)\bigr)$, and similarly for each of the 
subsequent set-theoretic subtractions.

We must show that at each stage the collection of loci 
$\T^{A_k}$, $\B^{A_{k_1}/A_{k_2}}$, and $\W^{A_k}$ that we excise is closed. 
We break this analysis into two steps:
In Corollaries \ref{C:Attaching1} and \ref{C:Attaching2}, 
we analyze how the attaching singularities of an $\alpha$-unstable subcurve degenerate, 
and in Lemmas \ref{L:HmLimits} and \ref{L:LimitChain}, we analyze degenerations of $\alpha$-unstable curves.
We combine these results to prove the desired statement in Proposition \ref{P:Openness}.

\begin{definition}[Inner/Outer Singularities]
\label{D:inner-outer} 
We say that an $A_k$-singularity $p \in C$ is \emph{outer} if it lies on two distinct irreducible components of $C$, 
and \emph{inner} if it lies on a single irreducible component. (N.B. If $k$ is even, then any $A_k$-singularity is necessarily inner.)
\end{definition}

Suppose $\C \ra \Delta$ is a family of curves with at worst $A$-singularities, 
where $\Delta$ is the spectrum of a DVR.
Denote by $C_{\bar{\eta}}$ the geometric generic fiber and by $C_0$ the central fiber. 
We are interested in how the singularities of $C_{\bar{\eta}}$ degenerate in $C_0$. 
By deformation theory, an $A_{k}$-singularity can deform to a collection 
of $\{A_{k_1}, \dots, A_{k_r}\}$ singularities if and only if $\sum_{i=1}^{r} (k_i+1) \leq k+1$. In the following proposition, we refine this result for 
outer singularities. 

\begin{prop}\label{P:limits-outer}
Let $p \in C_0$ be an $A_{m}$-singularity, and suppose that $p$ is the limit of an outer singularity $q \in C_{\bar{\eta}}$. 
Then $p$ is outer (in particular, $m$ is odd) and each singularity of $C_{\bar{\eta}}$ that approaches $p$ must be outer and must lie on the 
same two irreducible components of $C_{\bar{\eta}}$ as $q$. Moreover, the collection of singularities approaching $p$ is necessarily
of the form $\{A_{2k_1+1}, A_{2k_2+1}, \dots, A_{2k_r+1}\}$, where $\sum_{i=1}^r (2k_i+2)=m+1$, and there exists a simultaneous normalization of the family $\C \ra \Delta$ along this set of generic singularities.
\end{prop}
\begin{proof}
Suppose $q$ is an $A_{2k_1+1}$-singularity. We may take the local equation of $\C$ around $p$ to be 
\[
y^2=(x-a_1(t))^{2k_1+2} \prod_{i=2}^{r} (x-a_i(t))^{m_i}, \ \text{where $2k_1+2+\sum_{i=2}^{r} m_i=m+1$}.
\]
By assumption, the general fiber of this family has at least two irreducible components, and it follows that each $m_i$ must be even. Thus, we can rewrite the above equation as
\begin{equation}\label{E:family-outer}
y^2=\prod_{i=1}^{r} (x-a_i(t))^{2k_i+2},
\end{equation}
where $k_1, k_2, \dots, k_r$ satisfy $\sum_{i=1}^r (2k_i+2)=m+1$. It now follows by inspection that 
$C_{\bar{\eta}}$ contains outer singularities $\{A_{2k_1+1}, A_{2k_2+1}, \dots, A_{2k_r+1}\}$ joining the same
two irreducible components of $C_{\bar{\eta}}$ and approaching $p \in C_0$.
Clearly, the normalization of the family \eqref{E:family-outer} exists and is a union of two smooth families over $\Delta$. 
\end{proof}

Using the previous proposition, we can understand how the attaching singularities of a subcurve may degenerate.

\begin{corollary}\label{C:Attaching1}
Let $(\pi\co \C \rightarrow \Delta, \sigman)$ be a family of curves in $\U_{g,n}$. 
Suppose that $\tau$ is a section of $\pi$ such that $\tau(\bar{\eta}) \in \C_{\bar{\eta}}$ is a disconnecting $A_{2k+1}$-singularity of the geometric 
generic fiber. 
Then $\tau(0) \in C_0$ is also a disconnecting $A_{2k+1}$-singularity.
\end{corollary}

\begin{proof} A disconnecting singularity $\tau(\bar\eta)$ is outer and joins two irreducible components
which do not meet elsewhere. By Proposition \ref{P:limits-outer},
 $\tau(\bar{\eta})$ cannot collide with other singularities of $\C_{\bar{\eta}}$ in the special fiber and so must remain $A_{2k-1}$-singularity.
The normalization of $\C$ along $\tau$ separates $\C$ into two connected components, so $\tau(0)$ is disconnecting. 
\end{proof}

\begin{corollary}\label{C:Attaching2}
Let $(\pi\co \C \rightarrow \Delta, \sigman)$ be a family of curves in $\U_{g,n}$. Suppose that $\tau_1$, $\tau_2$ are sections of $\pi$ such that $\tau_1(\bar{\eta}), \tau_2(\bar{\eta}) \in \C_{\bar{\eta}}$ are $A_{2k_1+1}$ and $A_{2k_2+1}$-singularities of the 
geometric generic fiber. Suppose also that the normalization of $\C_{\bar{\eta}}$ along $\tau_1(\bar{\eta}) \cup \tau_2(\bar{\eta})$ consists of two connected components, while the normalization of $\C_{\bar{\eta}}$ along either $\tau_1(\bar{\eta})$ or $\tau_2(\bar{\eta})$ 
individually is connected. Then we have two possible cases for the limits $\tau_1(0)$ and $\tau_2(0)$:
\begin{enumerate}
\item $\tau_1(0)$ and $\tau_2(0)$ are distinct $A_{2k_1+1}$ and $A_{2k_2+1}$-singularities, respectively, or
\item $\tau_1(0)=\tau_2(0)$ is an $A_{2k_1+2k_2+3}$-singularity.
\end{enumerate}
\end{corollary}
\begin{proof}
Our assumptions imply that the singularities $\tau_1(\bar{\eta})$ and $\tau_2(\bar{\eta})$ are outer and are the only two singularities connecting the two connected components of the normalization of $\C_{\bar{\eta}}$ along $\tau_1(\bar{\eta}) \cup \tau_2(\bar{\eta})$. By Proposition \ref{P:limits-outer}, these two singularities cannot collide with any additional singularities of $\C_{\bar{\eta}}$ in the special fiber. If  $\tau_1(\bar{\eta})$ and $\tau_2(\bar{\eta})$ themselves do not collide, we have case (1). If they do collide, then, applying 
Proposition \ref{P:limits-outer} once more, we have case (2).
\end{proof}

\begin{lemma}[Limits of tails and bridges]\label{L:HmLimits}
\hfill
\begin{itemize}
\item[(1)]Let $(\H \rightarrow \Delta, \tau_1)$ be a family in $\U_{1,1}$ whose generic fiber is an elliptic tail. Then the special fiber $(H,p)$ is an elliptic tail.
\item[(2)]Let $(\H \rightarrow \Delta, \tau_1, \tau_2)$ be a family in $\U_{1,2}$ whose generic fiber is an elliptic bridge. Then the special fiber $(H,p_1,p_2)$ satisfies one of the following conditions:
\begin{itemize}
\item[(a)] $(H,p_1,p_2)$ is an elliptic bridge.
\item[(b)] $(H,p_1,p_2)$ contains an $A_1$-attached elliptic tail.
\end{itemize}
\item[(3)]Let $(\H \rightarrow \Delta, \tau_1)$ be a family in $\U_{2,1}$ whose generic fiber is a Weierstrass tail. Then the special fiber $(H,p)$ satisfies one of the following conditions:
\begin{itemize}
\item[(a)] $(H,p)$ is a Weierstrass tail.
\item[(b)] $(H,p)$ contains an $A_1$ or $A_3$-attached elliptic tail, or an $A_1/A_1$-attached elliptic bridge.
\end{itemize}
\end{itemize}
\end{lemma}
\begin{proof}

(1) For every $(H,p)\in \U_{1,1}$, the curve $H$ is irreducible, and $|2p|$ defines a degree $2$ map to $\P^1$ by Riemann-Roch. Hence $\U_{1,1}=\T^{A_1}$. 

For (2), the special fiber $(H,p_1, p_2)$ is a curve of arithmetic genus $1$ with  $\omega_{H}(p_1+p_2)$ ample. Since  $\omega_{H}(p_1+p_2)$ has degree $2$, $H$ has at most $2$ irreducible components. 
The possible topological types of $H$ are listed in the top row of Figure \ref{F:TopologicalTypes}. We see immediately that any curve with one of the first three topological types is an elliptic bridge, 
while any curve with the last topological type contains an $A_1$-attached elliptic tail.

Finally, for (3), the special fiber $(H,p)$ is a curve of arithmetic genus $2$ with $\omega_{H}(p)$ ample and $h^0(\omega_{H}(-2p))\geq 1$ by semicontinuity. 
Since  $\omega_{H}(p)$ has degree three, $H$ has at most three components, and the possible topological types of $H$ are listed in the bottom three rows of Figure \ref{F:TopologicalTypes}. 
One sees immediately that if $H$ does not contain an $A_1$ or $A_3$-attached elliptic tail or an $A_1/A_1$-attached elliptic bridge, 
there are only three possibilities for the topological type of $H$: either $H$ is irreducible or $H$ has topological type $(A)$ or $(B)$. 
However, topological types (A) and (B) do not satisfy $h^0(\omega_{H}(-2p))\geq 1$. Finally, if $(H,p)$ is irreducible, then it must be a Weierstrass tail. Indeed, the linear equivalence $\omega_{H} \sim 2p$ follows immediately from the corresponding linear equivalence on the general fiber. 
\end{proof}

\begin{figure}
\scalebox{.55}{\includegraphics{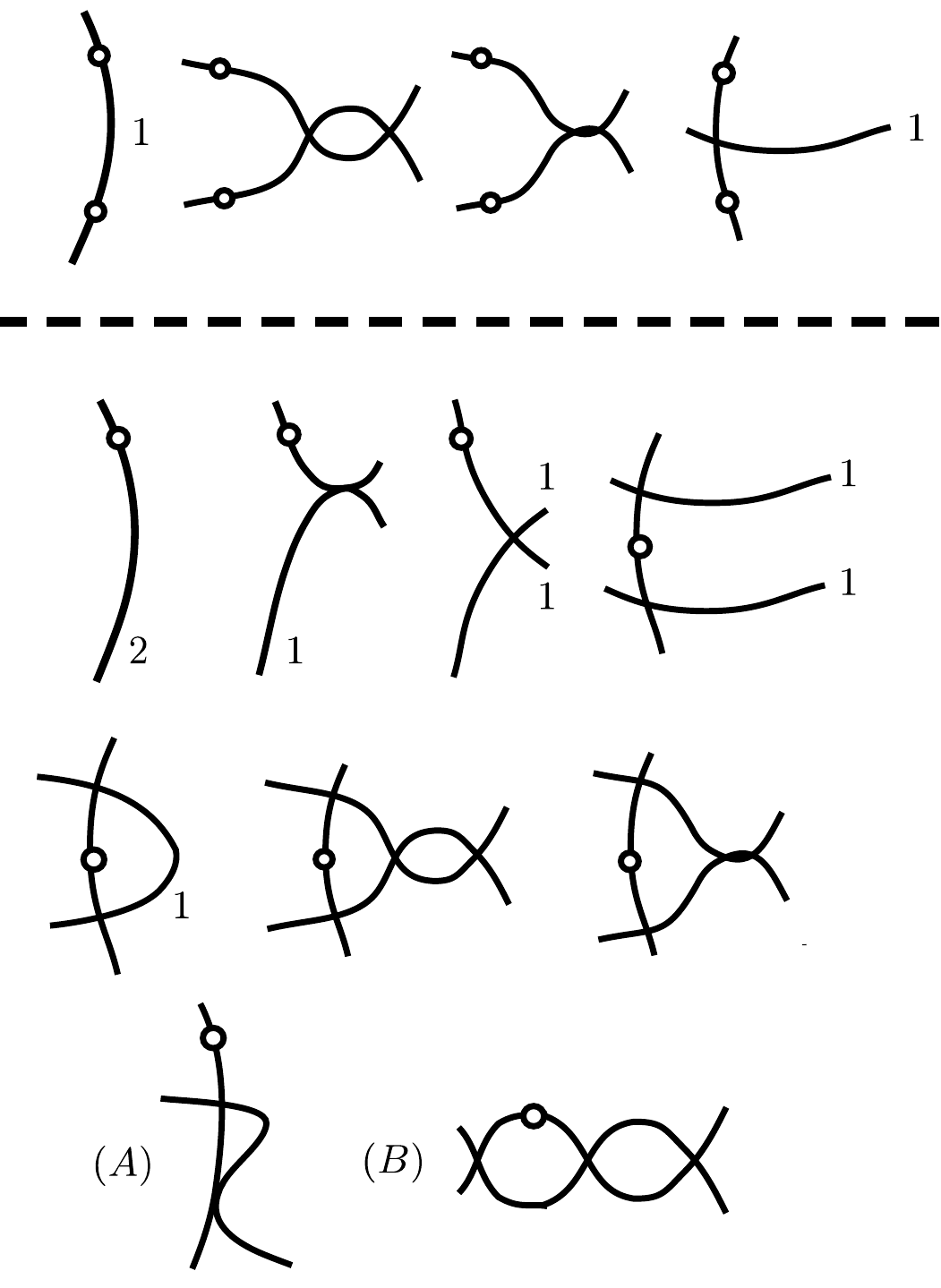}}
\caption{Topological types of curves in $\U_{1,2}(A_\infty)$ and $\U_{2,1}(A_\infty)$. For convenience, we have suppressed the data of 
inner singularities, and we record only the arithmetic genus of each component and the outer singularities
(which are either nodes or tacnodes, as indicated by the picture). Components without a label have arithmetic genus zero.}\label{F:TopologicalTypes} \label{F:TopTypes}
\end{figure}

\begin{lemma}[Limits of elliptic chains]\label{L:LimitChain}
Let $(\H \rightarrow \Delta, \tau_1, \tau_2)$ be a family in $\U_{2r-1,2}$ whose generic fiber is an elliptic chain of length $r$. Then the special fiber $(H,p_1,p_2)$ satisfies one of the following conditions:
\begin{itemize}
\item[(a)] $(H,p_1, p_2)$ contains an $A_1/A_1$-attached elliptic chain of length $\leq r$.
\item[(b)] $(H,p_1,p_2)$ contains an $A_1$-attached elliptic tail.
\end{itemize}
\end{lemma}
\begin{proof}
We will assume $(H,p_1,p_2)$ contains no $A_1$-attached elliptic tails, and prove that (a) holds. 
By Lemma \ref{L:HmLimits}, this assumption implies that if $(E,q_1,q_2)$ is a genus one subcurve of $H$, nodally attached at $q_1$ and 
$q_2$, and $\omega_{E}(q_1+q_2)$ is ample on $E$, then $(E, q_1, q_2)$ is an $A_1/A_1$-attached elliptic bridge.

To begin, let $\gamma_1, \ldots, \gamma_{r-1}$ be sections picking out the tacnodes in the general fiber at which the sequence of elliptic 
bridges are attached to each other. By Corollary \ref{C:Attaching1}, 
the limits $\gamma_1(0), \ldots, \gamma_{r-1}(0)$ remain tacnodes, so the normalization of $\phi\co \tilde{\H} \rightarrow \H$ 
along $\gamma_1, \ldots, \gamma_{r-1}$ is well-defined 
and we obtain $r$ flat families of $2$-pointed curves of arithmetic genus $1$, i.e. we have
\[
\tilde{\H}=\coprod_{i=1}^{r}(\E_i, \sigma_{2i-1}, \sigma_{2i}),
\]
where $\sigma_1:=\tau_1$, $\sigma_{2r}:=\tau_2$, and $\phi^{-1}(\gamma_i)=\{\sigma_{2i}, \sigma_{2i+1}\}$. 
The relative ampleness of $\omega_{\H/\Delta}(\tau_1+\tau_2)$ implies
\begin{enumerate}
\item $\omega_{E_1}(p_1+2p_2), \omega_{E_r}(2p_{2r-1}+p_{2r})$ ample on $E_1$, $E_r$ respectively.
\item $\omega_{E_i}(2p_{2i-1}+2p_{2i})$ ample on $E_{i}$ for $i=2, \ldots, r-1$.
\end{enumerate}
It follows that for each $1\leq i\leq r$, 
either $(E_{i}, p_{2i-1}, p_{2i})$ is an elliptic bridge or one of the following must hold:
\begin{itemize}
\item[(a)]$(E_{i}, p_{2i-1}, p_{2i})=(\PP^1,p_{2i-1}, q_{2i-1}') \cup (E_{i}',q_{2i-1},p_{2i})/(q_{2i-1}' \sim q_{2i-1})$, 
where  $(E_{i}',q_{2i-1},p_{2i})$ is an elliptic bridge. 
\item[(b)] $(E_{i},p_{2i-1},p_{2i})=(E_{i}',p_{2i-1}, q_{2i}) \cup (\PP^1,q_{2i}',p_{2i} )/(q_{2i} \sim q_{2i}')$, 
where  $(E_{i}',q_{2i-1},p_{2i})$ is an elliptic bridge. 
\item[(c)] $(E_{i},p_{2i-1},p_{2i})=(\PP^1,p_{2i-1}, q_{2i-1}') \cup (E_{i}',q_{2i-1}, q_{2i}) \cup (\PP^1,q_{2i}',p_{2i} )/
(q_{2i-1}' \sim q_{2i-1}, q_{2i} \sim q_{2i}')$, 
where  $(E_{i}',q_{2i-1},p_{2i})$ is an elliptic bridge. 
\end{itemize}
In the cases $(a), (b), (c)$ respectively, we say that  
$E_i$ sprouts on the \emph{left}, \emph{right}, or \emph{left and right}. Note that if $E_1$ or $E_r$ sprouts at all, then $E_1$ or $E_r$ contains an $A_1/A_1$-attached elliptic bridge. Similarly, if $E_i$ sprouts on both the left and right $(2 \leq i \leq r-1)$, then $E_i$ contains an $A_1/A_1$-attached elliptic bridge. Thus, we may assume without loss of generality that $E_1$ and $E_r$ do not sprout and that $E_i$ $(2 \leq i \leq r-1)$ sprouts on the left or right, but not both. We now observe that any collection $\{E_s, \dots, E_{s+t}\}$ such that $E_s$ sprouts on the left (or $s=0$), $E_{s+t}$ sprouts on the
right (or $s+t=r$), and $E_k$ does not sprout for $s<k<s+t$, contains an $A_1/A_1$-attached elliptic chain.
\end{proof}

\begin{lemma}[Limits of Weierstrass chains]\label{L:LimitWChain}
Let $(\H \rightarrow \Delta, \tau)$ be a family in $\U_{2r,1}$ whose generic fiber is a Weierstrass chain of length $r$. Then the special fiber satisfies one of the following conditions:
\begin{itemize}
\item[(a)] $(H,p)$ contains an $A_1$-attached Weirstrass chain of length $\leq r$
\item[(b)] $(H,p)$ contains an $A_1/A_1$-attached elliptic chain of length $< r$.
\item[(c)] $(H,p)$ contains an $A_1$ or $A_3$-attached elliptic tail.
\end{itemize}
\end{lemma}

\begin{proof} 
As in the proof of Lemma \ref{L:LimitChain}, let $\gamma_1, \ldots, \gamma_{r-1}$ be sections picking out the attaching tacnodes in the general fiber. By Corollary \ref{C:Attaching1}, 
the limits $\gamma_1(0), \ldots, \gamma_{r-1}(0)$ remain tacnodes, so the normalization $\phi\co \tilde{\H} \rightarrow \H$ 
along $\gamma_1, \ldots, \gamma_{r-1}$ is well-defined. 
We obtain $r-1$ families of $2$-pointed curves of arithmetic genus $1$ and a single family of $1$-pointed curves of genus 2: 
\[
\tilde{\H}=\coprod_{i=1}^{r-1}(\E_i, \sigma_{2i-1}, \sigma_{2i}) \coprod (\E_r, \sigma_{2r-1})
\]
where $\sigma_1:=\tau$ and $\phi^{-1}(\gamma_i)=\{\sigma_{2i}, \sigma_{2i+1}\}$. 

As in the proof of Lemma \ref{L:LimitChain}, we must consider the possibility that some $E_i$'s sprout in the special fiber. 
If $E_r$ sprouts on the left, then $E_r$ itself contains a Weierstrass tail, so we may assume that this does not happen. 
Now let $s<r$ be maximal such that $E_s$ sprouts. If $E_{s}$ sprouts on the left, then $E_{s} \cup E_{s+1} \cup \ldots \cup E_{r}$  
gives a Weierstrass chain in the special fiber. If $E_{s}$ sprouts on the right, then arguing as in Lemma \ref{L:LimitChain} 
produces an $A_1/A_1$-attached elliptic chain in $E_1 \cup \ldots \cup E_{s}$.
\end{proof}

\begin{proposition}\label{P:Openness} 
\hfill
\begin{enumerate}
\item $\T^{A_1}\cup \T^{A_m}$ is closed in $\U_{g,n}$ for any odd $m$.
\item $\B^{A_{1}/A_{1}}$ is closed in $\U_{g,n}-\bigcup_{i \in \{1,3\}}\T^{A_i}$.
\item $\B^{A_m/A_m}$ and $\B^{A_{1}/A_{m}}$ are closed in $\U_{g,n}(A_{m})-\T^{A_1}-\B^{A_1/A_1}$ for any even $m$.
\item $\W^{A_m}$ is closed in $\U_{g,n}(A_{m})-\bigcup_{i \in \{1,3\}}\T^{A_i}-\B^{A_1/A_1}$ for any odd $m$.
\end{enumerate}
\end{proposition}
\begin{proof}
The given loci are obviously constructible, so it suffices to show that they are closed under specialization.

For (1), let $(\C\to \Delta, \sigman)$ be a family in $\U_{g,n}$ whose generic fiber lies in $\T^{A_{2k+1}}$. 
Possibly after a finite base-change, let $\tau$ be the section picking out the attaching $A_{2k+1}$-singularity of the elliptic tail in the generic fiber. 
By Corollary \ref{C:Attaching1}, the limit $\tau(0)$ is also $A_{2k+1}$-singularity. 
Consider the normalization $\tilde{C} \rightarrow \C$ along $\tau$. 
Let $\H \subset \tilde{\C}$ be the component whose generic fiber is an elliptic tail and let $\alpha$ be the preimage of $\tau$ on $\H$. 
Then $\omega_{\H}((k+1)\alpha)$ is relatively ample. We conclude that either $\omega_{H_0}(\alpha(0))$ is ample, 
or $\alpha(0)$ lies on a rational curve attached nodally to the rest of $H_0$. In the former case, 
$(H_0, \alpha(0))$ is an elliptic tail by Lemma \ref{L:HmLimits}, so $C_0$ contains an elliptic tail with $A_{2k+1}$-attaching, as desired. 
In the latter case, $H_0$ contains an $A_1$-attached elliptic tail.
We conclude that $C_0\in  \T^{A_1}\cup \T^{A_{2k+1}}$, as desired.

For (2), let $(\C \to \Delta, \sigman)$ be a family in $\U_{g,n}$ whose generic fiber lies in $\B^{A_{1}/A_{1}}$. Possibly after a finite base change, let $\tau_1$, $\tau_2$ be the sections picking out the attaching nodes of an elliptic chain in the general fiber. 
By Proposition \ref{P:limits-outer}, $\tau_1(0)$ and $\tau_2(0)$ either remain 
nodes, or, if the elliptic chain has length $1$, can coalesce to form an outer $A_3$-singularity.
In either case there exists a normalization of $\C$ along $\tau_1$ and $\tau_2$. 
Since $C_{\bar\eta}$ becomes separated after normalizing along $\tau_1$ and $\tau_2$, we conclude that the limit of the elliptic chain is an 
connected component of $C_0$ attached either along two nodes, or, when $r=1$, 
along a separating $A_3$-singularity.
In the former case, $C_0$ has an elliptic chain by Lemma \ref{L:LimitChain}. In the latter case, $C_0$ has arithmetic genus $1$ connected 
component $A_3$-attached to the rest of the curve, so that $C_0 \in \T^{A_1}\cup \T^{A_3}$.

The proof of (3) is essentially identical to (2), making use of the observation that in $\U_{g,n}(A_k)$, the limit of an $A_k$-singularity must be an $A_k$-singularity.  The proof of (4) is essentially identical to (1), using Lemma \ref{L:LimitWChain} in place of Lemma \ref{L:HmLimits}.
\end{proof}

\begin{proof}[Proof of Theorem \ref{T:algebraicity}]
For $\alpha_c=\first, \second$, and $\third$,
Proposition \ref{P:Openness} implies that $\SM_{g,n}(\alpha_c)$ is obtained by excising closed substacks from 
$\U_{g,n}(A_2)$, $\U_{g,n}(A_3)$, and $\U_{g,n}(A_4)$, respectively.
Because
\[
\Mg{g,n}(\alpha_c+\epsilon)=\Mg{g,n}(\alpha_c)\setminus \overline{\{\text{the locus of curves with $\alpha_c$-critical singularities}\}},
\]
we conclude that $\Mg{g,n}(\alpha_c+\epsilon) \hookrightarrow \Mg{g,n}(\alpha_c)$ is an open immersion.
Finally, applying Proposition \ref{P:Openness} once more, we see that 
$\SM_{g,n}(\alpha_c-\epsilon)$ is obtained by excising closed substacks from 
$\Mg{g,n}(\alpha_c)$. 
\end{proof}

\subsection{Properties of $\alpha$-stability}
In this section, we record several elementary properties of $\alpha$-stability that will be needed in subsequent arguments. 
Recall that if $(C,\pn)$ is a Deligne-Mumford stable curve and $q \in C$ is a node, then the pointed normalization 
$(\tilde{C}, \pn, q_1, q_2)$ of $C$ at $q$ is Deligne-Mumford stable. The same statement holds for $\alpha$-stable curves.

\begin{lemma}\label{L:Norm}
Suppose $(C,\pn)$ is an $\alpha$-stable curve and $q \in C$ is a node. 
Then the pointed normalization $(\tilde{C}, \pn, q_1, q_2)$ of $C$ at $q$ is $\alpha$-stable.
\end{lemma}
\begin{proof} Follows immediately from the definition of $\alpha$-stability.
\end{proof}

Unfortunately, the converse of Lemma \ref{L:Norm} is false. Nodally gluing two marked points of an $\alpha$-stable curve may fail to preserve 
$\alpha$-stability if the two marked points are both on the same component, or both on rational components -- see Figure \ref{figure-gluing}.  
The following lemma says that these are the only problems that can arise.

\begin{lemma}\label{L:Glue}
\hfill
\begin{enumerate}
\item If $(\tilde{C}_1, \pn, q_1)$ and  $(\tilde{C}_2, \pn, q_2)$ are $\alpha$-stable curves, then 
\[
(\tilde{C}_1, \pn, q_1) \cup (\tilde{C}_2, \pn, q_2)/ (q_1 \sim q_2)
\] is $\alpha$-stable.
\item If $(\tilde{C}, \pn, q_1, q_2)$ is an $\alpha$-stable curve, then 
\[
(\tilde{C}, \pn, q_1, q_2)/ (q_1 \sim q_2)
\] is $\alpha$-stable provided one of the following conditions hold:
\begin{itemize}
\item $q_1$ and $q_2$ lie on disjoint irreducible components of $\tilde{C}$,
\item $q_1$ and $q_2$ lie on distinct irreducible components of $\tilde{C}$, and at least one of these components is not a smooth rational curve.
\end{itemize}
\end{enumerate}
\end{lemma}
\begin{figure}[htb] 
\begin{center}
\begin{tikzpicture}[scale=.45]
\tikzstyle{blackfill}=[circle,fill=black,draw=Black]
    \begin{pgfonlayer}{nodelayer}
        \coordinate (AL0) at (-8.12-5.111, 4) {};
        \coordinate (AL1) at (-8.12-5.111, -2) {};
        \coordinate (AL2) at (-.5-8.12-5.111, 2.2) {};
        \coordinate (AL3) at (3.5-8.12-5.111, 4.25) {};

        \coordinate (AR0) at (-8.12+8-5.111-1.5, 4) {};
        \coordinate (AR1) at (-8.12+8-5.111-1.5, -2) {};
        \coordinate (AR2) at (-0.5-5.111-1.5, 2.25) {};
        \coordinate (AR3) at (4.25-5.111-1.5, 4.25) {};
   
        \coordinate (Aa0) at (2-5.022-5.111-.5, 0+1.2) {};
        \coordinate (Aa1) at (1.5-5.022-5.111-.5, 0.5+1.2) {};
        \coordinate (Aa2) at (1.5-5.022-5.111-.5, -0.5+1.2) {};
        \coordinate (Aa3) at (0-5.022-5.111-.5, 0.25+1.2) {};
        \coordinate (Aa4) at (0-5.022-5.111-.5, -0.25+1.2) {};
        \coordinate (Aa5) at (0-5.022-5.111-.5, 0+1.2) {};
        \node at (-10.5-5.111,1.2)  {(A)};

        \filldraw  (-6.4-5.111,3.4) circle (3pt);       
        \filldraw  (-5.4-5.111, 3.9) circle (3pt);
        \node [font=\small] at (-6.2-5.111,2.7) {$q_{1}$};
        \node [font=\small] at  (-5.2-5.111, 3.2)  {$q_{2}$};

        \coordinate (l0) at (-3+5.1112, 3) {};
        \coordinate (l1) at (-4+5.1112, -2) {};
        \coordinate (l2) at (3-4.44+.2+5.1112, 3) {};
        \coordinate (l3) at (4-4.44+.5+5.1112, -2) {};
        \coordinate (l4) at (-4+.3+5.1112, .25+.6) {};
        \coordinate (l5) at (-1+5.1112, 5.3) {};
        \coordinate (l6) at (1-4.44+5.1112, 5.3) {};
        \coordinate (l7) at (4-4.44+5.1112, .25+.6) {};
       
        \coordinate (a0) at (2+6, 0+0.5+1.2) {};
        \coordinate (a1) at (1.5+6, 0.5+0.5+1.2) {};
        \coordinate (a2) at (1.5+6, -0.5+0.5+1.2) {};
        \coordinate (a3) at (0+6, 0.25+0.5+1.2) {};
        \coordinate (a4) at (0+6, -0.25+0.5+1.2) {};
        \coordinate (a5) at (0+6, 0+0.5+1.2) {};
       
        \coordinate (r0) at (-1+11, 3) {};
        \coordinate (r1) at (-2+11, -2) {};
        \coordinate (r2) at (1+11, 3) {};
        \coordinate (r3) at (2+11, -2) {};
        \coordinate (r4) at (-1.75+11, 1.5) {};
        \coordinate (r5) at (-1+11, 5) {};
        \coordinate (r6) at (1+11, 5) {};
        \coordinate (r7) at (1.75+11, 1.5) {};
       
        \filldraw (-2.5+.8+5.1112,2.5+.8+.85) circle (3pt);
        \filldraw (2.5-.8-4.44+.1+5.1112,2.5+.8+.85) circle (3pt);

        \node [font=\small] at  (-2.5+.8+.5+.07+5.1112,2.5+.8-.5+.85+.3)  {$q_{2}$};
        \node [font=\small] at  (2.5-.8-.5-4.44+.03+5.1112, 2.5+.8-.5+.85+.3)  {$q_{1}$};
        \node at (-.5,1.7) {(B)};
    \end{pgfonlayer}
    \begin{pgfonlayer}{edgelayer}
        \draw [style=very thick,in=-60, out=135, looseness=1.25] (AL1) to (AL0);
        \draw [style=very thick,](AL2) to (AL3);   
        \draw (Aa4) to (Aa3);
        \draw (Aa5) to (Aa0);
        \draw (Aa0) to (Aa1);
        \draw (Aa0) to (Aa2);
       
        \draw [style=very thick,in=-60, out=135, looseness=1.25] (AR1) to (AR0);
        \draw [style=very thick,bend left=15, looseness=3.75] (AR2) to (AR3);

        \draw [style=very thick,in=289, out=105, looseness=0.75] (l1) to (l0);
        \draw [style=very thick] (l6) to (l7);
        \draw [style=very thick,in=75, out=255] (l2) to (l3);
        \draw [style=very thick] (l4) to (l5);

        \draw (a4) to (a3);
        \draw (a5) to (a0);
        \draw (a0) to (a1);
        \draw (a0) to (a2);       
       
        \draw [style=very thick,in=289, out=105, looseness=0.75] (r1) to (r0);
        \draw [style=very thick,bend right=60, looseness=1.75] (r6) to (r7);
        \draw [style=very thick,in=75, out=255] (r2) to (r3);
        \draw [style=very thick,bend right=60, looseness=1.75] (r4) to (r5);
       
    \end{pgfonlayer}
\end{tikzpicture}
\end{center}
\caption{In (A), two marked points on a genus $0$ tail (resp., two conjugate points on an elliptic tail) are glued to yield an elliptic tail (resp., a Weierstrass tail).  In (B), two marked points on distinct rational components are glued to yield an elliptic bridge.}
\label{figure-gluing}
\end{figure}
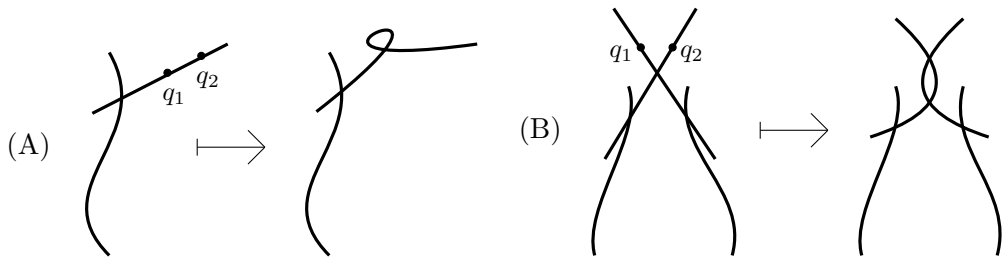

\begin{proof}
Let $C:=(\tilde{C}, q_1, q_2)/ (q_1 \sim q_2)$, and let $\phi\co \tilde{C} \rightarrow C$ be the gluing morphism which identifies $q_1, q_2$ to a
node $q \in C$. It suffices to show that if $E \subset C$ is an $\alpha$-unstable curve, then $\phi^{-1}(E)$ is an $\alpha$-unstable subcurve of 
$\tilde{C}$. The key observation is that any $\alpha$-unstable subcurve $E$ has the following property: If $E_1, E_2 \subset E$ are two distinct 
irreducible components of $E$, then the intersection $E_1 \cap E_2$ never consists of a \emph{single} node. 
Furthermore, if one of $E_1$ or $E_2$ is irrational, then the intersection $E_1 \cap E_2$ never contains any nodes. For elliptic tails, this 
statement is vacuous since elliptic tails are irreducible. For elliptic and Weierstrass chains, it follows from examining the topological types of
elliptic bridges and Weierstrass tails (see Figure \ref{F:TopTypes}).  From this observation, it follows that no $\alpha$-unstable $E \subset C$ 
can contain both branches of $q$. Indeed, the hypotheses of (1) and (2) each 
imply that either the two branches of the node $q \in C$ lie on distinct 
irreducible components whose intersection is precisely $q$, or else that that the two branches lie on distinct irreducible components, one of 
which is irrational. Thus, we may assume that $E \subset C$ is disjoint from $q$ or contains only one branch of $q$.

If $E \subset C$ is disjoint from $q$, then $\phi^{-1}$ is an isomorphism in a neighborhood of $E$ and the statement is clear. If $E \subset C$ contains only one branch of the node $q$, then $q$ must be an attaching point of $E$. We may assume without loss of generality that $E$ contains the branch labeled by $q_1$. Now $\phi^{-1}(E) \rightarrow E$ is an isomorphism away from $q_1$ and sends $q_1$ to the node $q$. Since an $\alpha$-unstable curve with nodal attaching is also $\alpha$-unstable with marked point attaching, $\phi^{-1}(E)$ is an $\alpha$-unstable subcurve of $\tilde{C}$.
\end{proof}

\begin{corollary}\label{C:Glue}
\hfill 
\begin{enumerate}
\item Suppose that $(C, \pn, q_1)$ is $\firstf$-stable and $(E,q_1')$ is a an elliptic tail. Then \\
$(C \cup E, \pn)/ (q_1 \sim q_1')$ is $\firstf$-stable.

\item Suppose $(C, \pn, q_1, q_2)$ is $\secondf$-stable and $(E, q_1',q_2')$ is an elliptic chain. Then \\
$(C \cup E, \pn)/(q_1 \sim q_1', q_2 \sim q_2')$   is $\secondf$-stable.

\item Suppose $(C_1, \{p_i\}_{i=1}^{m}, q_1)$ and $(C_2, \{p_i\}_{i=1}^{n-m}, q_2)$ are  $\secondf$-stable and $(E, q_1',q_2')$ is an elliptic chain. Then $(C_1 \cup C_2 \cup E, \pn)/(q_1 \sim q_1', q_2 \sim q_2')$ is $\secondf$-stable.

\item Suppose $(C, \pn, q_1)$ is $\secondf$-stable and $(E, q_1',q_2')$ is an elliptic chain. \\ Then $(C \cup E, \pn, q_2')/(q_1 \sim q_1')$ is $\secondf$-stable.

\item Suppose that $(C, \pn, q_1)$ is $\thirdf$-stable and $(E,q_1')$ is a Weierstrass chain. \\ Then $(C \cup E, \pn)/ (q_1 \sim q_1')$ is $\thirdf$-stable.
\end{enumerate}
\end{corollary}
\begin{proof}
(1), (3), (4), and (5) follow immediately from Lemma \ref{L:Glue}. For (2), one must apply Lemma \ref{L:Glue} twice: First apply Lemma \ref{L:Glue}(1) to glue $q_1 \sim q_1'$, then apply Lemma \ref{L:Glue}(2) to glue $q_2 \sim q_2'$, noting that if $q_2$ and $q_2'$ do not lie on disjoint irreducible components of $(C \cup E, \pn, q_2, q_2')/(q_1 \sim q_1')$, then $E$ must be an irreducible genus one curve, so $q_2'$ does not lie on a smooth rational curve.
\end{proof}

Next, we consider a question which does not arise for Deligne-Mumford stable curves: Suppose $(C, \pn)$ is an $\alpha$-stable curve 
and $q \in C$ is a \emph{non-nodal singularity} with $m\in \{1,2\}$ branches. 
When is the pointed normalization $(\tilde{C}, \pn, \qm)$ of $C$ at $q$ $\alpha$-stable? 
One obvious obstacle is that $\omega_{\tilde{C}}(\Sigma_{i=1}^{n}p_i+\Sigma_{i=1}^{m}q_i)$ need not be ample.
Indeed, one or both of the marked points $q_i$ may lie on a smooth $\P^1$ meeting the rest of the curve in a single node. We thus define the  
\emph{stable pointed normalization} of $(C,\pn)$ to be the (possibly disconnected) curve obtained from $\tilde{C}$ 
by contracting these semistable $\P^{1}$'s. 
This is well-defined except in several degenerate cases:  
First, when $(g,n)=(1,1), (1,2), (2,1)$, the stable pointed normalization of a cuspidal, tacnodal, and  
ramphoid cuspidal curve is a point. In these cases, we regard the stable pointed normalization as being undefined. Second, in the tacnodal case,  
it can happen that $(\tilde{C}, \pn, \qm)$ has two connected components, one of which is a smooth $2$-pointed $\PP^1$. In this case, we define 
the stable pointed normalization to be the curve obtained by deleting this component and taking the stabilization of the remaining connected 
component.

In general, the stable pointed normalization of an $\alpha$-stable curve at a non-nodal singularity need not be $\alpha$-stable. Nevertheless, there is one important case where this statement does hold, namely when $\alpha_c$ is a critical value and $q \in C$ is an $\alpha_c$-critical singularity.

\begin{lemma} \label{L:StableNorm} 
 Let $(C, \pn)$ be an $\alpha_c$-stable curve, and suppose $q \in C$ is an $\alpha_c$-critical singularity. Then the stable pointed normalization of $(C, \pn)$ at $q \in C$ is $\alpha_c$-stable if and only if $(C, \pn)$ is $\alpha_c$-stable.
\end{lemma}
\begin{proof} Follows from the definition of $\alpha$-stability by an elementary case-by-case analysis.
\end{proof}

\subsection{$\alpha_c$-closed curves}\label{s:closed}
We now give an explicit characterization of the closed points of $\SM_{g,n}(\alpha_c)$ when $\alpha_c\in \{\first, \second, \third\}$ is a critical value (see Theorem \ref{T:ClosedCurves}).

\begin{defn}[$\alpha_c$-atoms]  \label{defn-atom}
\hfill
\begin{enumerate} 
\item A \emph{$\firstf$-atom} is 
a $1$-pointed curve of arithmetic genus one obtained by gluing \\ 
$\Spec\CC[x,y]/(y^2-x^{3})$ and $\Spec \CC[n]$ 
via $x=n^{-2}$, $y=n^{-3}$, and marking the point $n=0$.

\item A \emph{$\secondf$-atom} is a $2$-pointed curve of arithmetic genus one obtained by gluing \\
$\Spec \CC[x,y]/(y^2-x^{4})$ and $\Spec \CC[n_1] \coprod \Spec \CC[n_2]$  
via $x=(n_1^{-1}, n_2^{-1})$, $y=(n_1^{-2}, -n_2^{-2})$, and marking the points $n_1=0$ and $n_2=0$.
\item A \emph{$\thirdf$-atom} is 
a $1$-pointed curve of arithmetic genus two obtained gluing \\ $\Spec
\CC[x,y]/(y^2-x^{5})$ and $\Spec \CC[n]$ 
via $x=n^{-2}$, $y=n^{-5}$, and marking the point $n=0$.
\end{enumerate}
We will often abuse notation by simply writing $E$ to refer to the $\alpha_c$-atom $(E, q)$ 
if $\alpha_c \in \{\third, \first\}$ (resp., $(E,q_1, q_2)$ if $\alpha_c = \second$).  
\end{defn}
Every $\alpha_c$-atom $E$ satisfies $\Aut(E) \simeq \GG_m$, where the action of $\GG_m=\Spec \CC[t, t^{-1}]$ is given by
\begin{equation}\label{gm-action-atoms}
\begin{aligned} 
\text{For $\alpha_c=\first$:} \quad \ &x \mapsto t^{-2}x,\, y \mapsto t^{-3} y,\, n \mapsto t n. \\
\text{For $\alpha_c=\second$:} \quad \ &x \mapsto t^{-1}x,\, y \mapsto t^{-2} y,\, n_1 \mapsto t n_1,\, n_2 \mapsto t n_2. \\
\text{For $\alpha_c=\third$:} \quad \ &x \mapsto t^{-2}x,\, y \mapsto t^{-5}y,\, n \mapsto t n.
\end{aligned}
\end{equation}
\begin{figure}[hbt]
\begin{center}
\begin{tikzpicture}[scale=.7]
\tikzstyle{blackfill}=[circle,fill=black,draw=Black]
	\begin{pgfonlayer}{nodelayer}
		\coordinate [style=black] (a0) at (-0.5, 2) {};
		\coordinate [style=black] (a1) at (1.5, 3) {};
		\coordinate [style=black] (a2) at (2.87, 3) {};
		\coordinate [style=black] (a3) at (4.87, 2) {};

		\node at (1.75,2.3)  {$A_3$};

		\filldraw [black] (-.42,1.5) circle (2pt);
		\node at (-.75, 1.5) {$q_1$};
		
		\filldraw [black] (4.8,1.5) circle (2pt);
		\node at (5.2, 1.5) {$q_2$};

		\coordinate [style=black] (b0) at (-7, 2) {};
		\coordinate [style=black] (b1) at (-2, 2) {};
		\coordinate [style=black] (b2) at (-3.75, 2.5) {};
		\filldraw [black] (-6.5,1.85) circle (2pt);
		\node at (-6.5, 1.5) {$q$};

		\node at (-3.75,2.85)  {$A_2$};

		\coordinate [style=black] (c0) at (6.5, 2) {};
		\coordinate [style=black] (c1) at (11.5, 2) {};
		\coordinate [style=black] (c2) at (9.75, 2.5) {};
		\filldraw [black] (7,1.85) circle (2pt);
		\node at (7, 1.5) {$q$};
		
		\node at (9.75,2.85)  {$A_4$};

	\end{pgfonlayer}
	\begin{pgfonlayer}{edgelayer}
		\draw [style=very thick,in=270, out=-15, looseness=2.75] (a1) to (a0);
		\draw [style=very thick,in=270, out=195, looseness=2.75] (a2) to (a3);
		
		\draw [style=very thick,in=270, out=-17] (b0) to (b2);
		\draw [style=very thick,in=240, out=251, looseness=1.25] (b2) to (b1);
		
		\draw [style=very thick,in=270, out=-17] (c0) to (c2);
		\draw [style=very thick,in=240, out=251, looseness=1.25] (c2) to (c1);
	\end{pgfonlayer}
\end{tikzpicture}
\end{center}
\vspace{-.8cm}
\caption{A $\firstf$-atom, $\secondf$-atom, and $\thirdf$-atom, respectively.}
\label{F:atoms}
\end{figure}

In order to describe the closed points of $\SM_{g,n}(\alpha_c)$ precisely, we need the following terminology. 
We say that $C$ admits a \emph{decomposition} $C=C_1 \cup \cdots \cup C_r$ if $C_1, \ldots, C_r$ 
are proper subcurves whose union is all of $C$, and either $C_i \cap C_j = \emptyset$ or $C_i$ meets $C_j$ nodally. 
When $(C, \pn)$ is an $n$-pointed curve, and $C=C_1 \cup \cdots \cup C_r$ is a decomposition of $C$, 
we always consider $C_i$ as a pointed curve by taking as marked points the subset of $\{p_i\}_{i=1}^{n}$ 
supported on $C_i$ and the attaching points $C_i \cap (\overline{C \backslash C_i})$.

\begin{defn}[$\alpha_c$-closed curves] \label{defn-alpha-closed} Let $\alpha_c \in \{\third, \second, \first\}$ be a critical value. 
We say that an $n$-pointed curve $(C, \pn)$ is \emph{$\alpha_c$-closed}  if there is a decomposition 
$C=K\cup E_1 \cup \cdots \cup E_r,$ where 
\begin{itemize}
	\item[(1)] $E_1, \ldots, E_r$ are $\alpha_c$-atoms.
	\item[(2)] $K$ is an $(\alpha_c \pe)$-stable curve containing no nodally-attached $\alpha_c$-tails.
	\item[(3)] $K$ is a closed curve in the stack of $(\alpha_c \pe)$-stable curves. 
\end{itemize}
We call $K$ the \emph{core of $(C, \pn)$}, and we call the decomposition $C=K \cup E_1 \cup \cdots \cup E_r$ 
the \emph{canonical decomposition of $C$}. 
Of course, we consider $K$ as a pointed curve where the set of marked points is the union of $\pn \cap K$ and 
$K \cap (\overline{C \setminus K})$.  
Note that we allow the possibility that $K$ is disconnected or empty.
\end{defn}

We can now state the main result of this section.

\begin{theorem}[Characterization of $\alpha_c$-closed curves]\label{T:ClosedCurves}
Let $\alpha_c \in \{\first, \second, \third\}$ be a critical value. An $\alpha_c$-stable curve $(C, \pn)$ is a closed point of $\SM_{g,n}(\alpha_c)$ if and only if $(C, \pn)$ is $\alpha_c$-closed.
\end{theorem}

To prove the above theorem, we need several preliminary lemmas.

\begin{lemma}\label{L:Closed1}
\hfill
\begin{enumerate}
\item Suppose $(E,q)$ is an elliptic tail. Then $(E, q)$ is a closed point of $\SM_{1,1}(\first)$ if and only if $(E, q)$ is a $\firstf$-atom.
\item Suppose $(E,q_1,q_2)$ is an elliptic bridge. Then $(E, q_1,q_2)$ is a closed point of $\SM_{1,2}(\second)$ if and only if $(C, q_1,q_2)$ is a $\secondf$-atom.
\item Suppose $(E,q)$ is a Weierstrass tail. Then $(C, q)$ is a closed point of $\SM_{2,1}(\third)$ if and only if $(C, q)$ is a $\thirdf$-atom.
\end{enumerate}
\end{lemma}
\begin{proof} 
Case (1) follows from the observation that
$\SM_{1,1}(\first) \simeq [\CC^2 / \GG_m]$, where $\GG_m$ acts with weights $4$ and $6$. 
Case (2) follows from the observation that $\SM_{1,2}(\second) \simeq [\CC^3 / \GG_m]$, where $\GG_m$
acts with weights $2$, $3$, and $4$. The proofs of these assertions parallel
our argument in case (3) below, so we leave the details to the reader.

We proceed to prove case (3). 
First, we show that if $(E, q)$ is any Weierstrass tail, then $(E, q)$ admits an isotrivial specialization to a $\thirdf$-atom. 
To do so, we can write any Weierstrass genus $2$ tail as a degree $2$ cover of $\PP^1$ with the equation on 
$\PP(1,3,1)$ given by
\[
 y^2 = x^{5}z+ a_3x^3z^3 + a_2 x^2z^4 + a_1xz^5 + a_0 z^6
\]
where $a_i \in \CC$, and the marked point $q$ corresponds to $y=z=0$. Acting 
by $\lambda \cdot (x, y, z)=(x, \lambda y, \lambda^2 z)$, we see that this cover is isomorphic to
\[
 y^2 = x^{5}z+ \lambda^4a_3x^3z^3 + \lambda^6a_2 x^2z^4 + \lambda^8a_1xz^5 + \lambda^{10}a_0 z^6
\]
for any $\lambda \in \CC^*$. Letting $\lambda \rightarrow 0$, we obtain an isotrivial specialization of $(E, q)$ to 
the double cover $y^2=x^5z$, which is a $\thirdf$-atom.

Next, we show that if $(E,q)$ is a $\thirdf$-atom, 
then $(E,q)$ does not admit any nontrivial isotrivial specializations in $\SM_{2,1}(\third)$. 
Let $(\E \rightarrow \Delta, \sigma)$  be an isotrivial specialization in $\SM_{2,1}(\third)$ with generic fiber isomorphic to $(E,q)$. 
Let $\tau$ be the section of $\E \rightarrow \Delta$ which picks out the unique ramphoid cusp of the generic fiber. 
Since the limit of a ramphoid cusp is a ramphoid cusp in $\SM_{2,1}(\third)$, $\tau(0)$ is also ramphoid cusp. 
Now let $\tau\co  \tilde{\E} \rightarrow \E$ be the simultaneous normalization of $\E$ along 
$\tau$, and let $\tilde{\tau}$ and $\tilde{\sigma}$ be the inverse images of $\tau$ and $\sigma$ respectively. 
Then $(\tilde{\E} \rightarrow \Delta, \tilde{\tau}, \tilde{\sigma})$ is an isotrivial specialization 
of $2$-pointed curves of arithmetic genus $0$ with smooth general fiber. 
To prove that the original isotrivial specialization is trivial, it suffices to prove that
$(\tilde{\E} \rightarrow \Delta, \tilde{\tau}, \tilde{\sigma})$ is trivial, i.e. we must show that the special fiber is smooth 
(equivalently, irreducible). The fact that $\omega_{\E/\Delta}(\sigma)$ is relatively ample on 
$\E$ implies that $\omega_{\tilde{\E}/\Delta}(3\tilde{\tau}+\tilde{\sigma})$ is relatively ample on $\tilde{\E}$, 
which implies that the special fiber of 
$\tilde{\E}$ is irreducible. 
\end{proof}


 
\begin{lemma}\label{L:Closed2}
Suppose $(C, \pn)$ is a closed point of $\SM_{g,n}(\alpha_c \pe)$. Then $(C, \pn)$ remains closed in $\SM_{g,n}(\alpha_c)$ if and only if $(C, \pn)$ contains no nodally attached $\alpha_c$-tail.
\end{lemma}

\begin{proof} We prove the case $\alpha_c = \third$ and leave the other cases to the reader. To lighten notation, 
we often omit marked points $\pn$ in the rest of the proof. 

First, we show that if $(C, \pn)$ has $A_1$-attached Weierstrass tail, then it does not remain closed in $\SM_{g,n}(\third)$. 
Suppose we have a decomposition $C=K\cup Z$, where $(Z, q)$ is an $A_1$-attached Weierstrass tail. By Lemma \ref{L:Closed1}, 
$(Z, q)$ admits an isotrivial specialization to a $\thirdf$-atom $(E, q_1)$. 
We may glue this specialization to the trivial family $K \times \Delta$ to obtain a nontrivial isotrivial specialization 
$C \rightsquigarrow K \cup E$, where $E$ is nodally attached at $q_1$.
By Lemma \ref{L:Glue}, $K \cup E$ is $\thirdf$-stable, so this is a nontrivial isotrivial specialization in $\SM_{g,n}(\third)$.

Next, we show that if $(C, \pn)$ has no nodally-attached Weierstrass tails, then it remains closed in $\SM_{g,n}(\third)$. 
In other words, if there exists a nontrivial isotrivial specialization $C \rightsquigarrow C_0$, 
then $C$ necessarily contains a nodally-attached Weierstrass tail. 
To begin, note that the special fiber $C_0$ of the nontrivial isotrivial specialization $\C \rightarrow \Delta$ must 
contain at least one ramphoid cusp. Otherwise, $(\C \rightarrow \Delta, \sigman)$ would constitute a nontrivial, isotrivial specialization in 
$\SM_{g,n}(\third +\epsilon)$, contradicting the hypothesis that $(C, \pn)$ is closed in $\SM_{g,n}(\third \pe)$. 
For simplicity, let us assume that the special fiber $C_0$ contains a single ramphoid cusp $q$. 
Locally around this point, we may write $\C$ as 
\[
y^2=x^5+a_{3}(t)x^3+a_2(t)x^2+a_1(t)x+a_0(t), 
\] 
where $t$ is the uniformizer of $\Delta$ at $0$ and $a_i(0)=0$. By \cite[Section 7.6]{radu-yano}, after possibly a finite base change, 
there exists a (weighted) blow-up $\phi\co \tilde{\C} \rightarrow \C$ such that the special fiber $\tilde{C}_0$ is isomorphic to the normalization of 
$C$ at $q$ attached nodally to the curve $T$, where $T$ is defined by an equation $y^2=x^5+b_{3}x^3z^2+b_2x^2z^3+b_1xz^4+b_0z^{5}$ 
on $\PP(2,5,2)$ for some $[b_3:b_2:b_1:b_0]\in \PP(4,6,8,10)$ (depending on the $a_i(t)$) and such that
$T$ is attached to $C$ at $[x:y:z]=[1:0:1]$. Evidently, $T$ is a genus $2$ double cover of $\PP^1\cong \PP(2,2)$ via the projection 
$[x:y:z] \mapsto [x:y]$ and $[1:0:1]$ is a ramification point of this cover. 
It follows that $\tilde{C}_0$ has a Weierstrass tail.
The special fiber of $\tilde{\C}$ is isomorphic to the stable pointed normalization of $C_0$ at $q$, 
together with a nodally attached Weierstrass tail. 
By Lemma \ref{L:StableNorm} and Corollary \ref{C:Glue}, $(\tilde{\C}_0, \pn)$ is $\alpha$-stable. 
Since it contains no ramphoid cusps, it is also $(\alpha_c \pe)$-stable. 
By hypothesis, $(C, \pn)$ is closed in $\SM_{g,n}(\alpha \pe)$, so the family 
$(\tilde{\C} \rightarrow \Delta, \sigman)$ must be trivial. 
This implies that the generic fiber $(C, \pn)$ must have a nodally-attached Weierstrass tail.
\end{proof}

The following lemma says that one can use isotrivial specializations to replace $\alpha_c$-critical singularities 
and $\alpha_c$-tails by $\alpha_c$-atoms.

\begin{lemma}\label{L:Specialize}
Let $(C, \pn)$ be an $n$-pointed curve, and let $E$ be the $\alpha_c$-atom.

(1) Suppose $q \in C$ is an $\alpha_c$-critical singularity. Then there exists an isotrivial specialization 
$
C \rightsquigarrow C_0 = \tilde{C} \cup E 
$
to an $n$-pointed curve $C_0$ which is the nodal union of $E$ and the stable
pointed normalization $\tilde{C}$ of $C$ at $q$ along the marked point(s) of $E$ and the pre-image(s) of $q$ in $\tilde{C}$.

(2) Suppose $C$ decomposes as $C=K \cup Z$, where $Z$ is an $\alpha_c$-tail. 
Then there exists an isotrivial specialization
$
C \rightsquigarrow C_0 = K \cup E
$
to an $n$-pointed curve $C_0$ which is the nodal union of $K$ and $E$ along the marked point(s) of $E$ and $K \cap Z$.

\end{lemma}
\begin{proof}
We prove the case $\alpha_c=\third$, and leave the remaining two cases to the reader. 
For (1), let $C \times \Delta$ be the trivial family, let $\tilde{\C} \rightarrow C \times \Delta$ be the normalization along $q \times \Delta$, 
and let $\tilde{\C}' \rightarrow \tilde{\C}$ be the blow-up of $\tilde{C}$ at the point lying over $(q,0)$. 
Let $\tau$ denote the strict transform of $q \times \Delta$ on $\tilde{\C}'$, 
and note that $\tau$ passes through a smooth point of the exceptional divisor. A local calculation, as in the proof of 
Proposition \ref{P:F1}, shows that we may `recrimp.' Namely, there exists a finite map $\psi\co \tilde{\C}' \rightarrow \C'$ such that $\psi$ is an 
isomorphism on $ \tilde{\C}'-\tau$, so that $\C'$ has a ramphoid cusp along $\psi \circ \tau$, and
the ramphoid cuspidal rational tail in the central fiber is an $\alpha_c$-atom, i.e., has trivial crimping. 
Blowing down any semistable $\PP^1$'s in the central fiber of $\C' \rightarrow \Delta$ (these appear, for example, when
$q$ lies on an unmarked $\PP^1$ attached nodally to the rest of the curve), we arrive at the desired isotrivial specialization. 
For (2), note that there exists an isotrivial specialization $(Z, q_1) \rightsquigarrow (E, q_1)$ by Lemma \ref{L:Closed1}. Gluing this to the trivial 
family $(K \times \Delta, q_1 \times \Delta)$ gives the desired isotrivial specialization.
\end{proof}

\begin{proof}[Proof of Theorem \ref{T:ClosedCurves}] We consider the case $\alpha_c=\third$, and leave the other two cases to the reader. 
First, we show that every $\thirdf$-closed curve $(C, \pn)$ is a closed point of $\SM_{g,n}(\third)$. 
Let $(\C \rightarrow \Delta, \sigman)$ be any isotrivial specialization of $(C, \pn)$ in $\SM_{g,n}(\third)$; we will show it must be trivial. 
Let $C=K \cup E_1\cup \cdots \cup E_r$ be the canonical decomposition and let $q_i=K \cap E_i$. 
Each $q_i$ is a disconnecting node in the general fiber of $\C \rightarrow \Delta$, 
so $q_i$ specializes to a node in the special fiber by Corollary \ref{C:Attaching1}. Possibly after a finite base change, we may normalize along the corresponding nodal sections to obtain isotrivial specializations $\K$ and $\E_1, \ldots, \E_r$. By 
Lemma \ref{L:Norm}, $\K$ is a family in $\SM_{g-2r, n+r}(\third)$ and $\E_1, \ldots, \E_r$ are families in $\SM_{2,1}(\third)$. 
Since $\K$ contains no Weierstrass tails in the general fiber, it is trivial by Lemma \ref{L:Closed2}. 
The families $\E_1, \ldots, \E_r$ are trivial by Lemma \ref{L:Closed1}. 
It follows that the original family $(\C \rightarrow \Delta, \sigman)$ is trivial, as desired.

Next, we show that if $(C, \pn) \in \SM_{g,n}(\third)$ is a closed point, then $(C, \pn)$ must be $\thirdf$-closed. First, we claim that every 
ramphoid cusp of $C$ must lie on a nodally attached $\thirdf$-atom. Indeed, if $q \in C$ is a ramphoid cusp which does not lie on a nodally
attached $\thirdf$-atom, then Lemma \ref{L:Specialize} gives an isotrivial specialization $(C, \pn) \rightsquigarrow (C_0, \pn)$ in which $C_0$ 
sprouts a nodally attached $\thirdf$-atom at $q$. Note that $(C_0, \pn)$ is $\thirdf$-stable by Lemma \ref{L:StableNorm} 
and Corollary \ref{C:Glue}, so this gives a nontrivial isotrivial specialization in $\SM_{g,n}(\third)$.
Second, we claim that $C$ contains no nodally-attached Weierstrass tails which are not $\thirdf$-atoms. 
Indeed, if it does, then Lemma \ref{L:Specialize} gives an isotrivial specialization $(C, \pn) \rightsquigarrow (C_0, \pn)$ which replaces this 
Weierstrass tail by a $\thirdf$-atom. Note that $(C_0, \pn)$ is $\thirdf$-stable by Lemma \ref{L:Norm} and Corollary \ref{C:Glue}, 
so this gives a nontrivial isotrivial specialization in $\SM_{g,n}(\third)$. It is now easy to see that $C$ is $\thirdf$-closed. 
Indeed, if $E_1, \ldots, E_r$ are the nodally attached $\thirdf$-atoms of $C$, then the complement $K$ has no ramphoid cusps 
and no nodally-attached Weierstrass tails. Since $K$ is $\thirdf$-stable and has no ramphoid cusps, it is $(\thirdf \pe)$-stable. Furthermore, $K$ 
must be closed in $\SM_{g,n}(\third \pe)$, since a nontrivial isotrivial specialization of $K$ in $\SM_{g,n}(\third \pe)$ would induce a nontrivial, 
isotrivial specialization of $(C, \pn)$ in $\SM_{g,n}(\third)$. We conclude that $(C, \pn)$ is $\thirdf$-closed as desired.
\end{proof}

\subsection{Combinatorial type of an $\alpha_c$-closed curve} \label{S:canonical}

In the previous section, we saw that every $\alpha_c$-stable curve which is closed in $\bar{\cM}_{g,n}(\alpha_c)$ has  a canonical decomposition $C=K\cup E_1 \cup \cdots \cup E_r$ where $E_1, \ldots, E_r$ are the $\alpha_c$-atoms of $C$. We wish to use this decomposition to compute the local VGIT chambers associated to $C$. For the two critical values $\alpha_c \in \{\second, \first\}$, the pointed curve $K$ does not have infinitesimal automorphisms and does not affect this computation. However, if $\alpha_c = \third$, then $K$ may have infinitesimal automorphisms due to the presence of rosaries (see Definition \ref{defn-rosary}), which leads us to consider a slight enhancement of the canonical decomposition. 
Once we have taken care of this wrinkle, we define the combinatorial type of an $\alpha_c$-closed curve
in Definition \ref{defn-canonical}). The key point of this definition is that it establishes the notation that will be used in carrying out the local VGIT calculations in Section \ref{section-local-vgit}.


\begin{defn}[Rosaries] \label{defn-rosary}
We say that $(R,p_1, p_{2})$ is a \emph{rosary of length $\ell$} if there exists a surjective gluing morphism
\[
\gamma \co \coprod_{i=1}^{\ell}(R_i,q_{2i-1},q_{2i}) \hookrightarrow (R,p_1,p_2)
\]
satisfying:
\begin{enumerate}
\item $(R_i,q_{2i-1},q_{2i})$ is a $2$-pointed smooth rational curve for $i=1, \ldots, \ell$.
\item $\gamma$ is an isomorphism when restricted to $R_i \setminus \{q_{2i-1}, q_{2i}\}$  for $i=1, \ldots, \ell$. 
\item $\gamma(q_{2i})=\gamma(q_{2i+1})$ is an $A_3$-singularity for $i=1, \ldots, \ell-1$.
\item $\gamma(q_1) = p_1$ and $\gamma(q_{2\ell}) = p_2$.
\end{enumerate}
We say that $(C,\pn)$ has an {\it $A_{k_1}/A_{k_2}$-attached (open) rosary of length $\ell$} 
if there exists a gluing morphism $\gamma \co (R,p_1,p_2) \hookrightarrow (C, \pn)$ such that
\begin{enumerate}
\item[(a)] $(R,p_1,p_2)$ is a rosary of length $\ell$.
\item[(b)] $\gamma(p_i)$ is an $A_{k_i}$-singularity of $C$, or $k_i=1$ and $\gamma(q_i)$ is a marked point of $(C,\pn)$.
\end{enumerate}
We say that $C$ is a \emph{closed rosary of length $\ell$} if $C$ has $A_3/A_3$-attached rosary 
$\gamma \co (R,p_1,p_2) \hookrightarrow C$ of length $\ell$ such that $\gamma(p_1)=\gamma(p_2)$
is an $A_3$-singularity of $C$.
\end{defn}
\begin{remark}\label{R:even-rosary}
A rosary of even length is an elliptic chain and thus can never appear in a $(\second - \epsilon)$-stable curve.
\end{remark}

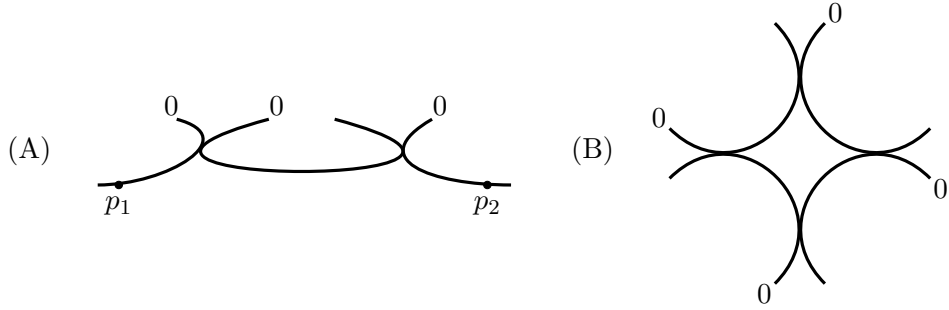
\begin{figure}[hbt]
\begin{center}
\begin{tikzpicture}[scale=.7]
\tikzstyle{blackfill}=[circle,fill=black,draw=Black]
	\begin{pgfonlayer}{nodelayer}
		\coordinate  (a0) at (-5.5-1.7, 0) {};
		\coordinate  (a1) at (-4-1.7, 1.25) {};
		\coordinate  (a2) at (-2.25-1.7, 1.25) {};
		\coordinate  (a3) at (-1-1.7, 1.25) {};
		\coordinate  (a6) at (-.25-.6, 1.25) {};
		\coordinate  (a7) at (1.25-.6, 0) {};
		
		\filldraw [black] (-6.8,-0) circle (2pt);
		\node  at (-6.8,-.4)  {$p_1$};

		\filldraw [black] (.2,-0) circle (2pt);
		\node  at (0.2,-.4)  {$p_2$};
		
		\node  at (-0.7,1.5)  {$0$};
		\node  at (-3.8,1.5)  {$0$};
		\node  at (-5.8,1.5)  {$0$};
		
		\coordinate (b0) at (-2.45+6.11, 0.5+.575) {};
		\coordinate (b1) at (-0.45+6.11, 2.5+.575) {};
		\coordinate (b2) at (0.5+6.11, 2.5+.575) {};
		\coordinate (b3) at (2.5+6.11, 0.5+.575) {};
		\coordinate (b4) at (2.5+6.11, -0.45+.575) {};
		\coordinate (b5) at (0.5+6.11, -2.45+.575) {};
		\coordinate (b6) at (-0.45+6.11, -2.45+.575) {};
		\coordinate (b7) at (-2.45+6.11, -0.45+.575) {};

		\node at (-2.45+6.11-.2, 0.5+.575+.2) {$0$};
		\node at (0.5+6.11+.2, 2.5+.575+.2) {$0$};
		\node at (2.5+6.11+.2, -0.45+.575-.2) {$0$};
		\node at (-0.45+6.11-.2, -2.45+.575-.2)  {$0$};

		\node at (-8.5,.65)  {(A)};
		\node at (2.2,.65)  {(B)};

	\end{pgfonlayer}
	\begin{pgfonlayer}{edgelayer}
		\draw [style=very thick, in=-15, out=0, looseness=1.75] (a0) to (a1);
		\draw [style=very thick, in=-15, out=195, looseness=10.5] (a2) to (a3);
		\draw [style=very thick, in=180, out=210, looseness=2.00] (a6) to (a7);
		
		\draw [style=very thick,bend right=90, looseness=1.75] (b0) to (b1);
		\draw [style=very thick,bend left=270, looseness=1.75] (b2) to (b3);
		\draw [style=very thick,bend left=90, looseness=1.75] (b5) to (b4);
		\draw [style=very thick,bend right=90, looseness=1.75] (b6) to (b7);

	\end{pgfonlayer}
\end{tikzpicture}
\end{center}
\vspace{-.4cm}
\caption{Curve (A) is a rosary of length $3$. Curve (B) is a closed rosary of length $4$.}\label{F:rational-chains}
\end{figure}

Note that if $(R, p_1, p_2)$ is a rosary, then $\Aut(R, p_1, p_{2}) \simeq \GG_m$. 
Hassett and Hyeon showed that all infinitesimal automorphisms of $(\second -\epsilon)$-stable curves 
are accounted for by rosaries \cite[Section 8]{HH2}. In Proposition \ref{P:Rosaries} and Corollary \ref{C:RosariesClosed}, 
we record a slight refinement of their result.

\begin{prop} \label{P:Rosaries}
Suppose $(C, \pn)$ is $(\second - \epsilon)$-stable with $\Aut(C, \pn)^{\circ} \simeq \GG_m^k$. 
Then one of the following holds:
\begin{enumerate}
\item There exists a decomposition $C = C_0 \cup R_1 \cup \cdots \cup R_k$, where each $R_i$ is an $A_1/A_1$-attached rosary of odd length,
and $C_0$ contains no $A_1/A_1$-attached rosaries. Note that we allow $C_0$ to be empty.
\item $k=1$ and $C$ is a closed rosary of even length.
\end{enumerate}
\end{prop}

\begin{proof}

Consider first the case in which $C$ is simply a chain of rational curves, say $R_1, \ldots, R_k$, where $R_i$ meets $R_{i+1}$ in a  single point,
 and $R_{k}$ meets $R_{1}$ in a single point. These attaching points may be either nodes or tacnodes. If every attaching point is a tacnode, 
then we are in case (2). If some of the attaching points are nodes, then the set of rational curves between any two consecutive nodes 
in the chain are tacnodally attached and thus constitute $A_1/A_1$-attached rosary. In other words, we are in case (1) with $C_0$ empty.

From now on, we may assume that not all components of $C$ are rational curves meeting the rest of the curve in two points. In particular, there exist components on which $\Aut(C, \pn)^{\circ}$ acts trivially. We proceed by induction on the dimension of $\Aut(C, \pn)^{\circ}$, 
noting that if dimension is $0$, there is nothing to prove.

Note that if $\Aut(C,\pn)^{\circ}$ acts nontrivially on a component $T_1$ and $T_1$ meets a component $S$ on which $\Aut(C, \pn)^{\circ}$ 
acts trivially, then their point of attachment must be a node (and not a tacnode). 
This follows immediately from the fact that an automorphism of $\P^1$ which fixes two points and the tangent space at one of these points must 
be trivial. Now let $T_1, \ldots, T_\ell$ be the maximal length chain containing $T_1$ on which $\Aut(C,\pn)^{\circ}$ acts nontrivially; we have just argued that $T_1$ and $T_\ell$ must be attached to the rest of $C$ at nodes. 
If each $T_i$ is tacnodally attached to $T_{i+1}$, then $R:=T_1 \cup \cdots \cup T_\ell$ is an $A_1/A_1$-attached rosary in $C$. 
If some $T_i$ is attached to $T_{i+1}$ at a node, then choosing minimal such $i$, we see that $R:=T_1 \cup \cdots \cup T_i$ is an 
$A_1/A_1$-attached rosary. Thus, $C$ contains an $A_1/A_1$-attached rosary $R$, necessarily of odd length by Remark \ref{R:even-rosary}. 
If it is not all of $C$, then the dimension of $\Aut(\overline{C \setminus R}, \pn)^{\circ}$ is one less than the dimension of $\Aut(C, \pn)^{\circ}$,
so we are done by induction.
\end{proof}

\begin{corollary}\label{C:RosariesClosed}
Suppose $(C, \pn)$ is a closed $(\second-\epsilon)$-stable curve with $\Aut(C, \pn)^{\circ} \simeq \GG_m^k$. Then there exists a decomposition $C = C_0 \cup R_1 \cup \cdots \cup R_k$ where each $R_i$ is $A_1/A_1$-attached rosary of length $3$.
\end{corollary}
\begin{proof}
This follows immediately from Proposition \ref{P:Rosaries} and two observations: 
\begin{itemize}
\item
If $R$ is a rosary of odd length $\ell \geq 5$, 
then $R$ admits an isotrivial specialization to the nodal union of a rosary of length $3$ and length $\ell-2$.
\item A closed rosary of even length $\ell$ admits an isotrivial specialization to the 
nodal union of $\ell/2$ rosaries of length $3$ arranged in a closed chain. 
\end{itemize}
\end{proof}

In order to compute the local VGIT chambers for an $\alpha_c$-closed curve, it will be useful to have the following notation.

\begin{defn}[Links] \label{defn-link}
A \emph{$\secondf$-link of length $\ell$} is a $2$-pointed curve $(E,p_1, p_{2})$ which admits a decomposition
\[
E=E_1 \cup \cdots \cup E_\ell \qquad \text{such that:}
\]
\begin{enumerate}
\item $(E_j,q_{j-1},q_{j})$ is a $\secondf$-atom for $j=1, \ldots, \ell$.
\item $q_j:=E_j \cap E_{j+1}$ is a node for $j=1, \ldots, \ell-1$.
\item $q_0:=p_1$ is a marked point of $E_1$ and $q_{\ell}:=p_2$ is a marked point of $E_{\ell}$.
\end{enumerate}

A \emph{$\thirdf$-link of length $\ell$} is a $1$-pointed curve $(E,p)$ which admits a decomposition
\[
E=R_1 \cup \cdots \cup R_{\ell-1} \cup E_{\ell} \qquad \text{such that:}
\]
\begin{enumerate}
\item $(R_j,q_{j-1},q_{j})$ is a rosary of length $3$ for $j=1, \ldots, \ell-1$, and $(E_\ell, q_\ell)$ is a $\thirdf$-atom.
\item $q_j:=R_j \cap R_{j+1}$ is a node for $j=1, \ldots, \ell-2$, and $q_{\ell-1}:=R_{\ell-1} \cap E_{\ell}$ is a node.
\item $q_0:=p$ is a marked point of $R_1$.
\end{enumerate}

When we refer to a $\secondf$-link $(E,p_1, p_{2})$ (resp., $\thirdf$-link $(E,p)$) as a subcurve of a larger curve, 
we always take it to be $A_1/A_1$-attached at $p_1$ and $p_2$ (resp., at $p$).
\end{defn}

\begin{figure}[htb]
\begin{center}
\begin{tikzpicture}[scale=.85]
\tikzstyle{blackfill}=[circle,fill=black,draw=Black]
	\begin{pgfonlayer}{nodelayer}
		\coordinate [style=black] (a0) at (-0.5, 2) {};
		\coordinate [style=black] (a1) at (1.5, 3) {};
		\coordinate [style=black] (a2) at (2.87, 3) {};
		\coordinate [style=black] (a3) at (4.87, 2) {};
		\node at (1.75,2.3)  {$A_3$};

		\coordinate [style=black] (b0) at (-0.5+4.8, 2) {};
		\coordinate [style=black] (b1) at (1.5+4.8, 3) {};
		\coordinate [style=black] (b2) at (2.87+4.8, 3) {};
		\coordinate [style=black] (b3) at (4.87+4.8, 2) {};
		\node at (1.75+4.8,2.3)  {$A_3$};

		\coordinate [style=black] (c0) at (-0.5-4.8, 2) {};
		\coordinate [style=black] (c1) at (1.5-4.8, 3) {};
		\coordinate [style=black] (c2) at (2.87-4.8, 3) {};
		\coordinate [style=black] (c3) at (4.87-4.8, 2) {};
		\node at (1.75-4.8,2.3)  {$A_3$};

		\filldraw [black] (-.42-4.8,1.5) circle (2pt);
		\node at (-.75-4.8, 1.5) {$p_1$};	
		\filldraw [black] (4.8+4.8,1.5) circle (2pt);
		\node at (5.2+4.8, 1.5) {$p_6$};
		\node at (-7,2)  {(A)};


		\node at (-7,-1)  {(B)};

		\node [style=black] (B0) at (-2.5-.3-1.99+.3, 1-1.001) {};
		\node [style=black] (B1) at (-3.5-.3-1.99+.3, 0-1.001) {};
		\node [style=black] (B2) at (0.5-.2-1.99+.4, 1-1.001) {};
		\node [style=black] (B3) at (1.5-.2-1.99+.4, 0-1.001) {};
		\node [style=black] (B4) at (-1.5-1.99+.35, 1-1.001) {};
		\node [style=black] (B5) at (-1-1.99+.35, 1-1.001) {};

		\node [style=black] (B10) at (-2.5-.3+2.591+.4, 1-1.001) {};
		\node [style=black] (B11) at (-3.5-.3+2.591+.4, 0-1.001) {};
		\node [style=black] (B12) at (0.5-.2+2.591+.5, 1-1.001) {};
		\node [style=black] (B13) at (1.5-.2+2.591+.5, 0-1.001) {};
		\node [style=black] (B14) at (-1.5+2.591+.45, 1-1.001) {};
		\node [style=black] (B15) at (-1+2.591+.45, 1-1.001) {};

		\coordinate [style=black] (B20) at (6.5-2.81, 2-3.45) {};
		\coordinate [style=black] (B21) at (11.5-2.81, 2-3.45) {};
		\coordinate [style=black] (B22) at (9.75-2.81, 2.5-3.45) {};
		\node at (9.75-2.81,2.85-3.45)  {$A_4$};

		\filldraw [black] (-5.73+.3, -1.64) circle (2pt);
		\node  at (-5.73-.3+.3, -1.64)  {$p$};



		\node at (-4.3,-1.1)  {$A_3$};
		\node at (-2.2,-1.1)  {$A_3$};
		\node at (.4,-1.1)  {$A_3$};
		\node at (2.5,-1.1)  {$A_3$};

	\end{pgfonlayer}
	\begin{pgfonlayer}{edgelayer}
		\draw [style=very thick,in=270, out=-15, looseness=2.75] (a1) to (a0);
		\draw [style=very thick,in=270, out=195, looseness=2.75] (a2) to (a3);
		
		\draw [style=very thick,in=270, out=-15, looseness=2.75] (b1) to (b0);
		\draw [style=very thick,in=270, out=195, looseness=2.75] (b2) to (b3);
		
		\draw [style=very thick,in=270, out=-15, looseness=2.75] (c1) to (c0);
		\draw [style=very thick,in=270, out=195, looseness=2.75] (c2) to (c3);

		\draw [style=very thick, in=270, out=315, looseness=3.25] (B0) to (B1);
		\draw [style=very thick, in=-90, out=-135, looseness=3.25] (B2) to (B3);
		\draw [style=very thick, bend right=135, looseness=11.25] (B4) to (B5);

		\draw [style=very thick, in=270, out=315, looseness=3.25] (B10) to (B11);
		\draw [style=very thick, in=-90, out=-135, looseness=3.25] (B12) to (B13);
		\draw [style=very thick, bend right=135, looseness=11.25] (B14) to (B15);
		
		\draw [style=very thick,in=270, out=-17] (B20) to (B22);
		\draw [style=very thick,in=240, out=251, looseness=1.25] (B22) to (B21);
		

	\end{pgfonlayer}
\end{tikzpicture}
\end{center}
\caption{Curve (A) (resp., (B)) is a $\secondf$-link (resp., $\thirdf$-link) of length $3$.  Each component above is a rational curve.}\label{F:link}
\end{figure}
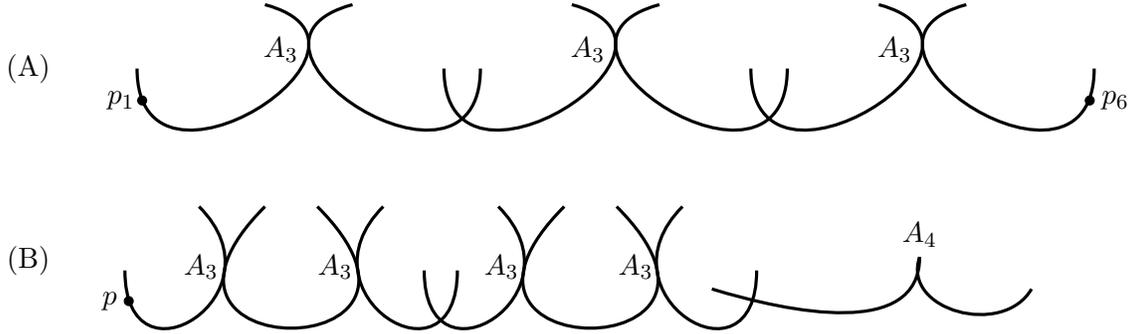

Now let $C=K \cup E_1 \cup \cdots \cup E_r$ be the canonical decomposition of an $\alpha_c$-closed curve $C$,
where $K$ is the core and $E_i$'s are $\alpha_c$-atoms.  
Observe that as long as $K \neq \emptyset$, then each $\secondf$-atom (resp., $\thirdf$-atom) $E_i$ of a $\secondf$-closed 
(resp., $\thirdf$-closed) curve is a component of a unique $\secondf$-link (resp., $\thirdf$-link) of maximal length.
When $\alpha_c=\third$, we make the following definition.

\begin{defn}[Secondary core for $\alpha_c=\third$]
Suppose $C=K \cup E_1 \cup \ldots \cup E_r$ is the canonical decomposition of an $\thirdf$-closed curve $C$.
For each $\thirdf$-atom $E_i$, let $L_i$ be the maximal length $\thirdf$-link containing $E_i$. We call
$K' := \overline{C \setminus (L_1 \cup \cdots \cup L_r)}$ \emph{the secondary core} of $C$, which we consider as a curve marked with the points $(\pn \cap K') \cup (K' \cap (\overline{C \setminus K'})$.  
The secondary core has the property that any $A_1/A_1$-attached rosary $R \subseteq K'$, 
satisfies $R \cap L_i = \emptyset$ for $i =1, \ldots, r$.
\end{defn}

We can now define combinatorial types of $\alpha_c$-closed curves. We refer the reader to
Figure \ref{F:closed-curves} for a graphical accompaniment of the following definition.

\begin{defn}[Combinatorial Type of $\alpha_c$-closed curve]\label{defn-canonical}  \quad \,

\noindent {\rm $\bullet \, $} A $\firstf$-closed curve $(C, \pn)$ has \emph{combinatorial type}
\begin{enumerate}
\item[(A)] If the core $K$ is nonempty. In this case, \[
C=K \cup E_1 \cup \cdots \cup E_r
\]
where each $E_i$ is a $\firstf$-atom meeting $K$ at a single node $q_i$.
\item[(B)] If $(g,n)=(2,0)$ and $C=E_1 \cup E_2$ where $E_1$ and $E_2$ are $\firstf$-atoms meeting each other in a single node $q \in C$.
 \item[(C)]If $(g,n)=(1,1)$ and $C=E_1$ is a $\firstf$-atom.
\end{enumerate}

\noindent {\rm $\bullet \, $} A $\secondf$-closed curve $(C, \pn)$ has  \emph{combinatorial type}
\begin{enumerate}
\item[(A)] If the core is nonempty. In this case, we have \[
C=K \cup L_1 \cup \cdots \cup L_r \cup L_{r+1} \cup \cdots \cup L_{r+s}
\] where
\begin{itemize}
\item For $i=1, \ldots, r$: \ $L_i = \bigcup_{j=1}^{\ell_i} E_{i,j}$ is a $\second$-link of length $\ell_i$ meeting $K$ at two distinct nodes.  
In particular, $E_{i,1}$ meets $K$ at a node $q_{i,0}$, $E_{i,\ell_i}$ meets $K$ at a node $q_{i,\ell_i}$, 
and $E_{i,j}$ meets $E_{i,j+1}$ at a node $q_{i,j}$.
\item For $i=r+1, \ldots, r+s$: \ $L_i = \bigcup_{j=1}^{\ell_i} E_{i,j}$ is a $\second$-link of length $\ell_i$ meeting $K$ at a single node and 
terminating in a marked point. In particular, $E_{i,1}$ meets $K$ at a node $q_{i,0}$, and $E_{i,j}$ meets $E_{i,j+1}$ at a node $q_{i,j}$.
\end{itemize}
\item[(B)] If $n=2$ and $(C,p_1,p_2)$ is a $\second$-link of length $g$, i.e.
$C=E_1 \cup \cdots \cup E_g$ 
where each $E_j$ is a $\secondf$-atom, $E_j$ meets $E_{j+1}$ at a node $q_j$, $p_1 \in E_1$ and 
$p_{2} \in E_g$.
\item[(C)] If $n=0$ and $C$ is a $\second$-link of length $g-1$, 
whose endpoints are nodally glued. In other words, $C = E_1 \cup \cdots \cup E_{g-1}$,
where each $E_j$ is a $\secondf$-atom, $E_j$ meets $E_{j+1}$ at a node $q_j$, and $E_1$ meets $E_{g-1}$ at a node $q_{0}$.
\end{enumerate}

\noindent {\rm $\bullet \, $} 
 A $\thirdf$-closed curve $(C, \pn)$ has \emph{combinatorial type}
\begin{enumerate}
\item[(A)] If the secondary core $K'$ is nonempty. In this case, we write 
\[
C=K' \cup L_1 \cup \cdots \cup L_r
\] where for $i=1, \ldots, r$, $L_i = \bigcup_{j=1}^{\ell_i-1} R_{i,j} \cup E_{i}$ is a $\thirdf$-link of length $\ell_i$. 
	In particular, $E_{i}$ is a $\thirdf$-atom and each $R_{i,j}$ a length $3$ rosary such that $R_{i,1}$ meets $K'$ at a node $q_{i,0}$, 
	$R_{i,j}$ meets $R_{i,j+1}$ at a node $q_{i,j}$, and $R_{i,\ell_i-1}$ meets $E_i$ in a node $q_{i,\ell_i-1}$.  
	We denote the tacnodes of the rosary $R_{i,j}$ by $\tau_{i,j,1}$ and $\tau_{i,j,2}$, and the unique ramphoid cusp of $E_{i}$ by $\xi_{i}$.

 \item[(B)] If $n=1$, $g=2\ell$ and $(C,p_1)$ is a $\thirdf$-link of length $\ell$, i.e. $C =  R_1 \cup \cdots \cup R_{\ell-1} \cup E_\ell$, 
 where $R_1, \ldots, R_{\ell-1}$ are rosaries of length $3$ with $p_1 \in R_1$ and $E_\ell$ is a $\thirdf$-atom.  
 For $j=1, \ldots, \ell-1$, we label the tacnodes of $R_j$ as $\tau_{j,1}$ and $\tau_{j,2}$, 
 the node where $R_j$ intersects $R_{j+1}$ as $q_j$, the node where $R_{\ell-1}$ intersects $E_\ell$ as $q_{\ell-1}$ and the unique ramphoid cusp of $E_\ell$ as $\xi$.
 
 \item[(C)] If $n=0$, $g=2\ell+2$ and $C$ is the nodal union of two $\thirdf$-links, i.e. $C = E_0 \cup R_1 \cup \cdots \cup R_{\ell-1} \cup E_\ell$, 
 where $E_0, E_\ell$ are $\thirdf$-atoms, and $R_1, \ldots, R_{\ell-1}$ are rosaries of length $3$.  
 For $j=1, \ldots, \ell-2$, $R_j$ intersects $R_{j+1}$ at a node $q_j$, $E_0$ intersects $R_1$ in a node $q_0$, 
 and $R_{\ell-1}$ intersects $E_\ell$ in a node $q_{\ell-1}$.  
 We label the ramphoid cusps of $E_0, E_\ell$ as $\xi_0, \xi_1$, and the tacnodes of $R_j$ as $\tau_{j,1}$ and $\tau_{j,2}$.
\end{enumerate}
\end{defn}

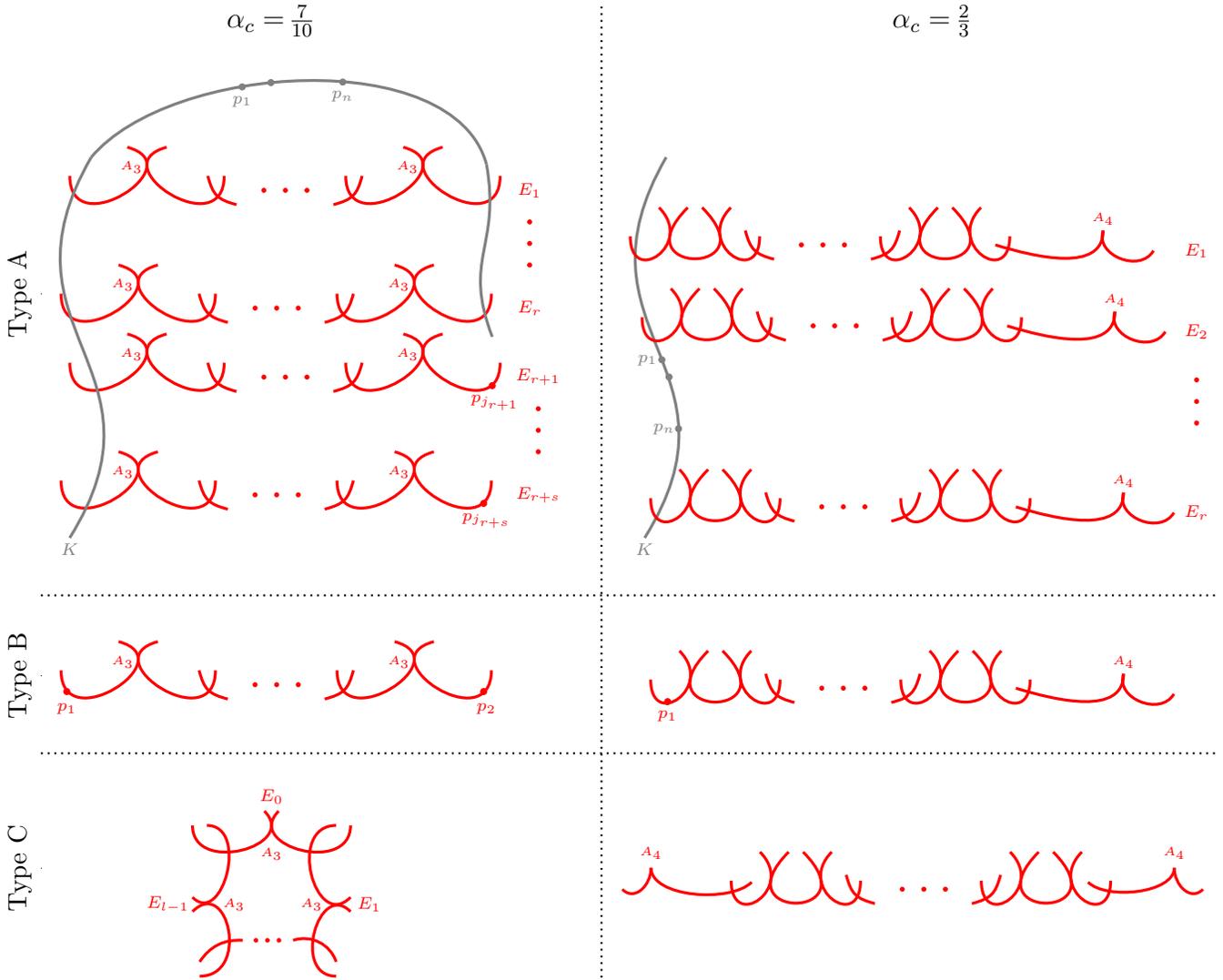
\begin{figure}[tbh]
\begin{center}
\begin{tikzpicture}[scale=.42]
\tikzstyle{blackfill}=[circle,fill=black,draw=Black]
	\begin{pgfonlayer}{nodelayer}
		\useasboundingbox (-.1,-.1) rectangle (3.1,1.1);
		\coordinate (a0) at (-0.5, 2) {};
		\coordinate (a1) at (1.5-1, 3-2) {};
		\coordinate (a2) at (2.87+1, 3-2) {};
		\coordinate (a3) at (4.87, 2) {};
		\filldraw [red] (1.45,1.5) circle (2pt);
		\filldraw [red] (2.2,1.5) circle (2pt);
		\filldraw [red] (2.95,1.5) circle (2pt);

		\coordinate (b0) at (-0.5+4.8, 2) {};
		\coordinate (b1) at (1.5+4.8, 3) {};
		\coordinate (b2) at (2.87+4.8, 3) {};
		\coordinate (b3) at (4.87+4.8, 2) {};
		\node [font=\miniscule, red] at (1.75+4.8-.1,2.3)  {$A_3$};

		\coordinate (c0) at (-0.5-4.8, 2) {};
		\coordinate (c1) at (1.5-4.8, 3) {};
		\coordinate (c2) at (2.87-4.8, 3) {};
		\coordinate (c3) at (4.87-4.8, 2) {};
		\node[font=\miniscule, red]at (1.75-4.8-.1,2.3)  {$A_3$};

		\node [font=\tiny, red] at (5.2+4.8+1, 1.5) {$E_r$};


		\coordinate (a01) at (-0.5+.304, 2+4.1) {};
		\coordinate (a11) at (1.5-1+.304, 3-2+4.1) {};
		\coordinate (a21) at (2.87+1+.304, 3-2+4.1) {};
		\coordinate (a31) at (4.87+.304, 2+4.1) {};
		\filldraw [red] (1.45+.304,1.5+4.1) circle (2pt);
		\filldraw [red] (2.2+.304,1.5+4.1) circle (2pt);
		\filldraw [red] (2.95+.304,1.5+4.1) circle (2pt);

		\coordinate (b01) at (-0.5+4.8+.304, 2+4.1) {};
		\coordinate (b11) at (1.5+4.8+.304, 3+4.1) {};
		\coordinate (b21) at (2.87+4.8+.304, 3+4.1) {};
		\coordinate (b31) at (4.87+4.8+.304, 2+4.1) {};
		\node [font=\miniscule, red] at (1.75+4.8+.304-.1,2.3+4.1)  {$A_3$};

		\coordinate (c01) at (-0.5-4.8+.304, 2+4.1) {};
		\coordinate (c11) at (1.5-4.8+.304, 3+4.1) {};
		\coordinate (c21) at (2.87-4.8+.304, 3+4.1) {};
		\coordinate (c31) at (4.87-4.8+.304, 2+4.1) {};
		\node [font=\miniscule, red] at (1.75-4.8+.304-.1,2.3+4.1)  {$A_3$};		

		\node [font=\tiny, red] at (5.2+4.8+1, 1.5+4.1) {$E_1$};
		
		\filldraw [red] (11,4.5) circle (2pt);
		\filldraw [red] (11,3.75) circle (2pt);
		\filldraw [red] (11,3) circle (2pt);

		\coordinate (d0) at (-0.5, 2-6.5023) {};
		\coordinate (d1) at (1.5-1, 3-2-6.5023) {};
		\coordinate (d2) at (2.87+1, 3-2-6.5023) {};
		\coordinate (d3) at (4.87, 2-6.5023) {};
		\filldraw [red] (1.45,1.5-6.5023) circle (2pt);
		\filldraw [red] (2.2,1.5-6.5023) circle (2pt);
		\filldraw [red] (2.95,1.5-6.5023) circle (2pt);

		\coordinate (e0) at (-0.5+4.8, 2-6.5023) {};
		\coordinate (e1) at (1.5+4.8, 3-6.5023) {};
		\coordinate (e2) at (2.87+4.8, 3-6.5023) {};
		\coordinate (e3) at (4.87+4.8, 2-6.5023) {};
		\node [font=\miniscule, red] at (1.75+4.8-.1,2.3-6.5023)  {$A_3$};

		\coordinate (f0) at (-0.5-4.8, 2-6.5023) {};
		\coordinate (f1) at (1.5-4.8, 3-6.5023) {};
		\coordinate (f2) at (2.87-4.8, 3-6.5023) {};
		\coordinate (f3) at (4.87-4.8, 2-6.5023) {};
		\node [font=\miniscule, red]at (1.75-4.8-.1,2.3-6.5023)  {$A_3$};

		\node [font=\tiny, red] at (5.2+4.8+1.3, 1.5-6.5023) {$E_{r+s}$};


		\coordinate (d01) at (-0.5+.304, 2+4.1-6.5023) {};
		\coordinate (d11) at (1.5-1+.304, 3-2+4.1-6.5023) {};
		\coordinate (d21) at (2.87+1+.304, 3-2+4.1-6.5023) {};
		\coordinate (d31) at (4.87+.304, 2+4.1-6.5023) {};
		\filldraw [red] (1.45+.304,1.5+4.1-6.5023) circle (2pt);
		\filldraw [red] (2.2+.304,1.5+4.1-6.5023) circle (2pt);
		\filldraw [red] (2.95+.304,1.5+4.1-6.5023) circle (2pt);

		\coordinate (e01) at (-0.5+4.8+.304, 2+4.1-6.5023) {};
		\coordinate (e11) at (1.5+4.8+.304, 3+4.1-6.5023) {};
		\coordinate (e21) at (2.87+4.8+.304, 3+4.1-6.5023) {};
		\coordinate (e31) at (4.87+4.8+.304, 2+4.1-6.5023) {};
		\node [font=\miniscule, red] at (1.75+4.8+.304-.1,2.3+4.1-6.5023)  {$A_3$};

		\coordinate (f01) at (-0.5-4.8+.304, 2+4.1-6.5023) {};
		\coordinate (f11) at (1.5-4.8+.304, 3+4.1-6.5023) {};
		\coordinate (f21) at (2.87-4.8+.304, 3+4.1-6.5023) {};
		\coordinate (f31) at (4.87-4.8+.304, 2+4.1-6.5023) {};
		\node [font=\miniscule, red] at (1.75-4.8+.304-.1,2.3+4.1-6.5023)  {$A_3$};		

		\node [font=\tiny, red] at (5.2+4.8+1.3, 1.5+4.1-6.5023) {$E_{r+1}$};
		
		\filldraw [red] (11.3,4.5-6.5023) circle (2pt);
		\filldraw [red] (11.3,3.75-6.5023) circle (2pt);
		\filldraw [red] (11.3,3-6.5023) circle (2pt);

		\filldraw [red] (9.7,-1.2) circle (3pt);
		\node [font=\tiny, red] at (9.8, -1.8) {$p_{j_{r+1}}$};
		\filldraw [red] (9.7-.3,-1.2-4.1) circle (3pt);
		\node [font=\tiny, red] at (9.8-.3, -1.8-4.1) {$p_{j_{r+s}}$};

		\coordinate  (k0) at (-3-1.5-.5, -2.5-3-.5-.5) {};
		\coordinate (k1) at (-2.75-1.5, 3.75+3) {};
		\coordinate (k2) at (2.5+7, 3.5+3) {};
		\coordinate (k3) at (2.7+7, 0.5) {};
		
		\node [font=\tiny, gray] at (-3-1.5-.5, -2.5-3-.5-.5-.4) {$K$};

		\filldraw [gray] (1, 9.2) circle (3pt);
		\filldraw [gray] (2, 9.35) circle (3pt);
		\filldraw [gray] (4.5, 9.37) circle (3pt);
		\node [font=\tiny, gray] at(1, 8.7) {$p_1$};
		\node [font=\tiny, gray] at(4.5, 8.87) {$p_n$};

		\coordinate  (K0) at (-5+22-2, -6.5) {};
		\coordinate (K1) at (-4.25+22-2, 6.75) {};

		\filldraw [gray] (17.6-2, -.3) circle (3pt);
		\filldraw [gray] (17.85-2, -.9) circle (3pt);
		\filldraw [gray] (18.18-2, -2.7) circle (3pt);
		\node [font=\tiny, gray] at (17.1-2, -.3) {$p_1$};
		\node [font=\tiny, gray] at (17.68-2, -2.7)  {$p_n$};

		\node [font=\tiny, gray] at (-5+22-2, -6.9) {$K$};

		\coordinate (Link00) at (15.5, 5) {};
		\coordinate (Link01) at (14.5, 4) {};
		\coordinate (Link02) at (18, 5) {};
		\coordinate (Link03) at (19, 4) {};
		\coordinate (Link04) at (16.5, 5) {};
		\coordinate (Link05) at (17, 5) {};
		
		\coordinate (Link10) at (15.5+8.7, 5) {};
		\coordinate (Link11) at (14.5+8.7, 4) {};
		\coordinate (Link12) at (18+8.7, 5) {};
		\coordinate (Link13) at (19+8.7, 4) {};
		\coordinate (Link14) at (16.5+8.7, 5) {};
		\coordinate (Link15) at (17+8.7, 5) {};
		
		\coordinate (LH00) at (-0.5+19.001, 2+2.201) {};
		\coordinate (LH01) at (1.5-1+19.001, 3-2+2.201) {};
		\coordinate (LH02) at (2.87+1+19.001, 3-2+2.201) {};
		\coordinate (LH03) at (4.87+19.001, 2+2.201) {};
		\filldraw [red] (1.45+19.001,1.5+2.201) circle (2pt);
		\filldraw [red] (2.2+19.001,1.5+2.201) circle (2pt);
		\filldraw [red] (2.95+19.001,1.5+2.201) circle (2pt);		
		
		\coordinate (LA0) at (28-.8, 4-.3) {};
		\coordinate  (LA1) at (33.5-.8, 3.8-.3) {};
		\coordinate  (LA2) at (31.75-.8, 4.5-.3) {};
		\node [font =\miniscule, red] at (19.75+12-.8, 4.9-.3) {$A_4$};
		\node [font =\tiny, red] at (34.2, 3.8-.3) {$E_1$};
		
		\coordinate (Link00b) at (15.5+.41, 5-2.82) {};
		\coordinate (Link01b) at (14.5+.41, 4-2.82) {};
		\coordinate (Link02b) at (18+.41, 5-2.82) {};
		\coordinate (Link03b) at (19+.41, 4-2.82) {};
		\coordinate (Link04b) at (16.5+.41, 5-2.82) {};
		\coordinate (Link05b) at (17+.41, 5-2.82) {};
		
		\coordinate (Link10b) at (15.5+8.7+.41, 5-2.82) {};
		\coordinate (Link11b) at (14.5+8.7+.41, 4-2.82) {};
		\coordinate (Link12b) at (18+8.7+.41, 5-2.82) {};
		\coordinate (Link13b) at (19+8.7+.41, 4-2.82) {};
		\coordinate (Link14b) at (16.5+8.7+.41, 5-2.82) {};
		\coordinate (Link15b) at (17+8.7+.41, 5-2.82) {};
		
		\coordinate (LH00b) at (-0.5+19.001+.41, 2+2.201-2.82) {};
		\coordinate (LH01b) at (1.5-1+19.001+.41, 3-2+2.201-2.82) {};
		\coordinate (LH02b) at (2.87+1+19.001+.41, 3-2+2.201-2.82) {};
		\coordinate (LH03b) at (4.87+19.001+.41, 2+2.201-2.82) {};
		\filldraw [red] (1.45+19.001+.41,1.5+2.201-2.82) circle (2pt);
		\filldraw [red] (2.2+19.001+.41,1.5+2.201-2.82) circle (2pt);
		\filldraw [red] (2.95+19.001+.41,1.5+2.201-2.82) circle (2pt);		
		
		\coordinate (LA0b) at (28-.8+.41, 4-.3-2.82) {};
		\coordinate  (LA1b) at (33.5-.8+.41, 3.8-.3-2.82) {};
		\coordinate  (LA2b) at (31.75-.8+.41, 4.5-.3-2.82) {};
		\node [font =\miniscule, red] at (19.75+12-.8+.41, 4.9-.3-2.82) {$A_4$};
		\node [font =\tiny, red] at (34.2, 3.8-.3-2.82) {$E_2$};
		
		\coordinate (Link00c) at (15.5+.71, 5-9.12) {};
		\coordinate (Link01c) at (14.5+.71, 4-9.12) {};
		\coordinate (Link02c) at (18+.71, 5-9.12) {};
		\coordinate (Link03c) at (19+.71, 4-9.12) {};
		\coordinate (Link04c) at (16.5+.71, 5-9.12) {};
		\coordinate (Link05c) at (17+.71, 5-9.12) {};
		
		\coordinate (Link10c) at (15.5+8.7+.71, 5-9.12) {};
		\coordinate (Link11c) at (14.5+8.7+.71, 4-9.12) {};
		\coordinate (Link12c) at (18+8.7+.71, 5-9.12) {};
		\coordinate (Link13c) at (19+8.7+.71, 4-9.12) {};
		\coordinate (Link14c) at (16.5+8.7+.71, 5-9.12) {};
		\coordinate (Link15c) at (17+8.7+.71, 5-9.12) {};
		
		\coordinate (LH00c) at (-0.5+19.001+.71, 2+2.201-9.12) {};
		\coordinate (LH01c) at (1.5-1+19.001+.71, 3-2+2.201-9.12) {};
		\coordinate (LH02c) at (2.87+1+19.001+.71, 3-2+2.201-9.12) {};
		\coordinate (LH03c) at (4.87+19.001+.71, 2+2.201-9.12) {};
		\filldraw [red] (1.45+19.001+.71,1.5+2.201-9.12) circle (2pt);
		\filldraw [red] (2.2+19.001+.71,1.5+2.201-9.12) circle (2pt);
		\filldraw [red] (2.95+19.001+.71,1.5+2.201-9.12) circle (2pt);		
		
		\coordinate (LA0c) at (28-.8+.71, 4-.3-9.12) {};
		\coordinate  (LA1c) at (33.5-.8+.71, 3.8-.3-9.12) {};
		\coordinate  (LA2c) at (31.75-.8+.71, 4.5-.3-9.12) {};
		\node [font =\miniscule, red] at (19.75+12-.8+.71, 4.9-.3-9.12) {$A_4$};
		\node [font =\tiny, red] at (34.2, 3.8-.3-9.12) {$E_r$};

		\filldraw [red] (34.2, -1) circle (2pt);
		\filldraw [red] (34.2,-1.75) circle (2pt);
		\filldraw [red] (34.2,-2.5) circle (2pt);

		\coordinate (BD0) at (-0.5, 2-13.111) {};
		\coordinate (BD1) at (1.5-1, 3-2-13.111) {};
		\coordinate (BD2) at (2.87+1, 3-2-13.111) {};
		\coordinate (BD3) at (4.87, 2-13.111) {};
		\filldraw [red] (1.45,1.5-13.111) circle (2pt);
		\filldraw [red] (2.2,1.5-13.111) circle (2pt);
		\filldraw [red] (2.95,1.5-13.111) circle (2pt);

		\coordinate (BE0) at (-0.5+4.8, 2-13.111) {};
		\coordinate (BE1) at (1.5+4.8, 3-13.111) {};
		\coordinate (BE2) at (2.87+4.8, 3-13.111) {};
		\coordinate (BE3) at (4.87+4.8, 2-13.111) {};
		\node [font=\miniscule, red] at (1.75+4.8-.1,2.3-13.111)  {$A_3$};

		\coordinate (BF0) at (-0.5-4.8, 2-13.111) {};
		\coordinate (BF1) at (1.5-4.8, 3-13.111) {};
		\coordinate (BF2) at (2.87-4.8, 3-13.111) {};
		\coordinate (BF3) at (4.87-4.8, 2-13.111) {};
		\node [font=\miniscule, red]at (1.75-4.8-.1,2.3-13.111)  {$A_3$};

		\filldraw [red] (-5.1,-1.2-10.65) circle (3pt);
		\node [font=\tiny, red] at (-5.1, -1.8-10.65) {$p_{1}$};
		\filldraw [red] (9.7-.3,-1.2-10.65) circle (3pt);
		\node [font=\tiny, red] at (9.8-.3, -1.8-10.65) {$p_{2}$};
		
		\coordinate (Link00d) at (15.5+.71, 5-15.432) {};
		\coordinate (Link01d) at (14.5+.71, 4-15.432) {};
		\coordinate (Link02d) at (18+.71, 5-15.432) {};
		\coordinate (Link03d) at (19+.71, 4-15.432) {};
		\coordinate (Link04d) at (16.5+.71, 5-15.432) {};
		\coordinate (Link05d) at (17+.71, 5-15.432) {};
		
		\coordinate (Link10d) at (15.5+8.7+.71, 5-15.432) {};
		\coordinate (Link11d) at (14.5+8.7+.71, 4-15.432) {};
		\coordinate (Link12d) at (18+8.7+.71, 5-15.432) {};
		\coordinate (Link13d) at (19+8.7+.71, 4-15.432) {};
		\coordinate (Link14d) at (16.5+8.7+.71, 5-15.432) {};
		\coordinate (Link15d) at (17+8.7+.71, 5-15.432) {};
		
		\coordinate (LH00d) at (-0.5+19.001+.71, 2+2.201-15.432) {};
		\coordinate (LH01d) at (1.5-1+19.001+.71, 3-2+2.201-15.432) {};
		\coordinate (LH02d) at (2.87+1+19.001+.71, 3-2+2.201-15.432) {};
		\coordinate (LH03d) at (4.87+19.001+.71, 2+2.201-15.432) {};
		\filldraw [red] (1.45+19.001+.71,1.5+2.201-15.432) circle (2pt);
		\filldraw [red] (2.2+19.001+.71,1.5+2.201-15.432) circle (2pt);
		\filldraw [red] (2.95+19.001+.71,1.5+2.201-15.432) circle (2pt);		
		
		\coordinate (LA0d) at (28-.8+.71, 4-.3-15.432) {};
		\coordinate  (LA1d) at (33.5-.8+.71, 3.8-.3-15.432) {};
		\coordinate  (LA2d) at (31.75-.8+.71, 4.5-.3-15.432) {};
		\node [font =\miniscule, red] at (19.75+12-.8+.71, 4.9-.3-15.432) {$A_4$};

		\filldraw [red] (15.8,-12.2) circle (3pt);
		\node [font=\tiny, red] at (15.8,-12.7){$p_{1}$};


		\coordinate (0) at (-2.75+2.021, 0.5-17.01) {};
		\coordinate (1) at (2.75+2.021, 0.5-17.01) {};
		\coordinate (2) at (-0.25+2.021, 1-17.01) {};
		\coordinate (3) at (0.25+2.021, 1-17.01) {};
		\coordinate (4) at (2.25+2.021, 0.5-17.01) {};
		\coordinate (5) at (2.75+2.021, -2-17.01) {};
		\coordinate (6) at (2.75+2.021, -2.5-17.01) {};
		\coordinate (7) at (2.25+2.021, -4.75-17.01) {};
		\coordinate (8) at (-2.25+2.021, 0.5-17.01) {};
		\coordinate (9) at (-2.75+2.021, -2-17.01) {};
		\coordinate (10) at (-2.75+2.021, -2.5-17.01) {};
		\coordinate (11) at (-2.5+2.021, -4.75-17.01) {};
		\coordinate (12) at (-2.5+2.021, -4.25-17.01) {};
		\coordinate (13) at (-1+2.021, -3.5-17.01) {};
		\coordinate (14) at (2.25+2.021, -4.25-17.01) {};
		\coordinate (15) at (0.75+2.021, -3.5-17.01) {};

		\node [font=\miniscule, red] at (2.021, -0.5-17.01)  {$A_3$};
		\node [font=\miniscule, red] at (1.4+2.021, -2.25-17.01)  {$A_3$};
		\node [font=\miniscule, red] at (-1.3+2.021, -2.25-17.01)  {$A_3$};
		
		\node [font=\tiny, red] at (2.021, 1.5-17.01)  {$E_0$};
		\node [font=\tiny, red] at (3.4+2.021, -2.25-17.01)  {$E_1$};
		\node [font=\tiny, red] at (-3.6+2.021, -2.25-17.01)  {$E_{l-1}$};

		\filldraw [red]  (3.5-2.021, -3.5-17.01) circle (1.8pt);
		\filldraw [red]  (3.9-2.021, -3.5-17.01) circle (1.8pt);
		\filldraw [red]  (4.3-2.021, -3.5-17.01) circle (1.8pt);

		\coordinate (Link00e) at (15.5+3.51, 5-22.432) {};
		\coordinate (Link01e) at (14.5+3.51, 4-22.432) {};
		\coordinate (Link02e) at (18+3.51, 5-22.432) {};
		\coordinate (Link03e) at (19+3.51, 4-22.432) {};
		\coordinate (Link04e) at (16.5+3.51, 5-22.432) {};
		\coordinate (Link05e) at (17+3.51, 5-22.432) {};
		
		\coordinate (Link10e) at (15.5+8.7+3.51, 5-22.432) {};
		\coordinate (Link11e) at (14.5+8.7+3.51, 4-22.432) {};
		\coordinate (Link12e) at (18+8.7+3.51, 5-22.432) {};
		\coordinate (Link13e) at (19+8.7+3.51, 4-22.432) {};
		\coordinate (Link14e) at (16.5+8.7+3.51, 5-22.432) {};
		\coordinate (Link15e) at (17+8.7+3.51, 5-22.432) {};
		
		\coordinate (LH00e) at (-0.5+19.001+3.51, 2+2.201-22.432) {};
		\coordinate (LH01e) at (1.5-1+19.001+3.51, 3-2+2.201-22.432) {};
		\coordinate (LH02e) at (2.87+1+19.001+3.51, 3-2+2.201-22.432) {};
		\coordinate (LH03e) at (4.87+19.001+3.51, 2+2.201-22.432) {};
		\filldraw [red] (1.45+19.001+3.51,1.5+2.201-22.432) circle (2pt);
		\filldraw [red] (2.2+19.001+3.51,1.5+2.201-22.432) circle (2pt);
		\filldraw [red] (2.95+19.001+3.51,1.5+2.201-22.432) circle (2pt);	
		
		\coordinate (CR0) at (-3.5+19.423-.7, 1.25-19.234) {};
		\coordinate (CR1) at (-4.5+19.423-.7, 0.5-19.234) {};
		\coordinate (CR2) at (3+19.423+11, 1.25-19.234) {};
		\coordinate (CR3) at (4+19.423+11, 0.5-19.234) {};
		\coordinate (CR4) at (19.423+11, 0.75-19.234) {};
		\coordinate (CR5) at (19.423-.7, 0.75-19.234) {};

		\node [font=\miniscule, red] at  (-3.5+19.423-.7, 1.25-19.234+.5) {$A_4$};	
		\node [font=\miniscule, red] at  (3+19.423+11, 1.25-19.234+.5) {$A_4$};	

		\node [font=\large] at (2,11.5) {$\alpha_c = \secondf$};
		\node [font=\large] at (25,11.5) {$\alpha_c = \thirdf$};

	\end{pgfonlayer}
	\begin{pgfonlayer}{edgelayer}
	
		\draw [style=very thick,color=red, bend left, looseness=1.25] (a1) to (a0);
		\draw  [style=very thick,color=red, bend right, looseness=1.25](a2) to (a3);
		\draw [style=very thick,color=red, in=270, out=-15, looseness=2.75] (b1) to (b0);
		\draw [style=very thick,color=red, in=270, out=195, looseness=2.75] (b2) to (b3);		
		\draw [style=very thick,color=red, in=270, out=-15, looseness=2.75] (c1) to (c0);
		\draw [style=very thick,color=red, in=270, out=195, looseness=2.75] (c2) to (c3);

		\draw [style=very thick,color=red, bend left, looseness=1.25] (a11) to (a01);
		\draw  [style=very thick,color=red, bend right, looseness=1.25](a21) to (a31);
		\draw [style=very thick,color=red, in=270, out=-15, looseness=2.75] (b11) to (b01);
		\draw [style=very thick,color=red, in=270, out=195, looseness=2.75] (b21) to (b31);		
		\draw [style=very thick,color=red, in=270, out=-15, looseness=2.75] (c11) to (c01);
		\draw [style=very thick,color=red, in=270, out=195, looseness=2.75] (c21) to (c31);		
		
		\draw [style=very thick,color=red, bend left, looseness=1.25] (d1) to (d0);
		\draw  [style=very thick,color=red, bend right, looseness=1.25](d2) to (d3);
		\draw [style=very thick,color=red, in=270, out=-15, looseness=2.75] (e1) to (e0);
		\draw [style=very thick,color=red, in=270, out=195, looseness=2.75] (e2) to (e3);		
		\draw [style=very thick,color=red, in=270, out=-15, looseness=2.75] (f1) to (f0);
		\draw [style=very thick,color=red, in=270, out=195, looseness=2.75] (f2) to (f3);

		\draw [style=very thick,color=red, bend left, looseness=1.25] (d11) to (d01);
		\draw  [style=very thick,color=red, bend right, looseness=1.25](d21) to (d31);
		\draw [style=very thick,color=red, in=270, out=-15, looseness=2.75] (e11) to (e01);
		\draw [style=very thick,color=red, in=270, out=195, looseness=2.75] (e21) to (e31);		
		\draw [style=very thick,color=red, in=270, out=-15, looseness=2.75] (f11) to (f01);
		\draw [style=very thick,color=red, in=270, out=195, looseness=2.75] (f21) to (f31);

		\draw [style=very thick,color=gray, in=240, out=58, looseness=1.35] (k0) to (k1);
		\draw [style=very thick, color=gray,  in=-80, out=111, looseness=1.25] (k3) to (k2);
		\draw [style=very thick,color=gray,  in=105, out=50, looseness=.8] (k1) to (k2);

		\draw [style=very thick,color=gray, in=240, out=58, looseness=1.35] (K0) to (K1);

		\draw [style=very thick, color=red, in=270, out=315, looseness=3.25] (Link00) to (Link01);
		\draw [style=very thick,color=red, in=-90, out=-135, looseness=3.25] (Link02) to (Link03);
		\draw [style=very thick,color=red, bend right=135, looseness=17.25] (Link04) to (Link05);
		\draw [style=very thick, color=red, in=270, out=315, looseness=3.25] (Link10) to (Link11);
		\draw [style=very thick,color=red, in=-90, out=-135, looseness=3.25] (Link12) to (Link13);
		\draw [style=very thick,color=red, bend right=135, looseness=17.25] (Link14) to (Link15);
		\draw [style=very thick,color=red, bend left, looseness=1.25] (LH01) to (LH00);
		\draw  [style=very thick,color=red, bend right, looseness=1.25](LH02) to (LH03);		
		\draw [style=very thick,color=red, in=270, out=-17] (LA0) to (LA2);
		\draw [style=very thick,color=red, in=240, out=251, looseness=1.25] (LA2) to (LA1);
		
		\draw [style=very thick, color=red, in=270, out=315, looseness=3.25] (Link00b) to (Link01b);
		\draw [style=very thick,color=red, in=-90, out=-135, looseness=3.25] (Link02b) to (Link03b);
		\draw [style=very thick,color=red, bend right=135, looseness=17.25] (Link04b) to (Link05b);
		\draw [style=very thick, color=red, in=270, out=315, looseness=3.25] (Link10b) to (Link11b);
		\draw [style=very thick,color=red, in=-90, out=-135, looseness=3.25] (Link12b) to (Link13b);
		\draw [style=very thick,color=red, bend right=135, looseness=17.25] (Link14b) to (Link15b);
		\draw [style=very thick,color=red, bend left, looseness=1.25] (LH01b) to (LH00b);
		\draw  [style=very thick,color=red, bend right, looseness=1.25](LH02b) to (LH03b);		
		\draw [style=very thick,color=red, in=270, out=-17] (LA0b) to (LA2b);
		\draw [style=very thick,color=red, in=240, out=251, looseness=1.25] (LA2b) to (LA1b);

		\draw [style=very thick, color=red, in=270, out=315, looseness=3.25] (Link00c) to (Link01c);
		\draw [style=very thick,color=red, in=-90, out=-135, looseness=3.25] (Link02c) to (Link03c);
		\draw [style=very thick,color=red, bend right=135, looseness=17.25] (Link04c) to (Link05c);
		\draw [style=very thick, color=red, in=270, out=315, looseness=3.25] (Link10c) to (Link11c);
		\draw [style=very thick,color=red, in=-90, out=-135, looseness=3.25] (Link12c) to (Link13c);
		\draw [style=very thick,color=red, bend right=135, looseness=17.25] (Link14c) to (Link15c);
		\draw [style=very thick,color=red, bend left, looseness=1.25] (LH01c) to (LH00c);
		\draw  [style=very thick,color=red, bend right, looseness=1.25](LH02c) to (LH03c);		
		\draw [style=very thick,color=red, in=270, out=-17] (LA0c) to (LA2c);
		\draw [style=very thick,color=red, in=240, out=251, looseness=1.25] (LA2c) to (LA1c);

		\draw [style=very thick,color=red, bend left, looseness=1.25] (BD1) to (BD0);
		\draw  [style=very thick,color=red, bend right, looseness=1.25](BD2) to (BD3);
		\draw [style=very thick,color=red, in=270, out=-15, looseness=2.75] (BE1) to (BE0);
		\draw [style=very thick,color=red, in=270, out=195, looseness=2.75] (BE2) to (BE3);		
		\draw [style=very thick,color=red, in=270, out=-15, looseness=2.75] (BF1) to (BF0);
		\draw [style=very thick,color=red, in=270, out=195, looseness=2.75] (BF2) to (BF3);
		
		\draw [style=very thick, color=red, in=270, out=315, looseness=3.25] (Link00d) to (Link01d);
		\draw [style=very thick,color=red, in=-90, out=-135, looseness=3.25] (Link02d) to (Link03d);
		\draw [style=very thick,color=red, bend right=135, looseness=17.25] (Link04d) to (Link05d);
		\draw [style=very thick, color=red, in=270, out=315, looseness=3.25] (Link10d) to (Link11d);
		\draw [style=very thick,color=red, in=-90, out=-135, looseness=3.25] (Link12d) to (Link13d);
		\draw [style=very thick,color=red, bend right=135, looseness=17.25] (Link14d) to (Link15d);
		\draw [style=very thick,color=red, bend left, looseness=1.25] (LH01d) to (LH00d);
		\draw  [style=very thick,color=red, bend right, looseness=1.25](LH02d) to (LH03d);		
		\draw [style=very thick,color=red, in=270, out=-17] (LA0d) to (LA2d);
		\draw [style=very thick,color=red, in=240, out=251, looseness=1.25] (LA2d) to (LA1d);


		\draw [style=very thick,color=red, in=270, out=-45, looseness=1.75] (2) to (0);
		\draw [style=very thick,color=red,in=-90, out=225, looseness=1.75] (3) to (1);
		\draw [style=very thick,color=red,in=225, out=180, looseness=1.75] (4) to (5);
		\draw [style=very thick,color=red,in=180, out=135, looseness=1.75] (6) to (7);
		\draw [style=very thick,color=red,in=0, out=-45, looseness=1.50] (9) to (8);
		\draw [style=very thick,color=red,in=0, out=45, looseness=2.00] (10) to (11);
		\draw [style=very thick,color=red,bend left] (12) to (13);
		\draw [style=very thick,color=red,bend left=45] (15) to (14);

		\draw [style=very thick,color=red,in=-90, out=225, looseness=0.75] (CR5) to (CR0);
		\draw [style=very thick,color=red,in=270, out=-45, looseness=0.75] (CR4) to (CR2);
		\draw [style=very thick,color=red,in=225, out=-90, looseness=1.25] (CR2) to (CR3);
		\draw [style=very thick,color=red,in=-45, out=-90, looseness=1.25] (CR0) to (CR1);
		\draw [style=very thick, color=red, in=270, out=315, looseness=3.25] (Link00e) to (Link01e);
		\draw [style=very thick,color=red, in=-90, out=-135, looseness=3.25] (Link02e) to (Link03e);
		\draw [style=very thick,color=red, bend right=135, looseness=17.25] (Link04e) to (Link05e);
		\draw [style=very thick, color=red, in=270, out=315, looseness=3.25] (Link10e) to (Link11e);
		\draw [style=very thick,color=red, in=-90, out=-135, looseness=3.25] (Link12e) to (Link13e);
		\draw [style=very thick,color=red, bend right=135, looseness=17.25] (Link14e) to (Link15e);
		\draw [style=very thick,color=red, bend left, looseness=1.25] (LH01e) to (LH00e);
		\draw  [style=very thick,color=red, bend right, looseness=1.25](LH02e) to (LH03e);	

		\draw[color=black, style=thick, dotted] (13.5,12) to (13.5,-22);
		\draw[color=black, style=thick, dotted] (-6,-8.5) to (35,-8.5);
		\draw[color=black, style=thick, dotted] (-6,-14) to (35,-14);

		\draw (-6,2) -- (-6,2) node [sloped, midway, above] {Type A};
		\draw (-6,-11.3) -- (-6,-11.3) node [sloped, midway, above] {Type B};
		\draw (-6,-18) -- (-6,-18) node [sloped, midway, above] {Type C};

	\end{pgfonlayer}
\end{tikzpicture}
\end{center}
\caption{The left (resp. right) column indicates the combinatorial types of $\secondf$-closed (resp. $\thirdf$-closed) curves.  
}
\label{F:closed-curves}
\end{figure}

%% file: 2nd-flip-Section3-submit-v9.tex
\section{Local description of the flips}
\label{section-local-vgit}

In this section, we give an \'{e}tale local description of the open immersions from Theorem \ref{T:algebraicity}
\[
\SM_{g,n}(\alpha_c+\epsilon) \hookrightarrow \SM_{g,n}(\alpha_c) \hookleftarrow \SM_{g,n}(\alpha_c\me)
\]
at each critical value $\alpha_c\in \{\third, \second, \first\}$. 

Roughly speaking, our main result says that, \'{e}tale locally around any closed point of
 $\SM_{g,n}(\alpha_c)$, these inclusions are induced by a variation of GIT problem. In Section \ref{section-local-quotient}, 
we develop the necessary background material on local quotient presentations and local VGIT in order to state our main result 
(Theorem \ref{theorem-etale-VGIT}).  In Section \ref{section-VGIT-basics}, 
we collect several basic facts concerning local variation of GIT which will be used in subsequent sections. 
In Section \ref{section-first-order}, 
we describe explicit coordinates on the formal miniversal deformation space of an $\alpha_c$-closed curve.
In Section \ref{section-variation-GIT}, we use these coordinates to compute the associated VGIT chambers and thus conclude 
the proof of Theorem \ref{theorem-etale-VGIT}.

\subsection{Local quotient presentations}
\label{section-local-quotient}
\begin{defn} \label{definition-etale-presentations}  \label{defn-local-quotient}
Let $\cX$ be an algebraic stack of finite type over $\Spec \CC$, and let $x \in \cX(\CC)$ be a closed point.
We say that $f\co \cW \to \cX$ is a \emph{local quotient presentation around $x$} if 
\begin{enumerate}
\item The stabilizer $G_x$ of $x\in \cX$ is reductive.
\item $\cW = [\Spec A \ / \ G_x]$, where $A$ is a finite type $\CC$-algebra.
\item $f$ is \'etale and affine.
\item There exists a point $w \in \cW$ such that $f(w)=x$ and $f$ induces an isomorphism $G_{w} \simeq G_x$. 
\end{enumerate}
We say that $\cX$ \emph{admits local quotient presentations} if there exist local quotient presentations around all closed points 
$x \in \cX(\CC)$. 
We sometimes write $f\co (\cW, w) \to (\cX,x)$ as a local quotient presentation to indicate the chosen preimage of $x$.
\end{defn}
It is an interesting (and still unsolved) problem to determine when an algebraic stack admits local quotient presentations.
Happily, the following result suffices for our purposes:

\begin{prop}\cite[Theorem 3]{alper_quotient}
\label{P:global-quotient}
Let $\cX$ be a normal algebraic stack of finite type over $\Spec \CC$ such that $\cX=[X/G]$ where 
$G$ is a connected algebraic group acting on a normal separated scheme $X$.
Then for any closed point $x \in \cX(\CC)$ with a reductive stabilizer, 
$\cX$ admits a local quotient presentation around $x$.
\end{prop}

\begin{corollary}\label{C:local-quotient}
For each $\alpha > \thirdme$, $\SM_{g,n}(\alpha)$ admits local quotient presentations.
\end{corollary}

\begin{proof}[Proof of Corollary \ref{C:local-quotient}]
By definition of $\alpha$-stability, each $\SM_{g,n}(\alpha)$ can be realized as $[X/G]$, 
where $X$ is a non-singular locally closed subvariety 
of the Hilbert scheme of some $\PP^N$ and $G=\PGL(N+1)$. By Proposition \ref{P:reductive-stabilizer},
stabilizers of $\alpha$-stable curves are reductive. 
Thus we can apply Proposition \ref{P:global-quotient}.
\end{proof}

Recall that if $G$ is a reductive group acting on an affine scheme $X=\Spec A$ 
by $\sigma \co G \times X \to X$, there is a natural correspondence between 
$G$-linearizations of the structure sheaf $\oh_{X}$ and characters $\chi\co G \to \GG_m=\Spec \CC[t,t^{-1}]$.  
Precisely, a character $\chi$ defines a $G$-linearization $\cL$ of the structure sheaf $\O_{X}$ as follows. 
The element $\chi^*(t) \in \Gamma(G, \oh_G^*)$ induces a $G$-linearization
$\sigma^* \oh_X \to p_2^* \oh_X$ defined by $p_1^* (\chi^*(t))^{-1} \in \Gamma(G \times X, \oh_{G \times X}^*)$.  
Therefore, we can associate to $\chi$ the semistable loci $X_{\cL}^{\ss}$ and $X_{\cL^{-1}}^{\ss}$ (cf. \cite[Definition 1.7]{git}).
The following definition describes explicitly the change in semistable locus as we move from 
$\chi$ to $\chi^{-1}$ in the character lattice of $G$. 
See \cite{thaddeus} and \cite{dolgachev-hu} for the general setup of variation of GIT.

\begin{defn}[VGIT chambers]\label{D:VGIT}
Let $G$ be a reductive group acting on an affine scheme $X=\Spec A$. Let $\chi\co G \rightarrow \GG_m$ be a character
and set $A_{n} := \{f \in A \mid \sigma^*(f) = \chi^*(t)^{-n} f\} = \Gamma(X, \cL^{\tensor n})^G$.
We define the \emph{VGIT ideals associated to $\chi$} to be:
\begin{align*}
I^{+}_{\chi} &:= (f \in A \mid f \in A_{n} \text{ for some } n > 0),\\
I^{-}_{\chi} &:= (f \in A \mid f \in A_{n} \text{ for some } n < 0).
\end{align*}
The \emph{VGIT $(+)$-chamber and $(-)$-chamber of $X$ associated to $\chi$} are the open subschemes 
\[
X_{\chi}^+:=X\setminus \mathbb{V}(I^+_{\chi}) \hookrightarrow X, \qquad X_{\chi}^-:=X \setminus \mathbb{V}(I^-_{\chi}) \hookrightarrow X.
\]

Since the open subsets $X_{\chi}^+$, $X_{\chi}^-$ are $G$-invariant, we also have stack-theoretic open immersions
\[
[X_{\chi}^+/G] \hookrightarrow [X/G] \hookleftarrow [X_{\chi}^-/G].
\]
We will refer to these open immersions as the \emph{VGIT $(+)$/$(-)$-chambers of $[X/G]$ associated to $\chi$}.
\end{defn}

\begin{remark}\label{R:affine-GIT}
For an alternative characterization of $X_{\chi}^+$, note that
$\chi^{-1}$ defines an action of $G$ on 
$X\times \AA^1$ via $g \cdot (x,s) = (g\cdot x, \chi(g)^{-1} \cdot s)$.  
Then $x \in X^+_{\chi}$ if and only if the orbit closure $\overline{G \cdot (x,1)}$ does not intersect the zero section $X \times \{0\}$.
\end{remark}

The natural inclusions of VGIT chambers induce projective morphisms of GIT quotients.
\begin{prop} \label{P:vgit-quotients}
Let $\cL$ be the $G$-linearization of the structure sheaf on $X$ corresponding to a character $\chi$.  
Then there are natural identifications of $X^+_{\chi}$ and $X^-_{\chi}$ with the semistable loci $X^{\ss}_{\cL}$ and $X^{\ss}_{\cL^{-1}}$,
respectively.  There is a commutative diagram
\[
\xymatrix{
X^+_{\chi} \ar[d] \ar@{^(->}[r]		& X \ar[d]		& X^-_{\chi} \ar@{_(->}[l] \ar[d]\\
X^+_{\chi}\gitq G := \Proj \bigoplus_{d \ge 0} A_{d} \ar[r]	& \Spec A_0	& \Proj \bigoplus_{d \ge 0} A_{-d} =: X^-_{\chi}\gitq G \ar[l]
}
\]
where $X \to \Spec A_0$, $X^+_{\chi} \to X^+_{\chi} \gitq G $ and $X^-_{\chi} \to X^-_{\chi}\gitq G $ are GIT quotients. 
The restriction of  $\cL$ to $X^+_{\chi}$ (resp., $\cL^{{-1}}$ to $X^-_{\chi}$) descends to line bundle $\oh(1)$ on $X^+_{\chi}\gitq G$ 
(resp., $\oh(1)$ on $X^-_{\chi}\gitq G$) relatively ample over $\Spec A_0$.  
In particular, for every point $x \in X^+_{\chi} \cup X^-_{\chi}$, the character of $G_x$ corresponding to $\cL|_{BG_x}$ is trivial.
\end{prop}
\begin{proof} 
This follows immediately from the definitions and \cite[Theorem 1.10]{git}.
\end{proof}

Next, we show how to use the data of a line bundle $\cL$ on a stack $\cX$ to define VGIT chambers associated to any local quotient presentation of $\cX$. In this situation, note that if $x \in \cX(\CC)$ is any point, then there is a natural action of the automorphism group $G_x$ on the fiber $\cL|_{BG_x}$ that induces a character $\chi_{\cL}\co G_x \to \GG_m$.

\begin{defn}[VGIT chambers of a local quotient presentation] \label{defn-plus-minus-quotient}
Let $\cX$ be an algebraic stack of finite type over $\Spec \CC$ and let $\cL$ be a line bundle on $\cX$.  
Let $x \in \cX$ be a closed point. 
If $f\co \cW \rightarrow \cX$ is a local quotient presentation around $x$, 
we define \emph{the chambers of $\cW$ associated to $\cL$}  
\[
\cW^+_{\cL} \hookrightarrow \W \hookleftarrow \cW^-_{\cL}
\]
to be the VGIT chambers associated to the character $\chi_{\cL}\co G_x \to \GG_m$. 
\end{defn}

\begin{defn}
\label{D:arise-from-VGIT}
Suppose $\cX$ is an algebraic stack of finite type over $\Spec CC$ that admits local quotient presentations
and $\cL$ is a line bundle on $\cX$. 
We say that open substacks $\cX^{+}$ and $\cX^{-}$ of $\cX$ \emph{arise from local VGIT with respect to $\cL$
at a point $x\in \cX$} if there exists a local quotient presentation $f\co \cW=[\Spec A \, / \,G_x] \ra \cX$ around $x$ such that  
$f^{*}\cL$ is the line bundle corresponding to the linearization of $\oh_{\Spec A}$ by $\chi_{\cL}$ and 
such that there is a Cartesian diagram:
\begin{equation} \label{diagram-etale-VGIT}
\xymatrix{
\cW^{+}_{\cL}  \ar[d] \ar@{^(->}[r]						&  \cW \ar[d]^f & \cW^{-}_{\cL}\ar@{_(->}[l] \ar[d]\\
 \cX^{+} \ar@{^(->}[r] 						& \cX & \cX^{-} \ar@{_(->}[l] 	.
}
\end{equation}
\end{defn}

The following key technical result allows to check that two given open substacks $\cX^{+}$ and $\cX^{-}$
arise from local VGIT with respect to a given line bundle $\cL$ on $\cX$ by working formally locally. 

\begin{prop}  \label{prop-ideals}
Let $\cX$ be a smooth algebraic stack of finite type over $\Spec \CC$ that admits local quotient presentations.  
Let  $\cL$ be a line bundle on $\cX$.  Let $\cX^+$ and $\cX^-$ be open substacks of $\cX$.  Let $x \in \cX$ be a closed point
and let $\chi\co G_x \to \GG_m$ be the character induced from the action of $G_x$ on the fiber of $\cL$ over $x$.  
Let $\TT^1(x)$ be the first-order deformation space of $x$, let $A=\CC[\TT^1(x)]$, 
and let $\hat{A} = \CC[[\TT^1(x)]]$ be the completion of $A$ at the origin.  
The affine space $T = \Spec A$ inherits an action of $G_x$.  Let $I_{\cZ^+}, I_{\cZ^-} \subseteq \hat{A}$ be the ideals defined by the reduced closed substacks $\cZ^+ = \cX \setminus \cX^+$ and $\cZ^- = \cX \setminus \cX^-$. 
Let $I^+ , I^- \subseteq A$ be the VGIT ideals associated to $\chi$.
If $I_{\cZ^+} = I^+ \hat{A}$ and $I_{\cZ^-} = I^- \hat{A}$, then $\cX^{+} \hookrightarrow \cX \hookleftarrow \cX^{-}$
arise from local VGIT with respect to $\cL$ at $x$. 
\end{prop} 

The proof Proposition \ref{prop-ideals} will be given in Section \ref{section-VGIT-basics}. 
We will now explain how this result is used in our situation.

On the stack $\SM_{g,n}(\alpha)$, there is a natural line bundle to use in conjunction with the VGIT formalism, namely $\delta-\psi$. 
Since this line bundle is defined over $\SM_{g,n}(\alpha)$ for each $\alpha$, there is an induced character 
$\chi_{\delta-\psi}\co \Aut(C, \pn) \rightarrow \GG_m$ for any $\alpha$-stable curve $(C, \pn)$. 

\begin{definition}[$I^+, I^-$] \label{defn-ideals}
If $(C, \pn)$ is an $\alpha_c$-closed curve, the affine space 
$$X=\Spec \CC[\TT^1(C, \pn)]$$
inherits an action of $\Aut(C, \pn)$, and  
we define $I^+$ and $I^-$ to be the VGIT ideals in $\CC[\TT^1(C, \pn)]$ 
associated to the character $\chi_{\delta-\psi}$.
\end{definition}

The main result of this section simply says that the VGIT chambers associated to $\delta-\psi$ locally cut out the inclusions 
$\SM_{g,n}(\alpha_c \pe) \hookrightarrow \SM_{g,n}(\alpha_c) \hookleftarrow \SM_{g,n}(\alpha_c \me)$.
\begin{theorem} 
\label{theorem-etale-VGIT}
Let $\alpha_c \in \{\third, \second, \first\}$.  
Then the open substacks 
\[
\SM_{g,n}(\alpha_c+\epsilon) \hookrightarrow \SM_{g,n}(\alpha_c) \hookleftarrow \SM_{g,n}(\alpha_c-\epsilon)
\]
arise from local VGIT with respect to $\delta-\psi$ 
at every closed point $(C, \pn)\in \SM_{g,n}(\alpha_c)$.
\end{theorem}

The remainder of Section \ref{section-local-vgit} is devoted to the proof of Theorem \ref{theorem-etale-VGIT}. 
In Section \ref{section-VGIT-basics}, we 
prove basic facts concerning the VGIT chambers defined above and, in particular, we prove Proposition \ref{prop-ideals}.  
In Section \ref{section-first-order}, we construct, for any $\alpha_c$-closed curve $(C,\pn)$, coordinates for $\hat{\Def}(C,\pn)$ and describe the ideals 
$I_{\cZ^+}$ and $I_{\cZ^-}$. In Section \ref{section-variation-GIT}, we use this coordinate description to compute the VGIT ideals $I^+$ and $I^-$.  
In Proposition \ref{prop-vgit-ideals} we prove that $I_{\cZ^+} = I^+ \hat{A}$ and $I_{\cZ^-} = I^- \hat{A}$, 
so that Theorem \ref{theorem-etale-VGIT} follows from Proposition \ref{prop-ideals}.

\subsection{Preliminary facts about local VGIT} \label{section-VGIT-basics}

In this section, we collect several basic facts concerning variation of GIT for the action of a reductive group on an affine scheme which will be needed in subsequent sections. In particular, we formulate a version of the Hilbert-Mumford criterion which will be useful for computing the  VGIT chambers associated to an $\alpha_c$-closed curve.

\begin{defn}
Recall that given a character $\chi\co G \to \GG_m$ and a one-parameter subgroup $\lambda\co \GG_m \to G$, 
the composition $\chi \circ \lambda\co \GG_m \to \GG_m$ is naturally identified with the integer $n$ such that $(\chi \circ \lambda)^*t = t^n$.  
We define the {\it pairing of $\chi$ and $\lambda$} as $\langle \chi, \lambda \rangle = n.$
\end{defn}
\begin{prop}[Affine Hilbert-Mumford criterion]
\label{prop-hilbert-mumford}
Suppose $G$ is a reductive group over $\Spec \CC$ acting on an affine scheme $X = \Spec A$ of finite type over $\Spec \CC$.  
Let $\chi\co G \to \GG_m$ be a character.  Let $x \in X(\CC)$.  Then $x \notin X^{+}_{\chi}$ (resp., $x \notin X^{-}_{\chi}$) 
if and only if there exists a one-parameter subgroup $\lambda\co \GG_m \to G$ with $\langle \chi, \lambda \rangle > 0$ 
(resp., $\langle \chi, \lambda \rangle < 0$) such that $\lim_{t \to 0} \lambda(t) \cdot x$ exists.
\end{prop}

\begin{proof}
Consider the action of $G$ on $X\times \AA^1$ induced by $\chi^{-1}$ as in Remark \ref{R:affine-GIT}.
Then $x\notin X^{+}_{\chi}$ if
and only if $\overline{G\cdot (x,1)}\cap( X\times \{0\} )\neq \emptyset$. 
By the Hilbert-Mumford criterion \cite[Theorem 2.1]{git},
this is equivalent to the existence of a one-parameter subgroup $\lambda\co \GG_m \to G$  such
$\lim_{t \to 0} \lambda(t) \cdot (x,1)  \in X\times \{0\}$. We are done by observing that 
$\lim_{t \to 0} \lambda(t) \cdot (x,1) = \lim_{t \to 0} (\lambda(t) \cdot x, t^{\langle \chi, \lambda \rangle}) \in X\times \{0\}$ if and only if $\lim_{t \to 0} \lambda(t) \cdot x$ exists
and $\langle \chi, \lambda \rangle > 0$.
\end{proof}
The following are three immediate corollaries of Proposition \ref{prop-hilbert-mumford}:
\begin{corollary} 
\label{lemma-vgit-product}
 Let $G_i$ be reductive groups acting on affine schemes $X_i$ of finite type over $\Spec \CC$ and $\chi_i \co G_i \to \GG_m$ be characters for $i=1, \ldots, n$.  Consider the diagonal action of $G = \prod_i G_i$ on $X = \prod_i X_i$ and the character $\prod_i \chi_i \co G \to \GG_m$.  Then 
$$\begin{aligned}
X \setminus X^{+}_{\chi} = & \bigcup_{i=1}^{n} X_1 \times \cdots \times (X_i \setminus (X_i)^{+}_{\chi_i}) \times \cdots \times X_n, \\
X \setminus X^{-}_{\chi} = & \bigcup_{i=1}^{n} X_1 \times \cdots \times (X_i \setminus (X_i)^{-}_{\chi_i}) \times \cdots \times X_n.
\end{aligned}$$ \epf
\end{corollary}


\begin{corollary}
\label{lemma-vgit-closed}
Let $G$ be a reductive group over $\Spec \CC$ acting on an affine $X = \Spec A$ of finite type over $\Spec \CC$.  Let $\chi \co G \to \GG_m$ be a character.  Let $Z \subseteq X$ be a $G$-invariant closed subscheme.  Then $Z^{+}_{\chi} = X^{+}_{\chi} \cap Z$ and $Z^{-}_{\chi} = X^{-}_{\chi} \cap Z$. \epf
\end{corollary}


\begin{corollary} \label{lemma-vgit-identity}
Let $G$ be a reductive group with character $\chi \co G \to \GG_m$.  Suppose $G$ acts on an affine scheme $X=\Spec A$ of finite type over $\Spec \CC$.  Let $G^{\circ}$ be the connected component of the identity and $\chi^{\circ} = \chi|_{G^{\circ}}$.  Then the VGIT chambers $X^+_{\chi}, X^-_{\chi}$ for the action of $G$ on $X$ are equal to the VGIT chambers $X^+_{\chi^{\circ}}, X^-_{\chi^{\circ}}$ for action of $G^{\circ}$ on $X$. \epf
\end{corollary}


\begin{prop} \label{P:specializations}
Let $G$ be a reductive group acting on an affine variety $X$ of finite type over $\Spec \CC$.  Let $\chi\co G \to \GG_m$ be a non-trivial character.  
Let $\lambda\co \GG_m \to G$ be a one-parameter subgroup and $x \in X^-_{\chi}(\CC)$ such that 
$x_0 = \lim_{t \to 0} \lambda(t) \cdot x \in X^G$ is fixed by $G$.  Then $\langle \chi, \lambda \rangle > 0$.
\end{prop}
\begin{proof}
As $x \in X_{\chi}^-$,  $\langle \chi, \lambda \rangle \ge 0$ by Proposition \ref{prop-hilbert-mumford}.  
Suppose $\langle \chi, \lambda \rangle = 0$. Considering the action of $G$ on $X\times \AA^1$ induced by $\chi$ as in Remark \ref{R:affine-GIT},
then $\lim_{t\to 0} \lambda(t)\cdot (x,1)=(x_0, 1)\in X^G\times \AA^1$.
But $X^{G}$ is contained in the unstable locus $X \setminus X^-_{\chi}$ since
$\chi$ is a nontrivial linearization.  It follows that $\overline{G\cdot (x,1)} \cap (X^G \times \{ 0 \}) \neq \emptyset$
which contradicts $x \in X^-_{\chi}$.
\end{proof}

\begin{lemma}  \label{lemma-vgit-pullback}
Let $G$ be a reductive group with character $\chi\co G \to \GG_m$
and $h\co \Spec A=X \to Y=\Spec B$ be a $G$-invariant morphism of affine schemes finite type over $\Spec \CC$.  Assume that $A = B \otimes_{B^G} A^G$.
Then $h^{-1}(Y^+_{\chi}) = X^+_{\chi}$ and $h^{-1}(Y^-_{\chi}) = X^-_{\chi}$.
\end{lemma}

\begin{proof} We use Proposition \ref{prop-hilbert-mumford}.  If $x \notin X^+_{\chi}$, then there exists 
$\lambda\co \GG_m \to G$ with $\langle \chi, \lambda \rangle > 0$ such that $x_0=\lim_{t \to 0} \lambda(t) \cdot x$ exists.  
It follows that $h(x_0) = \lim_{t \to 0} \lambda(t) \cdot h(x)$ exists, and so $h(x) \notin Y^+_{\chi}$. We conclude that 
$h^{-1}(Y^+_{\chi}) \subseteq X^+_{\chi}$. Conversely, suppose $h(x) \notin Y^+_{\chi}$. 
Then there exists $\lambda\co \GG_m \to G$ with $\langle \chi, \lambda \rangle > 0$ such that $ \lim_{t \to 0} \lambda(t) \cdot h(x)$ exists.   
Since $\lim_{t \to 0} \lambda(t) \cdot h(x)$ exists and since  both 
$\Spec A \to \Spec A^G$ and $\Spec B \to \Spec B^G$ are GIT quotients, there is a commutative diagram
$$\xymatrix{
\Spec \CC[t] \ar@/^/[drr] \ar@/_/[ddr] \ar@{-->}[dr]		\\
	& \Spec A \ar[r]^h \ar[d] 	& \Spec B \ar[d] \\
	&\Spec A^G \ar[r]			& \Spec B^G 
}$$
Since the square is Cartesian, the map $\GG_m=\Spec \CC[t,t^{-1}] \to \Spec A$ given by $t \mapsto \lambda(t) \cdot x$ extends to $\Spec \CC[t] \to \Spec A$. 
It follows that $x \notin X^+_{\chi}$. We conclude that $ X^+_{\chi} \subseteq h^{-1}(Y^+_{\chi})$.  
\end{proof}

In order to prove Proposition \ref{prop-ideals}, we need the following Lemma.
\begin{lemma} \label{lemma-vgit-comparison}
Let $G$ be a reductive group acting on a smooth affine variety  $W=\Spec A$ over $\Spec \CC$.  
Let $w \in W$ be a fixed point of $G$.  Let $\chi\co G \to \GG_m$ be a character.  
There is a Zariski-open affine neighborhood $W' \subseteq W$ containing $w$ and a $G$-invariant \'etale morphism 
$h\co W' \to T=\Spec \CC[T_{W,w}]$, where $T_{W,w}$ is the tangent space at $w$,  such that 
\[
h^{-1}(T^+_{\chi}) = W'^+_{\chi} \qquad h^{-1}(T^+_{\chi}) = W'^+_{\chi} \, .
\]
\end{lemma}

\begin{proof}
The maximal ideal $\fm \subseteq A$ of $w \in W$ is $G$-invariant.  Since $G$ is reductive, 
there exists a splitting $\fm/\fm^2 \hookarr \fm$ of the surjection $\fm \to \fm/\fm^2$ of $G$-representations. 
The inclusion $\fm/\fm^2 \hookarr \fm \subseteq A$ induces a morphism on algebras 
 $\Sym^* \fm / \fm^2 \to A$ which is $G$-equivariant which in turns gives a $G$-equivariant morphism 
 $h\co \Spec A \to  T$
\'etale at $w \in W$.  
By applying Luna's Fundamental Lemma (see \cite{luna}), there exists a $G$-invariant open affine 
$W'=\Spec A' \subseteq \Spec A$ containing $w$ such that the diagram
$$\xymatrix{
\Spec A' \ar[r] \ar[d]		& \Spec \CC[T_{W,w}] \ar[d] \\
\Spec A'^G \ar[r]			& \Spec  \CC[T_{W,w}]^G
}$$
is Cartesian with $\Spec A'^G \to \Spec \CC[T_{W,w}]^G$ \'etale.  
From Lemma \ref{lemma-vgit-pullback}, 
the induced map $h|_{W'}\co  W' \to T$ satisfies $h|_{W'}^{-1}(T^+_{\chi}) = W'^+_{\chi}$ and $h|_{W'}^{-1}(T^+_{\chi}) = W'^+_{\chi}$.  
\end{proof}

\begin{proof}[Proof of Proposition \ref{prop-ideals}]
Let $f\co \cW = [W / G_x] \to \cX$ be an \'etale local quotient presentation around $x$ where $W = \Spec A$.  By Lemma \ref{lemma-vgit-comparison}, after shrinking $\cW$, we may assume that there is an induced $G_x$-invariant morphism $h\co W \to T = \Spec \CC[\TT^1(x)]$ such that $h^{-1}(T^+_{\chi}) = W^+_{\chi}$ and $h^{-1}(T^+_{\chi}) = W^+_{\chi}$.  This provides a diagram
\[
\xymatrix{
\Spf \hat{A} \ar[r] 	& \cW = [\Spec A / G_x] \ar[d] \ar[rd]^{h} \\
& \cX 	& \cY = [\Spec \CC[\TT^1(x)] / G_x] 
}
\]
In particular, $I^+ A$ and $I^- A$ are the VGIT ideals in $A$ corresponding to $(+)/(-)$ VGIT chambers.   Since $I^+ \hat{A} = I_{\cZ^+}$ and $I^- \hat{A} = I_{\cZ^-}$, it follows that the ideals defining $\cZ^+, \cZ^-$ and $\cW \setminus \cW^+_{\chi}, \cW \setminus \cW^-_{\chi}$ must agree in a Zariski-open neighborhood $U \subseteq \Spec A$ of $w$.  By shrinking further, we may also assume that the pullback of $\cL$ to $U$ is trivial.  By Lemma \ref{lemma-shrinking}, 
we may assume that $U$ is affine scheme such that $\pi^{-1}(\pi(U)) = U$ where $\pi\co \Spec A \to \Spec A^G$.  If we set $\cU = [U/G]$, then the composition $\cU \hookrightarrow \cW \to \cX$ is a local quotient presentation.
By applying Lemma \ref{lemma-vgit-comparison}, $\cU^+ = \cW^+ \cap \cU$ and $\cU^- = \cW^- \cap \cU$ so that in $\cU$ the ideals defining $\cZ^+, \cZ^-$ and $\cU \setminus \cU^+, \cU \setminus \cU^-$ agree.  Moreover, the pullback of $\cL$ to $\cU$ is clearly identified with the linearization of $\oh_U$ by $\chi$. Therefore, $\cU \to \cX$ has the desired properties.
\end{proof}

\subsection{Deformation theory of $\alpha_c$-closed curves} \label{section-first-order}

Our goal in this section is to describe coordinates on the formal deformation space of an $\alpha_c$-closed curve $(C, \pn)$
in which the ideals $I_{\cZ^+}$ and $I_{\cZ^-}$ can be described explicitly, 
and which simultaneously diagonalize the natural action of $\Aut(C, \pn)$. We begin by 
describing the action of $\Aut(E)$ on the space of first-order deformations $\TT^1(E)$ of a single $\alpha_c$-atom $E$ 
(Lemma \ref{lem-deformation-atoms}) and a single rosary of length $3$ (Lemma \ref{lem-deformation-rosary}).  
Then we describe the action of $\Aut(C, \pn)$ on the first-order deformation space $\TT^1(C, \pn)$ 
for each combinatorial type of an $\alpha_c$-closed curve $(C, \pn)$
from Definition \ref{defn-canonical} (Proposition \ref{prop-first-order-action}). 
Finally, we pass from coordinates on the first-order deformation space to coordinates 
on the formal deformation space $\hat{\Def}(C, \pn)$ (Proposition \ref{prop-formal-ideals}).

Throughout this section, we let $\TT^1(C, \pn)$ denote the first-order deformation space of $(C, \pn)$ and $\TT^1(\hat{\oh}_{C, \xi})$ the first-order deformation space of a singularity $\xi \in C$. Finally, we let  $\Aut(C, \pn)^{\circ}$ denote the connected component of the identity of the automorphism group of $(C, \pn)$.  We sometimes write $\TT^1(C)$ (resp., $\Aut(C)$) for $\TT^1(C, \pn)$ (resp., $\Aut(C, \pn)$) if no confusion is likely.

\subsubsection{Action on the first-order deformation space for an $\alpha_c$-atom and rosary}
\label{S:Gm-action}

Suppose $(E, q)$ (resp., $(E, q_1, q_2)$) is an $\alpha_c$-atom (see Definition \ref{defn-atom}) with singular point $\xi \in E$.
By \eqref{gm-action-atoms}, 
we may fix an isomorphism $\Aut(E) \simeq  \GG_m = \Spec \CC[t,t^{-1}]$ and coordinates on $\hat{\oh}_{E,\xi}$ and $\hat{\oh}_{E,q}$ 
(resp.,  $\hat{\oh}_{E,q_1}$ and $\hat{\oh}_{E,q_2}$)  so that the action of $\Aut(E)$ is specified by: 

\noindent {\rm $\bullet \, \alpha_c=\first$: }
$\hat{\oh}_{E,\xi} \simeq \CC[[x,y]]/(y^2-x^3)$, $\hat{\oh}_{E,q} \simeq \CC[[n]]$, and $\GG_m$ acts by 
\begin{align*}
x \mapsto t^{-2}x,\, y \mapsto t^{-3} y,\, n \mapsto t n.
\end{align*}

\noindent {\rm $\bullet \, \alpha_c=\second$: } 
$\hat{\oh}_{E,\xi} \simeq \CC[[x,y]]/(y^2-x^4)$, $\hat{\oh}_{E,q_1} \simeq \CC[[n_1]]$, $\hat{\oh}_{E,q_2} \simeq \CC[[n_2]]$, and $\GG_m$ acts by 
\begin{align*}
x \mapsto t^{-1}x,\, y \mapsto t^{-2} y,\, n_1 \mapsto t n_1,\, n_2 \mapsto t n_2.
\end{align*}

\noindent {\rm $\bullet \, \alpha_c=\third$: }
$\hat{\oh}_{E,\xi} \simeq \CC[[x,y]]/(y^2-x^5)$, $\hat{\oh}_{E,q} \simeq \CC[[n]]$ and $\GG_m$, acts by 
\begin{align*}
x \mapsto t^{-2}x,\, y \mapsto t^{-5}y,\, n \mapsto t n.
 \end{align*}

We have an exact sequence of $\Aut(E)$-representations
\begin{equation*}
\begin{aligned}
0 \arr \Cr^1(E)   \xarr{\alpha}
        \TT^1(E)   \xarr{\beta}
       \TT^1(\hat{\cO}_{E,\xi}) \arr  0 \\
\end{aligned}
\end{equation*}
where   
$\Cr^1(E)$  denotes the space of first-order deformations which induces trivial deformations of $\xi$. 
In fact, since the pointed normalization of $E$ has no non-trivial deformations, 
we may identity $\Cr^1(E)$ with the space of crimping deformations, 
i.e., deformations which fix the pointed normalization and the analytic isomorphism type of the singularity. 
Note that in the cases $\alpha_c=\first$ and $\alpha_c=\second$, $\Cr^1(E)=0$, 
i.e., there is a unique way to impose a cusp on a $2$-pointed rational curve (resp., a tacnode on a pair of $2$-pointed rational curves).
 
\begin{lem} \label{lem-deformation-atoms}  Let $E$ be an $\alpha_c$-atom.  Fix  $\Aut(E) \simeq \GG_m$ as above.

\noindent {\rm $\bullet \, \alpha_c=\first$:}
$\TT^1(E) \simeq \TT^1(\hat{\cO}_{E,\xi})$ and there are coordinates $s_0, s_1$ on $\TT^1(\hat{\cO}_{E,\xi})$  with weights $-6, -4$.

\noindent {\rm $\bullet \, \alpha_c=\second$:}
$\TT^1(E) \simeq \TT^1(\hat{\cO}_{E,\xi})$ and there are coordinates $s_0, s_1, s_2$ on $\TT^1(\hat{\cO}_{E,\xi})$ with weights $-4,-3,-2$.

\noindent {\rm $\bullet \, \alpha_c=\third$:}
$\TT^1(E) \simeq \Cr^1(E) \oplus  \TT^1(\hat{\cO}_{E,\xi})$ and there are coordinates $c$ on $\Cr^1(E)$ and $s_0, s_1, s_2, s_3$ on $\TT^1(\hat{\cO}_{E,\xi})$  with with weights $1$ and $-10, -8, -6, -4$, respectively.
\end{lem}

\begin{proof}  We prove the case $\alpha_c = \third$ and leave the other cases to the reader. 
By deformation theory of hypersurface singularities, we have
\[
\TT^1(\hat{\cO}_{E,\xi}) \iso \CC^4, \quad \Spec \CC[[x,y,\eps]]/
(y^2-x^{5}-s_{3}^*\eps x^{3}-s_{2}^*\eps x^{2} - s_1^* \eps x - s_0^* \eps, \eps^2) 
\mapsto (s_0^*,s_1^*, s_2^*, s_3^*),
\]
and $\GG_m$ acts by  $s_k^* \mapsto t^{10-2k} s_k^*$. Thus, $\GG_m$ acts on $\TT^1(\hat{\cO}_{E,\xi})^{\vee}$ by $s_k \mapsto t^{2k-10} s_k$.  

From \cite[Example~1.111]{vdW}, we have
\[
\Cr^1(E) \iso \CC, \quad \Spec \CC[(s+c^* \eps s^2)^2, (s+c^* \eps s^2)^5,\eps]/(\eps)^2 \mapsto c^*,
\]
and $\GG_m$ acts by $c^* \rightarrow t^{-1}c^*$.  Thus, $\GG_m$ acts on $\Cr^1(E)^{\dual}$ by
$c \mapsto t c $.
\end{proof}

Now let $(R, p_1, p_2) = \coprod_{i=1}^3 (R_i, q_{2i-1}, q_{2i})$ be a rosary of length $3$ (see Definition \ref{defn-rosary}).  
Denote the tacnodes of $R$ as $\tau_1 :=q_2 = q_3$ and $\tau_2 :=q_4=q_5$.  We fix an isomorphism 
$\Aut(R, p_1, p_2) \simeq \GG_m = \Spec \CC[t,t^{-1}]$ such that $\GG_m$ acts on $\hat{\oh}_{R,\tau_i} = \CC[[x_i,y_i]]/(y_i^2-x_i^4)$ 
via $x_1 \mapsto t^{-1} x_1, y_1 \mapsto t^{-2} y_1$ and  $x_2 \mapsto t x_2, y_2 \mapsto t^{2} y_2$, 
and acts on $\hat{\oh}_{R,p_i} = \CC[[n_i]]$ via $n_1 \mapsto t n_1$ and $n_2 \mapsto t^{-1} n_2$.

\begin{lemma} \label{lem-deformation-rosary}  Let $(R,p_1,p_2)$ be a rosary of length $3$.  Fix  $\Aut(R, p_1, p_2) \simeq \GG_m$ as above.  
Then $\TT^1(R, p_1, p_2) = \TT^1(\hat{\cO}_{R, \tau_1}) \oplus  \TT^1(\hat{\cO}_{R, \tau_2})$ and
there are coordinates on $\TT^1(\hat{\cO}_{R,\tau_1})$ (resp., $\TT^1(\hat{\cO}_{R,\tau_2})$)  with weights $-2, -3, -4$ (resp., $2,3,4$).
\end{lemma}

\begin{proof}  This is established similarly to Lemma \ref{lem-deformation-atoms}.
\end{proof}

The above lemmas immediately imply a description for the action of $\Aut(C, \pn)^{\circ}$ on $\TT^1(C, \pn)$ for any $\alpha_c$-closed curve. 

\begin{prop}[Diagonalized Coordinates on $\TT^1(C, \pn)$]\label{prop-first-order-action}
Let $(C, \pn)$ be an $\alpha_c$-closed curve.  Depending on the combinatorial
type of $(C, \pn)$ from Definition \ref{defn-canonical}, the following statements hold:

\noindent
{\rm $\bullet \, \alpha_c= \first$ of Type A:}
There are decompositions
\begin{align*}
\Aut(C)^{\circ} &= \prod_{i=1}^r \Aut(E_i) 
\TT^1(C) & = \TT^1(K)
         \oplus \left[\bigoplus_{i=1}^r 
         \TT^1(E_i)\right]
         \oplus \left[\bigoplus_{i=1}^r \TT^1(\hat{\cO}_{C,q_i})\right]
\end{align*}
For $1 \leq i \leq r$, let $t_i$ be the coordinate on
$\Aut(E_i) \simeq \GG_m$. There
are coordinates 
\[
\begin{array}{lllll}
\text{``{\bf s}ingularity''}	&\bs_i=(s_{i,0},s_{i,1}) & \text{on} & \TT^1(\hat{\cO}_{E_i,\xi_i}) & \text{for $1 \leq i \leq r$} \\
\text{``{\bf n}ode''}	&
n_i & \text{on} & \TT^1(\hat{\cO}_{C,q_i}) &
 \text{for $1 \leq i \leq r$}
\end{array}
\]
such that $\Aut(C)^{\circ}$ acts trivially on $\TT^1(K)$ and on the coordinates $\bs_i, n_i$ by
\[
s_{i,0}  \mapsto  t_i^{-6} s_{i,0} \qquad
s_{i,1}  \mapsto  t_i^{-4} s_{i,1}\qquad
n_i   \mapsto  t_i n_i.
\]

\noindent  {\rm $\bullet \, \alpha_c =\first$ of Type B: }
There are decompositions
\begin{align*}
\Aut(C)^{\circ} = \Aut(E_1)\times \Aut(E_2) \qquad \quad 
\TT^1(C)  =  \TT^1(E_1)\oplus \TT^1(E_2) \oplus \TT^1(\hat{\cO}_{C,q})
\end{align*}
For $1 \leq i \leq 2$, let $t_i$ be the coordinate on
$\Aut(E_i) \simeq \GG_m$. There are coordinates
$\bs_i=(s_{i,0},s_{i,1})$ 
on $\TT^1(E_i)$  and a coordinate $n$ on $\TT^1(\hat{\oh}_{C,q})$ such that the
action of $\Aut(C)^{\circ}$ on $\TT^1(C)$ is
given by 
\[
s_{i,0}  \mapsto  t_i^{-6} s_{i,0} \qquad
s_{i,1}  \mapsto  t_i^{-4} s_{i,1} \qquad
n \mapsto t_1 t_2 n.
\]

\noindent 
{\rm $\bullet \, \alpha_c = \first$ of Type C: }
This case is described in Lemma \ref{lem-deformation-atoms}.

\noindent {\rm $\bullet \, \alpha_c = \second$ of Type A: } There are decompositions
\begin{align*}
\Aut(C)^{\circ} &=  \prod_{i=1}^{r+s} \prod_{j=1}^{\ell_i} \Aut(E_{i,j})\\
\TT^1(C) = & \TT^1(K)
         \oplus \bigoplus_{i=1}^{r+s}\left[  \bigoplus_{j=1}^{\ell_{i}}
         \TT^1(E_{i,j})  \oplus  \bigoplus_{j=0}^{\ell_i-1} \TT^1(\hat{\cO}_{C,q_{i,j}})    \right] \oplus
       \bigoplus_{i=1}^{r} \TT^1(\hat{\cO}_{C,q_{i,\ell_i}})
\end{align*}
Let $t_{i,j}$ be the coordinate on
$\Aut(E_{i,j}) \simeq \GG_m$.
  There are coordinates
\[
\begin{tabular}{lllll}

``{\bf s}ingularity''	&$\bs_{i,j}=\bigl(s_{i,j,k}\bigr)_{k=0}^{2}$
& on & $\TT^1(E_{i,j})$  \quad \, &
 $1 \leq i \leq r+s$, $1 \leq j \leq \ell_{i}$ \\
``{\bf n}ode''	&$n_{i,j}$ & \text{on} & $\TT^1(\hat{\cO}_{C,q_{i,j}})$ &
$1 \leq i \leq r+s$, $0 \leq j \leq \ell_{i}-1$\\
``{\bf n}ode''	&$n_{i,\ell_i}$ & \text{on} & $\TT^1(\hat{\cO}_{C,q_{i,\ell_i}})$ &
$1 \leq i \leq r$\\
\end{tabular}
\]
such that $\Aut(C)^{\circ}$ acts trivially on $\TT^1(K)$ and on $\bs_{i,j}, n_{i,j}$ by
\[
\begin{array}{lllllllllll}
s_{i,j,k} &  \mapsto & t_{i,j}^{k-4} s_{i,j,k} &
\\
n_{i,0} & \mapsto & t_{i,1} n_{i,0} & \quad
n_{i,\ell_i} & \mapsto &  t_{i,\ell_i} n_{i,\ell_i} & \quad
n_{i,j}  &\mapsto  &t_{i,j} t_{i,j+1} n_{i,j} \,\,(j \neq 0, \ell_i).
\end{array}
\]

\noindent {\rm $\bullet \, \alpha_c = \second$ of Type B: } There are decompositions
$$\begin{aligned}
\Aut(C, p_1, p_2)^{\circ} & = \prod_{i=1}^g \Aut(E_i) \\
\TT^1(C,p_1,p_2) & = \bigoplus_{i=1}^g
    \TT^1(E_i)
  \oplus \bigoplus_{i=1}^{g-1} \TT^1(\hat{\cO}_{C,q_{i}})
\end{aligned}$$
Let $t_{i}$ be the coordinate on
$\Aut(E_{i}) \simeq \GG_m$. There are coordinates
$\bs_i=(s_{i,0},s_{i,1},s_{i,2})$ 
on $\TT^1(E_i)$ and coordinates $n_i$ on $\TT^1(\hat{\oh}_{C,q_{i}})$ such that the action of $\Aut(C,\pn)^{\circ}$ on $\TT^1(C,\pn)$ is given by
\[
s_{i,k} \mapsto t_i^{k-4} s_{i,k} \qquad
n_i \mapsto t_i t_{i+1} n_i.
\]

\noindent {\rm $\bullet \, \alpha_c = \second$ of Type C: } There are decompositions
$$\begin{aligned}
\Aut(C)^{\circ} &= \prod_{i=1}^{g-1} \Aut(E_i)\\
\TT^1(C) &= \bigoplus_{i=1}^{g-1}
    \TT^1(E_i)
  \oplus \bigoplus_{i=0}^{g-2} \TT^1(\hat{\cO}_{C,q_{i}})\end{aligned}$$
Let $t_{i}$ be the coordinate on
$\Aut(E_{i}) \simeq \GG_m$. There are coordinates
$\bs_i=(s_{i,0},s_{i,1}, s_{i,2})$ 
on $\TT^1(E_i)$ and coordinates $n_i$ on $\TT^1(\hat{\oh}_{C,q_{i}})$ such that the action of $\Aut(C,\pn)^{\circ}$ on $\TT^1(C,\pn)$ is given by
\[
s_{i,k} \mapsto t_i^{k-4} s_{i,k} \qquad
n_i \mapsto t_i t_{i+1} n_i,
\]
and where $t_{0}:=t_{g-1}$.\\

\noindent
{\rm $\bullet \, \alpha_c = \third$ of Type A: }
There exist decompositions
\begin{align*}
\Aut(C,\pn)^{\circ} &= \Aut(K')^{\circ} \times \prod_{i=1}^r \Aut(L_{i})  \\
&= \Aut(K')^{\circ} \times \prod_{i=1}^r \left[  \prod_{j=1}^{\ell_i-1} \Aut(R_{i,j}) \times \Aut(E_{i}) \right]\\
\TT^1(C, \pn) & = \TT^1(K')
         \oplus \bigoplus_{i=1}^r  \TT^1(L_{i}) \bigoplus_{i=1}^{r}  \TT^1(\hat{\cO}_{C,q_{i,0}})  \\
         & = \TT^1(K')
         \oplus \bigoplus_{i=1}^r \left[ \bigoplus_{j=1}^{\ell_i-1}   \TT^1(R_{i,j})
         \oplus  \bigoplus_{j=0}^{\ell_i-1}  \TT^1(\hat{\cO}_{C,q_{i,j}}) 
         \oplus \TT^1(E_i) \right]
\end{align*}
where $\Aut(K')^{\circ}$ acts trivially on $\bigoplus_{i=1}^r  \TT^1(L_{i}) \bigoplus_{i=1}^{r}  \TT^1(\hat{\cO}_{C,q_{i,0}})$ and $\prod_{i=1}^r \Aut(L_{i})$ acts trivially on $\TT^1(K')$. For $1 \leq i \leq r$, $1 \leq j \leq \ell_i-1$, let $t_{i,j}$ denote the coordinate on
$\Aut(R_{i,j}) \simeq \GG_m$, and let $t_{i,\ell_i}$ denote the coordinate on $\Aut(E_i) \simeq \GG_m$. Then there exist coordinates
\[
\begin{array}{lllll}
\text{``{\bf r}osary''}	&\br_{i,j}=(r_{i,j,k})_{k=0}^2, \, \br'_{i,j}=(r'_{i,j,k})_{k=0}^2& \text{on} & \TT^1(R_{i,j}) & \text{for $1 \leq i \leq r , 1 \le j < \ell_i$} \\
\text{``{\bf s}ingularity''}		&\bs_{i}=(s_{i,k})_{k=0}^3 & \text{on} & \TT^1(\hat{\oh}_{C, \xi_{i}}) \subset \TT^1(E_i) & \text{for $1 \leq i \leq r $} \\

\text{``{\bf c}rimping''}	&c_i	 & \text{on} & \Cr^1(E_{i}) \subset \TT^1(E_i) & \text{for $1 \leq i \leq r$} \\
\text{``{\bf n}ode''}	&
n_{i,j} & \text{on} & \TT^1(\hat{\cO}_{C,q_{i,j}}) &
 \text{for $1 \leq i \leq r, 0 \leq j < \ell_i$}
\end{array}
\]
such that the action of $\prod_{i=1}^r \Aut(L_{i})$ on $\bigoplus_{i=1}^r  \TT^1(L_{i})$ is given by
$$\begin{array}{llllllllllllll}
r_{i,j,k} &  \mapsto & t_{i,j}^{k-4} r_{i,j,k} & \qquad r'_{i,j,k} &  \mapsto & t_{i,j}^{4-k} r'_{i,j,k} & \qquad 
s_{i,k}  &\mapsto&  t_{i,\ell_i}^{2k-10} s_{i,k}\\
c_i  &\mapsto & t_{i,\ell_i} c_i &
\qquad n_{i,0} & \mapsto & t_{i,1} n_{i,0} &
\qquad n_{i,j}  &\mapsto  &t_{i,j}^{-1} t_{i,j+1} n_{i,j} \quad (0 < j < \ell_i).
\end{array}
$$
Note that we need not specify the action of $\Aut(K')^{\circ}$ on $\TT^{1}(K')$ as this will be irrelevant for the calculation of the VGIT chambers associated to $(C, \pn)$.

\noindent
{\rm $\bullet \, \alpha_c = \third$ of Type B: }
There exist decompositions
\begin{align*}
\Aut(C,\pn)^{\circ} &= \prod_{i=1}^{\ell-1} \Aut(R_{i}) \times \Aut(E_\ell)\\
\TT^1(C, \pn) & = \bigoplus_{i=1}^{\ell-1} \left[  \TT^1(R_{i})
         \oplus   \TT^1(\hat{\cO}_{C,q_{i}}) \right] \oplus \TT^1(E_\ell)
\end{align*}
For $1 \leq i \leq \ell-1$, let $t_{i}$ be the coordinate on
$\Aut(R_{i}) \simeq \GG_m$, and let $t_\ell$ be the coordinate on $\Aut(E_\ell) \simeq \GG_m$. Then there
are coordinates 
\[
\begin{array}{lllll}
\text{``{\bf r}osary''}	&\br_{i}=(r_{i,k})_{k=0}^2 , \, \br'_{i}=(r'_{i,k})_{k=0}^2& \text{on} & \TT^1(R_{i}) & \text{for $1 \leq i \leq \ell-1$} \\
\text{``{\bf s}ingularity''}		&\bs=(s_{k})_{k=0}^3 & \text{on} & \TT^1(\hat{\oh}_{C, \xi}) \subset \TT^1(E_\ell) \\

\text{``{\bf c}rimping''}	&c	 & \text{on} & \Cr^1(E_{\ell}) \subset \TT^1(E_\ell) \\
\text{``{\bf n}ode''}	&
n_{i} & \text{on} & \TT^1(\hat{\cO}_{C,q_{i}}) &
 \text{for $1 \leq i \leq \ell-1$}
\end{array}
\]
such that the
action of $\Aut(C)^{\circ}$ on $\TT^1(C)$ is
given by
$$\begin{array}{llllllllllllll}
r_{i,k} &  \mapsto & t_{i}^{k-4} r_{i,k} & \qquad r'_{i,k} &  \mapsto & t_{i}^{4-k} r'_{i,k} & 
s_{k}  &\mapsto&  t_{\ell}^{2k-10} s_{k}\\
c  &\mapsto & t_{\ell} c &
\qquad n_{i}  &\mapsto  &t_{i}^{-1} t_{i+1} n_{i} \,\, (0 \le i < \ell).
\end{array}
$$

\noindent
{\rm $\bullet \, \alpha_c = \third$ of Type C: }
There exist decompositions
\begin{align*}
\Aut(C)^{\circ} &=\Aut(E_0) \times \Aut(E_\ell) \times \prod_{i=1}^{\ell-1} \Aut(R_{i})\\
\TT^1(C) & = \TT^1(E_0) \oplus \TT^1(E_\ell) \oplus \bigoplus_{i=1}^{\ell-1}   \TT^1(R_{i})
         \oplus \bigoplus_{i=0}^{\ell-1}   \TT^1(\hat{\cO}_{C,q_{i}})
\end{align*}
Let $t_0, t_\ell$ be coordinates on $\Aut(E_0)\simeq \GG_m$ and $\Aut(E_\ell)\simeq \GG_m$, and for $0 \leq i \leq \ell$, 
let $t_{i}$ be the coordinate on
$\Aut(R_{i}) \simeq \GG_m$. Then there
are coordinates 
\[
\begin{array}{lllll}
\text{``{\bf r}osary''}	&\br_{i}=(r_{i,k})_{k=0}^2 , \, \br'_{i}=(r'_{i,k})_{k=0}^2& \text{on} & \TT^1(R_{i}) & \text{for $1 \leq i \leq \ell-1$} \\
\text{``{\bf s}ingularity''}		&\bs_{i}=(s_{i,k})_{k=0}^3 & \text{on} & \TT^1(\hat{\oh}_{C, \xi_{i}}) \subset \TT^1(E_i) & \text{for $i=0, \ell$} \\

\text{``{\bf c}rimping''}	&c_i	 & \text{on} & \Cr^1(E_{i}) \subset \TT^1(E_i) & \text{for $i=0, \ell$} \\
\text{``{\bf n}ode''}	&
n_{i} & \text{on} & \TT^1(\hat{\cO}_{C,q_{i}}) &
 \text{for $0 \leq i \leq \ell-1$}
\end{array}
\]
such that the
action of $\Aut(C)^{\circ}$ on $\TT^1(C)$ is
given by
$$\begin{array}{llllllllllllll}
r_{i,k} &  \mapsto & t_{i}^{k-4} r_{i,k} & \ r'_{i,k} &  \mapsto & t_{i}^{4-k} r'_{i,k} & \
s_{i,k}  &\mapsto&  t_{i}^{2k-10} s_{i,k}\\
c_i  &\mapsto & t_{i} c_i & \
n_{0} & \mapsto & t_{0}t_{1} n_{0} & \
n_{i}  &\mapsto  &t_{i}^{-1} t_{i+1} n_{i} \, (0 < i < \ell) & \ n_{\ell} & \mapsto & t_{\ell-1}t_{\ell} n_{0} &
\end{array}
$$

\end{prop}

\begin{proof}
This follows easily from Lemmas \ref{lem-deformation-atoms} and \ref{lem-deformation-rosary}. 
\end{proof}

It is evident that the coordinates of Proposition \ref{prop-first-order-action} on $\TT^1(C, \pn)$ diagonalize the natural action of $\Aut(C, \pn)^{\circ}$. However, we need slightly more. We need coordinates which diagonalize the natural action of $\Aut(C, \pn)^{\circ}$ \emph{and} which cut out the natural geometrically-defined loci on 
$\hat{\Def}(C, \pn) =\Spf \CC[[ \TT^1(C, \pn)]]$. For example, when $\alpha_c=\third$,  the $\{s_i\}$ coordinates should cut out the locus of formal deformations preserving the singularities and the $\{c_i, n_i\}$ coordinates should cut out the locus of formal deformations preserving a Weierstrass tail. This is almost a purely formal statement (see Lemma \ref{L:CoordinateExtension} below); however there is one non-trivial geometric input. We must show that the crimping coordinate which defines the locus of ramphoid cuspidal deformations with trivial crimping can be extended to a global coordinate which vanishes on the locus of Weierstrass tails. This is essentially a first-order statement which we prove below in Lemma \ref{splitting}.

The $\thirdf$-atom $E$ defines a point in $\cZ^{+} \cap \cZ^{-} \subseteq \bar{\cM}_{2,1}(\third)$ using the notation of $\cZ^+, \cZ^{-}$ from Proposition \ref{prop-ideals}. 
If we denote this point by $0$, we have natural inclusions of $\Aut(E)$-representations
\begin{equation*}
i\co \TT^1_{\cZ^{+},0} \hookarr \TT^1_{\SM_{2,1}(2/3),0}=\TT^1(E)  \quad \text{ and } \quad j\co \TT^1_{\cZ^{-},0} \hookarr \TT^1_{\SM_{2,1}(2/3),0} =\TT^1(E).
\end{equation*}
On the other hand, recall that we have the exact sequence of $\Aut(E, q)$-representations. 
\begin{equation}\label{sequence3}
\begin{aligned}
0 \arr \Cr^1(E)   \xarr{\alpha}
        \TT^1(E)   \xarr{\beta}
       \TT^1(\hat{\cO}_{E,\xi}) \arr  0 \\
\end{aligned}
\end{equation}
 where   $\TT^1(\hat{\cO}_{E,\xi})$ denotes the space of first-order deformations of the singularity $\xi \in E$, and $\Cr^1(E)$  denotes the space of first-order crimping deformations. The key point is that the tangent spaces of these global stacks are naturally identified as deformations of the singularity and the crimping respectively.
\begin{lem}\label{splitting}
With notation as above, there exist  isomorphisms of $\Aut(E)$-representations
\begin{align*}
 \TT^1_{\cZ^{-},0} &\simeq \TT^1(\hat{\oh}_{E, \xi})  \\
\TT^1_{\cZ^{+},0}  &\simeq \Cr^1(E)  
\end{align*}
inducing a splitting of (\ref{sequence3}) with $i=\alpha$ and $j=\beta^{-1}$.
\end{lem}

\begin{proof}
It suffices to
show that the composition
\begin{align*}
\alpha \circ i\co \TT_{\cZ^-,0} \to
\TT_{\overline{\cM}_{2,1}(2/3),0} &= 
\TT^1(E) \to
\TT^1(\hat{\cO}_{E,\xi})
\end{align*}
is an isomorphism, and that the composition
\begin{align*}
\alpha \circ j \co \TT_{\cZ^+,0} \to
\TT_{\overline{\cM}_{2,1}(2/3),0} &= 
\TT^1(E) \to
\TT^1(\hat{\cO}_{E,\xi}) 
\end{align*}
is zero. The latter follows from the former by transversality of $\TT_{\cZ^-,0}$ and $\TT_{\cZ^+,0}$. To see that $\alpha \circ i$ is an isomorphism, observe that
$\cZ^{-} \simeq [\AA^{4}/\GG_m]$ with weights $-4$,$-6$,
$-8$,$-10$, where the universal family is given by
 $$
(y^2-x^{5}-a_{3} \eps x^{3} - a_2 \eps x^2 - a_1 \eps x - a_0 \eps, \eps^2) \, : a_{3},\ldots,a_{0}\in\CC \}.
 $$
On the other hand, there is a natural isomorphism
$$ \begin{aligned}
\TT^1(\widehat{\cO}_{E,\xi}) = &
\{
\Spec \CC[[x,y,\eps]]/
(y^2-x^{5}-a_{3} \eps x^{3} - a_2 \eps x^2 - a_1 \eps x - a_0 \eps, \eps^2) \, : a_{3},\ldots,a_{0}\in\CC \}.
\end{aligned}$$
Evidently, $\alpha \circ i$ is the identity map in the
given coordinates.
\end{proof}

\begin{lemma}\label{L:CoordinateExtension}
Let $V$ be a finite-dimensional representation of a torus $G$, let $X=\Spf \CC[[V]]$, and let $\fm \subseteq \CC[[V]]$ be the maximal ideal.  Suppose we are a given a collection of $G$-invariant formal smooth closed subschemes $Z_i:=\Spf \CC[[V]]/I_i, (i=1, \ldots, r)$ which intersect transversely at $0$, and a basis $x_{1}, \ldots, x_n$ for $V$ such that:
\begin{enumerate}
\item $x_{1}, \ldots, x_n$ diagonalize the action of $G$.
\item $I_i/\fm I_i$ is spanned by a subset of $x_1, \ldots, x_n$.
\end{enumerate}
Then there exist coordinates $X \simeq \Spf \CC[[x_1', \ldots, x_k']]$ such that
\begin{enumerate}
\item $x_1', \ldots, x_n'$ diagonalize the action of $G$.
\item $x_1', \ldots, x_n'$ reduce modulo $\fm$ to $x_1, \ldots, x_n$.
\item $I_i$ is generated by a subset of $x_1', \ldots, x_n'$.
\end{enumerate}
\end{lemma}
\begin{proof}
Let $x_{i,1}, \ldots, x_{i,d_i}$ be a diagonal basis for $ I_i/ \fm I_i$ as a $G$-representation. Consider the surjection
$$
I_i \rightarrow I_i/ \fm I_i
$$
and choose an equivariant section, i.e., choose $x_{i,1}', \ldots, x_{i,d_i}'$ such that each spans a one-dimensional sub-representation of $G$. By Nakayama's Lemma, these elements generate $I_i$. Repeating this procedure for each $Z_i$, we obtain 
$x'_{i,j}$ for $i =1 , \ldots, r$ and $j = 1, \ldots, d_i$.
Since the $Z_i$'s intersect transversely, these coordinates induce linearly independent elements of $V$. 
Thus they may be completed to a diagonal basis, and this gives the necessary coordinate change.
\end{proof}

\begin{prop}[Explicit Description of $I_{\cZ^+}$, $I_{\cZ^-}$]\label{prop-formal-ideals}
Let $(C,\pn)$ be an $\alpha_c$-closed curve.  There exist coordinates $n_i, \bs_i, c_i$ (resp., $n_{i,j}, \bs_{i,j}$) on $\hat{\Def}(C,\pn)$ such that the action of $\Aut(C, \pn)^{\circ}$ on  $\hat{\Def}(C,\pn) = \Spf \hat{A}$ is given as in Proposition \ref{prop-first-order-action}, and such that the ideals $I_{\cZ^+}$, $I_{\cZ^-}$ are given as follows:

\noindent {\rm $\bullet \,\alpha_c=\first$, Type A: } $I_{\cZ^+} = \bigcap_{i=1}^r (\bs_i)$, $I_{\cZ^-} = \bigcap_{i=1}^r (n_i)$.

\noindent {\rm $\bullet \,\alpha_c=\first$, Type B: } $I_{\cZ^+} = (\bs_1) \cap (\bs_2)$, $I_{\cZ^-} = (n)$.

\noindent {\rm $\bullet \, \alpha_c=\first$, Type C: }  $I_{\cZ^+} = (\bs)$, $I_{\cZ^-} = (0)$.

\noindent {\rm $\bullet \, \alpha_c = \second$, Type A: } $I_{\cZ^+} = \bigcap_{i,j} (\bs_{i,j})$ , $I_{\cZ^-} = \bigcap_{i,\mu,\nu \in S} J_{i,\mu,\nu}$ where 
$$\begin{aligned}
	S & :=\{i, \mu,\nu: 1 \leq i \leq r+s, 1 \leq \mu \leq \left\lceil\frac{\ell_i}{2}\right\rceil,\,0 \leq \nu \leq \ell_i-2\mu+1\}\\
	J_{i,\mu,\nu}&:= (n_{i,\nu},  \bs_{i,\nu+2},   \ldots, \bs_{i,\nu+2\mu-2},  n_{i,\nu+2\mu-1}),  \quad \text{for $i=1, \ldots, r$} \\
	J_{i,\mu,\nu}&:= (n_{i,\nu},  \bs_{i,\nu+2},  \ldots, \bs_{i,\nu+2\mu-2}),   \quad \text{for $i=r+1, \ldots, r+s$}.
\end{aligned}$$

\noindent {\rm $\bullet \,\alpha_c = \second$, Type B: } $I_{\cZ^+} = \bigcap_{i} (\bs_{i})$ , $I_{\cZ^-} = \bigcap_{\mu,\nu \in S} J_{\mu,\nu}$ where 
$$\begin{aligned}
	S & :=\{\mu,\nu: 1 \leq \mu \leq \left\lceil\frac{g}{2}\right\rceil,\,0 \leq \nu \leq g-2\mu+1\}\\
	J_{\mu,\nu}&:= (n_{\nu}, \bs_{\nu+2},   \ldots, \bs_{\nu+2\mu-2},  n_{\nu+2\mu-1}),
\end{aligned}$$
and $n_0:=0$ and $n_g:=0$.

\noindent {\rm $\bullet \, \alpha_c = \second$, Type C: } $I_{\cZ^+} = \bigcap_{i} (\bs_{i})$ , $I_{\cZ^-} = \bigcap_{\mu,\nu \in S} J_{\mu,\nu}$ where 
$$\begin{aligned}
	S & :=\{\mu,\nu: 1 \leq \mu \leq \left\lceil\frac{g-1}{2}\right\rceil,\,0 \leq \nu \leq g-2\}\\
	J_{\mu,\nu}&:= (n_{\nu},  \bs_{\nu+2},   \ldots, \bs_{\nu+2\mu-2},  n_{\nu+2\mu-1}),
\end{aligned}$$
and the subscripts are taken modulo $g-1$.

\noindent {\rm $\bullet \,\alpha_c=\third$, Type A: } $I_{\cZ^+} = \bigcap_{i=1}^r (\bs_i)$, 
$$I_{\cZ^-} = \bigcap_{i=1}^r \bigcap_{j=0}^{\ell_i-1} (n_{i,j}, \br'_{i,j+1}, \br'_{i,j+2}, \ldots, \br'_{i,\ell_i-1}, c_i).$$

\noindent {\rm $\bullet \,\alpha_c=\third$, Type B: } $I_{\cZ^+} = (\bs) $,
$$I_{\cZ^-} = \bigcap_{i=1}^{\ell-1}(n_i, \br'_{i+1}, \br'_{i+2}, \ldots, \br'_{\ell-1}, c) \cap (\br'_1, \br'_2, \ldots, \br'_{\ell-1}, c).$$

\noindent {\rm $\bullet \,\alpha_c=\third$, Type C: } $I_{\cZ^+} = (\bs_1) \cap (\bs_2)$,
$$I_{\cZ^-} = \bigcap_{i=0}^{\ell-1}(n_i, \br_i, \br_{i-1}, \ldots, \br_1, c_0) \cap \bigcap_{i=0}^{\ell-1}(n_i, \br'_{i+1}, \br'_{i+2}, \ldots, \br'_{\ell-1}, c_\ell).$$
\end{prop}

\begin{proof}
We prove the statement when $(C, \pn)$ is a $\thirdf$-closed curve of combinatorial type A; the other cases are similar and left to the reader.  Let $\hat{\Def}(C,\pn) = \Spf \hat{A} \to \bar{\cM}_{g,n}(\third)$ be a miniversal deformation space of $(C,\pn)$.  For $i = 1, \ldots, r$, we define
\begin{itemize}
\item $Z^+_i = \Spf \hat{A}/I_{Z_i^+}$ is the locus of deformations preserving the $i^{th}$ ramphoid cusp $\xi_{i}$.
\item $Z^{-}_{i} = \Spf \hat{A}/I_{Z_{i}^-}$ is the locus of deformations preserving the $i^{th}$ Weierstrass tail.
\end{itemize}
Since $Z_i^+$ (resp., $Z_i^-$) are smooth, $G$-invariant, formal closed subschemes of $\Spf \hat{A}$, the conormal space of $Z_i^+$ (resp., $Z_i^-$) is canonically identified with  $I_{Z_i^+} /\fm_{\hat{A}} I_{Z_i^+}$ (resp.,  $I_{Z_i^-} /\fm_{\hat{A}} I_{Z_i^-}$). Thus, in the notation of Proposition \ref{prop-first-order-action}, we have $
I_{Z_i^+} /\fm_{\hat{A}} I_{Z_i^+} \simeq \TT^1(\hat{\cO}_{E_i,\xi_i})^\vee $.  Moreover, if $\ell_i = 1$, we have
$$
I_{Z_i^-} /\fm_{\hat{A}} I_{Z_i^-} \simeq \Cr^1(E_i)^{\vee} \oplus \TT^1(\hat{\cO}_{E_i,q_i})^\vee
$$
using Lemma \ref{splitting} to identify $\Cr^1(E_i)^{\vee}$ as the conormal space of the locus of deformations of $E_i$ for which the attaching point remains Weierstrass. 

If $\ell_i > 1$ (i.e., $E_i$ is not a nodally-attached Weierstrass tail), we define
\begin{itemize}
\item $T_{i,j} = \Spf \hat{A}/I_{T_{i,j}}$ as the locus of deformations preserving the tacnode $\tau_{i,j,2}$, for $j=1, \ldots, \ell_i - 2$.
\item $W_{i} = \Spf \hat{A}/I_{W_i}$ as the closure of the locus of deformations preserving the tacnode $\tau_{i,\ell_i-1,2}$ such that the tacnodally attached genus $2$ curve is attached at a Weierstrass point.
\item $N_{i,j} = \Spf \hat{A}/I_{N_{i,j}}$ as the locus of deformations preserving the node $q_{i,j}$, for $j=0, \ldots, \ell_i - 1$.
\end{itemize}
We observe that for each $i$ with $\ell_i > 1$, $W_i$ is a smooth, $G$-invariant formal subscheme, and
there is an identification
$$I_{W_i} /\fm_{\hat{A}} I_{W_i} \simeq \Cr^1(E_{i})^{\vee} \oplus \TT^1(\hat{\cO}_{C,\tau_{i,\ell_i-1,2}})^\vee.$$
If we choose coordinates $c_i \in\Cr^1(E_{i})^{\vee}$ and $s_{i,0}, s_{i,1}, s_{i,2}, s_{i,3} \in  \TT^1(\hat{\cO}_{C,\tau_{i,\ell_i-1,2}})^\vee$
cutting out $W_i$ and a coordinate $n_{i, \ell_i-1}$ cutting out $N_{i, \ell_i-1}$, 
then it is easy to check that $Z_i^-$ is necessarily cut out by $c_i$ and $n_{i, \ell_i-1}$.

Formally locally around $(C, \pn)$, $\cZ^+$ and $\cZ^-$ decompose as
$$ \begin{aligned}
\cZ^+ \times_{\bar{\cM}_{g,n}(2/3)} \Spf \hat{A} &=Z^{+}_1 \cup \cdots \cup Z^{+}_r, \\
 \cZ^- \times_{\bar{\cM}_{g,n}(2/3)} \Spf \hat{A} &= \bigcup_{i=1}^r \bigg(Z^{-}_i \cup \bigcup_{j=0}^{\ell_i-2} \big( W_i \cap \bigcap_{k=j+1}^{\ell_i-2} T_{i,k} \cap N_{i,j} \big) \bigg)
\end{aligned}$$

For each $i=1, \ldots, r$, we consider the cotangent space of $Z_i^+$ and either the cotangent space of $Z_i^-$ if $\ell_i=1$ or the set of cotangents spaces of $T_{i,j}, W_i, N_{i,j}$ if $\ell_i > 1$.  
Since this collection of subspaces of $\TT^1(C, \pn)$ as $i$ ranges from $1$ to $r$  is linearly independent, 
we may apply Lemma \ref{L:CoordinateExtension} to this collection of formal closed subschemes to obtain coordinates with the required properties.
\end{proof}

\subsection{Local VGIT chambers for an $\alpha_c$-closed curve}
\label{section-variation-GIT}
In this section, we explicitly compute the VGIT ideals $I^+,I^- \subseteq 
\CC[T^1(C, \pn)]$ (Definition \ref{defn-ideals}) for any $\alpha_c$-closed curve. The main result (Proposition \ref{prop-vgit-ideals}) states that the VGIT ideals agree formally locally with the ideals $I_{\cZ^+}$, $I_{\cZ^-}$. By Proposition \ref{prop-ideals}, this suffices to establish Theorem \ref{theorem-etale-VGIT}. In order to carry out the computation of $I^+$ and $I^-$, we must do two things: First, we must explicitly identify the character 
$\chi_{\delta-\psi} \co \Aut(C, \pn) \rightarrow \GG_m$ for any $\alpha_c$-closed curve. Second, we must compute the ideals of positive and negative semi-invariants with respect to this character.

\begin{definition}\label{defn-char}
Let $E_1, \ldots, E_r$ be the $\alpha_c$-atoms of $(C, \pn)$, and let $t_i \in \Aut(E_i)$ 
be the coordinate specified in Proposition \ref{prop-first-order-action}.  Let
$$\chi_{\star}\co \Aut(C,\pn)^{\circ} \rightarrow \GG_m = \Spec \CC[t,t^{-1}]$$
be the character defined by $t \mapsto t_1 t_2 \cdots t_r.$ Note that $\chi_{\star}$ is trivial on automorphisms fixing the $\alpha_c$-atoms.
\end{definition}

The following proposition shows that $\chi_{\delta-\psi}$ is simply a positive multiple of $\chi_{\star}$. Since it will be important in Proposition \ref{P:Proj1}, we also prove now that the character of $K_{\bar{\cM}_{g,n}(\alpha_c)} + \alpha_c \delta + (1-\alpha_c)\psi$ is trivial for $\alpha_c$-closed curves.

\begin{prop} \label{prop-character-comparison} Let $\alpha_c \in \{\first, \second, \third\}$ be a critical value and let $(C, \pn)$ be an $\alpha_c$-closed curve. Then there exists a positive integer $N$ such that $\chi_{\delta-\psi}|_{\Aut(C,\pn)^{\circ}} = \chi_{\star}^N$ for every $\alpha_c$-closed curve $(C, \pn)$.  Specifically, 
$$
N = \left\{
	\begin{array}{rl}
	11					& \text{if } \alpha_c = \first \\
	10					& \text{if } \alpha_c = \second \\
	39					& \text{if } \alpha_c = \third \\
	\end{array} \right.
$$
In particular, $I^{\pm}_{\chi_{\delta-\psi}}=I^{\pm}_{\chi_{\star}}$.
\end{prop}
\begin{proof}
We prove the case when $\alpha_c = \third$ for an $\alpha_c$-closed curve $(C, \pn)$ of Type A. Let $C = K' \cup L_1 \cup \cdots \cup L_r$ be the decomposition of $C$ as in Definition \ref{defn-canonical}, and suppose that the rank of $\Aut(K')$ is $k$.
Corollary \ref{C:RosariesClosed} implies that there exist length three rosaries 
$R'_1, \ldots, R'_k$ such that $\Aut(K')^{\circ} \simeq \prod_{i=1}^k \Aut(R'_i)$. Thus, we have
\begin{align*}
\Aut(C)^{\circ}&=\Aut(K')^{\circ} \times \prod_{i=1}^{r} \Aut(L_i) \\
&= \prod_{i=1}^k \Aut(R'_i) \times \prod_{i=1}^r \left[ \prod_{j=1}^{\ell_i-1} \Aut(R_{i,j}) \times  \Aut(E_{i}) \right].
\end{align*}
Let $\rho'_i \co \GG_m \to  \Aut(C)$ (resp. $\rho_{i,j}$, $\varphi_{i}$) be the one-parameter subgroup corresponding 
to $\Aut(R'_i) \subset \Aut(C)$ (resp. $\Aut(R'_{i,j}), \Aut(E_i) \subset \Aut(C)$). By \cite[Sections 3.1.2--3.1.3]{afs}, we have
\[
\langle \chi_{\delta-\psi}, \rho'_{i} \rangle =0, \qquad
 \langle \chi_{\delta-\psi}, \rho_{i,j} \rangle =0, \qquad
\langle \chi_{\delta - \psi}, \varphi_{i} \rangle = 39.
\]
On the other hand, the definition of $\chi_{\star}$ obviously implies
\[
\langle \chi_{\star}, \rho'_{i} \rangle =0, \qquad
\langle \chi_{\star}, \rho_{i,j} \rangle =0, \qquad
\langle \chi_{\star}, \varphi_{i} \rangle = 1.
\]
It follows that $\chi_{\delta-\psi}=\chi_{\star}^{39}$ as desired.
\end{proof}

\begin{prop}\label{P:trivial-characters}
For any $\alpha_c$-closed curve $(C, \pn)$, the action of $\Aut(C, \pn)^{\circ}$ on the fiber of $K_{\bar{\cM}_{g,n}(\alpha_c)} + \alpha_c \delta + (1-\alpha_c)\psi$ is trivial.
\end{prop}
\begin{proof}
We prove the case when $\alpha_c = \third$ for an $\alpha_c$-closed curve $(C, \pn)$ of Type A.  
Let $\rho'_i$, $\rho_{i,j}$, $\varphi_{i}$ be the one-parameter subgroups of $\Aut(C, \pn)$ 
as in the proof of Proposition \ref{prop-character-comparison}. By \cite[Sections 3.1.2--3.1.3]{afs}, we have

\[
\begin{array}{lllllll}
\langle \chi_{\lambda}, \rho'_{i} \rangle=0 \qquad &  \langle \chi_{\lambda}, \rho_{i,j} \rangle=0 \qquad & \langle \chi_\lambda, \varphi_{i} \rangle = 4 \\
 \langle \chi_{\delta-\psi}, \rho'_{i} \rangle =0 \qquad & \langle \chi_{\delta-\psi}, \rho_{i,j} \rangle =0 \qquad &\langle \chi_{\delta - \psi}, \varphi_{i} \rangle = 39.
\end{array}
\]
Using the identity
\begin{equation} \label{eqn-K}
K_{\bar{\cM}_{g,n}(\alpha_c)}  + \alpha_c \delta +(1-\alpha_c)\psi= 13 \lambda + (\alpha_c-2) (\delta-\psi) \, 
\end{equation}
one easily computes $$\langle\chi_{K_{\bar{\cM}_{g,n}(\alpha_c)}  + \alpha_c \delta +(1-\alpha_c)\psi}, \rho_i \rangle= \langle\chi_{K_{\bar{\cM}_{g,n}(\alpha_c)}  + \alpha_c \delta +(1-\alpha_c)\psi}, \rho_{i,j} \rangle= \langle\chi_{K_{\bar{\cM}_{g,n}(\alpha_c)}  + \alpha_c \delta +(1-\alpha_c)\psi}, \varphi_{i} \rangle=0,$$ and the claim follows.
\end{proof}

Proposition \ref{prop-character-comparison} and Corollary \ref{lemma-vgit-identity} imply that we can compute the VGIT ideals $I^-$ and $I^+$ as the ideals of semi-invariants associated to $\chi_{\star}$. In the following proposition, we compute these explicitly, and show that they are identical to the ideals $I_{\cZ^+}$ and $I_{\cZ^-}$, as described in Proposition \ref{prop-formal-ideals}.

\begin{proposition}[Description of VGIT ideals]\label{prop-vgit-ideals}
Let $(C, \pn)$ be an $\alpha_c$-closed curve for a critical value $\alpha_c \in \{\third, \second, \first\}$. Then $I^+ \hat{A} = I_{\cZ^+}$ and $I^- \hat{A} = I_{\cZ^-}$.

\end{proposition}

We  establish the proposition first in the case of an $\alpha_c$-atom, then in the case of an $\alpha_c$-link, 
and finally for each of the distinct combinatorial types of $\alpha_c$-closed curves.

\subsubsection{The case of an $\alpha_c$-atom}  
\begin{lem} \label{lemma-chambers-atoms}
Let $E$ be an $\alpha_c$-atom. Using the notation of Lemma \ref{lem-deformation-atoms} 
for the action of $\Aut(E)$ on $\TT^1(E)$, we have

\begin{tabular}{l l l}
{\rm $\bullet \,\alpha_c=\first$:}  		&$I^+ = (s_0, s_1)$,  		&$I^- = (0)$. \\
{\rm $\bullet \, \alpha_c = \second$:} 	& $I^+ = (s_0, s_1, s_2)$, 	&$I^- = (0)$.\\
{\rm $\bullet \, \alpha_c = \third$:} 	&$I^+ = (s_0,s_1, s_2, s_3)$,	& $I^- = (c)$.
\end{tabular}
\end{lem}

\begin{proof}  This is a direct computation from the definitions.  The $I^+$ (resp., $I^-$) ideal is 
generated by all semi-invariants of negative (resp., positive) weight.
\end{proof}

\subsubsection{The case of a $\secondf$-link.}
We handle the special case when $C$ has one nodally-attached $\secondf$-link, i.e., $C$ is a $\secondf$-closed curve of type $A$ with $r=1$ and $s=0$. Using Proposition \ref{prop-first-order-action}, we have 
$$
\Aut(C)^{\circ} = \Aut(L_1) \qquad
\TT^1(C) =  \TT^1(K) \oplus \TT^1(L_1)
$$
with coordinates $t_1, \ldots, t_\ell$ on $\Aut(L_1)$ and coordinates
$\bs_j = (s_{j,0}, s_{j,1}, s_{j,2})$ ($j =1, \ldots, \ell$), 
$n_j$ ($j=0, \ldots, \ell$) on $\TT^1(L_1)$ so that the action of $\Aut(C, \pn)^{\circ}$ on $\TT^1(L_1)$ is given by 
$$s_{j,k} \mapsto t_j^{k-4} s_{j,k}, \quad 
	n_0 \mapsto t_1 n_0, \quad n_\ell \mapsto t_\ell n_\ell,  \quad 
	n_{j} \mapsto t_j t_{j+1}n_{j}  \text{ for $j \neq 0, \ell$} \, .$$

 \begin{lem} 
 \label{lemma-second-link}
  With the above notation, the vanishing loci of $I^+$ and $I^-$ are
 $$
 	V(I^+) =  \bigcup_{j=1}^\ell V( \bs_j) \qquad \qquad
 	V(I^-) =  \bigcup_{\mu \ge 1} \bigcup_{\nu = 0}^{\ell-2\mu+1} V_{\mu, \nu}
$$ 
where $V_{\mu, \nu} = V(n_{\nu} , 
\bs_{\nu+2},  
\ldots , \bs_{\nu+2\mu-2} ,
n_{\nu+2\mu-1})$.
\end{lem}

\begin{remark*}  For instance, $V_{1,\nu} = V(n_{\nu}, 
n_{\nu+1})$ and $V_{2, \nu} = V(n_{\nu},
\bs_{\nu+2} , 
n_{\nu+3})$.
\end{remark*}

\begin{proof}
We will use the Hilbert-Mumford criterion of Proposition \ref{prop-hilbert-mumford}.  
For the $V(I^+)$ case, suppose  $x \in V(\bs_j)$ for some $j$.  
Set $\lambda = (\lambda_i)\co \GG_m \to \GG_m^\ell \simeq \prod_{i=1}^\ell \Aut(E_i)$ 
where $\lambda_i = 1$ for $i \neq j$ and $\lambda_j =\id$.  Then 
$\langle \chi_{\star}, \lambda \rangle = 1$ and $\lim_{t \to 0} \lambda(t) \cdot x$ exists 
so $x \in V(I^+)$.  Conversely, let $\lambda = (\lambda_i)$ be a one-parameter 
subgroup with $\langle \chi_{\star}, \lambda \rangle =\sum_i \lambda_i > 0$ such that $\lim_{t \to 0} \lambda(t) \cdot x$ exists.  
Then for some $j$, we have $\lambda_j > 0$ which implies that $\bs_j(x)  = 0$.

For the $V(I^-)$ case, the inclusion $\supseteq$ is easy: suppose that $x \in V_{\mu, \nu}$ 
for $\mu \ge 1$ and $\nu = 0, \ldots, \ell-2\mu+1$.  Set
$$
\lambda = \big( \underbrace{0, \ldots, 0}_{\nu} , \underbrace{-1, 1, -1, \ldots, 1, -1}_{2\mu-1} , 
\underbrace{0, \ldots, 0}_{\ell-2\mu-\nu+1} \big)
$$
Then $\langle \chi_{\star}, \lambda \rangle =\sum_i \lambda_i = -1$ and $\lim_{t \to 0} \lambda(t) \cdot x$ exists 
so $x \in V(I^-)$.  For the $\subseteq$ inclusion, we will use induction on $\ell$.  
If $\ell=1$, then $V(I^-) = V(n_0,  n_1)$.  For $\ell > 1$, 
suppose $x \in V(I^-)$ and $\lambda = (\lambda_i)\co \GG_m \to \GG_m^\ell$ is a 
one-parameter subgroup with $\sum_{i=1}^\ell \lambda_i < 0$ such that 
$\lim_{t \to 0} \lambda(t) \cdot x$ exists.  If $\lambda_\ell \ge 0$, then 
$\sum_{i=1}^{\ell-1} \lambda_\ell < 0$ so by the induction hypothesis 
$x \in V_{\mu, \nu}$ for some $\mu \ge 1$ and $\nu = 0, \ldots, \ell-2\mu$. 
 If $\lambda_\ell < 0$, then we immediately conclude that $n_\ell(x) =  
 0$.  If $\lambda_{\ell-1} + \lambda_\ell < 0$, then $n_{\ell-1}(x) = 0$ so $x \in V_{1, {\ell-1}}$.  
 If $\lambda_{\ell-1} + \lambda_\ell \ge 0$, then $\lambda_{\ell-1} \ge 0$ so $\bs_{\ell-1}(x) = 0$.  
 Furthermore, $\sum_{i=1}^{\ell-2} \lambda_i < 0$ 
 so by applying the induction hypothesis and restricting to the locus $V(n_{\ell-2}, \bs_{\ell-1},  n_{\ell-1}, \bs_{\ell},   n_{\ell})$,  
we can conclude either:  (1) $x \in V_{\mu, \nu}$ for $\mu \ge 1$ and 
$\nu = 0, \ldots, \ell-2\mu-1$, or (2) $x \in V(n_{\ell-\mu-4} ,  
\bs_{\ell-\mu-2} ,  
\ldots , \bs_{\ell-3})$ 
for some $\mu \ge 1$.  In case (2), since  $\bs_{\ell-1}(x) =   
n_\ell(x) = 0 $, we have $x \in V_{\mu+1,\ell-\mu-4}$.
 \end{proof}
 
 \begin{remark*} The chamber $V(I^+)$ is the closed locus in the deformation 
 space consisting of curves with a tacnode while $V(I^-)$ consists of curves 
 containing an elliptic chain.
 \end{remark*}
 
 \subsubsection{The case of a $\thirdf$-link.}
We now handle the special case when $C$ has one nodally-attached $\thirdf$-link of length $\ell$, i.e., $C$ is a $\thirdf$-closed curve of combinatorial type $A$ with $r=1$.
Using Proposition \ref{prop-first-order-action}, we have 
$$
\Aut(C)^{\circ} = \Aut(K') \times \Aut(L_1) \qquad
\TT^1(C) =  \TT^1(K') \oplus \TT^1(L_1)
$$
with coordinates $t_1, \ldots, t_\ell$ on $\Aut(L_1)$ and  coordinates
$\br_j = (r_{j,0}, r_{j,1}, r_{j,2}), \br'_j = (r'_{j,0}, r'_{j,1}, r'_{j,2})$, $n_j$ ($j=0, \ldots, \ell-1$),
$\bs = (s_0, s_1, s_2, s_3)$, $c$  on $\TT^1(L_1)$, 
so that the action of $\Aut(L_1)$ on $\TT^1(L_1)$ is given by
 $$\begin{array}{llllllllllllll}
r_{j,k} &  \mapsto & t_{j}^{k-4} r_{j,k}, &r'_{j,k} &  \mapsto & t_{j}^{4-k} r'_{j,k}, & 
s_{k}  &\mapsto&  t_{\ell}^{2k-10} s_{k}\\
c &\mapsto & t_{\ell} c &
n_{0} & \mapsto & t_{1} n_{0}, &
n_{j}  &\mapsto  &t_{j}^{-1} t_{j+1} n_{j} \, (0 < j < \ell).
\end{array}
$$
The character $\chi_{\star}$ is given by
$$
\Aut(C)^{\circ} \simeq \GG_m^\ell  \to \GG_m, \quad (t_1, \ldots, t_\ell)  \mapsto t_\ell
$$

 \begin{lem} 
 \label{lemma-third-link}
  With the above notation, the vanishing loci of $I^+$ and $I^-$ are
 $$
 	V(I^+) =  V( \bs) \qquad \qquad
 	V(I^-) = \bigcup_{j=0}^{\ell-1} V(n_j, \br'_{j+1}, \br'_{j+2}, \ldots, \br'_{\ell-1}, c)
$$ 
\end{lem}

\begin{remark*} For instance, if $\ell=2$, $V(I^-) = V(n_1, c) \cup V(n_0, \br'_1, c)$.
\end{remark*}

\begin{proof} The first equality is obvious.  We use the Hilbert-Mumford criterion to verify the second.
Suppose $x \in V(n_j, \br'_{j+1}, \ldots, \br'_{\ell-1}, c)$ for some $j=0, \ldots, \ell-1$.  If we set
$$
\lambda = \big( \underbrace{0, \ldots, 0}_{j} , \underbrace{-1, -1, \ldots,  -1}_{\ell-j} \big)
$$
then $\langle \chi_{\star}, \lambda \rangle = -1 < 0$ and $\lim_{t \to 0} \lambda(t) \cdot x$ exists.  Therefore, $x \in V(I^-)$.  Conversely, suppose $x \in V(I^-)$ and $\lambda = (\lambda_i) \co \GG_m \to \GG_m^\ell$ is a one-parameter subgroup with $\langle \chi_{\star}, \lambda \rangle = \lambda_\ell < 0$ such that $\lim_{t \to 0} \lambda(t) \cdot x$ exists.  Clearly, we may assume that $\lambda_\ell = -1$.  First, it is clear that $c(x) = 0$.  If $n_{\ell-1}(x) = 0$, then $x \in V(n_{\ell-1}, c)$.  Otherwise, as the limit exists, $\lambda_{\ell-1} \le -1$ so that $\br'_{\ell-1}(x) =0$.  If $n_{\ell-2}(x) = 0$, then $x \in V(n_{\ell-2}, \br'_{\ell-1}, c)$.  Continuing by induction, we see that there must be some $j=0, \ldots, \ell-1$ with $x \in V(n_j, \br'_{j+1}, \br'_{j+2}, \ldots, \br'_{\ell-1}, c)$ which establishes the lemma.
\end{proof}

\subsubsection{The general case} We are now ready thanks to Lemmas \ref{lemma-second-link} and \ref{lemma-third-link} as well as 
Corollaries \ref{lemma-vgit-product} and \ref{lemma-vgit-closed} to establish Proposition \ref{prop-vgit-ideals} in full generality.

\noindent {\it Proof of Proposition \ref{prop-vgit-ideals}.}
Let $(C, \pn)$ be an $\alpha_c$-closed curve and consider the action of 
$\Aut(C, \pn)$ on $\TT^1(C, \pn)$ described in Proposition \ref{prop-first-order-action}.   
We split the proof into the types of $\alpha_c$-closed curves according to 
Definition \ref{defn-canonical}.

\noindent
{\it $\bullet \, \alpha_c = \first$ of Type A.}  By using 
Corollary \ref{lemma-vgit-product}, one may assume that $r=1$ in which case the statement is clear.

\noindent
{\it $\bullet \,  \alpha_c = \first$ of Type B.}  A simple application of 
Proposition \ref{prop-hilbert-mumford} shows that $V(I^+) = (\bs_1, \bs_2)$, and 
$V(I^-) = (n)$.

\noindent
{\it $\bullet \,  \alpha_c = \first$ of Type C.}  This is 
Lemma \ref{lemma-chambers-atoms}.

\noindent
{\it $\bullet \,  \alpha_c = \second$  of Type A.}  By Corollary \ref{lemma-vgit-product}, 
it is enough to consider the case when either $r = 1, s=0$ or $r=0, s=1$.   
The case of $r=1$ and $s=0$ is the example worked out in Lemma \ref{lemma-second-link}. 
If $r=1, s=0$, 
the action of $\Aut(C, \pn)^{\circ}$ on $\Def(C, \pn)$ is same as the action 
given in Lemma \ref{lemma-second-link} restricted to the closed subscheme $V(n_{\ell}) = 0$.  
This case therefore follows from Corollary \ref{lemma-vgit-closed} and Lemma \ref{lemma-second-link}.

\noindent {\it $\bullet \, \alpha_c = \second$ of Type B.} 
The action of $\Aut(C,\pn)^{\circ}$ on $\TT^1(C, \pn)$ is the same action as in 
Lemma \ref{lemma-second-link} restricted to the closed subscheme $V(n_0,n_{r+1}) = 0$ 
so this case follows from 
Corollary \ref{lemma-vgit-closed} and Lemma \ref{lemma-second-link}.

\noindent {\it $\bullet \,  \alpha_c = \second$  of Type C.} 
This follows from an argument similar to the proof of Lemma \ref{lemma-second-link}.

\noindent
{\it $\bullet \,  \alpha_c = \third$  of Type A.}  By Corollary \ref{lemma-vgit-product}, 
it is enough to consider the case when  $r = 1$ which
 is the example worked out in Lemma \ref{lemma-third-link}. 

\noindent {\it $\bullet \, \alpha_c = \third$ of Type B.} 
The action here is the same action as in Lemma \ref{lemma-third-link} restricted to the closed subscheme $V(n_0)$ so 
 this case follows from Corollary \ref{lemma-vgit-closed} and Lemma \ref{lemma-third-link}.

\noindent {\it $\bullet \,  \alpha_c = \third$  of Type C.} This case can be handled by an argument similar to the proof of Lemma \ref{lemma-third-link}.

\epf

\begin{proof}[Proof of Theorem \ref{theorem-etale-VGIT}]
Proposition \ref{prop-vgit-ideals} implies that $I_{\cZ^+} = I^+ \hat{A}$ and $I_{\cZ^-} = I^- \hat{A}$ so we 
may apply Proposition \ref{prop-ideals} to conclude the statement of the theorem.
\end{proof}

%% file: 2nd-flip-Section4-submit-v9.tex

\section{Existence of good moduli spaces}
\label{S:existence}
In this section, we prove that the algebraic stacks $\SM_{g,n}(\alpha)$ possess good moduli spaces 
(Theorem \ref{T:Existence}). In Section \ref{S:general-existence}, we prove three general 
existence results for good moduli spaces. The first of these,  Theorem \ref{T:keel-mori}, gives 
conditions under which one may use a local quotient presentation to construct a 
good moduli space. As we explain below, this may be considered as an analog of the Keel-Mori 
theorem \cite{keel-mori} for algebraic stacks, but in practice the hypotheses of the theorem are much harder to verify 
than those of the Keel-Mori theorem. Our second existence result, Theorem \ref{T:vgit-existence}, 
gives one situation in which the hypotheses of Theorem \ref{T:keel-mori} are satisfied. It says that 
if $\cX$ is an algebraic stack and $\cX^+ \hookrightarrow \cX \hookleftarrow \cX^-$ is a pair of open immersions 
locally cut out by VGIT, then $\cX$ admits a good moduli space if $\cX^+$, 
$\cX \smallsetminus \cX^+$, and $\cX \smallsetminus \cX^-$ do. The third existence result, 
Proposition \ref{P:finite-existence}, proves that one can check existence of a good moduli space 
after passing to a finite cover.  These results pave the way for the argument in 
Section \ref{S:Existence}  which proves the existence of good moduli spaces for
$\SM_{g,n}(\alpha)$ inductively.

\subsection{General existence results}  \label{S:general-existence}
In this section, we prove the following three results.  
Recall the definition of a local quotient presentation from Definition \ref{definition-etale-presentations}.  
Note that if an algebraic stack $\cX$ of finite type over $\Spec \CC$ admits local quotient presentations around every closed point, 
then $\cX$ necessarily has affine diagonal.

\begin{theorem} \label{T:keel-mori} \label{T:general-existence}
 Let $\cX$ be an algebraic stack of finite type over $\Spec \CC$.  Suppose that:
\begin{enumerate}
\item For every closed point $x \in \cX$, there exists a local quotient presentation 
$f \co \cW \to \cX$ around $x$ such that:
\begin{enumerate}
\item $f$ is stabilizer preserving at closed points of $\cW$.
\item  $f$ sends closed points to closed points.
\end{enumerate}
\item For any $\CC$-point $x \in \cX$, the closed substack $\overline{ \{x\}}$ admits a good moduli space.
\end{enumerate}
 Then $\cX$ admits a good moduli space. 
 \end{theorem}

\begin{theorem}\label{T:vgit-existence}
Let $\cX$ be an algebraic stack of finite type over $\Spec \CC$, and 
let $\cL$ be a line bundle on $\cX$. 
Let $\cX^+, \cX^- \subset \cX$ be open substacks, and 
let $\cZ^+=\cX \smallsetminus \cX^+$ and $\cZ^-=\cX \smallsetminus \cX^-$ 
be their reduced complements.  Suppose that
\begin{enumerate}
\item $\cX^+$, $\cZ^+$, $\cZ^-$ admit good moduli spaces.
\item For all closed points $x \in \cZ^+ \cap \cZ^-$, there exists 
a local quotient presentation $\cW \rightarrow \cX$ around $x$ 
and a Cartesian diagram
\begin{equation} 
\begin{split}
\label{diagram-etale-pres}
\xymatrix{
\cW_{\cL}^{+}  \ar[d] \ar@{^(->}[r]			&  \cW \ar[d] & \cW_{\cL}^{-}\ar@{_(->}[l] \ar[d]\\
\cX^+ \ar@{^(->}[r] 				& \cX & \cX^-  \ar@{_(->}[l] 	
}
\end{split}
\end{equation}
where $\cW_{\cL}^+, \cW_{\cL}^-$ are the VGIT chambers of $\cW$ with respect to $\cL$.
\end{enumerate}
Then there exist good moduli spaces $\cX \to X$ and 
$\cX^- \to X^-$ such that $X^+ \to X$ and $X^- \to X$ are 
proper and surjective. 
In particular, if $X^+$ is proper over $\Spec \CC$, 
then $X$ and $X^-$ are also proper over $\Spec \CC$.
\end{theorem}

Recall that an algebraic stack $\cX$ is called a {\it global quotient stack} 
if $\cX \simeq [Y/\GL_n]$, where $Y$ is an algebraic space with an action of $\GL_n$.

\begin{prop} \label{P:finite-existence}
Let $f \co \cX \to \cY$ be a morphism of algebraic stacks of finite type over  $\CC$.  Suppose that:
\begin{enumerate}
\item $f \co \cX \to \cY$ is finite and surjective.
\item There exists a good moduli space $\cX \to X$ with $X$ separated.
\item  $\cY$ is a global quotient stack and admits local quotient presentations.
\end{enumerate}
Then there exists a good moduli space $\cY \to Y$ with $Y$ separated.  
Moreover, if $X$ is proper, so is $Y$.
\end{prop}

Both Theorem \ref{T:vgit-existence} and Proposition \ref{P:finite-existence} are proved using 
Theorem \ref{T:keel-mori}.
In order to motivate the statement of Theorem \ref{T:keel-mori}, 
let us give an informal sketch of the proof. If $\cX$ admits local quotient presentations, 
then every closed point $x \in \cX$ admits an \'{e}tale neighborhood of the form 
\[
[\Spec A_x/G_x] \rightarrow \cX,
\]
where $A_x$ is a finite-type $\CC$-algebra and $G_x$ is the stabilizer of $x$. 
The union $\coprod_{x \in \cX} [\Spec A_x/G_x] $ defines an \'etale cover of $\cX$; 
reducing to a finite subcover, we obtain an atlas $f\co \cW \rightarrow \cX$ with the following properties:
\begin{enumerate}
\item $f$ is affine and \'etale.
\item $\cW$ admits a good moduli space $W$.
\end{enumerate}
Indeed, (2) follows simply by taking invariants 
$[\Spec A_x/G_x] \rightarrow \Spec A_x^{G_x}$ and since $f$ is affine, 
the fiber product $\cR := \cW \times_{\cX} \cW$ admits a good moduli space $R$.  
We may thus consider the following diagram:

\begin{equation} \label{diagram-groupoid}
\begin{split}
\xymatrix{
\cR  \ar@<.5ex>[r]^{p_1} \ar@<-.5ex>[r]_{p_2} \ar[d]^{\varphi}  
& \cW \ar[r]^f \ar[d]^{\phi} \ar[r]  & \cX \\
R  \ar@<.5ex>[r]^{q_1} \ar@<-.5ex>[r]_{q_2}  & W &
}
\end{split}
\end{equation}

The crucial question is: can we choose $f \co \cW \to \cX$ to guarantee 
that the projections $q_1, q_2 \co R \rightrightarrows W$ define an \'etale equivalence relation.
If so, then the algebraic space quotient $X = W/R$ 
gives a good moduli space for $\cX$. 

If $\cX$ is separated, we can always do this.  Indeed, if $\cX$ 
is separated, the atlas $f$ may be chosen to be 
stabilizer preserving.\footnote{The set of points $\omega \in \cW$ where $f$ is not stabilizer preserving is 
simply the image of the complement of the open substack 
$I_{\cW} \subset I_{\cX} \times_{\cX} \cW$ in $\cW$ 
and therefore is closed since $I_{\cX} \to \cX$ is proper.  
By removing this locus from $\cW$, $f \co \cW \to \cX$ may be chosen to be stabilizer preserving.} 
Thus, we may take the projections $\cR \rightrightarrows \cW$ to be stabilizer preserving and \'etale, 
and this implies that the projections $R \rightrightarrows W$ are \'etale.\footnote{To see this, note 
that if $r \in R$ is any closed point and $\rho \in \cR$ is its preimage, then 
$\hat{\oh}_{R,r} \simeq D_{\rho}^{G_{\rho}}$, where $D_{\rho}$ denotes the 
miniversal formal deformation space of $\rho$ and $G_{\rho}$ is the stabilizer 
of $\rho$; similarly $\hat{\oh}_{W, q_i(r)} \simeq D_{p_i(\rho)}^{G_{p_i(\rho)}}$.  
Now $p_i$ \'etale implies $D_{\rho} \simeq D_{p_i(\rho)}$ and $p_i$ stabilizer preserving 
implies $G_{\rho} \simeq G_{p_i(\rho)}$, so $\hat{\oh}_{R,r} \simeq \hat{\oh}_{W, q_i(r)}$, 
i.e. $q_i$ is \'etale.} This leads to a direct proof of the Keel-Mori theorem for separated 
Deligne-Mumford stacks of finite type over $\Spec \CC$ (one can show directly that such stacks always admit local quotient presentations).
In general, of course, algebraic stacks need not be separated so we must 
find weaker conditions which ensure that the projections $q_1, q_2$ are \'etale. In particular, we must identify a set of sufficient conditions which can be directly verified for 
geometrically-defined stacks such as $\Mg{g,n}(\alpha)$.

Our result gives at least one plausible answer to this problem. To begin, note that if $\omega \in \cW$ is a {\it closed} $\CC$-point with image $w \in W$, 
then the formal neighborhood $\hat{\oh}_{W,w}$ can be identified with the $G_\omega$-invariants $D_{\omega}^{G_{\omega}}$ of the miniversal deformation space $D_{\omega}$ of $\omega$. Thus, we may ensure that $q_i$ is \'etale 
at a $\CC$-point $r \in R$, or equivalently that the induced map 
$\hat{\oh}_{W,q_i(r)} \to \hat{\oh}_{R,r}$ is an isomorphism, by manually imposing the following conditions:  $p_i(\rho)$ should be a closed point, where  $\rho \in \cR$ is the unique closed point in the 
preimage of $r \in R$, and $p_i$ should induce an isomorphism of stabilizer groups $G_{\rho} \simeq G_{p_i(\rho)}$.  Indeed, we then have $\hat{\oh}_{W, q_i(r)}=D_{p_i(\rho)}^{G_{p_i(\rho)}}\simeq D_{\rho}^{G_{\rho}}=\hat{\oh}_{R, r}$, where the middle isomorphism follows from the hypothesis that $p_i$ is \'etale and stabilizer preserving. In sum, we have identified two key conditions that will imply that 
$R \rightrightarrows W$ is an \'etale equivalence relation:
\begin{enumerate}
\item[($\star$)] The morphism $f \co \cW \to \cX$ is stabilizer preserving at closed points.
\item[($\star \star$)] The projections $p_1, p_2\co  \cW \times_{\cX} \cW \rightrightarrows \cW$ send closed 
points to closed points.
\end{enumerate}

Condition ($\star$) is precisely hypothesis (1a) of Theorem \ref{T:keel-mori}. 
In practice, it is difficult to directly verify condition ($\star \star$), 
but it turns out that it is implied by conditions (1b) and (2), which are often 
easier to verify. 


Section \ref{S:local-quotient-existence} is devoted to making the above argument precise.  
Then in Sections \ref{S:vgit-existence} and \ref{S:finite-existence}, we prove 
Theorem \ref{T:vgit-existence} and Proposition \ref{P:finite-existence} by 
showing that after suitable reductions, their hypotheses imply that conditions 
(1a), (1b) and (2) of Theorem \ref{T:keel-mori} are satisfied.

\subsubsection{Definitions and preparatory material}

\begin{definition}
Let $f \co \cX \to \cY$ be a morphism of algebraic stacks of finite type over $\Spec \CC$. We say that
\begin{itemize}
\item $f$ \emph{sends closed points to closed points} if for every closed point 
$x \in \cX$, $f(x) \in \cY$ is closed.
\item $f$ is \emph{stabilizer preserving at $x \in \cX(\CC)$} if
$\Aut_{\cX(\CC)}(x) \to \Aut_{\cY(\CC)}(f(x))$ is an isomorphism.
\item For a closed point $x \in \cX$, $f$ is \emph{strongly \'etale at $x$} if $f$ is \'etale at $x$, $f$ is stabilizer preserving at $x$ and $f(x) \in \cY$ is closed.
\item $f$ is \emph{strongly \'etale} if $f$ is strongly \'etale at all closed points of $\cX$.
\end{itemize}
\end{definition}
\begin{definition}
Let $\phi \co \cX \to X$ be a good moduli space. 
We say that an open substack $\cU \subset \cX$ is \emph{saturated} if $\phi^{-1}(\phi(\cU)) = \cU$.  
\end{definition}

The following proposition is simply a stack-theoretic formulation of Luna's well-known
results in invariant theory \cite[Chapitre II]{luna} often referred to as Luna's fundamental lemma. It justifies the terminology \emph{strongly \'etale} by showing that strongly \'etale morphisms induce \'etale morphisms of good moduli spaces.  It is also shows that for a morphism of algebraic stacks admitting good moduli spaces, strongly \'etale is an open condition.

\begin{prop} \label{etale-preserving}
Consider a commutative diagram 
\begin{equation} \label{diagram-luna}
\begin{split}
\xymatrix{ 
\cW \ar[r]^f \ar[d]^{\varphi}		& \cX \ar[d]^{\phi} \\
W \ar[r]^g					& X
}
\end{split}
\end{equation}
where $f$ is a representable, separated morphism between algebraic stacks of finite type over $\Spec \CC$.  
Suppose $\varphi \co \cW \to W$ and $\phi \co \cX \to X$ are good moduli spaces.  Then
\begin{enumerate}
	\item If $f$ is strongly \'etale at $w \in \cW$, then $g$ is \'etale at $\varphi(w)$.
	\item If $f$ is strongly \'etale, then $g$ is \'etale and Diagram \eqref{diagram-luna} is Cartesian.
	\item There exists a saturated open substack $\cU \subset \cW$ such that:
		\begin{enumerate}
			\item $f|_{\cU} \co \cU \to \cX$ is strongly \'etale and $f(\cU) \subset \cX$ is saturated.
			\item If $w \in \cW$ is a closed point such that $f$ is strongly \'etale at $w$, then $w \in \cU$.
		\end{enumerate}
\end{enumerate}
\end{prop}

\begin{proof}
\cite[Theorem 5.1]{alper_good} gives part (1) and that $g$ is \'etale in (2).  
The hypotheses in (2) imply that the induced morphism $\Psi\co \cW \to W \times_X \cX$ 
is representable, separated, quasi-finite and sends closed points to closed points.   
\cite[Proposition 6.4]{alper_good} implies that $\Psi$ is finite.  
Moreover, since $f$ and $g$ are \'etale, so is $\Psi$.  But since $\cW$ and $W \times_X \cX$ 
both have $W$ as a good moduli space, it follows that a closed point in $W \times_X \cX$ 
has a unique preimage under $\Psi$.  Therefore, $\Psi$ is an isomorphism and the 
diagram is Cartesian.  Statement (3) follows from \cite[Theorem 6.10]{alper_quotient}.
\end{proof}	

\begin{lem} \label{lemma-shrinking}
Let $\cX$ be an algebraic stack of finite type over $\Spec \CC$ and $\phi \co \cX \to X$ be a good moduli space.  
Let $x \in \cX$ be a closed point 
and $\cU \subset \cX$ be an open substack containing $x$.  Then there exists a 
saturated open substack $\cU_1 \subset \cU$ containing $x$.  Moreover, if $\cX \simeq [\Spec A / G]$ with $G$ reductive, then $\cU_1$ can be chosen to be of the form $[\Spec B / G]$ for a $G$-invariant open affine subscheme $\Spec B \subset \Spec A$.
\end{lem}
\begin{proof}  The substacks $\{x\}$ and $\cX \setminus \cU$ are closed and disjoint. 
 By \cite[Theorem 4.16]{alper_good}, $\phi(\{x\})$ and $Z:=\phi(\cX \setminus \cU)$ 
 are closed and disjoint.   Therefore, we take $\cU_1 = \phi^{-1}(X \setminus Z)$.  For the second statement, take $\cU_1 = \phi^{-1}(U_1)$ for an affine open subscheme $U_1 \subset X \setminus Z$.
\end{proof}

\begin{lem} \label{lem-strongly-etale-bc}
Let $f \co \cX \to \cY$ be a strongly \'etale morphism of algebraic stacks of finite type over $\Spec \CC$.   Suppose that $\cX$ admits a good moduli space and for any point $y \in \cY(\CC)$, $\overline{ \{y \} }$ admits a good moduli space.  Then for any finite type morphism $g \co \cY' \to \cY$, the base change $f' \co \cX \times_{\cY} \cY' \to \cY'$ is strongly \'etale.
\end{lem}

\begin{proof}
Clearly, $f' $ is \'etale.  Let $x' \in \cX \times_{\cY} \cY'$ be a closed point.  To check that $f'$ is stabilizer preserving at $x'$ and $f'(x') \in \cY'$ is closed, we may replace $\cY$ with $\overline{ \{g(f'(x'))\}}$ and $\cX$ with $\overline{ \{ g'(x') \}}$ where $g' \co \cX \times_{\cY} \cY' \to \cX$.  Since $f$ is strongly \'etale, Proposition \ref{etale-preserving}(2) implies that $f$ is in fact an isomorphism in which case the desired statements regarding $f'$ are clear.
\end{proof}

\subsubsection{Existence via local quotient presentations} \label{S:local-quotient-existence}

In this section, we prove Theorem \ref{T:general-existence}.

\begin{prop}  \label{prop-existence2}
Let $\cX$ be an algebraic stack of finite type over $\Spec \CC$.   Suppose that:
\begin{enumerate}
 \item There exists an affine, strongly \'etale, surjective morphism $f \co \cX_1 \to \cX$ from an 
algebraic stack $\cX_1$ admitting a good moduli space $\phi_1 \co \cX_1 \to X_1$.
\item For any $\CC$-point $x \in \cX$, the closed substack 
	$\overline{ \{x\}}$ admits a good moduli space.
\end{enumerate}

Then $\cX$ admits a good moduli space $\phi \co \cX \to X$. 
 \end{prop}

\begin{proof} 
Set $\cX_2 = \cX_1 \times_{\cX} \cX_1$.  By Lemma \ref{lem-strongly-etale-bc}, the 
projections $p_1, p_2\co   \cX_2 \to \cX_1$ are strongly \'etale.  As $f$ is affine, there exists a 
good moduli space $\phi_2\co  \cX_2 \to X_2$ with projections $q_1, q_2 \co X_2 \to X_1$.  Similarly,  $\cX_3 := \cX_1 \times_{\cX} \cX_1 \times_{\cX} \cX_1$ admits
a good moduli space $\phi_3 \co \cX_3 \to X_3$.  By Proposition \ref{etale-preserving}(2), the induced diagram
\[
\xymatrix{
\cX_3  \ar@<1ex>[r] \ar@<-1ex>[r] \ar[r] \ar[d]^{\phi_3}
& \cX_2 \ar@<.5ex>[r] \ar@<-.5ex>[r] \ar[d]^{\phi_2}  
& \cX_1 \ar[r]^f \ar[d]^{\phi_1}    & \cX \\
X_3  \ar@<1ex>[r] \ar@<-1ex>[r] \ar[r]  & X_2 \ar@<.5ex>[r] \ar@<-.5ex>[r]  & X_1
}
\]
is Cartesian.
Moreover, by the universality of good moduli spaces, 
there is an induced identity map $X_1 \to X_2$, 
an inverse $X_2 \to X_2$ and a composition 
$X_2 \times_{q_1, X_1, q_2} X_2 \to X_2$  giving 
$X_2 \rrarrows X_1$ an \'etale groupoid structure.

To check that $\Delta\co  X_2 \to X_1 \times X_1$ is a monomorphism, 
it suffices to check that there is a unique pre-image of 
$(x_1, x_1) \in X_1 \times X_1$ where $x_1 \in X_1(\CC)$.  
Let $\xi_1 \in \cX_1$ be the unique closed point in $\phi_1^{-1}(x_1)$.  
Since $\cX_1 \to \cX$ is stabilizer preserving at $\xi_1$, we can set 
$G := \Aut_{\cX_1(\CC)}(\xi_1) \simeq \Aut_{\cX(\CC)}(f(\xi_1))$.  There are diagrams
\[
\xymatrix{
BG \ar[r] \ar[d]                & BG \times BG \ar[d] \\
\cX_2 \ar[r] \ar[d]         & \cX_1 \times \cX_1 \ar[d]  \\
\cX \ar[r]                  & \cX \times \cX
}
\qquad
\qquad
\xymatrix{
\cX_2 \ar[r]^{(p_1,p_2) \ \ } \ar[d]^{\phi_2}	
& \cX_1 \times \cX_1 \ar[d]^{\phi_1 \times \phi_1}\\
X_2 \ar[r]^{\Delta \ \ }			& X_1 \times X_1
}
\]
where the squares in the left diagram are Cartesian.  
Suppose $x_2 \in X_2(\CC)$ is a preimage of $(x_1, x_1)$ 
under $\Delta\co  X_2 \to X_1 \times X_1$.  
Let $\xi_2 \in \cX_2$ be the unique closed point in $\phi_{2}^{-1}(x_2)$.  
Then $(p_1(\xi_2), p_2(\xi_2)) \in \cX_1 \times \cX_1$ is closed and 
is therefore the unique closed point $(\xi_1, \xi_1)$ in the $(\phi_1 \times \phi_1)^{-1}(x_1,x_1)$.
But by Cartesianness of the left diagram, $\xi_2$ is the unique point in $\cX_2$ which 
maps to $(\xi_1, \xi_1)$ under  $\cX_2 \to \cX_1 \times \cX_1$.  
Therefore, $x_2$ is the unique preimage of $(x_1,x_1)$.

Since $X_2 \times_{q_1, X_1, q_2} X_2 \to X_2$ is an \'etale equivalence relation, 
there exists an algebraic space quotient $X$ and induced maps $\phi \co \cX \to X$ and $X_1 \to X$.  
Consider
\[\xymatrix{
\cX_2 \ar[r] \ar[d]  & \cX_1 \ar[r] \ar[d]     & X_1 \ar[d] \\
\cX_1 \ar[r] &\cX \ar[r]              & X
}\]
Since $\cX_2 \simeq \cX_1 \times_{X_1} X_2$ and $X_2 \simeq X_1 \times_X X_1$, 
the left and outer squares above are Cartesian.  Since $\cX_1 \to \cX$ is \'etale and surjective, 
it follows that the right square is Cartesian.  By descent (\cite[Prop. 4.7]{alper_good}), 
$\phi \co \cX \to X$ is a good moduli space.   \end{proof}

\noindent
\emph{Proof of Theorem \ref{T:keel-mori}.}
After taking a disjoint union of finitely many local quotient presentations, there 
exists a strongly \'etale, affine and surjective morphism $f \co \cW \to \cX$ where $\cW$ 
admits a good moduli space.  The theorem now follows from 
Proposition \ref{prop-existence2}. \epf

\subsubsection{Existence via local VGIT} \label{S:vgit-existence}

In this section, we prove Theorem \ref{T:vgit-existence}.  We will need the following lemma on isotrivial specializations.

\begin{lemma}\label{L:closure}
Let $\cX$ be an algebraic stack of finite type over $\Spec \CC$, and let $\cL$ be a line 
bundle on $\cX$. Let $\cX^+, \cX^- \subset \cX$ be open substacks, and let
$\cZ^+=\cX \smallsetminus \cX^+$, $\cZ^-=\cX \smallsetminus \cX^-$ be their reduced 
complements. Suppose that for all closed points $x \in \cX$, there exists a local quotient 
presentation $f \co \cW \rightarrow \cX$ around $x$ and a Cartesian diagram
\begin{equation} \begin{split} \label{D:specializations}
\xymatrix{
\cW^{+}  \ar[d] \ar@{^(->}[r]						
	&  \cW \ar[d]^f & \cW^{-}\ar@{_(->}[l] \ar[d]\\
\cX^+ \ar@{^(->}[r] 						& \cX & \cX^-  \ar@{_(->}[l] 	
} \end{split}
\end{equation}
where $\cW^+ = \cW_{\cL}^{+}$ and $\cW^- = \cW_{\cL}^{-}$ are the VGIT chambers of $\cW$ with respect to $\cL$.  Then
\begin{enumerate}
\item If $z \in \cX^+(\CC) \cap \cX^-(\CC)$, then the closure of $z$ in $\cX$ is 
contained in $\cX^+ \cap \cX^-$.
\item If $z \in \cX(\CC)$ is a closed point, then either 
$z \in \cX^+ \cap \cX^-$ or $z \in \cZ^+ \cap \cZ^-$.
\end{enumerate}
\end{lemma}
\begin{proof}
For (1),  if the closure of $z$ in $\cX$ is not contained in $\cX^+ \cap \cX^-$, 
there exists an isotrivial specialization $z\rightsquigarrow x$ to a closed point in 
$\cX \setminus (\cX^+ \cap \cX^-)$.  
Choose a local quotient presentation $f\co \cW = [W/G_x] \to \cX$ around $x$ 
such that \eqref{D:specializations} is Cartesian. 
Since $f^{-1}(x) \not\subset \cW^{+} \cap \cW^{-}$,
the character $\chi = \cL|_{BG_{x}}$ is non-trivial.  
By the Hilbert-Mumford criterion (\cite[Theorem 2.1]{git}), there exists 
a one-parameter subgroup $\lambda\co \GG_m \to G_{x}$ such that 
$\lim_{t \to 0} \lambda(t) \cdot w = w_0$ where $w\in W$ and $w_0\in W^{G_x}$  are points over $z$
and $x$, respectively.  As $w \in W^+_{\chi} \cap W^-_{\chi}$  and $w_0 \in W^{G_x}$, 
by applying Proposition \ref{P:specializations} twice with the characters $\chi$ and $\chi^{-1}$, 
we see that both $\langle \chi, \lambda \rangle < 0$ and $\langle \chi, \lambda \rangle > 0$, a contradiction.

For (2), choose a local quotient presentation $f\co (\cW,w) \to \cX$ around $z$ with $\cW = [W/G_x]$.
Let $\chi = \cL|_{BG_{x}}$ be the character of $\cL$. 
Since $w\in W^{G_x}$,  $w$ can be semistable with respect to $\chi$
if and only if $\chi$ is trivial. It follows that either $w\in \W^{+}\cap \W^{-}$ in the case $\chi$ is trivial, 
or $w\notin \W^{+} \cup\W^{-}$ in the case $\chi$ is non-trivial. 
\end{proof}

\noindent
{\it Proof of Theorem \ref{T:vgit-existence}. } 
We show that $\cX$ has a good moduli space by verifying the 
hypotheses of Theorem \ref{T:general-existence}.  Let $x_0 \in \cX$ be a closed point.
By Lemma \ref{L:closure}(2), we have either $x_0 \in \cX^+ \cap \cX^-$ or $x_0 \in \cZ^+ \cap \cZ^-$. 
Suppose first that $x_0 \in \cX^+ \cap \cX^-$. Since $\cX^+$ admits a good moduli space, 
Proposition \ref{etale-preserving}(3) implies we may choose a local quotient presentation $f \co \cW \rightarrow \cX^+$ 
which is strongly \'etale. By applying Lemma \ref{lemma-shrinking}, we may shrink further to assume that $f(\W) \subset \cX^+ \cap \cX^-$. Then Lemma \ref{L:closure}(1) implies that 
the composition $f\co \cW \rightarrow \cX^+ \hookrightarrow \cX$ is also strongly \'etale.

On the other hand,  suppose $x_0 \in \cZ^+ \cap \cZ^-$.  Choose a local quotient presentation 
$f\co (\cW,w_0) \rightarrow \cX$ around
 $x_0$ inducing a Cartesian diagram
\begin{equation}
\begin{split}
 \label{D:vgit}
\xymatrix{
\cW^{+}  \ar[d] \ar@{^(->}[r]						
	&  \cW \ar[d]^f & \cW^{-}\ar@{_(->}[l] \ar[d]\\
\cX^+ \ar@{^(->}[r] 						& \cX &  \cX^- \, \ar@{_(->}[l] 	
}
\end{split}
\end{equation}
with $\cW^+ = \cW_{\cL}^{+}$ and $\cW^- = \cW_{\cL}^{-}$.  
We claim that, after shrinking suitably, we may assume that 
$f$ is strongly \'etale. In proving this claim, we make implicit repeated use of 
Lemma \ref{lemma-shrinking} in conjunction with Lemma \ref{lemma-vgit-pullback} 
to argue that if $\cW' \subset \cW$ is an open substack containing $w_0$, there exists 
open substack $\cW'' \subset \cW'$ containing $w_0$ such that $\cW'' \to \cX$ is a local 
quotient presentation inducing a Cartesian diagram as in \eqref{D:vgit}.

Using the hypothesis that $\cZ^+,\cZ^-,$ and $\cX^+$ 
admit good moduli spaces, we will first show that $f$ may be chosen to satisfy:
\begin{enumerate}
\item[(A)] $f|_{f^{-1}(\cZ^+)}$, $f|_{f^{-1}(\cZ^-)}$ is strongly \'etale
\item[(B)] $f|_{\cW^{+}}$ is strongly \'etale.
\end{enumerate}
If $f$ satisfies (A) and (B), then $f$ is also strongly \'etale. Indeed, if $w \in \cW$ is a closed point, 
then either $w \in f^{-1}(\cZ^+) \cup f^{-1}(\cZ^-)$ or $w \in f^{-1}(\cX^+) \cap f^{-1}(\cX^-)$.
In the former case, (A) immediately implies that $f$ is stabilizer preserving at $w$ and $f(w)$ is 
closed in $\cX$. In the latter case, (B) implies that $f$ is stabilizer preserving at $w$ and that $f(w)$ is 
closed in $\cX^+$. Since $f(w) \in \cX^+ \cap \cX^-$ however, Lemma \ref{L:closure}(1) implies 
that $f(w)$ remains closed in $\cX$. 

It remains to show that $f$ can be chosen to satisfy (A) and (B).
For (A), Proposition \ref{etale-preserving}(3) implies the existence of an open substack 
$\cQ \subset f^{-1}(\cZ^+)$ containing $w_0$ such that $f|_{\cQ}$ is strongly \'etale.  
After shrinking $\cW$ suitably, we may assume 
$\cW \cap f^{-1}(\cZ^+) \subset \cQ$. One argues similarly for $f|_{f^{-1}(\cZ^-)}$.

For (B), Proposition \ref{etale-preserving}(3) implies there exists an open substack 
$\cU \subset \cW^{+}$ such that $f|_{\cU}\co  \cU \to \cX^+$ is strongly \'etale; 
moreover, $\cU$ contains all closed points 
$w \in \cW^{+}$ such that $f|_{\cW^+} \co \cW^+ \to \cX^+$ is strongly \'etale at $w$.
Let $\cV = \cW^{+} \setminus \cU$ and let $\bar{\cV}$ be the closure of $\cV$ in $\cW$.  
We claim that $w_0 \notin \bar{\cV}$. Once this is established, we may replace $\cW$ 
by an appropriate open substack of $\cW \setminus \overline{\cV}$ to obtain a local 
quotient presentation satisfying (B).  Suppose, by way of contradiction, that $w_0 \in \bar{\cV}$. 
Then there exists a specialization diagram
\[\xymatrix{
\Spec K = \Delta^* \ar[r] \ar[d]       & \cV \ar[d] \\
\Spec R = \Delta \ar[r]^{\qquad h}            & \cW
}\]
such that $h(0) = w_0$. By Proposition \ref{P:vgit-quotients}, there exist good moduli 
spaces $\cW \rightarrow W$ and $\cW^{+} \rightarrow W^+$, and the induced morphism 
$W^+ \rightarrow W$ is proper. Since the composition $\cW^{+} \to W^+ \to W$ is 
universally closed, there exists, after an extension of the fraction field $K$, a diagram
\[
\xymatrix{
 \Delta^* \ar[r] \ar[d]          & \cW^{+} \ar@{^(->}[d] \ar[r]      & W^+ \ar[d]  \\
 \Delta \ar[r]^h \ar@{-->}[ur]^{\tilde h}  & \cW \ar[r]    & W
}\]
and a lift $\tilde h\co  \Delta \to \cW^{+}$ that extends $\Delta^* \to \cW^{+}$ with 
$\tilde{w} = \tilde h(0) \in \cW^{+}$ closed.  There is an isotrivial specialization 
$\tilde{w} \rightsquigarrow w_0$.  It follows from Lemma \ref{L:closure}(1) that 
$\tilde{w} \in f^{-1}(\cZ^-)$.  By assumption (A), $f|_{\cU}\co  \cU \to \cX^+$ is strongly \'etale
at $\tilde{w}$ so that $\tilde{w} \in \cU$.  
On the other hand, the generic point of the specialization $\tilde{h}\co  \Delta \to \cW^{+}$ 
lands in $\cV$ so that $\tilde{w} \in \cV$, a contradiction.  
Thus, $w_0 \notin \bar{\cV}$ as desired. 

We have now shown that $\cX$ satisfies condition (1) in 
Theorem \ref{T:general-existence}, and it remains to verify condition (2).
Let $x \in \cX(\CC)$.  If $x \in \cZ^+$ (resp., $x \in \cZ^-$), then $\bar{\{ x \}} \subset \cZ^+$ 
(resp., $\bar{\{ x \}} \subset \cZ^-$).  Therefore, since $\cZ^+$ (resp., $\cZ^-$) admits a 
good moduli space, so does $\bar{\{x\}}$.  On the other hand, if $x \in \cX^- \cap \cX^+$, 
then Lemma \ref{L:closure}(1) implies the closure of $x$ in $\cX$ is contained in $\cX^+$.  
Since $\cX^+$ admits a good moduli space, so does $\bar{\{x\}}$. Now 
Theorem \ref{T:general-existence} implies that $\cX$ admits a good moduli space $\phi\co\cX \to X$.

Next, we use Theorem \ref{T:keel-mori} to show that $\cX^-$ admits a good moduli space.  Let $x \in \cX^-$
be a closed point and $x \rightsquigarrow x_0$ be the isotrivial specialization to the unique closed point $x_0 \in \cX$ in its closure.
By Proposition \ref{etale-preserving}, there exists a strongly \'etale local quotient presentation $f\co \cW \rightarrow \cX$
inducing a Cartesian diagram as in \eqref{diagram-etale-pres}.  By
Lemma \ref{lem-strongly-etale-bc}, the base change $f^-\co \cW^{-} \rightarrow \cX^-$ is strongly \'etale.  As $\cW^-$ 
admits a good moduli space, we may shrink $\cW^-$ further so that $f^- \co \cW^- \rightarrow \cX^-$ is a 
strongly \'etale local quotient presentation about $x$.

 It remains to check that if $x \in \cX^-(\CC)$ is any point, then its closure $\overline{\{ x \}}$ in 
 $\cX^-$ admits a good moduli space.  Let $x\rightsquigarrow x_0$ be the isotrivial specialization to the unique closed point $x_0 \in \cX$ in 
 the closure of $x$.  We claim in fact that $\phi^{-1}(\phi(x_0)) \cap \cX^-$ admits a good moduli space. 
Clearly this claim implies that $\overline{ \{x\}} \subset \cX^-$ does as well.  
We can choose a local quotient presentation $f \co (\cW, w_0) \to \cX$ 
about $x_0$ inducing a Cartesian diagram as in \eqref{diagram-etale-pres}.  
After shrinking, we may assume by Proposition \ref{etale-preserving}(3) that $f$ is strongly \'etale 
and we may also assume that $w_0$ is the unique preimage of $x_0$.  If we set $\cZ = \phi^{-1}(\phi(x_0))$, 
then $f|_{f^{-1}(\cZ)} \co f^{-1}(\cZ) \to \cZ$ is in fact an isomorphism as both $f^{-1}(\cZ)$ and $\cZ$ 
have $\Spec(\CC)$ as a good moduli space.  As $\cW_{\cL}^-$ admits a good moduli space, 
so does $\cW_{\cL}^- \cap f^{-1}(\cZ) = \cX^- \cap \cZ$.  This establishes that $\cX^-$ admits a good moduli space.

Finally, we argue that $X^+ \to X$ and $X^- \to X$ are proper and surjective. By taking a disjoint union 
of local quotient presentations and applying Proposition \ref{etale-preserving}(3), there exists a
strongly \'etale, affine, stabilizer preserving and surjective morphism 
$f \co \cW \to \cX$ from an algebraic stack admitting a good moduli space $\cW \to W$ 
such that $\cW = \cX \times_X W$.
Moreover, if we set $\cW^+:= f^{-1}(\cX^+)$ and $\cW^-:= f^{-1}(\cX^-)$, then (see 
Proposition \ref{P:vgit-quotients}) $\cW^{+}$ and  $\cW^{-}$ admit 
good moduli spaces $W^+$ and $W^-$ such that $W^- \to W$ and $W^+ \to W$ are proper and surjective.
This gives commutative cubes
\begin{equation} \label{diagram-big-cube}
\begin{split}
{\def\objectstyle{\scriptstyle}
\def\labelstyle{\scriptstyle}
\xymatrix@=30pt{
                &\cW^{+} \ar@{^(->}[rr] \ar[dd] \ar[dl]          &              
                	 & \cW \ar[dd] \ar[dl]&  &\cW^{-} \ar@{_(->}[ll] \ar[dd] \ar[dl]  \\
\cX^+\ar@{^(->}[rr] \ar[dd]&                       & \cX \ar[dd]	  & &\cX^-\ar[dd] \ar@{_(->}[ll]& \\
                &W^+ \ar[rr] \ar[dl]                &               & W\ar[dl] &&W^-\ar[ll] \ar[dl]\\
X^+\ar[rr]  &                           &X      &&X^-\ar[ll]&
}}
\end{split}
\end{equation}
The same argument as in the proof that $\cX^-$ admits a good moduli space 
shows that $f|_{\cW^{+}}\co  \cW^{+} \rightarrow \cX^+$ and $f|_{\cW^{-}}\co  \cW^{-} \rightarrow \cX^-$  
send closed points to closed points. By Proposition \ref{etale-preserving}(2), the left and right faces are Cartesian squares.  Since the top faces are also Cartesian, we have  
$\cW^{+} = \cX^+ \times_X W$ and $\cW^{-} = \cX^- \times_X W$. In particular, 
$\cW^{+} \to X^+ \times_X W$ and $\cW^{-} \to X^- \times_X W$ are good moduli spaces.  
By uniqueness of good moduli spaces, we have $X^+ \times_X W = W^+$ and 
$X^- \times_X W = W^-$.  Since $W^+ \to W$ and $W^- \to W$ are proper and 
surjective, $X^+ \to X$ and $X^- \to X$ are proper and surjective by \'etale descent.  
 \epf

\subsubsection{Existence via finite covers} \label{S:finite-existence}

In proving Proposition \ref{P:finite-existence}, we will appeal to following lemma:

 \begin{lem} \label{lemma-finite-affine-existence}
 Consider a commutative diagram
\[
\xymatrix{
 \cX \ar[r] \ar@/_1pc/[rr] & \cY \ar[r]  & X
}\]

of algebraic stacks of finite type over $\Spec \CC$ where $X$ is an algebraic space.  Suppose that:
\begin{enumerate}
\item $\cX \to \cY$ is finite and surjective.
\item $\cX \to X$ is cohomologically affine.
\item $\cY$ is a global quotient stack. 
\end{enumerate}
Then $\cY \to X$ is cohomologically affine.
 \end{lem}

 \begin{proof}  We may write $\cY = [V/\GL_n]$, where $V$ is an 
 algebraic space with an action of $\GL_n$.  Since $\cX \to \cY$ is affine, 
 $\cX$ is the quotient stack $\cX = [U/G]$ where  $U = \cX \times_{\cY} V$.  
 Since $U \to \cX$ is affine and $\cX \to X$ is cohomologically affine, $U \to X$ is affine by 
 Serre's criterion.  The morphism $U \to V$ is finite and surjective so by Chevalley's theorem, 
 we can conclude that $V \to X$ is affine.  Therefore $\cY \to X$ is cohomologically affine.
 \end{proof}

 \noindent
{\it Proof of Proposition \ref{P:finite-existence}. } Let $\cZ$ be the scheme-theoretic image of 
$\cX \to X \times \cY$.  Since $\cX \to \cY$ is finite and $X$ is separated,  $\cX \to \cZ$ is finite.  
As $\cZ$ is a global quotient stack since $\cY$ is, we may apply
Lemma \ref{lemma-finite-affine-existence} to conclude that the projection $\cZ \to X$ is cohomologically 
affine which implies that $\cZ$ admits a separated good moduli space.  The composition 
$\cZ \hookarr X \times \cY \to \cY$ is finite, surjective and stabilizer preserving at closed points.  Therefore, by replacing $\cX$ with $\cZ$, 
to prove the proposition, we may assume that $f \co \cX \to \cY$ is stabilizer preserving at closed points.  

We will now show that the hypotheses
of Theorem \ref{T:keel-mori} are satisfied.
Let $y_0 \in \cY$ be a closed point and 
 $g \co (\cY', y'_0) \to \cY$ be a local quotient presentation about $y_0$.   
 Consider the Cartesian diagram
\[
\xymatrix{
 \cX' \ar[r]^{f'} \ar[d]^{g'}       & \cY' \ar[d]^g   \\
 \cX \ar[r]^f                   & \cY 
 }\]
We claim that $g'$ is strongly \'etale at each point $x' \in f'^{-1}(y'_0)$.
Indeed, $g'$ is stabilizer preserving at 
 $x'$ by hypothesis (1) together with the fact that $g$ is stabilizer preserving at $y'_0$, and $g'(x')$ is a closed point of $\cX$ because $f(g'(x'))$ is closed.  By Proposition \ref{etale-preserving},
 there exists an open substack $\cU' \subset \cX'$ containing the fiber of $y'_0$ such
 that $g'|_{\cU'}$ is strongly \'etale.  Therefore, $y'_0 \notin \cZ = \cY' \setminus f'(\cX' \setminus \cU')$ and 
 $g|_{\cY' \setminus \cZ}$ is strongly \'etale.  By shrinking further using Lemma \ref{lemma-shrinking}, we
 obtain a local quotient presentation $g \co \cY' \to \cY$ about $y_0$ which is strongly \'etale.
 
 Finally, let $y \in \cY(\CC)$ and $x \in \cX(\CC)$ be any preimage.  Set $\cX_0 = \overline{ \{x\} } \subset \cX$ and $\cY_0 = \overline{ \{y \}} \subset \cY$.
 As $\cX_0 \to \cY_0$ is finite and surjective, $\cX_0 \to \Spec(\CC)$ is a good moduli space and $\cY_0$ is a global quotient stack,
 we may conclude using Lemma \ref{lemma-finite-affine-existence} that $\cY_0$ admits a good moduli space.  Therefore, we may apply Theorem \ref{T:keel-mori} to establish the proposition. \epf

\begin{remark*}
The hypothesis that $X$ is separated in Proposition \ref{P:finite-existence} is necessary.  
For example, let $X$ be the affine line with $0$ doubled and let $\ZZ_2$ act on $X$ 
by swapping the points at $0$ and fixing all other points.  Then $X \to [X/\ZZ_2]$ 
satisfies the hypotheses but $[X/\ZZ_2]$ does not admit a good moduli space.
\end{remark*}

\subsection{Application to $\SM_{g,n}(\alpha)$} \label{S:Existence}
In this section, we apply Theorem \ref{T:vgit-existence} to prove that the algebraic stacks 
$\SM_{g,n}(\alpha)$ admit good moduli spaces (Theorem \ref{T:Existence}). 
We have already proved that the inclusions $\SM_{g,n}(\alpha \pe) \hookrightarrow \SM_{g,n}(\alpha)\hookleftarrow \SM_{g,n}(\alpha \me)$ 
arise from local VGIT with respect to $\delta-\psi$ (Theorem \ref{theorem-etale-VGIT}). Thus, it only remains to show that for each critical 
value $\alpha_c \in \{\first, \second, \third\}$, the closed substacks
$$\begin{aligned}
\bar{\cS}_{g,n}(\alpha_c) :=  \bar{\cM}_{g,n}(\alpha_c) \setminus \bar{\cM}_{g,n}(\alpha_c\pe) \\
\bar{\cH}_{g,n}(\alpha_c) := \bar{\cM}_{g,n}(\alpha_c) \setminus \bar{\cM}_{g,n}(\alpha_c\me) 
\end{aligned}$$
admit good moduli spaces. We will prove this statement by induction on $g$. Like the boundary strata of $\SM_{g,n}$,  
$\bar{\cH}_{g,n}(\alpha_c)$ can be described (up to a finite cover) as a product of moduli spaces of $\alpha_c$-stable curves of lower genus. 
Likewise, $\bar{\cS}_{g,n}(\alpha_c)$ can be described (up to a finite cover) 
as stacky projective bundles over moduli spaces of $\alpha_c$-stable curves 
of lower genus. We use induction to deduce that these products and projective bundles admit good moduli spaces,
and then apply Proposition \ref{P:finite-existence} to conclude that $\bar{\cS}_{g,n}(\alpha_c)$ 
and $\bar{\cH}_{g,n}(\alpha_c)$ admit good moduli spaces.

\subsubsection{Existence for $\bar{\cS}_{g,n}(\alpha_c)$}\label{S:Crimp}

\begin{lem}\label{L:S-packets}
We have:
\begin{align*}
 \bar{\cS}_{1,1}(\first)&\simeq B\GG_m\\
\bar{\cS}_{1,2}(\second)& \simeq B\GG_m\\
\bar{\cS}_{2,1}(\third)& \simeq [\AA^1/\GG_m], \text{ where $\GG_m$ acts with weight $1$}.
\end{align*}
In particular, the algebraic stacks $\bar{\cS}_{1,1}(\first)$, $\bar{\cS}_{1,2}(\second)$, $\bar{\cS}_{2,1}(\third)$ admit good moduli spaces.
\end{lem}

\begin{proof} 
The algebraic stacks $\bar{\cS}_{1,1}(\first)$ and $\bar{\cS}_{1,2}(\second)$ each contain a unique $\CC$-point, namely the $\firstf$-atom and the $\secondf$-atom, and each of these curves have a $\GG_m$-automorphism group. The stack $\bar{\cS}_{2,1}(\third)$ contains two isomorphism classes of curves, namely the $\thirdf$-atom, and the rational ramphoid cuspidal curve with non-trivial crimping. We construct this stack explicitly as follows: start with the constant family $(\P^1 \times \AA^1, \infty \times \AA^1)$, let $c$ be a coordinate on $\AA^1$, and $t$ a coordinate on $\PP^1-\infty$. Now let $\P^{1} \times \AA^1 \rightarrow \C$ be the map defined by the inclusion of algebras $\CC[t^2+ct^3, t^5] \subset \CC[c,t]$ on the complement of the infinity section, and defined as an isomorphism on the complement of the zero section. Then $(\C \rightarrow \AA^1, \infty \times \AA^1)$ is a family of rational ramphoid cuspidal curves whose fiber over zero is a $\thirdf$-atom. Furthermore, $\GG_m$ acts on the base and total space of this family by $t \rightarrow \lambda^{-1} t, c \rightarrow \lambda c$, since the subalgebra $\CC[t^2+ct^3, t^5] \subset \CC[c,t]$ is invariant under this action. Thus, the family descends to $[\AA^1/\GG_m]$ and there is an induced map $[\AA^1/\GG_m] \rightarrow \SM_{2,1}(\third)$. This map is a locally closed immersion by \cite[Theorem 1.109]{vdW}, and the image is precisely $\bar{\cS}_{2,1}(\third)$. Thus, $\bar{\cS}_{2,1}(\third) \simeq [\AA^1/\GG_m]$ as desired.
\end{proof}

For higher values of $(g,n)$, the key observation is that every curve in $\bar{\cS}_{g,n}(\alpha_c)$ can be obtained from an $\alpha_c$-stable curve by `sprouting' an appropriate singularity. We make this precise in the following definition.

\begin{defn}\label{D:Sprout}
If $(C,p_1)$ is a 1-pointed curve, we say that $C'$ is a \emph{(ramphoid) cuspidal sprouting} of $(C,p_1)$ if $C'$ contains a (ramphoid) cusp $q \in C'$, and the pointed normalization of $C'$ at $q$ is isomorphic to one of:
\begin{itemize}
\item[(a)] $(C,p_1)$.
\item[(b)] $(C \cup \P^1, \infty)$ where $C$ and $\P^1$ are glued nodally by identifying $p_1 \sim 0$.
\end{itemize}

\noindent If $(C,p_1,p_2)$ is a 2-pointed curve, we say that $C'$ is a \emph{tacnodal sprouting} of $(C,p_1,p_2)$ if $C'$ contains a tacnode $q \in C'$, and the pointed normalization of $C'$ at $q$ is isomorphic to one of:
\begin{enumerate}
\item[(a)] $(C,p_1, p_2)$.
\item[(b)] $(C \cup \P^1, p_1, \infty)$ where $C$ and $\P^1$ are glued nodally by identifying $p_2 \sim 0$.
\item[(c)] $(C \cup \P^1, p_2, \infty)$ where $C$ and $\P^1$ are glued nodally by identifying $p_1 \sim 0$.
\item[(d)] $(C \cup \P^1 \cup \P^1, \infty_1, \infty_2)$ where $C$ is glued nodally to two copies of $\P^1$ along $p_1 \sim 0, p_2 \sim 0$.
\end{enumerate}
In this definition, we allow the possibility that $(C,p_1,p_2)=(C_1,p_1) \coprod (C_2, p_2)$ is disconnected, with one marked point on each connected component.

\noindent If $(C,p_1)$ is a 1-pointed curve, we say that $C'$ is a \emph{one-sided tacnodal sprouting} of $(C,p_1)$ if $C'$ contains a tacnode $q \in C'$, and the pointed normalization of $C'$ at $q$ is isomorphic to one of:
\begin{itemize}
\item[(a)] $(C,p_1) \coprod (\P^1,0)$.
\item[(b)] $(C \cup \P^1, \infty) \coprod (\P^1,0)$ where $C$ and $\P^1$ are glued nodally by identifying $p_1 \sim 0$.\\
\end{itemize}
\end{defn}

\begin{remark*} Suppose $C'$ is a cuspidal sprouting, one-sided tacnodal sprouting or ramphoid cuspidal sprouting of $(C, p_1)$ (resp., tacnodal sprouting of $(C, p_1, p_2)$) with $\alpha_c$-critical singularity $q \in C'$.    Then $(C, p_1)$ (resp., $(C, p_1, p_2)$) is the stable pointed normalization of $C'$ along $q$.  By Lemma \ref{L:StableNorm}, $C'$ is $\alpha_c$-stable 
if and only if $(C, p_1)$ (resp., $(C, p_1, p_2)$) is $\alpha_c$-stable.
\end{remark*}

\begin{lemma}\label{L:Crimp}
Fix $\alpha_c \in \{\first, \second, \third\}$, and suppose $(C, \pn) \in \bar{\cS}_{g,n}(\alpha_c)$.
\begin{enumerate}
\item If $(g,n) \neq (1,1)$, then $(C,\pn)$ is a cuspidal sprouting of a $\first$-stable curve in $\SM_{g-1,n+1}(\first)$.
\item If $(g,n) \neq (1,2)$, then one of the following holds:
\begin{enumerate}
\item $(C,\pn)$ is a tacnodal sprouting of a $\second$-stable curve in $\SM_{g-2,n+2}(\second)$.
\item $(C,\pn)$ is a tacnodal sprouting of a $\second$-stable curve in $\SM_{g-i-1,n-m+1}(\second) \times \SM_{i,m+1}(\second)$.
\item $(C, \pn)$ is a one-sided tacnodal sprouting of a $\second$-stable curve in $\SM_{g-1,n}(\second)$. 
\end{enumerate}
\item If $(g,n) \neq (2,1)$, then $(C,\pn)$ is a ramphoid cuspidal sprouting of a $\third$-stable curve in $\SM_{g-2,n+2}(\third)$.
\end{enumerate}
\end{lemma}
\begin{proof}
If $(C, \pn) \in \bar{\cS}_{g,n}(\alpha_c)$, then $(C, \pn)$ contains an $\alpha_c$-critical singularity $q \in C$. The stable pointed normalization of $(C, \pn)$ along $q$ is well-defined by our hypothesis on $(g,n)$, and is  $\alpha_c$-stable by Lemma \ref{L:StableNorm}.
\end{proof}

Lemma \ref{L:Crimp} gives a set-theoretic description of $\bar{\cS}_{g,n}(\alpha_c)$, and we must now augment this to a stack-theoretic description. This means constructing universal families of cuspidal, tacnodal, and ramphoid cuspidal sproutings. A nearly identical construction was carried out in \cite{smyth_elliptic2} for elliptic $m$-fold points (in particular, cusps and tacnodes), and for all curve singularities in \cite{vdW}. The only key difference is that here we allow \emph{all} branches to sprout $\PP^1$'s rather than a restricted subset.  Therefore, we obtain non-separated, stacky compactifications (rather than Deligne-Mumford compactifications) of the associated crimping stack of the singularity. In what follows, if $\C \rightarrow T$ is any family of curves with a section $\tau$, we say that \emph{$\C$ has an $A_k$-singularity along $\tau$} if, \'{e}tale locally on the base, the Henselization of $\C$ along $\tau$ is isomorphic to the Henselization of $T \times \CC[x,y]/(y^2-x^{k+1})$ along the zero section (cf. \cite[Definition 1.64]{vdW}). 

\begin{definition}
Let $\Sprout_{g,n}(A_k)$ denote the stack of flat families of curves $(\C \rightarrow T, \{\sigma_i\}_{i=1}^{n+1})$ satisfying
\begin{enumerate}
\item $(\C \rightarrow T, \sigman)$ is a $T$-point of $\U_{g,n}(A_k)$.
\item $\C$ has an $A_k$-singularity along $\sigma_{n+1}$.
\end{enumerate}
\end{definition}

The fact that $\Sprout_{g,n}(A_k)$ is an algebraic stack over $(\text{Schemes}/\CC)$ is verified in \cite{vdW}. There are obvious forgetful functors
$$
F_k\co  \Sprout_{g,n}(A_k) \rightarrow \U_{g,n}(A_k),
$$
given by forgetting the section $\sigma_{n+1}$.
\begin{proposition}\label{P:F2}
$F_k$ is representable and finite.
\end{proposition}
\begin{proof}
It is clear that $F_k$ is representable.  The fact that $F_2$ is quasi-finite follows from the observations that a curve $(C, \pn)$ in $\cU_{g,n}(A_k)$ has only a finite number of $A_k$-singularities and that for a $\CC$-point $x \in \Sprout_{g,n}(A_k)$, the induced map $\Aut_{\Sprout_{g,n}(A_k)}(x) \to \Aut_{\cU_{g,n}(A_k)}(F_k(x))$ on automorphism groups has finite cokernel.  To show that $F_2$ is finite, it now suffices to verify the valuative criterion for properness: let $\Delta$ be the spectrum of a discrete valuation ring, let $\Delta^*$ denote the spectrum of its fraction field, and suppose we are given a diagram
\[
\xymatrix{
 \Delta^*\ar[r] \ar[d]&\Sprout_{g,n}(A_k) \ar[d]^{F_k}\\
 \Delta \ar[r]&\U_{g,n}(A_k)\\
 }
\]
\noindent This corresponds to a diagram of families,
\[
\xymatrix{
\C_{\Delta^*}\ar[d]^{\pi_{\Delta^*}} \ar[r]&\C\ar[d]_{\pi}\\
\Delta^*\ar[r] \ar@/^1pc/[u]^{\sigma_{n+1}}&\Delta\\
}
\]
such that $\C_{\Delta^*}$ has $A_k$-singularity along $\sigma_{n+1}$. Since $\C \rightarrow \Delta$ is proper, $\sigma_{n+1}$ extends uniquely to a section of $\pi$, and since the limit of an $A_k$-singularity in $\U_{g,n}(A_k)$ is necessarily an $A_k$-singularity, $\C$ has an $A_k$-singularity along $\sigma_{n+1}$. This induces a unique lift $\Delta \rightarrow \Sprout_{g,n}(A_k)$, cf. \cite[Theorem 1.109]{vdW}.
\end{proof}

The algebraic stacks $\Sprout_{g,n}(A_k)$ also admit stable pointed normalization functors, given by forgetting the crimping data of the singularity along $\sigma_{n+1}$. To be precise, if 
$(\C \rightarrow T, \{\sigma\}_{i=1}^{n+1})$ is a $T$-point of  $\Sprout_{g,n}(A_k)$, there exists a commutative diagram
\[
\xymatrix{
&&\tilde{\C} \ar[dll]_{\phi} \ar[dd] \ar[dr]^{\psi}&\\
\C^{s} \ar[drr]&&&\C\ar[dl]\\
&&T \ar@/^1pc/[uu]^{\{\tilde{\sigma}_i\}_{i=1}^{n+\bar{k}}} \ar@/^1pc/[ull]^{\{\sigma^s_i\}_{i=1}^{n+\bar{k}}}  
\ar@/_1pc/[ur]_{\{\sigma_i\}_{i=1}^{n+\bar{k}}}&\\
}
\]
satisfying:
\begin{enumerate}
\item $(\tilde{\C} \rightarrow T, \{\tilde{\sigma}_i\}_{i=1}^{n+\bar{k}})$ is a family of $(n+\bar{k})$-pointed curves, where $\bar{k} \in \{1,2\}$.
\item $\psi$ is the pointed normalization of $\C$ along $\sigma_{n+1}$, i.e. $\psi$ is finite and restricts to an isomorphism on the open set $\tilde{\C}-\cup_{i=1}^{\bar{k}}\tilde{\sigma}_{n+i}$.
\item $\phi$ is the stabilization of $(\tilde{\C}, \{\tilde{\sigma}_i\}_{i=1}^{n+\bar{k}})$, i.e. $\phi$ is the morphism associated to a high multiple of the line bundle $\omega_{\tilde{\C}/T}(\Sigma_{i=1}^{n+\bar{k}}\tilde{\sigma}_i)$.
\end{enumerate}

\begin{remark}
Issues arise when defining the stable pointed normalization for $(g,n)$ small relative to $k$. From now on, we assume $k \in \{2,3,4\}$, and that $(g,n) \neq (1,1), (1,2), (2,1)$ when $k=2,3,4$, respectively. This ensures that the stabilization morphism $\phi$ is well-defined. Indeed, under these hypotheses, $\omega_{\tilde{\C}}(\Sigma_i\tilde{\sigma}_i)$ will be relatively big and nef, and the only components of fibers of $(\tilde{\C}, \{\tilde{\sigma}_i\}_{i=1}^{n+\bar{k}})$ on which $\omega_{\tilde{\C}}(\Sigma_i\tilde{\sigma}_i)$ has degree zero will be $\P^1$'s which meet the rest of the curve in a single node and are marked by one of the sections $\tilde{\sigma}_{n+i}$. The effect of $\phi$ is simply to blow-down these $\P^1$'s. 
\end{remark}

Since normalization and stabilization are canonically defined, the association
$$
(\C \rightarrow T, \sigman) \mapsto (\C^{s} \rightarrow T, \{\sigma_i^{s}\}_{i=1}^{n+\bar{k}})
$$
is functorial, and we obtain normalization functors:
\begin{align*}
N_2\co &\Sprout_{g,n}(A_2)\rightarrow \U_{g-1, n+1}(A_2)\\
N_3\co & \Sprout_{g,n}(A_3)\rightarrow \!\!\!\! \coprod_{\substack{g_1+g_2=g\\ n_1+n_2=n}}\!\!\!\! \big(\U_{g_1, n_1+1}(A_3) \times \U_{g_2,n_2+1}(A_3) \big)\, \coprod \,  \U_{g-2,n+2}(A_3) \, \coprod \, \U_{g-1,n+1}(A_3)\\
 N_4\co & \Sprout_{g,n}(A_4)\rightarrow \U_{g-2, n+1}(A_4)
\end{align*}
The connected components of the range of $N_{3}$ correspond to the different possibilities for the stable pointed normalization of $\C$ along $\sigma_{n+1}$. Note that the last case  $\U_{g-1,n+1}(A_3)$ corresponds to a one-sided tacnodal sprouting, i.e. one connected component of the pointed normalization of $\C$ along $\sigma_{n+1}$ is a family of 2-pointed $\P^{1}$'s. It is convenient to distinguish these possibilities by defining:
\begin{align*}
 \Sprout_{g,n}^{ns}(A_3)&=N_{3}^{-1}(\U_{g-2,n+2}(A_3))\\
\Sprout_{g,n}^{g_1,n_1}(A_3)&=N_3^{-1}\left( \U_{g_1, n_1+1}(A_3) \times \U_{g_2,n_2+1}(A_3)\right)\\
 \Sprout_{g,n}^{0,2}(A_3)&=N_3^{-1}(\U_{g-1,n+1}(A_3))
\end{align*}

The following key proposition shows that $N_k$ makes  $\Sprout_{g,n}(A_k)$ a stacky projective bundle over the moduli stack of pointed normalizations.  

We will use the following notation:  if $\E$ is a locally free sheaf on an algebraic stack $\cX$, we let $V(\E)$ denote the total space of the associated vector bundle, $[V(\E)/\GG_m]$ the quotient stack for the natural action of $\GG_m$ on the fibers of $V(\E)$, and $p\co  [V(\E)/\GG_m] \rightarrow T$ the natural projection.

\begin{proposition}\label{P:F1}
In  the following statements, we let $(\pi\co \C \rightarrow \U_{g,n}(A_k), \{\sigma_i\}_{i=1}^{n})$ denote the universal family over $\U_{g,n}(A_k)$, and $(\pi\co \C \rightarrow \U_{g_1,n_1}(A_k) \times \U_{g_2,n_2}(A_k), \{\sigma_i\}_{i=1}^{n_1}, \{\tau_i\}_{i=1}^{n_2})$ the universal family over $\U_{g_1,n_1}(A_k) \times \U_{g_2,n_2}(A_k)$.
\begin{enumerate}
\item[]
\item Let $\E$ be the invertible sheaf on $\U_{g-1,n+1}(A_2)$ defined by 
\[
\E:= \pi_*\left(\O_{\cC}(-2\sigma_{n+1})/\O_{\cC}(-3\sigma_{n+1})\right)
\]
Then there exists an isomorphism $$\gamma \co  [V(\E)/\GG_m] \simeq  \Sprout_{g,n}(A_2)$$ such that $N_2 \circ \gamma =p$.\\

\item Let $\E$ be the locally free sheaf on $\U_{g-2,n+2}(A_3)$ defined by 
\[
\E:= \pi_*\left(\O_{\cC}(-\sigma_{n+1})/\O_{\cC}(-2\sigma_{n+1})\oplus \O_{\cC}(-\sigma_{n+2})/\O_{\cC}(-2\sigma_{n+2})\right)
\]
Then there exists an isomorphism $$\gamma \co  [V(\E)/\GG_m] \simeq  \Sprout_{g,n}^{ns}(A_3)$$ such that $N_3 \circ \gamma =p$.\\

\item Let $\E$ be the locally free sheaf on $\U_{g_1,n_1+1}(A_3) \times \U_{g_2,n_1+1}(A_3)$ defined by 
\[
\E:= \pi_*\left(\O_{\cC}(-\sigma_{n_1+1})/\O_{\cC}(-2\sigma_{n_1+1})\oplus \O_{\cC}(-\tau_{n_2+1})/\O_{\cC}(-2\tau_{n_2+1})\right)
\]
Then there exists an isomorphism $$\gamma \co  [V(\E)/\GG_m] \simeq  \Sprout_{g,n}^{g_1,n_1}(A_3) $$ such that $N_3 \circ \gamma =p$.\\

\item Let $\E$ be the locally free sheaf on $\U_{g-1,n+1}(A_3)$ defined by 
\[
\E:= \pi_*\left(\O_{\cC}(-\sigma_{n+1})/\O_{\cC}(-2\sigma_{n+1})\right)
\]
Then there exists an isomorphism $$\gamma \co  [V(\E)/\GG_m] \simeq  \Sprout_{g,n}^{0,2}(A_3) $$ such that $N_3 \circ \gamma =p$.\\

\item Let $\E$ be the locally free sheaf on $\U_{g-2,n+1}(A_4)$ defined by 
$$
\E:= \pi_*\left(\O_{\cC}(-2\sigma_{n+1})/\O_{\cC}(-4\sigma_{n+1})\right)
$$
Then there exists an isomorphism $$\gamma \co  [V(\E)/\GG_m] \simeq  \Sprout_{g,n}(A_4)$$ such that $N_4 \circ \gamma =p$.
\end{enumerate}
\end{proposition}
\begin{proof}
We prove the hardest case (5), and leave the others as an exercise to the reader. To construct a map $\gamma \co [V(\E)/\GG_m] \rightarrow \Sprout_{g,n}(A_4)$, we start with a family $(\pi\co \C \rightarrow X, \{\sigma_i\}_{i=1}^{n+1})$ in $\U_{g-2,n+1}(A_4)$, and construct a family of ramphoid cuspidal sproutings over $[V(\E_X)/\GG_m]$, where
$$
\E_{X}:= \pi_*\left(\O_{\cC}(-2\sigma_{n+1})/\O_{\cC}(-4\sigma_{n+1})\right).
$$
Let $V:=V(\E_X)$, $p\co  V \rightarrow X$ the natural projection, and $( \C_V \rightarrow V, \sigma_V)$ the family obtained  from $(\C \rightarrow X, \sigma_{n+1})$ by base change along $p$. As the construction is local around $\sigma_{n+1}$, we will not keep track of $\sigman$ for the remainder of the argument. If we set  $\E_{V}=p^*\E_X$,  there exists a tautological section $e\co \O_{V} \rightarrow \E_{V}$.
Let $Z \subset V$ denote the divisor along which the composition $$\O_{V} \rightarrow \E_{V} \rightarrow  (\pi_V)_*\left(\O_{\C_V}(-2\sigma_V)/\O_{\C_V}(-3\sigma_V)\right)$$ vanishes, and let $\phi \co \tilde{\C} \rightarrow \C_V$ be the blow-up of $\C_V$ along $\sigma_{V}(Z)$. Since $\sigma_V(Z) \subset \C_{V}$ is a regular subscheme of codimension 2, the exceptional divisor $E$ of the blow-up is a $\P^{1}$-bundle over $\sigma_{V}(Z)$. In other words, for all $z \in Z$, we have
\[
\tilde{\C}_{z}=\C_{z} \cup E_{z}=\C_{z} \cup \P^1.
\]
Let $\tilde{\sigma}$ be the strict transform of $\sigma_V$ on $\tilde{\C}$, and observe that $\tilde{\sigma}$ passes through a smooth point of the $\P^1$ component in every fiber over $Z$. We will construct a map $\tilde{\C} \rightarrow \C'$ which crimps $\tilde{\sigma}$ to a ramphoid cusp, and $\C' \rightarrow X$ will be the desired family of ramphoid cuspidal sproutings.

Setting $\tilde{\pi}\co  \tilde{\C} \to \cC_V \to V$ and
$$\tilde{\E}=(\tilde{\pi})_*\left(\O_{\tilde{\C}}(-2\tilde{\sigma})/\O_{\tilde{\C}}(-4\tilde{\sigma})\right)$$
\noindent we claim that $e$ induces a section $\tilde{e}\co  \O_{V} \rightarrow \tilde{\E}$ with the property that the composition
$$
\O_{V} \rightarrow  \tilde{\E} \rightarrow \tilde{\pi}_*\left(\O_{\tilde{\C}}(-2\tilde{\sigma})/\O_{\tilde{\C}}(-3\tilde{\sigma})\right)
$$
is never zero. To see this, let
$U=\Spec R \subset X$ be an open affine along which $\E$ is trivial, and choose local coordinates on $a, b$ on $p^{-1}(U)=\Spec R[a,b]$ such that the tautological section $e$ is given by $at^2+bt^3$, where $t$ is a local equation for $\sigma_V$ on $\C_V$. In these coordinates, $\phi$ is the blow-up along $a=t=0$. Let $\tilde{a}, \tilde{t}$ be homogeneous coordinates for the blow-up and note that on the chart $\tilde{a} \neq 0$, $t':=\tilde{t}/\tilde{a}$ gives a local equation for $\tilde{\sigma}_V$. In these coordinates, $\phi$ is given by
$$
(a,b,t') \rightarrow (a,b,at')
$$
The section $at^2+bt^3$ pulls back to $a^3(t'^2+bt'^3)$, and $t'^2+bt'^3$ is a section of $\tilde{\E}$ over $p^{-1}(U)$ with the stated property.

We will use $\tilde{e}$ to construct a map $\psi\co  \tilde{\C} \rightarrow \C'$ such that $\C'$ has a ramphoid cusp along $\psi \circ \tilde{\sigma}$. It is sufficient to define $\psi$ locally around $\tilde{\sigma}$, so we may assume $\tilde{\pi}$ is affine, i.e.
$\tilde{\C}:=\Spec_{V} \tilde{\pi}_*\O_{\tilde{\cC}} \, .$ 
We specify a sheaf of $\O_{V}$-subalgebras of $\tilde{\pi}_*\O_{\tilde{\cC}}$ as follows: Consider the exact sequence
$$
0 \rightarrow \tilde{\pi}_*\cO_{\tilde{\C}_V}(-4\tilde{\sigma}) \rightarrow \tilde{\pi}_*\cO_{\tilde{\C}_V}(-2\tilde{\sigma}) \rightarrow \tilde{\E} \rightarrow 0
$$
and let $\F \subset \tilde{\pi}_*\O_{\tilde{\cC}}$ be the sheaf of $\O_{V}$-subalgebras generated by any inverse image of $\tilde{e}$ and all functions in  $\tilde{\pi}_*\cO_{\tilde{\C}}(-4\tilde{\sigma})$. We let $\psi\co  \Spec_{V} \tilde{\pi}_*\O_{\tilde{\cC}} \rightarrow \C':=\Spec_{V} \F$ be the map corresponding to the inclusion $\F \subset \tilde{\pi}_*\O_{\tilde{\cC}}$. By construction, the complete local ring $\hat{\O}_{C'_v, (\psi \circ \tilde{\sigma})(v)} \subset \hat{\O}_{\tilde{C}_v, \tilde{\sigma}(v)} \simeq \CC[[t]]$ is of the form $\CC[[t^2+bt^3, t^5]] \subset \CC[[t]]$, and this subalgebra is isomorphic to $\CC[[x,y]]/(y^2-x^5)$. 

Finally, we claim that $\C' \rightarrow V$ descends to a family of ramphoid cuspidal sproutings over the quotient stack $[V/\GG_m]$. It suffices to show that the subsheaf $\F \subset \tilde{\pi}_*\O_{\tilde{\C}}$ is invariant under the natural action of $\GG_m$ on $V$. Using the same local coordinates introduced above, the sheaf $\F$ is given over the open set $\Spec R[a,b]$ by the $R[a,b]$-algebra generated by $t'^2+bt'^3$ and $t'^5$, where $t'$ is a local equation for $\tilde{\sigma}$ on $\tilde{\C}$. To see that this algebra is $\GG_m$-invariant, 
note that the $\GG_m$-action on $V=\Spec R[a,b]$ (acting with weight 1 on $a$ and $b$) extends canonically 
to a $\GG_m$-action on the blow-up, 
where $\GG_m$ acts on $\tilde{a}$, $\tilde{t}$ with weight $1$ and $0$, respectively. 
Thus, $\GG_m$ acts on $t'=\tilde{t}/\tilde{u}$ with weight $-1$, 
so the section $t'^2+bt'^3$ is a semi-invariant. It follows that the algebra generated by $t'^2+bt'^3$ and $t'^5$ is $\GG_m$-invariant as desired. Thus, we obtain a family $(\C' \rightarrow [V/\GG_m],\psi \circ \tilde{\sigma})$ in $\Sprout_{g,n}(A_4)$ as desired.

To define an inverse map $i^{-1}\co  \Sprout_{g,n}(A_4) \rightarrow [V/\GG_m]$, we start with a family $(\C \rightarrow X, \sigma)$ in $\U_{g,n}(A_4)$ such that $\C$ has an $A_4$-singularity along $\sigma$.  We must construct a map $X \rightarrow [V(\E)/\GG_m]$. By taking the stable pointed normalization of $\C$ along $\sigma$, we obtain a diagram 
\[
\xymatrix{
&&\tilde{\C} \ar[dll]_{\phi} \ar[dd] \ar[dr]^{\psi}&\\
\C^{s} \ar[drr]&&&\C\ar[dl]\\
&&X \ar@/^1pc/[uu]^{\tilde{\sigma}} \ar@/^1pc/[ull]^{\sigma^s }  \ar@/_1pc/[ur]_{ \sigma }&\\
}
\]
satisfying
\begin{enumerate}
\item $(\tilde{\C} \rightarrow X, \tilde{\sigma})$ is a family of $(n+1)$-pointed curves.
\item $\psi$ is the pointed normalization of $\C$ along $\sigma$, i.e. $\psi$ is finite and restricts to an isomorphism on the open set $\tilde{\C}-\tilde{\sigma}$.
\item $\phi$ is the stabilization of $(\tilde{\C}, \tilde{\sigma})$, i.e. $\phi$ is the morphism associated 
to a high multiple of the relatively nef line bundle $\omega_{\tilde{\C}/X}(\tilde{\sigma})$.
\end{enumerate}
By Lemma \ref{L:StableNorm}, $(\C^{s}\rightarrow X, \sigma_i^s)$ induces a map $X \rightarrow \U_{g-2,n+1}(A_4)$, and we must show that this lifts to define a map $X \rightarrow [V(\E)/\GG_m]$. To see this, let $\F$ be the coherent sheaf defined by the following exact sequence
\[
0 \rightarrow \pi_*\O_{\C} \cap \tilde{\pi}_*\O_{\tilde{\C}}(-4\tilde{\sigma}) \subset \pi_*\O_{\C} \cap \tilde{\pi}_*\O_{\tilde{\C}}(-2\tilde{\sigma}) \rightarrow \F \rightarrow 0.
\]
The condition that $\C$ has a ramphoid cusp along $\psi \circ \tilde{\sigma}$ implies that $\F \subset \tilde{\pi}_*\O_{\tilde{\C}}(-2\sigma)/\O_{\tilde{\C}}(-4\sigma)$ is a rank one subbundle. In particular, $\F$ induces a subbundle of $\pi^s_*\O_{\C^s}(-2\sigma^s)/\O_{\C^s}(-4\sigma^s)$ over the locus of fibers on which $\phi$ is an isomorphism. A local computation, similar to the one performed in the definition of $\gamma$, shows that $\F$ extends to a subsheaf of  $\pi^s_*\O_{\C^s}(-2\sigma^s)/\O_{\C^s}(-4\sigma^s)$ over all of $X$ (though not a subbundle; the morphism on fibers is zero precisely where $\phi$ fails to be an isomorphism). The subsheaf $\F \subset \E$ induces the desired morphism $X \rightarrow [V/\GG_m]$.
\end{proof}

\begin{prop} \label{P:S-existence}
Let $\alpha_c \in \{\first, \second, \third\}$ and suppose that $\bar{\cM}_{g',n'}(\alpha_c)$ admits a proper good moduli space for all $(g',n')$ with $g' < g$.  Then $\bar{\cS}_{g,n}(\alpha_c)$ admits a proper good moduli space.
\end{prop}

\begin{proof}
Let $\alpha_c=\first$. By Lemma \ref{L:S-packets}, we may assume $(g,n) \neq (1,1)$. By Proposition \ref{P:F1}(1), there is a locally free sheaf $\E$ on $\SM_{g-1, n+1}(\first)$ such that $[V(\E)/\GG_m]$ is the base of the universal family of cuspidal sproutings of curves in $\SM_{g-1, n+1}(\first)$. By Lemma \ref{L:StableNorm}, the fibers of this family are $\first$-stable so there is an induced map
$$
\Psi\co  [V(\E)/\GG_m] \rightarrow \SM_{g,n}(\first).
$$
By Lemma \ref{L:Crimp}, $\Psi$ maps surjectively onto $\bar{\cS}_{g,n}(\first)$. Furthermore, $\Psi$ is finite by Proposition \ref{P:F2}. By hypothesis, $\SM_{g-1,n+1}(\first)$ and therefore $[V(\E)/\GG_m]$  admits a proper good moduli space.  Thus, $\bar{\cS}_{g,n}(\first)$ admits a proper good moduli space by Proposition \ref{P:finite-existence}.

Let $\alpha_c=\second$. By Lemma \ref{L:S-packets}, we may assume $(g,n) \neq (1,2)$.   If $g \geq 2$, Proposition \ref{P:F1}(2) provides a locally free sheaf $\E$ on $\SM_{g-2, n+2}(\second)$ such that $[V(\E)/\GG_m]$ is the base of the universal family of tacnodal sproutings of curves in $\SM_{g-2, n+2}(\second)$, and there is an induced map
$
[V(\E)/\GG_m] \rightarrow \SM_{g,n}(\second).
$
Similarly, for every pair of integers $(i,m)$ such that $\SM_{g-i-1, n-m+1}(\second) \times \SM_{i, m+1}(\second)$ is defined, by Proposition \ref{P:F1}(3), there is a locally free sheaf $\cE$ on $\SM_{g-i-1, n-m+1}(\second) \times \SM_{i, m+1}(\second)$ such that 
$[V(\E)/\GG_m]$
is the universal family of tacnodal sproutings.   By Lemma \ref{L:StableNorm}, there are induced maps
$
[V(\E)/\GG_m] \rightarrow \SM_{g,n}(\second).
$
Finally, Proposition \ref{P:F1}(4) provides a locally free sheaf on $\SM_{g-1, n}(\second)$ such that $[V(\E)/\GG_m]$ is the base of 
the universal family of one-sided tacnodal sproutings of curves in $\SM_{g-1, n}(\second)$. By Lemma \ref{L:StableNorm}, there is an induced map
$
[V(\E)/\GG_m] \rightarrow \SM_{g,n}(\second).
$
The union of the maps $[V(\E)/\GG_m] \rightarrow \SM_{g,n}(\second)$
cover $\bar{\cS}_{g,n}(\second)$ by Lemma \ref{L:Crimp}. Furthermore, each map is finite by Proposition \ref{P:F2}. By hypothesis, each of the stacky projective bundles $[V(\E)/\GG_m]$ admits a proper good moduli space, and therefore so does  $\bar{\cS}_{g,n}(\second)$ by Proposition \ref{P:finite-existence}.

Let $\alpha_c=\third$. By Lemma \ref{L:S-packets}, we may assume $(g,n) \neq (2,1)$.  By Proposition  \ref{P:F1}(5), there is a locally free sheaf $\E$ on $\SM_{g-2, n+1}(\third)$ such that $[V(\E)/\GG_m]$ is the base of the universal family of ramphoid cuspidal sproutings of curves in $\SM_{g-2, n+1}(\third)$. By Lemma \ref{L:StableNorm}, there is an induced map $\Psi\co  [V(\E)/\GG_m] \rightarrow \SM_{g,n}(\third)$
which maps surjectively onto $\bar{\cS}_{g,n}(\third)$ by Lemma \ref{L:Crimp}. Furthermore, $\Psi$ is finite by Proposition \ref{P:F2}. Thus, $\bar{\cS}_{g,n}(\third)$ admits a proper good moduli space by Proposition \ref{P:finite-existence}.
\end{proof}

\subsubsection{Existence for $\bar{\H}_{g,n}(\alpha_c)$}\label{S:HExistence}
In this section, we use induction on $g$ to prove that $\bar{\cH}_{g,n}(\alpha_c)$ 
admits a good moduli space. The base case is handled by the following easy lemma.
\begin{lem}\label{L:H-packets}
We have:
\begin{align*}
\bar{\cH}_{1,1}(\first)&=[\AA^2/\GG_m], \text{ with weights $4,6$.}\\
\bar{\cH}_{1,2}(\second)&=[\AA^3/\GG_m], \text{ with weights $2, 3, 4$.}\\
\bar{\cH}_{2,1}(\third)&=[\AA^4/\GG_m], \text{ with weights $4, 6, 8,10$.}
\end{align*}
In particular, $\bar{\cH}_{1,1}(\first)$, $\bar{\cH}_{1,2}(\second)$, $\bar{\cH}_{2,1}(\third)$ 
each admit a good moduli space.
\end{lem}
\begin{proof} 
We describe the case of $\bar{\cH}_{2,1}(\third)$, as the other two are essentially identical. 
Consider the family of Weierstrass tails over $\AA^4$ given by:
\[
y^2=x^5z+a_{3}x^{3}z^3+a_2x^2z^4+a_1xz^5+a_0z^6,
\]
where the Weierstrass section is given by $[1,0,0]$. Since $\GG_m$ acts on the base and 
total space of this family by
\[
x \rightarrow \lambda^2x,\,\, y \rightarrow \lambda^5y,\,\, a_i \rightarrow \lambda^{10-2i}a_i,
\]
the family descends to $[\AA^4/\GG_m]$. One checks that the induced map 
$[\AA^4/\GG_m] \rightarrow \bar{\cH}_{2,1}(\third)$ is an isomorphism.
\end{proof}

Lemma \ref{L:H-packets} gives an explicit description of the stack of elliptic tails, 
elliptic bridges, and Weierstrass tails. In the case $\alpha_c=\second$ (resp., $\alpha_c = \third$), we will also need an 
explicit description of the stack of elliptic chains (resp., Weierstrass chains) of length $r$.

\begin{lemma}\label{lem-existence-chains}
Let $r \geq 1$ be an integer, and let 
$$
\cEC_{r} \subset \SM_{2r-1, 2}(\second) \qquad (\text{resp., } \cWC_{r} \subset \SM_{2r, 1}(\third) \,)
$$
denote the closure of the locally closed substack of elliptic chains (resp., Weierstrass chains) of length $r$. Then $\cEC_{r}$ (resp., $\cWC_{r}$) 
admits a good moduli space.
\end{lemma}
\begin{proof}
For elliptic chains, Lemma \ref{L:H-packets} handles the case $r=1$ as $\cEC_{1} = \bar{\cH}_{1,2}(\second)$. By induction on $r$, we may 
assume that $\cEC_{r-1}$ admits a good moduli space.  By Proposition \ref{P:F1}(3), 
there is a locally free sheaf $\E$ on $\cEC_{r-1} \times \bar{\cH}_{1,2}(\second)$ such that $[V(\E)/\GG_m ]$ 
is the base of the universal family of tacnodal sproutings over  $\cEC_{r-1} \times \bar{\cH}_{1,2}(\second)$.  
By Lemma \ref{L:StableNorm}, there is an induced morphism
$
\Psi\co  [\V(\E)/\GG_m] \rightarrow \SM_{2r-1, 2}(\second).
$
The image of $\Psi$ is $\cEC_{r}$, and $\Psi$ is finite by Proposition \ref{P:F2}.   Since $\cEC_{r-1} \times \bar{\cH}_{1,2}(\second)$ admits a good moduli space, Proposition \ref{P:finite-existence} implies that $\cEC_{r}$ admits a good moduli space.

For Weierstrass chains, Lemma \ref{L:H-packets} again handles the case $r=1$ as $\cWC_{1} = \bar{\cH}_{2,1}(\third)$.  By induction, we may assume that $\cWC_{r-1}$ admits a good moduli space.  By Proposition \ref{P:F1}(3), 
there is a locally free sheaf $\E$ on $\bar{\cH}_{1,2}(\second) \times \cWC_{r-1}$ such that $[V(\E)/\GG_m ]$ 
is the base of the universal family of tacnodal sproutings over  $\bar{\cH}_{1,2}(\second) \times \cWC_{r-1}$. Indeed, we may take $
\E$ to be \[\pi_*\left(\O_{\cC}(-\sigma)/\O_{\cC}(-2\sigma)\oplus \O_{\cC}(-\tau)/\O_{\cC}(-2\tau)\right),\] where $\pi\co \cC \rightarrow \bar{\cH}_{1,2}(\second) \times \cWC_{r-1}$ is the universal family, $\sigma$ corresponds to one of the universal sections over $\bar{\cH}_{1,2}(\second)$, and $\tau$ corresponds to the universal section over $\cWC_{r-1}$.  If $\cV \subset [V(\E)/\GG_m]$ is the open locus parameterizing sproutings which do not introduce an elliptic bridge, then $\V$ is the complement of the subbundle $[V(\pi_*\O_{\C}(-\tau))/\GG_m] \subset [V(\E)/\GG_m]$. Since  $\bar{\cH}_{1,2}(\second) \times \cWC_{r-1}$ admits a good moduli space, and $V(\E) \backslash V(\pi_*\O_{\C}(-\tau))$ is affine over $\bar{\cH}_{1,2}(\second) \times \cWC_{r-1}$, $\V$ admits a good moduli space.
By Lemma \ref{L:StableNorm}, there is an induced morphism
$
\Psi\co  \cV \rightarrow \SM_{2r, 1}(\third).
$
The image of $\Psi$ is $\cWC_{r}$ and $\Psi$ is finite by Proposition \ref{P:F2} so Proposition \ref{P:finite-existence} implies that $\cWC_{r}$ admits a good moduli space.
\end{proof}

For higher $(g,n)$, we can use gluing maps to decompose $\bar{\cH}_{g,n}(\alpha_c)$ 
into products of lower-dimensional moduli spaces.
\begin{lem}\label{L:Gluing1}
Let $\alpha_c \in \{\first, \second, \third\}$. There exist finite gluing morphisms
$$
\Psi\co  \SM_{g_1, n_1+1}(\alpha_c) \times \SM_{g_2, n_2+1}(\alpha_c) 
\rightarrow \SM_{g_1+g_2, n_1+n_2}(\alpha_c)
$$
obtained by identifying $(C, \{p_i\}_{i=1}^{n_1+1})$ and $(C', \{p'_i\}_{i=1}^{n_2+1})$ 
nodally at $p_{n_1+1} \sim p'_{n_2+1}$
\end{lem}
\begin{proof}
$\Psi$ is well-defined by Lemma \ref{L:Glue}. To see that $\Psi$ is finite, 
first observe that $\Psi$ is clearly representable and quasi-finite. Furthermore, 
since the limit of a disconnecting node is a disconnecting node in 
$\bar{\cM}_{g,n}(\alpha_c)$ (Corollary \ref{C:Attaching1}), $\Psi$ satisfies 
the valuative criterion for properness.  
\end{proof}

In the case $\alpha_c=\second$, we will need two additional gluing morphisms.
\begin{lem}\label{L:Gluing2} There exist finite gluing morphisms
$$
\SM_{g,n+2}(\second) \times \cEC_{r} \rightarrow \SM_{g+2r,n}(\second) \, , \qquad
\cEC_{r} \rightarrow \SM_{2r}(\second),
$$
where the first map is obtained by nodally gluing $(C, \{p_i\}_{i=1}^{n+2})$ 
and an elliptic chain $(Z, q_1,q_2)$ at $p_{n+1} \sim q_1$ and $p_{n+2} \sim q_2$, 
and the second map is obtained by nodally self-gluing an elliptic chain $(Z, q_1, q_2)$ at $q_1 \sim q_2$.  
\end{lem}

\begin{proof}
These gluing maps are well-defined by Lemma \ref{L:Glue}, and finiteness 
follows as in Lemma \ref{L:Gluing1}.
\end{proof}

\begin{prop} \label{P:H-existence} Let $\alpha_c \in \{\first, \second, \third\}$ 
and suppose that $\bar{\cM}_{g',n'}(\alpha_c)$ admits a proper good moduli space 
for all $(g',n')$ satisfying $g' < g$.  Then $\bar{\cH}_{g,n}(\alpha_c)$ admits a 
proper good moduli space.
\end{prop}
\begin{proof}
Let $\alpha_c = \first$. By Lemma \ref{L:H-packets}, we may assume $(g,n) \neq (1,1)$. 
By Lemma \ref{L:Gluing1}, there exists a finite gluing morphism
\[
 \Psi\co  \bar{\cM}_{g-1,n+1}(\first) \times  \bar{\cH}_{1,1}(\first) \to \SM_{g,n}(\first),
\]
whose image is precisely $\bar{\cH}_{g,n}(\first)$. Now $\bar{\cH}_{g,n}(\first)$ admits a 
proper good moduli space by Proposition \ref{P:finite-existence}.
 
Let $\alpha_c = \second$. For every $r$ such that $\bar{\cM}_{g-2r,n+2}(\second)$  
(resp.,  $\bar{\cM}_{g-2r-1,n}(\second)$) exists, Lemma \ref{L:Gluing2} 
(resp., Lemma \ref{L:Gluing1}) gives a finite gluing morphism
\begin{align*}
 & \bar{\cM}_{g-2r,n+2}(\second) \times \cEC_{r} \to \bar{\cH}_{g,n}(\second) \\
\bigl(\text{resp.,} \ \ & \bar{\cM}_{g-2r-1,n}(\second) \times \cEC_{r} \to \bar{\cH}_{g,n}(\second) \bigr),
\end{align*}
that identifies $(C, \{p_i\}_{i=1}^{n+2})$ (resp., $(C, \{p_i\}_{i=1}^{n})$) to $(Z,q_1,q_2)$ 
at $p_{n+1} \sim q_1$, $p_{n+2} \sim q_2$ (resp., $p_n \sim q_1$). In addition, for every 
triple of integers $(i,m,r)$ such that 
$\bar{\cM}_{i,m+1}(\second) \times \bar{\cM}_{g-i-2r+1,n-m+1}(\second)$ exists, 
Lemma \ref{L:Gluing1} gives a finite gluing morphism
\begin{align*}
\bar{\cM}_{i,m+1}(\second) \times \bar{\cM}_{g-i-2r+1,n-m+1}(\second) \times \cEC_{r} \to \bar{\cH}_{g,n}(\second), 
 \end{align*}
 which identifies $(C, \{p_i\}_{i=1}^{m+1})$, $(C',\{p'_i\}_{i=1}^{n-m+1})$, 
 $(Z,q_1,q_2)$ nodally at $p_{m+1} \sim q_1$, $p'_{n-m+1} \sim q_2$. 
 Finally, if $(g,n)=(2r,0)$, Lemma \ref{L:Gluing2} gives a finite gluing morphism
\begin{align*}
\cEC_{r} \to \bar{\cH}_{2r}(\second),
\end{align*}
which nodally self-glues $(Z,q_1,q_2)$ at $q_1 \sim q_2$.
The union of these gluing morphisms covers $\bar{\cH}_{g,n}(\second)$. Thus,  
$\bar{\cH}_{g,n}(\second)$ admits a proper good moduli space by Proposition \ref{P:finite-existence} and Lemma \ref{lem-existence-chains}.

Let $\alpha_c = \third$. By Lemma \ref{L:H-packets}, we may assume $(g,n) \neq (2,1)$. 
For each $r=1, \ldots, \lfloor \frac{g}{2} \rfloor$, Lemma \ref{L:Gluing1} provides a finite gluing morphism
\[
 \bar{\cM}_{g-2r,n+1}(\third) \times  \cWC_{r}(\third) \to \SM_{g,n}(\third)
\]
(if $r=g/2$ and $n=1$, we consider $ \bar{\cM}_{g-2r,n+1}(\third)$ as the emptyset).
The union of these gluing morphisms cover $\bar{\cH}_{g,n}(\third)$. Now $\bar{\cH}_{g,n}(\third)$ 
admits a proper good moduli space by Proposition \ref{P:finite-existence} and Lemma \ref{lem-existence-chains}.
\end{proof}

\subsubsection{Existence for $\SM_{g,n}(\alpha)$}
\label{S:ExistenceFinal}

\begin{theorem}\label{T:Existence}
For every $\alpha \in (\thirdme, 1]$, $\SM_{g,n}(\alpha)$ admits a good moduli space $\gM_{g,n}(\alpha)$ which is a proper algebraic space. 
Furthermore, for each critical value $\alpha_c \in \{\third, \second, \first\}$, 
there exists a diagram
\[\xymatrix{
\SM_{g,n}(\alpha_c \pe)\ar[d] \ar@{^(->}[r]		& \SM_{g,n}(\alpha_c)  \ar[d]	
	&\SM_{g,n}(\alpha_c \me) \ar@{_(->}[l] \ar[d]\\
\gM_{g,n}(\alpha_c \pe)  \ar[r]	& \gM_{g,n}(\alpha_c)	& \gM_{g,n}(\alpha_c \me) \ar[l]
}
\]
where $\SM_{g,n}(\alpha_c) \to  \gM_{g,n}(\alpha_c)$, $\SM_{g,n}(\alpha_c \pe) \to \gM_{g,n}(\alpha_c \pe)$ and 
$\SM_{g,n}(\alpha_c \me) \to \gM_{g,n}(\alpha_c \me)$ are good moduli spaces, and where 
$\gM_{g,n}(\alpha_c \pe) \to  \gM_{g,n}(\alpha_c)$ and $\gM_{g,n}(\alpha_c \me) \to  \gM_{g,n}(\alpha_c)$ 
are proper morphisms of algebraic spaces.
\end{theorem}

\begin{remark*} The reader should not confuse $\gM_{g,n}(\alpha)$ with the projective variety $\bar{M}_{g,n}(\alpha)$
defined in \eqref{E:log-canonical-models-2}.  
The goal of Section \ref{S:projectivity} is to establish the isomorphism $\gM_{g,n}(\alpha)  \simeq \M_{g,n}(\alpha)$.
\end{remark*}

\begin{proof}
Fix $\alpha_c \in \{\first, \second, \third\}$. Note that $\SM_{0,n}(\alpha_c)=\SM_{0,n}$, 
so $\SM_{0,n}(\alpha_c)$ admits a proper good moduli space for all $n$. By induction on $g$, 
we may assume that $\SM_{g',n'}(\alpha_c)$ admits a proper good moduli space for all $(g',n')$ with $g'<g$.  
Note that $\SM_{g,n}(\alpha)=\Mg{g,n}$ for $\alpha > \first$. 
By descending induction on $\alpha$, 
we may now assume that $\SM_{g,n}(\alpha)$ admits a good moduli space for all $\alpha \geq \alpha_c \pe$. 
By Theorem \ref{theorem-etale-VGIT}, the inclusions 
$\SM_{g,n}(\alpha \pe) \hookrightarrow \SM_{g,n}(\alpha)\hookleftarrow \SM_{g,n}(\alpha \me)$ 
arise from local VGIT with respect to $\delta-\psi$, and Propositions \ref{P:H-existence} and \ref{P:S-existence} 
imply that
 $\bar{\cH}_{g,n}(\alpha_c)=\bar{\cM}_{g,n}(\alpha_c) \setminus \bar{\cM}_{g,n}(\alpha_c \me)$ and $\bar{\cS}_{g,n}(\alpha_c)=\bar{\cM}_{g,n}(\alpha_c) \setminus \bar{\cM}_{g,n}(\alpha_c \pe)$ admit proper good moduli spaces. 
Now Theorem \ref{T:vgit-existence} implies that  $\SM_{g,n}(\alpha_c)$ and $\SM_{g,n}(\alpha_c \me)$ admit proper good moduli spaces fitting into the stated diagram.
\end{proof}

%% file: 2nd-flip-Section5-submit-v9.tex
\section{Projectivity of the good moduli spaces} 
\label{S:projectivity}


Theorem \ref{T:Existence} establishes the existence of the good moduli space 
$\phi_{\alpha} \co \SM_{g,n}(\alpha) \rightarrow \gM_{g,n}(\alpha)$ for $\alpha>\thirdme$.
Since $\Mg{g,n}(\alpha)$ parameterizes unobstructed curves, it is a smooth algebraic stack and so has a
canonical divisor $K_{\Mg{g,n}(\alpha)}$. 
Because non-nodal curves in $\Mg{g,n}(\alpha)$ form a closed substack of codimension $2$,
the standard formula gives $K_{\Mg{g,n}(\alpha)}=13\lambda-2\delta+\psi$,
cf. \cite[Theorem 2.6]{logan-kodaira}.  
The main result of this section says that 
$\gM_{g,n}(\alpha)$ is projective and isomorphic to the log canonical model $\M_{g,n}(\alpha)$ defined by \eqref{E:log-canonical-models-2}:
\begin{theorem}\label{T:Projectivity}
For $\alpha>\thirdme$, the following statements hold:
\begin{enumerate}
\item The line bundle $K_{\Mg{g,n}(\alpha)}+\alpha\delta+(1-\alpha)\psi$ descends to an ample line bundle
on $\gM_{g,n}(\alpha)$. 
\item $\displaystyle{\gM_{g,n}(\alpha) \simeq \M_{g,n}(\alpha)}$.
\end{enumerate}
\end{theorem}

We proceed to prove this result \emph{assuming} Propositions \ref{P:Proj-discrepancy},
\ref{P:Proj2}, \ref{P:Proj1} and Theorem \ref{T:main-positivity},
which will be proved subsequently. Of these, Theorem \ref{T:main-positivity} is the most involved and its proof will 
occupy \S\S \ref{S:positivity-a}--\ref{S:positivity-b}.  Note that throughout this section, we make use of the
following standard abuse of notation: Whenever $\cL$ is a line bundle on $\SM_{g,n}(\alpha)$ that 
descends to the good moduli space, we denote the corresponding line bundle
on $\gM_{g,n}(\alpha)$ also by $\cL$. In this situation, 
pullback defines a natural isomorphism $\HH^0\bigl(\gM_{g,n}(\alpha), \cL\bigr)  \simeq \HH^0 \bigl(\SM_{g,n}(\alpha), \cL \bigr)$.

\begin{proof}[Proof of Theorem \ref{T:Projectivity}] 
First, we show that Part (2) follows from Part (1).
Indeed, suppose $K_{\Mg{g,n}(\alpha)}+\alpha\delta+(1-\alpha)\psi$ descends to an ample line bundle on $\gM_{g,n}(\alpha)$.
Then 
\begin{align*}
\gM_{g,n}(\alpha) \simeq \Proj \RS\bigl(\gM_{g,n}(\alpha), K_{\Mg{g,n}(\alpha)}+\alpha\delta+(1-\alpha)\psi\bigr)
\simeq \overline{M}_{g,n}(\alpha),
\end{align*}
where 
the second isomorphism is given by Proposition \ref{P:Proj-discrepancy}.

The proof of Part (1) proceeds by descending induction on $\alpha$ beginning with the known case $\alpha > 9/11$, 
when $\Mg{g,n}(\alpha)=\Mg{g,n}$. Let $\alpha_c \in \{\alpha_1=\first, \alpha_2=\second, \alpha_3=\third\}$ 
and take $\alpha_0=1$. Suppose we know Part (1) for all
$\alpha\geq \alpha_{c-1} -\epsilon$. By Theorem \ref{T:main-positivity}, 
the line bundle $K_{\Mg{g,n}(\alpha_{c-1}-\epsilon)}+\alpha_c \delta+(1-\alpha_c)\psi$ is nef on $\Mg{g,n}(\alpha_{c-1}-\epsilon)$
and all curves on which it has degree $0$ are contracted by $\gM_{g,n}(\alpha_{c-1}-\epsilon) \rightarrow \gM_{g,n}(\alpha_c)$.
It follows by Proposition \ref{P:Proj2} that the statement of Part (1) holds for all $\alpha\geq \alpha_c$.
Finally, Proposition \ref{P:Proj1} gives the statement of Part (1) for $\alpha \geq \alpha_c - \epsilon$.
\end{proof}

\begin{proposition}\label{P:Proj-discrepancy} Let $\alpha > 2/3-\epsilon$.
Suppose that $K_{\Mg{g,n}(\alpha)}+\beta\delta+(1-\beta)\psi$ descends to $\gM_{g,n}(\alpha)$ for some
$\beta \leq \alpha$. Then we have
\[
\Proj \RS \bigl(\gM_{g,n}(\alpha), K_{\Mg{g,n}(\alpha)}+\beta\delta+(1-\beta)\psi \bigr)\simeq \M_{g,n}(\beta).
\]
\end{proposition}

\begin{proof}
Consider the rational map $f_{\alpha} \co \M_{g,n} \dra \gM_{g,n}(\alpha)$. If $\alpha>9/11$,
then $f_\alpha$ is an isomorphism.
If $7/10 < \alpha \leq 9/11$, then 
$f_{\alpha}\vert_{\M_{g,n} \setminus \delta_{1,0}}$  
is an isomorphism onto the complement of the codimension $2$ locus of cuspidal curves in $\gM_{g,n}(\alpha)$. 
If $\alpha \leq 7/10$, then 
$f_{\alpha}\vert_{\M_{g,n} \setminus (\delta_{1,0}\, \cup \, \delta_{1,1})}$ 
is an isomorphism onto the complement of the codimension $2$ locus of cuspidal and tacnodal curves in $\gM_{g,n}(\alpha)$. 
(If $n=0$, then $\delta_{1,1}=\emptyset$). 
It follows that we have a discrepancy equation
\begin{equation}\label{E:disc}
f_{\alpha}^*\bigl(K_{\Mg{g,n}(\alpha)}+\beta\delta+(1-\beta)\psi\bigr) \simeq K_{\Mg{g,n}}+\beta\delta+(1-\beta)\psi +c_0 \delta_{1,0}+c_1\delta_{1,1},
\end{equation}
where $c_0=0$ if $\alpha >9/11$ and $c_1=0$ if $\alpha>7/10$. 

Let $T_1 \subset \Mg{g,n}$ be a non-trivial family of elliptic tails and $T_2 \subset \Mg{g,n} \setminus \delta_{1,0}$
be a non-trivial family of $1$-pointed elliptic tails.
Then $f_{\alpha}$ is regular along $T_1$, and for $\alpha \leq 9/11$ contracts $T_1$ to a point.
Similarly, $f_{\alpha}$ is regular along $T_2$, and for $\alpha\leq 7/10$ contracts $T_2$ to a point.
By intersecting both sides of \eqref{E:disc} with $T_1$ and $T_2$, we obtain
$c_0=11\beta-9\leq 0$ if $\alpha \leq 9/11$, and $c_1=10\beta-7\leq 0$ if $\alpha\leq 7/10$.
It follows that 
\begin{align*}
\Proj \RS\bigl(\gM_{g,n}(\alpha), K_{\Mg{g,n}(\alpha)}+\beta\delta+(1-\beta)\psi\bigr) 
\simeq \Proj \RS\bigl(\M_{g,n}, K_{\Mg{g,n}}+\beta\delta+(1-\beta)\psi\bigr) &=\M_{g,n}(\beta).
\end{align*}
\end{proof}

\begin{proposition}\label{P:Proj2}
Fix $\alpha_c \in \{\alpha_1=\first, \alpha_2=\second, \alpha_3=\third\}$ and take $\alpha_0=1$. 
Suppose that for all $0 < \epsilon \ll 1$, 
\[
K_{\SM_{g,n}(\alpha_{c-1}-\epsilon)}+(\alpha_{c-1} - \epsilon)\delta+(1-\alpha_{c-1} +\epsilon)\psi
\] 
descends to an ample line bundle on $\gM_{g,n}(\alpha_{c-1}-\epsilon)$.
In addition, suppose that 
\[
K_{\Mg{g,n}(\alpha_{c-1}-\epsilon)}+\alpha_c\delta+(1-\alpha_c)\psi
\]
is nef on $\Mg{g,n}(\alpha_{c-1} - \epsilon)$ and all curves on which it has degree $0$ are 
contracted by $\gM_{g,n}(\alpha_{c-1} - \epsilon) \rightarrow \gM_{g,n}(\alpha_c)$. Then
$K_{\SM_{g,n}(\alpha)}+\alpha \delta+(1-\alpha)\psi$ descends to an ample line
bundle on $\gM_{g,n}(\alpha)$ for all $\alpha \in [\alpha_c, \alpha_{c-1})$.
\end{proposition}
\begin{proof} 
By Proposition \ref{P:trivial-characters}, for any $\alpha_c$-closed curve $(C, \pn)$,
the action of $\Aut(C, \pn)^{\circ}$ on the fiber of $K_{\SM_{g,n}(\alpha_c)} + \alpha_c \delta + (1-\alpha_c)\psi$ is trivial.
It follows that $K_{\SM_{g,n}(\alpha_c)} + \alpha_c \delta + (1-\alpha_c)\psi$ descends to $\gM_{g,n}(\alpha_c)$. 
Consider the open immersion of stacks 
$\Mg{g,n}(\alpha_{c-1}-\epsilon) \ra \Mg{g,n}(\alpha_c)$ and the induced map on the good moduli 
spaces $j\col \gM_{g,n}(\alpha_{c-1}-\epsilon) \ra \gM_{g,n}(\alpha_c)$. We have that
\[
j^*\bigl(K_{\SM_{g,n}(\alpha_c)} + \alpha_c \delta + (1-\alpha_c)\psi\bigr)=K_{\SM_{g,n}(\alpha_{c-1}-\epsilon)}+\alpha_c\delta+(1-\alpha_c)\psi.
\]

It follows by assumption that $K_{\SM_{g,n}(\alpha_{c-1}-\epsilon)}+\alpha_c\delta+(1-\alpha_c)\psi$ descends to 
a nef line bundle on the projective variety $\gM_{g,n}(\alpha_{c-1}-\epsilon)$. 
First, we show that $K_{\SM_{g,n}(\alpha_{c-1}-\epsilon)}+\alpha_c\delta+(1-\alpha_c)\psi$ is semiample on 
$\gM_{g,n}(\alpha_{c-1}-\epsilon)$. To bootstrap from nefness to semiampleness, we first consider the case $n=0$ and $g \geq 3$.
By Proposition \ref{P:Proj-discrepancy}, the section ring of 
$K_{\Mg{g}(\alpha_{c-1}-\epsilon)}+\alpha_c\delta$ on $\gM_{g}(\alpha_{c-1}-\epsilon)$ 
is identified with the section ring of $K_{\Mg{g}}+\alpha_c\delta$ on $\M_g$. 
The latter line bundle is big, by standard bounds on the effective cone of $\M_{g}$, and finitely generated 
by \cite[Corollary 1.2.1]{BCHM}. 
We conclude that $K_{\Mg{g}(\alpha_{c-1}-\epsilon)}+\alpha_c\delta$ 
is big, nef, and finitely generated, and so is semiample
by \cite[Theorem 2.3.15]{Laz1}. When $n\geq 1$, simply note that $K_{\SM_{g+hn}(\alpha_{c-1}-\epsilon)}+\alpha_c\delta$
pulls back to $K_{\SM_{g,n}(\alpha_{c-1}-\epsilon)}+\alpha_c\delta+(1-\alpha_c)\psi$
under the morphism $\Mg{g,n}(\alpha_{c-1}-\epsilon) \ra \Mg{g+nh}(\alpha_{c-1}-\epsilon)$ defined by attaching a
fixed general curve of genus $h\geq 3$ to every marked point. 

We have established that
\[
j^*\bigl(K_{\SM_{g,n}(\alpha_c)} + \alpha_c \delta + (1-\alpha_c)\psi\bigr)=K_{\SM_{g,n}(\alpha_{c-1}-\epsilon)}+\alpha_c\delta+(1-\alpha_c)\psi
\]
is semiample on $\gM_{g,n}(\alpha_{c-1}-\epsilon)$. By assumption, it has degree $0$ only on curves 
contracted by $\gM_{g,n}(\alpha_{c-1}-\epsilon) \rightarrow \gM_{g,n}(\alpha_c)$.
We conclude that $K_{\SM_{g,n}(\alpha_c)} + \alpha_c \delta + (1-\alpha_c)\psi$ is semiample and is positive on all curves in 
$\gM_{g,n}(\alpha_c)$. Therefore, $K_{\SM_{g,n}(\alpha_c)} + \alpha_c \delta + (1-\alpha_c)\psi$ is ample on $\gM_{g,n}(\alpha_c)$.

The statement for $\alpha \in (\alpha_c, \alpha_{c-1})$ follows by interpolation.
\end{proof}

\begin{proposition}\label{P:Proj1}
Fix $\alpha_c \in \{\first, \second, \third\}$. 
Suppose that $K_{\SM_{g,n}(\alpha_c)}+\alpha_c \delta+(1-\alpha_c)\psi$ descends to an ample line bundle on 
$\gM_{g,n}(\alpha_c)$.
Then for all $0<\epsilon \ll 1$, \[
K_{\SM_{g,n}(\alpha_c-\epsilon)}+(\alpha_c-\epsilon)\delta_c+(1-\alpha_c+\epsilon)\psi
\] descends to an ample line bundle on $\gM_{g,n}(\alpha_c -\epsilon)$.
\end{proposition}
\begin{proof}
Consider the proper morphism $\pi\co  \gM_{g,n}(\alpha_c \me) \to \gM_{g,n}(\alpha_c)$
given by Theorem \ref{T:Existence}. 
Our assumption implies that $K_{\SM_{g,n}(\alpha_c -\epsilon)}+\alpha_c\delta+(1-\alpha_c)\psi$
descends to a line bundle on $\gM_{g,n}(\alpha_c \me)$ which is a pullback of an ample line bundle
on $\gM_{g,n}(\alpha_c)$ via $\pi$. To establish the proposition, it suffices to show that a positive multiple of 
$\psi-\delta$ on $\bar{\cM}_{g,n}(\alpha_c \me)$ descends to a $\pi$-ample line bundle on $\gM_{g,n}(\alpha_c \me)$.

For every $(\alpha_c \me)$-stable curve $(C, \{p_i\}_{i=1}^n)$, the induced character of $\Aut(C, \{p_i\}_{i=1}^n)^{\circ}$ 
on $\delta-\psi$ is trivial by Proposition \ref{prop-character-comparison}. It follows by \cite[Theorem 10.3]{alper_good} 
that a positive multiple of $\delta-\psi$ descends to a line bundle $\cN$ on $\gM_{g,n}(\alpha_c \me)$.

To show that $\cN^{\dual}$ is relatively ample over $\gM_{g,n}(\alpha_c)$, consider the commutative cube
{\large
\begin{equation}\label{eqn-cube}
{\def\objectstyle{\scriptstyle}
\def\labelstyle{\scriptstyle}
\xymatrix@=20pt{
				&\cW \ar[dd] \ar[dl]_f			&				& \cW^-_{\chi} \ar@{_(->}[ll] \ar[dd] \ar[dl] \\
\bar{\cM}_{g,n}(\alpha_c)  \ar[dd]^(0.4){\phi_{\alpha_c}}&				& \bar{\cM}_{g,n}(\alpha_c - \epsilon)  \ar[dd]^(0.4){\phi_{\alpha_c-\epsilon}} \ar@{_(->}[ll]	& \\
				&W\gitq G \ar[dl]				& 				& W^-_{\chi} \gitq G \ar[dl]  \ar[ll]\\
\gM_{g,n}(\alpha_c)  	&							&\gM_{g,n}(\alpha_c - \epsilon) \ar[ll]_{\pi}	&
}}\end{equation}
}
where $\cW = [\Spec A / G] \to W\gitq G = \Spec A^G $ and $\cW^-_{\chi_{\delta-\psi}} \to W_{\chi_{\delta-\psi}}^-\gitq G 
= \Proj \bigoplus_{d \ge 0} A_{d}$ 
are the good moduli spaces as in Proposition \ref{P:vgit-quotients}.  Since the vertical arrows are good moduli spaces, 
by Proposition \ref{etale-preserving} and Lemmas \ref{lemma-vgit-pullback} and \ref{lemma-shrinking}, 
after shrinking $\cW$ by a saturated open substack  such that $f$ sends closed points to closed points and is stabilizer preserving at 
closed points, we may assume that the left and right faces are Cartesian. The argument in the proof of Theorem \ref{T:vgit-existence} 
concerning Diagram \eqref{eqn-cube} 
shows that the bottom face is Cartesian.

The restriction of $\cN^{\dual}$ to $\cW^-_{\chi_{\delta-\psi}}$ descends to the relative $\oh(1)$ on $W_{\chi_{\delta-\psi}}^-\gitq G$.  
Therefore, the pullback of $\cN^{\dual}$ on $\gM_{g,n}(\alpha_c \me)$ to $W_{\chi}^-\gitq G$ is $\cO(1)$ and, in particular, is 
relatively ample over $W\gitq G$.  Since the bottom face is Cartesian, it follows by descent that $\cN^{\dual}$ is relatively ample over 
$\gM_{g,n}(\alpha_c)$. The proposition follows.
\end{proof}

\subsection{Main positivity result}
\label{S:main-result-positivity}

A well-known result of Cornalba and Harris says that 
\begin{theorem*}[\cite{CH}] \label{T:CH-positivity} 
$K_{\Mg{g,n}}+\frac{9}{11}\delta+\frac{2}{11}\psi\sim 11\lambda-\delta+\psi$ is nef on $\Mg{g,n}$ for all $(g,n)$, 
and has degree $0$ precisely on families whose only non-isotrivial components are $A_1$-attached elliptic tails.
\end{theorem*}
In a similar vein, Cornalba proved that $12\lambda-\delta+\psi$ is ample on $\Mg{g,n}$ 
and thus obtained a direct intersection-theoretic proof of the projectivity of $\M_{g,n}$ \cite{cornalba}.
We refer the reader to \cite[Chapter 14]{GAC} for the comprehensive treatment of intersection-theoretic 
approaches to projectivity of $\M_{g,n}$, many of which make appearance in the sequel. 
In the introduction to \cite{cornalba}, the author says that ``... it is hard to see how [these techniques] could be extended to other situations.'' 
In what follows, 
we do precisely that by giving intersection-theoretic proofs of projectivity for $\gM_{g,n}(\secondme)$ and $\gM_{g,n}(\thirdme)$. 

\begin{theorem}[Positivity of log canonical divisors] \label{T:main-positivity}
\hfill
\begin{enumerate}
\item[(a)] 
$K_{\Mg{g,n}(9/11 - \epsilon)}+\frac{7}{10}\delta+\frac{3}{10}\psi \sim 10\lambda-\delta+\psi$ 
is nef on $\Mg{g,n}(9/11 - \epsilon)$, and, if $(g,n)\neq (2,0)$, has degree $0$ precisely on families whose only non-isotrivial components 
are $A_1/A_1$-attached elliptic bridges. It is trivial if $(g,n)=(2,0)$. 
\item[(b)]
$K_{\Mg{g,n}(7/10 - \epsilon)}+\frac{2}{3}\delta+\frac{1}{3}\psi \sim \frac{39}{4}\lambda-\delta+\psi$ 
is nef on $\Mg{g,n}(7/10 - \epsilon)$, and has 
degree $0$ precisely on families whose only non-isotrivial components are $A_1$-attached Weierstrass chains.
\end{enumerate}
\end{theorem}

Our proof of this theorem is organized as follows. In Section \ref{S:simultaneous-norm}, 
we develop a theory of simultaneous normalization of families of at-worst tacnodal curves. 
By tracking how the relevant divisor classes change under normalization, we can reduce to proving a (more complicated) positivity result for 
families of generically smooth curves. In Section \ref{S:positivity-prelim}, we collect several preliminary positivity results, stemming from three 
sources: the Cornalba-Harris inequality, the Hodge Index Theorem, and some ad hoc divisor calculations on $\Mg{0,n}$. 
Finally, in Sections \ref{S:positivity-a} and \ref{S:positivity-b}, we combine these ingredients to prove 
parts (a) and (b) of Theorem \ref{T:main-positivity}, respectively. 

The following terminology will be in force throughout the rest of this section: 
We let $\tilde{\cU}_g$ denote the stack of connected curves of arithmetic genus $g$ with only $A$-singularities, 
and let $\tilde\cU_{g}(A_\ell)\subset \tilde\cU_{g}$ be the open substack parameterizing curves with at worst
$A_1,\dots, A_{\ell}$ singularities.
Since $\tilde\cU_g$ is smooth, we may freely alternate between line bundles and divisor classes. 
In addition, any relation between divisor classes on $\tilde\cU_g$ that holds on the open
substack of at-worst nodal curves extends to $\tilde\cU_g$. 

Let $\pi \co \cC \ra \tilde\cU_g$ be the universal family. 
We define the \emph{Hodge class} as $\lambda:=c_1(\pi_*\omega_{\pi})$ and the \emph{kappa class}
as $\kappa:=\pi_*(c_1(\omega_{\pi})^2)$. The divisor parameterizing singular curves in $\tilde\cU_g$ 
is denoted $\delta$; it can be further decomposed as $\delta=\delta_{\irr}+\delta_{\red}$, 
where $\delta_{\red}$ is the closed (by Corollary \ref{C:Attaching1}) locus of curves with disconnecting nodes. 
By the preceding remarks, Mumford's relation
$\kappa=12\lambda-\delta$ holds on $\tilde\cU_g$.
Note that the higher Hodge bundles $\pi_*(\omega_{\pi}^m)$ for $m\geq 2$ are well-defined on the open
locus in $\tilde{\cU}_g$ of curves with nef dualizing sheaf (it is the complement of the closed locus of curves with 
rational tails). 
On this locus, the Grothendieck-Riemann-Roch formula gives
\begin{equation}\label{GRR-higher-hodge}
c_1(\pi_*(\omega_{\pi}^m))=\lambda+\frac{m^2-m}{2}\kappa.
\end{equation}

Now let $\C \ra B$ be a family of curves in $\tilde\cU_g$. 
If $\sigma\co B \ra \C$ is any section of the family, we define $\psi_{\sigma}:=\sigma^*{\omega_{\C/B}}$. 
We say that $\sigma$ is \emph{smooth} if it avoids the relative singular locus of $\C/B$. 

From now on, we work only with one-parameter families $\C\to B$ over a smooth and proper curve $B$. 
If $\sigma\co B \to \C$ is generically smooth and the only singularities of fibers that $\sigma(B)$ passes through are nodes, then
$\sigma(B)$ is a $\QQ$-Cartier divisor on $\C$, and we define the \emph{index of $\sigma$} to be
\begin{equation}\label{index-section}
\iota(\sigma):=(\omega_{\C/B}+\sigma)\cdot \sigma.
\end{equation}
Notice that the index $\iota(\sigma)$ is non-negative, and if $\sigma$ is smooth, then $\iota(\sigma)=0$. We also have the following standard result:
\begin{lemma}\label{L:psi-class}
Suppose $\C \ra B$ is a generically smooth non-isotrivial family of curves in $\tilde\cU_g$.
\begin{enumerate}
\item If $g\geq 1$ and $\sigma\col B \ra \C$ is
a smooth section, then $\sigma^2 < 0$.
\item If $g=0$ and $\sigma, \sigma', \sigma'' \col B \ra \C$ are $3$ smooth sections such that $\sigma$
is disjoint from $\sigma'$ and $\sigma''$, then 
$\sigma^2 < 0$.
\end{enumerate}
\end{lemma}

Let $\C\ra B$ be a one-parameter family of curves in $\tilde\U_{g}$.
If $p\in \C$ is a node of its fiber, then 
the local equation of $\C$ at $p$ is $xy=t^{e}$, for some $e\in \ZZ$ called \emph{the index of $p$} and denoted $\ind(p)$. 
A \emph{rational tail} (resp., a \emph{rational bridge}) of a fiber 
is a $\PP^1$ meeting the rest of the fiber in exactly one (resp., two) nodes.
If $E\subset C_{b}$ is a rational tail and $p=E\cap \overline{(C_{b}\setminus E)}$, then \emph{the index of $E$} is defined to be $\ind(p)$.
Similarly, if $E\subset C_{b}$ is a rational bridge and $\{p,q\}=E\cap \overline{(C_{b}\setminus E)}$, then the index of $E$ is defined to be
$\min\{\ind(p), \ind(q)\}$. We also denote the index of $E$ by $\ind(E)$. 
We say that a rational bridge $E\subset C_{b}$ is \emph{balanced} 
if $\ind(p)=\ind(q)$. 

\subsection{Degenerations and simultaneous normalization} 
\label{S:simultaneous-norm}
Our first goal is to develop a theory of simultaneous normalization along generic 
singularities in families of at-worst tacnodal curves. In contrast to the situation for 
nodal curves, where normalization along a nodal section can always be performed because
a node is not allowed to degenerate to a worse singularity,
we must now deal with families where a node degenerates
to a cusp or a tacnode, where two nodes degenerate to a tacnode, or where a cusp degenerates to a tacnode. 

The following result, stated in the notation of \S\ref{S:deformation-openness}, describes all possible degenerations
of singularities in one-parameter families of tacnodal curves.
\begin{prop}\label{P:limits-singularities}
Suppose $\C \ra \Delta$ is a family of at-worst tacnodal curves over $\Delta$, the spectrum of a DVR.
Denote by $C_{\bar\eta}$ the geometric generic fiber and by $C_0$ the central fiber. 
Then the only possible limits in $C_0$ of the singularities of $C_{\bar\eta}$ are the following:
\begin{enumerate}
\item[(1)] A limit of a tacnode of $C_{\bar\eta}$ is necessarily a tacnode of $C_0$. 
Moreover, a limit of an outer tacnode is necessarily an outer tacnode. 
\item[(2)] A limit of a cusp of $C_{\bar\eta}$ is either a cusp or a tacnode of $C_0$.
\item[(3)] A limit of an inner node of $C_{\bar\eta}$ is either a node, a cusp, or a tacnode of $C_0$.
\item[(4)] A limit of an outer node of $C_{\bar\eta}$ is either an outer node of $C_0$ or an outer tacnode of $C_0$.
Moreover, if an outer tacnode of $C_0$ is a limit of an outer node, it must be a limit 
of two outer nodes, necessarily joining the same components.
\end{enumerate}
\end{prop}
\begin{proof}
By deformation theory of $A$-singularities, a cusp deforms only to a node, 
a tacnode deforms only either to a cusp, or to a node, or to two nodes. 
Given this, the result follows directly from 
Proposition \ref{P:limits-outer}.
\end{proof}

We describe the operation of normalization along the generic singularities for each of the following degenerations:
\begin{enumerate}
\item[(A)] Inner nodes degenerate to cusps and tacnodes (see Proposition \ref{P:generic-inner-node}).
\item[(B)] Outer nodes degenerate to tacnodes (see Proposition \ref{P:generic-outer-node}). 
\item[(C)] Cusps degenerate to tacnodes (see Proposition \ref{P:generic-cusp}).
\end{enumerate}

We begin with a preliminary result concerning normalization along a collection of generic nodes.
Suppose $\pi\co \X\to B$ is a family in $\tilde\cU_g$ with sections $\{\sigma_i\}_{i=1}^k$ such that $\sigma_i(t)$ 
are distinct nodes of $\X_b$ for a generic $b\in B$ and such that $\{\sigma_i(B)\}_{i=1}^{k}$ do not meet
any other generic singularities. (The last condition will be automatically satisfied 
when $\{\sigma_i\}_{i=1}^k$ is the collection of all inner or all outer nodes.)
Let $\nu \co \Y \ra \X$ be the normalization of $\X$ along $\cup_{i=1}^k \sigma_i(B)$. Denote by 
$\{\eta_i^{+}, \eta_i^{-}\}$ the two preimages of $\sigma_i$ (which exist after a base change). 
Let $R^{+}_i\co \nu_*\O_{\Y} \ra \O_{\sigma_i(B)}$ (resp., $R^{-}_i\co \nu_*\O_{\Y} \ra \O_{\sigma_i(B)}$) 
be the morphisms of sheaves on $\X$ induced by pushing forward
the restriction maps $\O_{\Y} \ra \O_{\eta_i^{\pm}(B)}$ and composing with the natural isomorphisms 
$\nu_*(\O_{\eta_i^{\pm}(B)})\simeq \O_{\sigma_i(B)}$.
We let $R_i := R_{i}^{+} - R_{i}^{-}$ be the difference map, and
$R:=\oplus_{i=1}^{k} R_i\co \nu_*\O_\Y \longrightarrow \oplus_{i=1}^{k} \O_{\sigma_i(B)}$. In this notation,
we have the following result. 

\begin{lemma}\label{L:nodal-normalization} 
There is an exact sequence 
\begin{equation}\label{E:pushforward}
0 \ra \O_\X \stackrel{\nu^\#}{\longrightarrow} \nu_*\O_\Y \stackrel{R}{\longrightarrow} \oplus_{i=1}^{k} \O_{\sigma_i(B)} \ra \cK \ra 0,
\end{equation}
where $\cK$ is supported on the finitely many points of $\X$ at which the generic nodes $\{\sigma_i(B)\}_{i=1}^k$ degenerate
to worse singularities. Consequently, 
\[
\lambda_{\X/B}=\lambda_{\Y/B}+\operatorname{length}(\pi_* \cK).
\]
\end{lemma} 
\begin{proof}
Away from finitely many points on $\X$ where the generic nodes degenerate, $\im(\nu^\#)=\ker(R)$ and 
$R$ is surjective.
Consider now a point $p\in \X$ where $a$ generic nodes coalesce. A local chart of $\X$ around $p$ can be taken 
to be
\[
\Spec \CC[[x,y,t]]/\bigl(y^2-(x-s_1(t))^2\cdots (x-s_a(t))^2 f(x,t)\bigr),
\]
where $x=s_i(t)$ are the equations of generic nodes. By assumption on the generic nodes, $f(x,t)$ is a square-free polynomial.
Hence $\Y=\Spec \CC[[x,u,t]]/\bigl(u^2- f(x,t)\bigr)$ and the normalization map is $y\mapsto u\prod_{i=1}^{a}(x-s_i(t))$. 

Without loss of generality, the equation 
of $\eta_i^{\pm}$ is $u=\pm v_i(t)$, where $v_i(t)^2=f(s_i(t),t)$.
It follows that $R_i \co \CC[[x,u,t]]/(u^2-f(x,t)) \to \CC[[t]]$
is given by 
\[
R_i\bigl(g(x,u,t)\bigr)=g(s_i(t), v_i(t), t)-g(s_i(t), -v_i(t), t).
\]
Write $\CC[[x,u,t]]/(u^2-f(x,t))=\CC[[x,t]]+u\CC[[x,t]]$. Clearly, $\CC[[x,t]] \subset \ker(R) \cap \im(\nu^\#)$.
Note that $ug(x,t) \in \ker(R)$ if and only if
$R_i\bigl(ug(x,t)\bigr)=2v_i(t)g(s_i(t),t)=0$ for every $i$ if and only if $g(x,t)\in (x-s_i(t))$ for every $i$.
Since the generic nodes are distinct, we conclude that $ug(x,t) \in \ker(R)$ if and only if $\prod_{i=1}^{a} (x-s_i(t)) \mid g(x,t)$
if and only if $ug(x,t) \in y\CC[[x,t]] \subset \im(\nu^\#)$. The exactness of \eqref{E:pushforward} follows.

Pushing forward \eqref{E:pushforward} to $B$ and noting that
$c_1((\pi\circ \nu)_{*} \O_\Y)=c_1(\pi_*\O_X)=c_1(\pi_*\O_{s_i(B)})=0$,
we obtain 
\[
c_1(R^1(\pi\circ \nu)_{*} \O_\Y)=c_1(R^1\pi_*\O_\X)+c_1(\pi_* \cK).
\]
The formula relating Hodge classes now follows by relative Serre duality.
\end{proof}

\begin{prop}[Type A degeneration]
\label{P:generic-inner-node}
Suppose $\X/B$ is a family in $\tilde\cU_g(A_3)$ with sections $\{\sigma_i\}_{i=1}^k$ such that $\sigma_i(b)$ 
are distinct inner nodes of $\X_b$ for a generic $b\in B$, degenerating to cusps and tacnodes over a finite set of points of $B$.
Denote by $\Y$ the normalization of $\X$ along $\cup_{i=1}^k \sigma_i(B)$ and by 
$\{\eta_i^{+}, \eta_i^{-}\}$ the two preimages of $\sigma_i$. Then $\{\eta_i^{\pm}\}$ are 
sections of $\Y/B$ satisfying: 
\begin{enumerate}
\item If $\sigma_i(b)$ is a cusp of $\X_b$, then $\eta_i^{+}(b)=\eta_i^{-}(b)$ is a smooth point of $\Y_b$. 
\item If $\sigma_i(b)$ is a tacnode of $\X_b$ and $\sigma_j(b)\neq \sigma_i(b)$ for all $j\neq i$, then 
$\eta_i^{+}(b)=\eta_i^{-}(b)$ is a node of $\Y_b$ and $\eta_i^{+}+\eta_i^{-}$ is Cartier at $b$. 
\item If $\sigma_i(b)=\sigma_j(b)$ is a tacnode of $\X_b$ for some $i\neq j$, then (up to $\pm$)
$\eta_i^{+}(b)=\eta_j^{+}(b)$ and $\eta_i^{-}(b)=\eta_j^{-}(b)$ are smooth and distinct points of $\Y_b$. 
\end{enumerate}
Set $\eta_i:=\eta_i^{+}+\eta_i^{-}$ and $\psi_{\eta_i}:=\omega_{\Y/B}\cdot \eta_i=\psi_{\eta_i^+}+\psi_{\eta_i^-}$. 
Define
\[
\psi_{inner}:=\sum_{i=1}^k\psi_{\eta_i}, \quad \delta_{tacn}:= \sum_{i\neq j} \bigl(\eta_i \cdot \eta_j \bigr),  \quad \text{and} \quad \delta_{inner}=\sum_{i=1}^k (\eta_i^{+}\cdot \eta_i^{-}).
\]
Then we have the following formulae:
\begin{align*}
\lambda_{\X/B} &=\lambda_{\Y/B}+\frac{1}{2}\delta_{tacn}+\delta_{inner}+\sum_{i=1}^{k} \iota(\eta^{+}_i), \\
\delta_{\X/B} &=\delta_{\Y/B}-\psi_{inner}+4\delta_{tacn}+10\delta_{inner}+10\sum_{i=1}^{k}\iota(\eta^{+}_i).
\end{align*}
A pair of sections $\{\eta_i^{+},\eta_i^{-}\}$ arising from the normalization 
of a generic inner node will be called \emph{inner nodal pair} and $\eta^{\pm}_i$ will be called 
\emph{inner nodal transforms}. 
\end{prop}

\begin{proof}
The formula for the Hodge class follows from Lemma \ref{L:nodal-normalization}, whose notation we keep, 
once we analyze the torsion sheaf $\cK$ on $\X$. Consider the following loci in $\X$:
\begin{enumerate}
\item[(a)] $\Cu$ is the locus of cusps in $\X/B$ which are limits of generic inner nodes.
\item[(b)] $\Tn$ is the locus of tacnodes in $\X/B$ which are limits of a single generic inner node.
\item[(c)] $\Tnt$ is the locus of tacnodes in $\X/B$ which are limits of two generic inner nodes.
\end{enumerate}

(a) A local chart of $\X$ around a point $p\in \Cu$ can be taken to be 
\[
\Spec \CC[[x,y,t]]/\bigl(y^2-(x-t^{2m})^2(x+2t^{2m})\bigr),
\]
where $x=t^{2m}$ is the equation of the generic node $\sigma$ degenerating to the cusp $p$.
Then $\Y=\Spec \CC[[x,u,t]]/\bigl(u^2-x-2t^{2m}\bigr)$ and the normalization map is $y \mapsto u(x-t^{2m})$. 
The preimages $\eta^{+}$ and $\eta^{-}$ of the generic node $\sigma$
have equations $u=\sqrt{3} t^{m}$ and $u=-\sqrt{3} t^m$. Note that $\Y$ is smooth and the intersection multiplicity of 
$\eta^{+}$ and $\eta^{-}$ at the preimage of $p$ is $m$. It follows that the contribution of $p$ to $\delta_{inner}$ is $m$.

The elements of $\CC[[x,u,t]]/\bigl(u^2-x-2t^{2m}\bigr)$ that do not lie in $\ker(R)$ are of the form $ug(x,t)$ and we have
$R\bigl(ug(x,t)\bigr)=2\sqrt{3}t^{m} g(t^{2m},t)$. It follows that 
$\im(R)=(t^m) \subset \CC[[t]]$. Hence $\cK_p=\CC[[t]]/\im(R)$ has length $m$. 

(b) A local chart of $\X$ around a point $p\in \Tn$ can be taken to be 
\[
\Spec \CC[[x,y,t]]/\bigl(y^2-(x-t^{m})^2(x^2+t^{2c})\bigr),
\]
where $x=t^{m}$ is the equation of the generic node $\sigma$ degenerating to the tacnode $p$.
Then $\Y=\Spec \CC[[x,u,t]]/(u^2-x^2-t^{2c})$ is a normal surface with $A_{2c-1}$-singularity at the preimage of $p$,
and the normalization map is given by $y \mapsto u(x-t^{m})$.
The preimage of $\sigma$ is the bi-section given by the equation $u^2=t^{2m}+t^{2c}$, which splits into two 
sections given by the equations $u=\pm v(t)$, where the valuation of $v(t)$ is equal to $\min\{m, c\}$. 
The map $R\co \CC[[x,u,t]]/(u^2-x^2-t^{2c}) \to \CC[[t]]$ sends an element of the form $ug(x,t)$ to $2v(t)g(t^m,t)$ and everything else to $0$.
We conclude that $\cK_p=\CC[[t]]/\im(R)$ has length $\min\{m, c\}$.

It remains to show that the contribution of $p$ to $\bigl(\eta^{+}\cdot \eta^{-}+\iota(\eta^{+})\bigr)$ is $\min\{m,c\}$.
There are two cases to consider. First, suppose $c\leq m$. Then the equations of $\eta^{+}$ and $\eta^{-}$ 
are $u=\alpha t^c$ and $u=-\alpha t^c$ where $\alpha \neq 0$ is a unit in $\CC[[t]]$. The minimal
resolution $h\co \wt{\Y} \ra \Y$ has the exceptional divisor
\[
E_1 \cup \cdots \cup E_{2c-1},
\] 
which is a chain of $(-2)$-curves. 
The strict transforms
$\wt{\eta}^{\, +}$ and $\wt{\eta}^{\, -}$ meet the central $(-2)$-curve $E_c$ at two distinct points. 
Clearly,  $h^*\omega_{\Y/B}=\omega_{\wt{\Y}/B}$ and
a straightforward computation shows that 
\[
h^*(\eta^{+}+\eta^{-})=\wt{\eta}^{\, +}+\wt{\eta}^{\, -}+\sum_{i=1}^{c-1} i(E_i+E_{2c-i})+ cE_c.
\]
It follows that the contribution of $p$ to 
$\bigl(\eta^{+}\cdot \eta^{-}+\iota(\eta^{+})\bigr)=(\omega_{\X/B}+\eta^{+}+\eta^{-})\cdot \eta^{+}$ 
is $c$.
 
Suppose now that $c>m$. Then the equations of $\eta^{+}$ and $\eta^{-}$ 
are $u=\alpha t^m$ and $u=-\alpha t^m$, respectively, where $\alpha \neq 0$ is a unit in $\CC[[t]]$. The exceptional divisor of the minimal
resolution $h\co \wt{\Y} \ra \Y$ is still a chain of $(-2)$-curves of length $2c-1$. However, 
$\wt{\eta}^{\, +}$ and $\wt{\eta}^{\, -}$ now meet $E_{m}$ and $E_{2c-m}$, respectively. 
It follows that the contribution of $p$ to $\bigl(\eta^{+}\cdot \eta^{-}+\iota(\eta^{+})\bigr)$ is $m$.

(c) A local chart of $\X$ around a point $p\in \Tnt$ can be taken to be 
\[
\Spec \CC[[x,y,t]]/\bigl(y^2-(x-t^{m})^2(x+t^{m})^2\bigr),
\]
where $x=t^{m}$ and $x=-t^m$ are the equations of the generic nodes $\{\sigma_1, \sigma_2\}$ coalescing to the tacnode $p$.
Then $\Y=\Spec \CC[[x,u,t]]/(u^2-1)$ is a union of two smooth sheets, and the normalization map is given by $y \mapsto u(x-t^{m})(x+t^{m})$. 
The preimages $\eta_1^{+}$ and $\eta_1^{-}$ of the generic node $\sigma_1$
have equations $\{u=1, x=t^m\}$ and $\{u=-1, x=t^m\}$. The preimages $\eta_2^{+}$ and $\eta_2^{-}$ of the generic node $\sigma_2$
have equations $\{u=1, x=-t^m\}$ and $\{u=-1, x=-t^m\}$. In particular, $\eta^{\pm}_j$ are smooth sections,
with $\eta^{+}_1$ meeting $\eta^{+}_2$, and $\eta^{-}_1$ meeting $\eta^{-}_2$, each with intersection multiplicity $m$.
It follows that the contribution of $p$ to $\delta_{tacn}$ is $2m$.

The elements of $\CC[[x,u,t]]/(u^2-1)$ that do not lie in $\ker(R)$ are of the form $ug(x,t)$ and we have
$R(ug(x,t))=(2g(t^m,t), 2g(-t^m,t)) \in \CC[[t]]\oplus \CC[[t]]$. It follows that 
\[
\im(R)= \langle (1,1), (t,t), \dots, (t^{m-1}, t^{m-1})\rangle + (t^m)\times (t^m) \subset \CC[[t]]\times \CC[[t]].
\]
Hence  $\cK_p = (\CC[[t]]\oplus \CC[[t]]) /\im(R)$ has length $m$.

It remains to prove the formula for the boundary classes. 
To do this, note that
$\nu^*\omega_{\X/B}=\omega_{\Y/B}\bigl(\sum_{i=1}^{k} (\eta_i^{+}+\eta_i^{-})\bigr)$.
Therefore, 
\begin{align*}
\kappa_{\X/B}&=\kappa_{\Y/B}
+2\sum_{1\leq i< j\leq k} \bigl((\eta_i^{+}+\eta_i^{-}) \cdot (\eta_j^{+}+\eta_j^{-})\bigr)+2\omega_{\Y/B}\cdot \sum_{i=1}^k (\eta_i^{+}+\eta_i^{-})
+\sum_{i=1}^k (\eta_i^{+}+\eta_i^{-})^2
\\ =\kappa_{\Y/B} &+2\delta_{tacn}+\omega_{\Y/B}\cdot \sum_{i=1}^k (\eta_i^{+}+\eta_i^{-})+\sum_{i=1}^k\left(\omega_{\Y/B}\cdot \eta_i^{+}+(\eta_{i}^{+})^2+\omega_{\Y/B}\cdot \eta_i^{-}+(\eta_{i}^{-})^2\right)
+2\sum_{i=1}^k \bigl(\eta_i^{+}\cdot \eta_i^{-}\bigr)\\ 
&= \kappa_{\Y/B}+2\delta_{tacn}+\psi_{inner}+2\sum_{i=1}^k\iota(\eta^{+}_i)+2\delta_{inner}.
\end{align*}
Using Mumford's relation $\kappa=12\lambda-\delta$ and the already established relation between $\lambda_{\X/B}$ and $\lambda_{\Y/B}$, 
we obtain the desired relation between $\delta_{\X/B}$ and $\delta_{\Y/B}$.
\end{proof}


\begin{prop}[Type B degeneration]
\label{P:generic-outer-node}
Suppose $\X/B$ is a family in $\tilde\cU_g(A_3)$ with sections $\{\sigma_i\}_{i=1}^k$ such that $\sigma_i(b)$ 
are outer nodes of $\X_b$ for a generic $b\in B$, degenerating to outer tacnodes over a finite set of points of $B$.
Denote by $\Y$ the normalization of $\X$ along $\cup_{i=1}^k \sigma_i(B)$ and by 
$\{\zeta_i^{+}, \zeta_i^{-}\}$ the two preimages of $\sigma_i$. Then $\{\zeta_i^{\pm}\}_{i=1}^{k}$ are smooth 
sections of $\Y$ such that $\zeta_i^{+}$ and $\zeta_i^{-}$ lie on different irreducible components of $\Y$.
Setting
\[
\delta_{tacn}:= \sum_{i\neq j} (\zeta_i^{+}+\zeta_i^{-})\cdot (\zeta_j^{+}+\zeta_j^{-}),
\]
we have the following formulae:
\begin{align*}
\lambda_{\X/B}&=\lambda_{\Y/B}+\frac{1}{2}\delta_{tacn}, \\
\delta_{\X/B}&=\delta_{\Y/B}-\sum_{i=1}^k(\psi_{\zeta_i^+}+\psi_{\zeta_i^-})+4\delta_{tacn}.
\end{align*}
The sections $\{\zeta_i^{+},\zeta_i^{-}\}_{i=1}^{k}$ will be called \emph{outer nodal transforms}. 
\end{prop}
\begin{proof}
By Proposition \ref{P:limits-singularities}, outer nodes can degenerate only to outer tacnodes. Moreover, 
an outer tacnode which is a limit of one outer node is a limit of two outer nodes. 
The statement now follows by repeating verbatim the proof of Proposition \ref{P:generic-inner-node} beginning with part (c),
and using Lemma \ref{L:nodal-normalization}.
\end{proof}

\begin{prop}[Type C degeneration]
\label{P:generic-cusp}
Suppose $\X/B$ is a family in $\tilde\cU_g$ with sections $\{\sigma_i\}_{i=1}^{k}$ such that $\sigma_i(b)$ 
is a cusp of $\X_b$ for a generic $b\in B$, degenerating to a tacnode over a finite set of points in $B$.
Denote by $\Y$ the normalization of $\X$ along $\cup_{i=1}^k \sigma_i(B)$ 
and by $\xi_i$ the preimage of $\sigma_i$. Then $\xi_i$ is a section of $\Y/B$ such that $\xi_i(t)$ is a node of $\Y_b$ whenever
$\sigma_i(b)$ is a tacnode of $\X_b$ and $\xi_i(b)$ is a smooth point of $\Y_b$ otherwise. Moreover, $2\xi_i$ is Cartier
and we have the following formulae:
\begin{align*}
\lambda_{\X/B}&=\lambda_{\Y/B}-\sum_{i=1}^{k} \psi_{\xi_i}+2\sum_{i=1}^{k} \iota(\xi_i), \\
\delta_{\X/B}&=\delta_{\Y/B}-12\sum_{i=1}^{k}\psi_{\xi_i}+20\sum_{i=1}^{k} \iota(\xi_i).
\end{align*}
The sections $\xi_i$ will be called \emph{cuspidal transforms}. 
\end{prop}

\begin{proof} 
The proof of this proposition is easier than the previous two results because 
a generic cusp cannot collide with another generic singularity. In particular, we can consider the case of a single generic cusp $\sigma$. 
Let $\nu \co \Y \ra \X$ be the normalization along $\sigma$. Suppose $\sigma(b)$ is a tacnode. Then the local equation of $\X$ around 
$\sigma(b)$ is 
\[
y^2=(x-a(t))^3(x+3a(t)),
\]
where $x=a(t)$ is the equation of the generic cusp.
It follows that $\Y$ has local equation $u^2=(x-a(t))(x+3a(t))$ and $\nu$ is given by $y \mapsto u(x-a(t))$. 
The preimage of $\sigma$ is a section $\xi\co B \ra \Y$ given by $x-a(t)=u=0$. Note that $\xi(b)=\{x=u=t=0\}$ is a node of 
$\Y_b$, and consequently $\xi$ is not Cartier at $\xi(b)$. 

Clearly, $\nu^*\omega_{\X/B}=\omega_{\Y/B}(2\xi)$ and by duality theory for singular curves 
\[
\pi_* \omega_{\X/B}=(\pi\circ \nu)_*(\omega_{\Y/B}(2\xi)).
\]
Therefore, 
\[
\kappa_{\X/B}=(\omega_{\Y/B}+2\xi)^2=(\omega_{\Y/B})^2+4(\xi^2+\xi\cdot \omega_{\Y/B})
=\kappa_{\Y/B}+4\iota(\xi),
\]
and by the Grothendieck-Riemann-Roch formula $\lambda_{\X/B}=c_1\bigl((\pi\circ \nu)_*(\omega_{\Y/B}(2\xi))\bigr)
=\lambda_{\Y/B}-\psi_{\xi}+2\iota(\xi)$. The claim follows.
\end{proof}

\subsection{Preliminary positivity results}
\label{S:positivity-prelim}

\begin{prop}[Cornalba-Harris inequality]
\label{P:CH-inequality} Let $g \geq 2$.
Suppose $f\co \C\ra B$ is a generically smooth  
family in $\tilde{\cU}_g(A_3)$, over a smooth and proper curve $B$, 
with $\omega_{\C/B}$ relatively nef. Then
\begin{equation*}
\left(8+\frac{4}{g}\right)\lambda_{\C/B}-\delta_{\C/B} \geq 0.
\end{equation*}
Moreover, if the general fiber of $\C/B$ is non-hyperelliptic and $\C/B$ is non-isotrivial, then the inequality is strict. 
\end{prop}
\begin{remark*}
When the total space $\C$ is \emph{smooth}, this result was proved in \cite{xiao} and  
\cite[Theorem 2.1]{stoppino}, with no restrictions on fiber singularities.
\end{remark*}

\begin{proof} 
As in \cite[Theorem 2.1]{stoppino}, if the general fiber of $\C/B$ is non-hyperelliptic, the result is obtained 
by the original argument of Cornalba and Harris \cite{CH}, which we now recall.

Suppose $C_{b}$ for some $b\in B$ is a non-hyperelliptic
curve of genus $g\geq 3$. After a finite base change, we can assume that $\lambda\in \Pic(B)$ is $g$-divisible. 
Then the line bundle $\cL:=\omega_{\C/B}\otimes f^*(-\lambda/g)$ on $\C$ satisfies the following conditions:
\begin{enumerate}
\item  $\det(f_*(\cL))\simeq \O_B$. 
\item $f_*(\cL^m)$ is a vector bundle of rank $(2m-1)(g-1)$ for all $m\geq 2$.
\item $\Sym^m f_*(\cL) \ra f_*(\cL^m)$ is generically surjective for all $m\geq 1$.
\end{enumerate}
For $m\geq 2$ and general $b\in B$, the map $\Sym^m \HH^0(C_b,\omega_{C_b}) \ra \HH^0(C_b,\omega^m_{C_b})$ 
defines the $m^{th}$ Hilbert point of $C_b$. Since
the canonical embedding of $C_{b}$ has a stable $m^{th}$ Hilbert point 
for some $m\gg 0$ by \cite[Lemma 14]{morrison-git}, the proof of \cite[Theorem 1.1]{CH} gives
$c_1(f_*(\cL^m))\geq 0$. 
Using \eqref{GRR-higher-hodge}, we obtain
\begin{equation}\label{E:polarization}
\left(8+\frac{4}{g}-\frac{2(g-1)}{gm}+\frac{2}{gm(m-1)}\right)\lambda -\delta=c_1(f_*(\cL^m)) \geq 0.
\end{equation}
To conclude we note that $\delta\geq 0$, and if $\delta=0$, then $\lambda>0$ 
for any non-isotrivial family by the existence of the Torelli morphism
$\M_g \ra \overline{\cA}_g$. 
We conclude that $\left(8+4/g\right)\lambda-\delta > 0$.

Suppose now that $\C\ra B$ is a family of at-worst tacnodal curves with a relatively nef $\omega_{\C/B}$ and 
a smooth hyperelliptic generic fiber. 
To prove the requisite inequality, we construct $\C/B$ explicitly as a double cover of a family of $(2g+2)$-pointed curves, 
and prove a corresponding inequality on families of rational pointed curves.

Suppose that $(\Y/B, \{\sigma_{i}\}_{i=1}^{2g+2})$ is a family of $(2g+2)$-pointed at-worst nodal rational curves where 
$\sigma_i$ are smooth sections and no more than $4$ sections meet at a point. 
We say that an irreducible 
component $E$ in the fiber $Y_{b}$ of $\Y/B$ is an odd bridge if the following conditions hold:
\begin{enumerate}
\item $E$ meets the rest of the fiber $\overline{Y_{b}\setminus E}$ in two nodes of equal index, 
\item $E\cdot \sum_{i=1}^{2g+2} \sigma_i=2$, 
\item the degree of $\sum_{i=1}^{2g+2} \sigma_i$ on each of the connected 
components of $\overline{Y_{b}\setminus E}$ is odd. 
\end{enumerate}
Suppose $h\co \Y\ra \Z$ is a blow-down of some collection of odd bridges.
The image of $\sum_{i=1}^{2g+2} \sigma_i$ in $\Z$ will be denoted by $\Sigma$.
Note that while the individual images of $\sigma_i$'s are not Cartier on $\Z$
along the image of blown-down odd bridges, the total class of $\Sigma$ is Cartier on $\Z$.
We say that a node $p\in \Z_b$ (resp., $p\in \Y_b$) is an odd node if the degree of $\Sigma$ (resp., $\sum_{i=1}^{2g+2} \sigma_i$)
on each of the connected component of the normalization of $\Z_{b}$ (resp., $\Y_b$) at $p$ is odd. We denote by $\delta_{\odd}$
the Cartier divisor on $B$ associated to all odd nodes of $\Z/B$ (resp., $\Y/B$).

The hyperelliptic involution on the generic fiber of $f\co \C \to B$ extends to all of $\C$ and  
realizes $\C/B$ as a double cover of a family $(\Z/B, \Sigma)$ described above
in such a way that $\C \ra \Z$ ramifies over $\Sigma$. Let $\delta_{\odd}$ be the divisor of odd nodes of $\Z/B$.
We have the following 
standard formulae:
\begin{align*}
\lambda_{\C/B}&=\frac{1}{8}\left(\Sigma^2+2\omega_{\Z/B} \cdot \Sigma-\delta_{odd}\right)_{\Z/B}, \\
\delta_{\C/B}&=\left(\Sigma^2+\omega_{\Z/B} \cdot \Sigma+2\omega_{\Z/B}^2-\frac{3}{2}\delta_{odd}\right)_{\Z/B}.
\end{align*}
Consider $h\col \Y\ra \Z$. 
Then $h^*(\Sigma)=\sum_{i=1}^{2g+2}\sigma_i+E$, where $E$ is a collection of odd bridges, and $h^* \omega_{\Z/B}=\omega_{\Y/B}$. Set
$\psi_{\Y/B}:=\omega_{\Y/B}\cdot \sum_{i=1}^{2g+2}\sigma_i$, $\delta_{inner}:=\sum_{i\neq j} (\sigma_i\cdot \sigma_j)$, 
and $e:=-\frac{1}{2}E^2$.
Then
\begin{align*}
\lambda_{\C/B}&=\left(\frac{1}{8}(\psi_{\Y/B}+2\delta_{inner}-\delta_{odd})+\frac{1}{2}e\right)_{\Y/B}, \\
\delta_{\C/B}&=\left(2\delta_{inner}+2\delta_{even}+\frac{1}{2}\delta_{odd}+5e\right)_{\Y/B}.
\end{align*}
We obtain
\begin{equation*}
\left(8+\frac{4}{g}\right)\lambda_{\C/B}-\delta_{\C/B} 
=\left(\frac{2g+1}{2g}\psi+\frac{1}{g}\delta_{inner}+\left(\frac{2}{g}-1\right)e-2\delta_{even}-\left(\frac{3}{2}+\frac{1}{2g}\right)\delta_{odd}\right)_{\Y/B}.
\end{equation*}
Multiplying by $2g$, we need to show that on $\Y/B$ we have
\[
(2g+1)\psi+2\delta_{inner} -4g \delta_{even}-(3g+1)\delta_{odd} -(2g-4)e \geq 0.
\]
Noting that
\[
(2g+1)\psi+2\delta_{inner}=\sum_{i=2}^{g+1}i(2g+2-i)\delta_{i},
\]
and using the inequality $2e\leq \delta_{odd}$, we obtain the desired claim.
\end{proof}

\subsubsection*{Hodge Index Theorem Inequalities}
We apply a method of Harris \cite{harris-smooth} to obtain inequalities between the $\psi$-classes,
indices of cuspidal and inner nodal transforms, and the $\kappa$ class.  
In the following lemmas, we use the  following variant of Hodge Index Theorem for singular surfaces.
\begin{lemma}
Let $S$ be a proper reduced algebraic space of dimension $2$. Suppose there exists a Cartier divisor 
$H$ on $S$ such that $H^2>0$. Then the intersection pairing on $\operatorname{NS}(S)$
has signature $(1, \ell)$. 
\end{lemma}
\begin{proof}
Let $\pi\co \wt{S} \ra S$ be the minimal desingularization of the normalization of $S$. Then $\wt{S}$ is a smooth 
projective surface.  
Note that $\pi^*\co \operatorname{NS}(S) \ra \operatorname{NS}(\wt{S})$ is an injection 
preserving the intersection pairing. The statement now follows
from the Hodge Index Theorem for smooth projective surfaces. 
\end{proof}

\begin{lemma}\label{L:cuspidal-section-high} 
Suppose $\X/B$ is a family of Gorenstein curves of arithmetic genus
$g\geq 2$ with a section $\xi$. Let $\iota(\xi)=(\xi+\omega_{\X/B})\cdot \xi$ be the index of $\xi$. Then
\begin{equation}
\psi_{\xi} \geq \frac{(g-1)}{g}\iota(\xi)+\frac{\kappa}{4g(g-1)}.
\end{equation}

\end{lemma}
\begin{proof}
Apply the Hodge Index Theorem to the three classes
$\langle F, \xi, \omega_{\X/B}\rangle$, where $F$ is the fiber class. 
Since $\xi+k F$ has positive self-intersection for $k\gg 0$, the determinant of the following intersection pairing matrix is non-negative:
\[
\left(\begin{array}{ccc} 0 & 1 & 2g-2 \\ 
1 & -\psi_{\xi}+\iota(\xi) & \psi_{\xi} \\
2g-2  & \psi_{\xi} & \kappa \end{array}\right).
\]
The claim follows by expanding the determinant.
\end{proof}

\begin{lemma}\label{L:inner-sections-high} Suppose $\X/B$ is a family of Gorenstein curves of arithmetic genus
$g\geq 2$ with a pair of sections $\eta^{+}, \eta^{-}$.
Then
\begin{equation}
\psi_{\eta^+}+\psi_{\eta^-} \geq \frac{2(g-1)}{g+1}\bigl((\eta^+\cdot \eta^-)+\iota(\eta^+)\bigr)+\frac{\kappa}{g^2-1}.
\end{equation}
\end{lemma}
\begin{proof}
Consider the three divisor classes $\langle F, \eta=\eta^{+}+\eta^{-}, \omega_{\X/B}\rangle$,
where $F$ is the fiber class. 
Since $\eta+kF$ has positive self-intersection for $k\gg 0$,
the Hodge Index Theorem implies that the determinant of the following intersection pairing matrix
is non-negative:
\[
\left(\begin{array}{ccc} 0 & 2 & 2g-2 \\ 
2 & -\psi_{\eta^+}-\psi_{\eta^-} +2(\eta^+\cdot\eta^-)+\iota(\eta^{+})+\iota(\eta^-) & \psi_{\eta^+}+\psi_{\eta^-} \\
2g-2  & \psi_{\eta^+}+\psi_{\eta^-} & \kappa \end{array}\right).
\]
The claim follows by expanding the determinant.
\end{proof}

\begin{lemma}\label{L:genus-2-equality}
Suppose $\X/B$ is a family in $\tilde{\cU}_2(A_3)$ with a smooth section $\tau$. Then 
\begin{equation}
8\psi_{\tau} \geq \kappa.
\end{equation}
Moreover, if $\delta_{\red}=0$, then the equality is satisfied if and only if $(\X/B, \tau)$ is a family of Weierstrass tails
in $\Mg{2,1}(7/10-\epsilon)$.
\end{lemma}
\begin{proof}
The inequality follows directly from Lemma \ref{L:cuspidal-section-high} by taking $g=2$. 
Moreover, the proof of Lemma \ref{L:cuspidal-section-high} shows that equality holds if and only if the intersection pairing on
$\langle F, \tau, \omega_{\X/B}\rangle$ is degenerate.  
Assuming $\delta_{\red}=0$, there is a global hyperelliptic involution $h\co \X\ra \X$.
Hence $\omega_{\X/B}\equiv \tau+h(\tau)+xF$, for some $x\in \ZZ$. 
Observe that $\omega_{\X/B}\cdot \tau =\omega_{\X/B}\cdot h(\tau)$
and $F\cdot \tau=F\cdot h(\tau)$. Since no combination of $\omega$ and $F$ is in the kernel of the intersection pairing, 
we conclude that \[
\tau^2=\tau\cdot h(\tau).\]
However, the intersection number on the left is negative by Lemma \ref{L:psi-class} 
and the intersection number on the right is non-negative 
whenever $\tau\neq h(\tau)$. We conclude that equality holds if only if $h(\tau)=\tau$, that is 
$\tau$ is a Weierstrass section.
\end{proof}

We will need special variants of Lemmas \ref{L:cuspidal-section-high} and \ref{L:inner-sections-high} 
for the case of relative genus $1$ and $0$.

\begin{lemma}\label{L:inner-sections-1} 
Let $\X/B$ be a family of Gorenstein curves of arithmetic genus
$1$ with a pair of sections $\eta^{+},\eta^{-}$, and suppose that $\eta^{+}$ and $\eta^{-}$ 
are disjoint from $N$ smooth pairwise disjoint section of $\X/B$.
Then 
\[
(\eta^+ \cdot \eta^-)+\iota(\eta^+) \leq \frac{N+2}{2N}(\psi_{\eta^+}+\psi_{\eta^-})
+\frac{1}{2N^2}\delta_{\red}.
\]
\end{lemma}
\begin{proof}
Let 
$\Sigma$ be the sum of $N$ pairwise disjoint smooth sections of $\X/B$ disjoint from $\{\eta^{+}, \eta^{-}\}$.
Then $(\omega_{\X/B}+2\Sigma)^2=\omega_{\X/B}^2=\kappa$. 
Apply the Hodge Index Theorem  to 
$\langle F, \eta^{+}+\eta^{-}, \omega_{\X/B}+2\Sigma\rangle$, where $F$ is the fiber class. 
The determinant of the matrix
\[
\left(\begin{array}{ccc}
0 & 2 & 2N \\ 
-\psi_{\eta^+}-\psi_{\eta^-} +2(\eta^+ \cdot \eta^-)+\iota(\eta^+)+\iota(\eta^-)  & \psi_{\eta^+}+\psi_{\eta^-} \\
2N  & \psi_{\eta^+}+\psi_{\eta^-} & \kappa \end{array}\right)
\]
is non-negative. Therefore
\begin{align*}
-4\kappa+
8N(\psi_{\eta^+}+\psi_{\eta^-})+4N^2(\psi_{\eta^+}+\psi_{\eta^-})\geq 8N^2\bigl((\eta^+ \cdot \eta^-)+\iota(\eta^+)\bigr),
\end{align*}
which gives the desired inequality using $\kappa=-\delta_{\red}$.
\end{proof}

\begin{lemma}\label{L:cuspidal-section-1}
Let $\X/B$ be a family of Gorenstein curves of arithmetic genus
$1$ with a section $\xi$, and suppose that $\xi$ is disjoint from $N$ smooth pairwise disjoint
sections of $\X$.  
Then 
\begin{align*}
\iota(\xi)\leq \frac{N+1}{N}\psi_{\xi}+\frac{1}{4N^2}\delta_{\red}.
\end{align*}
Furthermore, suppose $N=1$, with $\tau$ being a smooth section disjoint from $\xi$, 
and $\delta_{\red}=0$. Then equality holds if and only if $2\xi \sim 2\tau$.
\end{lemma}
\begin{proof}
Let $\Sigma$ be the collection of smooth sections of $\X/B$ disjoint from $\xi$.
By the Hodge Index Theorem applied to 
$\langle F, \xi, \omega_{\X/B}+2\Sigma \rangle$,
the determinant of the matrix
\[
\left(\begin{array}{ccc}
0 & 1 & 2N \\ 
1 & -\psi_{\xi}+\iota(\xi) & \psi_{\xi} \\
2N  & \psi_{\xi} & \kappa \end{array}\right)
\]
is non-negative. Therefore
\begin{align*}
\iota(\xi)\leq \psi_{\xi}+\frac{1}{N}\psi_{\xi}-\frac{1}{4N^2}\kappa.
\end{align*}
This gives the desired inequality using $\kappa=-\delta_{\red}$.

To prove the last assertion observe that because
$\delta_{\red}=0$ all fibers of $\X/B$ are irreducible curves of genus $1$. In particular, $\omega_{\X/B}=\lambda F$ 
and it follows from the existence
of the group law on the set of sections of $\X/B$ that there exists a section $\tau'$ such 
that $2\xi-\tau=\tau'$. Since $\tau \cap \xi=\emptyset$, we have $\tau'\cap \xi=\emptyset$.
If equality holds, then the intersection pairing matrix on the classes $F, \xi, \tau$ is
degenerate. Hence some linear combination $(x\xi+y\tau+zF)$
intersects $F, \xi, \tau$ trivially. Clearly, $y\neq 0$. Intersecting with $\tau$, we obtain $y(\tau\cdot \tau)+z=0$; and intersecting with $\tau'$, 
we obtain $y(\tau\cdot \tau')+z=0$. Hence $\tau^2=\tau\cdot\tau'$. This leads to a contradiction if $\tau\neq \tau'$.
\end{proof}

\subsubsection{An inequality between divisor classes on $\Mg{0,N}$.}

The proof of Theorem \ref{T:main-positivity} will require the following ad hoc effectivity result on $\Mg{0,N}$.

\begin{lemma}\label{L:inner-sections} Suppose $\{\eta_i^{+}, \eta_i^{-}\}_{i=1}^{a}$ 
are sections of a family of $N$-pointed Deligne-Mumford stable rational curves. 
Let $\psi_{inner}:=\sum_{i=1}^{a} \left(\psi_{\eta^+_i}+\psi_{\eta^-_i}\right)$ and 
$\delta_{inner}:=\sum_{i=1}^{a}\delta_{\{\eta_i^+, \eta_i^-\}}$.
If $a \geq 2$,
then for any generically smooth one-parameter family in $\Mg{0,N}$, we have
\[
\psi_{inner} \geq 4 \delta_{inner} + 4\sum_{i=1}^{a} \sum_{\beta\notin \{\eta_i^{+}, \eta_i^{-}\}_{i=1}^{a}} 
\delta_{\{\eta_i^{+}, \eta_i^{-},\beta\}}
+ 2\frac{a-2}{a-1} \sum_{i\neq j} \delta_{\{\eta^{\pm}_i, \, \eta^{\pm}_j\}}
+\frac{5a-9}{a-1}\sum_{i=1}^{a} \sum_{j\neq i}  \delta_{\{\eta_i^{+}, \eta_i^{-}, \eta_j^{\pm}\}}.
\]
\end{lemma}
\begin{proof}
For any two distinct $\psi$-classes on $\Mg{0,N}$,
we have the following standard relation:
\begin{equation}\label{keel-relation}
\psi_\sigma+\psi_\tau =\sum_{S: \ \sigma\in S, \ \tau\notin S} \delta_{S}.
\end{equation}
We apply \eqref{keel-relation} to the right-hand side of
\[
(a-1)\psi_{inner}= \sum_{1\leq i <  j \leq a} (\psi_{\eta^{\pm}_{i}}+\psi_{\eta^{\pm}_{j}})-(a-1)\sum_{i=1}^{a} (\psi_{\eta^{+}_i}+\psi_{\eta^{-}_i}).
\]
This gives us a formula of the following form:
\[
(a-1)\psi_{inner}=  \sum c_{S} \delta_{S}.
\]
We now estimate the coefficients of the boundary divisors appearing on the right-hand side. Suppose
there are $x$ pairs $\{\eta_{i}^{+}, \eta_{i}^{-}\}$ such that $\eta_{i}^{+}\in S$ and $\eta_{i}^{-} \notin S$, or vice versa,
and that $S$ contains $y$ pairs $\{\eta_{i}^{+}, \eta_{i}^{-}\}$. Set $z=a-x-y$. 
Then 
\[
c_{S}=((x+2y)(x+2z)-x)-(a-1)x=x(y+z)+4yz.
\]
We have that
\begin{enumerate}
\item $c_{S} \geq 0$ for every $S$.
\item If $S=\{\eta_i^+, \eta_i^-\}$ or $S=\{\eta_i^+, \eta_i^-,\beta\}$, where $\beta\notin \{\eta_i^{+}, \eta_i^{-}\}_{i=1}^{a}$,
then $x=0$ and $y=1$, and so $c_{S}=4(a-1)$.
\item If $S=\{\eta^{\pm}_i, \, \eta^{\pm}_j\}$ for $i\neq j$, then $x=2$ and $y=0$, and so $c_S=2(a-2)$.
\item If $S=\{\eta_i^{+}, \eta_i^{-}, \eta_j^{\pm}\}$ for $j\neq i$, then $x=1$ and $y=1$,
and so $c_{S}=5a-9$.
\end{enumerate}
The claim follows.
\end{proof}

\subsection{Proof of Theorem {\ref{T:main-positivity}(a)}}
\label{S:positivity-a} Notice that $10\lambda-\delta+\psi=0$ on $\Mg{2,0}(\first-\epsilon)$
by the standard relation $10\lambda=\delta_{\irr}+2\delta_{\red}$ that holds for all families in $\cU_2$. 

We now prove that $10\lambda-\delta+\psi$ is nef on $\Mg{g,n}(\first-\epsilon)$ and has degree $0$ precisely 
on families whose only non-isotrivial components are $A_1/A_1$-attached elliptic bridges, for all $(g,n)\neq (2,0)$. 
Let $(\C/B, \{\sigma_i\}_{i=1}^{n})$ be a $(\first-\epsilon)$-stable family. 
The proof proceeds by normalizing 
$\C$ along generic singularities to arrive at a family of generically smooth curves,
where the Cornalba-Harris inequality holds, or at a family of low genus curves, where the requisite 
inequality is established by ad-hoc methods. 
Keeping in mind that generic outer nodes and generic cusps of $\C/B$ do not degenerate, but 
generic inner nodes of $\C/B$ can degenerate to cusps, we begin by normalizing generic 
outer nodes, then normalize generic cusps, and finally normalize generic inner nodes.

\subsubsection{Reduction 1: Normalization along generic outer nodes}
\label{first-reduction-1}
Let $\X$ be the normalization of $\C$ along generic outer nodes, marked by nodal transforms.  
By Lemma \ref{L:Norm} every connected component of $\X/B$ is a family of generically irreducible 
$(9/11-\epsilon)$-stable curves. By Proposition \ref{P:generic-outer-node}, we have
\[
(10\lambda-\delta+\psi)_{\C/B}=(10\lambda-\delta+\psi)_{\X/B}.
\]
We have reduced to proving $10\lambda-\delta+\psi\geq 0$ for a family with generically irreducible fibers.

\subsubsection{Reduction 2: Normalization along generic cusps}
\label{first-reduction-2}
Suppose $(\X/B, \sigman)$ is a family of $(9/11-\epsilon)$-stable curves with generically irreducible fibers.
Let $\Y$ be the normalization of $\X$ along generic cusps. Denote
by $\{\xi_i\}_{i=1}^{c}$ the cuspidal transforms on $\Y$.
Set $\psi_{cusp}:=\sum_{i=1}^{c} \psi_{\xi_i}$ and 
$\psi_{\Y/B}:=\psi_{\X/B}+ \psi_{cusp}$.
Then by Proposition \ref{P:generic-cusp}, we have 
\[
(10\lambda-\delta+\psi)_{\X/B}=(10\lambda-\delta+\psi)_{\Y/B}+\psi_{cusp}.
\]
We have reduced to proving $10\lambda-\delta+\psi+\psi_{cusp}\geq 0$ for a family 
$(\Y/B, \sigman, \{\xi_i\}_{i=1}^{c})$, where
\begin{enumerate}
\item The fibers are at-worst cuspidal and the generic fiber is irreducible and at-worst nodal.
\item $\sigman, \{\xi_i\}_{i=1}^{c}$ are smooth sections and 
$\omega_{\Y/B}(\sum_{i=1}^n \sigma_i +\sum_{i=1}^c \xi_i)$ is relatively ample.\footnote{
A priori, only $\omega_{\Y/B}(\sum_{i=1}^n \sigma_i +2\sum_{i=1}^c \xi_i)$ is relatively ample. 
However, a rational tail cannot meet just a single cuspidal transform because the original family $\X/B$
cannot have cuspidal elliptic tails.}
\end{enumerate}

\subsubsection{Reduction 3: Normalization along generic inner nodes}
\label{first-reduction-3}

Consider $(\Y/B, \sigman, \{\xi_i\}_{i=1}^{c})$ as in \ref{first-reduction-2}.
Let $a$ be the number of generic inner nodes of $\Y/B$.
We let $\Z\ra \Y$ be the normalization and denote
by $\eta^{+}_i$ and $\eta^{-}_i$ the inner nodal transforms of the $i^{th}$ generic node.
We obtain a family
$\bigl(\Z/B, \{\sigma_i\}_{i=1}^{n}, \{\eta_i^{\pm}\}_{i=1}^{a}, \{\xi_i\}_{i=1}^{c}\bigr)$, where
\begin{enumerate}
\item The fibers are at-worst cuspidal curves and the generic fiber is smooth.
\item The sections $\{\sigma_i\}_{i=1}^{n}, \{\eta_i^{\pm}\}_{i=1}^{a}, \{\xi_i\}_{i=1}^{c}$
are all smooth and pairwise disjoint, except that $\eta_i^{+}$ can intersect $\eta_i^{-}$ for each $i$.
\item $\omega_{\Z/B}\left(\sum_{i=1}^n \sigma_i+\sum_{i=1}^{a} (\eta_i^{+}+\eta_i^{-})+
\sum_{i=1}^c \xi_i\right)$ is relatively ample.  
\end{enumerate}

By Proposition \ref{P:generic-inner-node}, we have that 
\[
(10\lambda-\delta+\psi+\psi_{cusp})_{\Y/B}=(10\lambda-\delta+\psi+\psi_{cusp})_{\Z/B},
\]
where $\psi_{cusp}=\sum_{i=1}^{c} \psi_{\xi_i}$ and 
$\psi_{\Z/B}=\psi_{\Y/B}+\sum_{i=1}^{a} (\psi_{\eta_i^+}+\psi_{\eta_i^-})$.

We let $N=n+2a+c$ be the total number of sections of $\Z/B$, including cuspidal and inner nodal transforms.
Our proof that $(10\lambda-\delta+\psi+\psi_{cusp})_{\Z/B}\geq 0$ will depend on the relative genus $h$ of $\Z/B$.

\paragraph{Suppose $h\geq 2$.} 
Passing to the relative minimal model of $\Z/B$ 
only decreases the degree of $(10\lambda-\delta+\psi+\psi_{cusp})$.
Hence we will assume that $\omega_{\Z/B}$ is relatively nef. We still have $N$ smooth and distinct sections
(which can now intersect pairwise). 
With $\omega_{\Z/B}$ relatively nef, we can apply the Cornalba-Harris inequality. 
If $h\geq 3$, then $10>8+4/h$ and so $10\lambda-\delta>0$ by Proposition \ref{P:CH-inequality}. 
If $h=2$, then Proposition \ref{P:CH-inequality} gives $10\lambda-\delta \geq 0$. 
Lemma \ref{L:psi-class} gives $\psi+\psi_{cusp}> 0$ since we must have $N\geq 1$
(if $N=0$, then $\C/B$ was a family in $\Mg{2,0}(\first-\epsilon)$).

\paragraph{Suppose $h=1$.} Using relations on the stack on $N$-pointed Gorenstein genus $1$ curves 
inherited from standard relations in $\Pic(\Mg{1,N})$ given by \cite[Theorem 2.2]{AC}, we have $\lambda=\delta_{\irr}/12$,
and $\psi=N\delta_{irr}/12+\sum_{S}|S| \delta_{0,S}\geq N\delta_{\irr}/12+2\delta_{\red}$.
If $N\geq 3$, we obtain 
\begin{align*}
10\lambda+\psi -\delta\geq 10\delta_{\irr}/12+N\delta_{\irr}/12+2\delta_{\red}-(\delta_{\irr}+\delta_{\red})>0.
\end{align*}
If $N=2$, we obtain $10\lambda-\delta+\psi\geq \delta_{\red} \geq 0$ and $\psi_{cusp}\geq 0$. We conclude that 
$10\lambda-\delta+\psi+\psi_{cusp}\geq 0$ with the equality holding if only if $\psi_{cusp}=\delta_{\red}=0$.
This is possible if and only if all fibers are irreducible and there are no cuspidal transforms (by Lemma \ref{L:psi-class}),
which implies that $\X/B=\Y/B$ is a family of $A_1/A_1$-attached elliptic bridges.

\paragraph{Suppose $h=0$.} Then all fibers of $\Z/B$ are in fact at-worst nodal. Because $\lambda=0$, we can write
$(10\lambda-\delta+\psi+\psi_{cusp})_{\Z/B}=\psi-\delta+\psi_{cusp}$. 
Blow-up the points of intersection of $\eta^{+}_i$ and $\eta^{-}_i$ for each $i$. 
We obtain a family $\bigl(\W/B, \{\sigma_i\}_{i=1}^{n}, \{\eta_i^{\pm}\}_{i=1}^{a}, \{\xi_i\}_{i=1}^{c}\bigr)$
in $\Mg{0,N}$. Setting $\delta_{inner}:=\sum_{i=1}^{a}\delta_{\{\eta_i^+, \eta_i^-\}}$, 
we have
\[
\left(\psi-\delta+\psi_{cusp}\right)_{\Z/B}=\left(\psi-\delta-\delta_{inner}+\psi_{cusp}\right)_{\W/B}.
\] 
If $a=0$, then $\delta_{inner}=0$ and we are done because $\psi-\delta>0$ for any family of Deligne-Mumford stable rational curves,
for example by \cite[Lemma 3.6]{KMc}. 
If $a\geq 2$, then by Lemma \ref{L:inner-sections}, 
$\sum_{i=1}^a (\psi_{\eta_i^+}+\psi_{\eta_i^-})\geq 4\delta_{inner}$. 
In addition, $3\psi \geq 4\delta$ by a similar argument. 
It follows that $\psi > \delta+\delta_{inner}$ and so we are done.

Finally, if $a=1$, then 
$\bigl(\Y/B, \{\sigma_i\}_{i=1}^{n}, \{\xi_i\}_{i=1}^{b}\bigr)$ obtained in \ref{first-reduction-2} is a  
family of arithmetic genus $1$ (generically nodal) curves and
the proof in the case of $h=1$ above goes through without any modifications to show that
$(10\lambda-\delta+\psi+\psi_{cusp})_{\Y/B} \geq 0$ with the equality if and only if $\X/B=\Y/B$
is a (generically nodal) elliptic bridge.

\subsection{Proof of Theorem \ref{T:main-positivity}(b)}
\label{S:positivity-b}
In the remaining part of the paper, we prove Theorem \ref{T:main-positivity}(b).  
Let $(\C/B, \sigman)$ be a $(\secondme)$-stable generically non-isotrivial family of curves.
We begin by dealing with the case when $\C/B$ has a generic rosary, or a generic $A_1/A_3$ or $A_3/A_3$-attached
elliptic bridge. In both cases, generic tacnodes come into play and we will repeatedly use the following
result that explains what happens under normalization of a generic tacnode: 
\begin{prop}\label{P:generic-tacnode}
Suppose $\X/B$ is a family in $\tilde\cU_g$ with a section $\tau$ such that $\tau(b)$ 
is a tacnode of $\X_b$ for all $b\in B$. Denote by $\Y$ the normalization of $\X$ along $\tau$ 
and by $\tau^{+}$ and $\tau^{-}$ the preimages of $\tau$. Then $\tau^{\pm}$ are smooth sections 
satisfying $\psi_{\tau^{+}}=\psi_{\tau^{-}}$
and we have the following formulae:
\begin{align*}
\lambda_{\X/B}&=\lambda_{\Y/B}-\frac{1}{2}(\psi_{\tau^+}+\psi_{\tau^-}), \\
\delta_{\X/B}&=\delta_{\Y/B}-6(\psi_{\tau^+}+\psi_{\tau^-}).
\end{align*}
\end{prop}
\begin{proof} This is \cite[Proposition 3.4]{smyth_elliptic2} (although it is stated there only 
in the case of $g=1$).
\end{proof}

\subsubsection{Reduction 1: The case of generic rosaries}
\label{S:generic-rosary} Let $C$ be the geometric generic fiber of $\C/B$ and 
consider a maximal length rosary $R=R_1\cup \cdots \cup R_\ell$
of $C$ (see Definition \ref{defn-rosary}). Since $\C/B$ is 
non-isotrivial, the rosary cannot be closed. 
Let $T:=\overline{C \setminus R}$.
The point $T\cap R_1$ (resp., $T\cap R_\ell$) is either an outer node or an outer tacnode,
so its limit in every fiber is the same singularity by Proposition \ref{P:limits-outer}. 
Similarly, the limits of the tacnodes $R_{i}\cap R_{i+1}$, for $i=1,\dots, \ell-1$, remain tacnodes
in every fiber. We then have that $\C=\mathcal{T}\cup \mathcal{R}_1 \cup \cdots  \cup\mathcal{R}_\ell$, 
where the geometric generic fiber of $\mathcal{R}_i$ and $\mathcal{T}$ is $R_i$ and $T$ respectively. 
Let $\chi_{1}$ (resp., $\chi_{2}$) be the nodal or tacnodal section along which $\mathcal{T}$ and $\mathcal{R}_1$
(resp., $\mathcal{R}_\ell$) meet.
Let $\tau_i$, for $i=1,\dots, \ell-1$, be the tacnodal section along which $\mathcal{R}_i$ and $\mathcal{R}_{i+1}$
meet. In the rest of the proof we use the fact that self-intersections of $2$ disjoint smooth sections on 
a $\PP^1$-bundle over $B$ are equal of opposite signs. Together with Proposition \ref{P:generic-tacnode}, this gives
\begin{equation*}
(\psi_{\chi_1})_{\mathcal{R}_1/B} = -(\psi_{\tau_1})_{\mathcal{R}_1/B} = -(\psi_{\tau_1})_{\mathcal{R}_2/B}
=(\psi_{\tau_2})_{\mathcal{R}_2/B}=\cdots =(-1)^{\ell-1}(\psi_{\tau_{\ell-1}})_{\mathcal{R}_\ell/B}=(-1)^{\ell} (\psi_{\chi_2})_{\mathcal{R}_\ell/B}.
\end{equation*}
In what follows, we set $\psi_{\T/B}=\sum_{i=1}^n\psi_{\sigma_i}+\psi_{\chi_1}+\psi_{\chi_2}=\psi_{\C/B}+\psi_{\chi_1}+\psi_{\chi_2}$.

\emph{Case I: $R$ is $A_1/A_1$-attached rosary.} By Remark \ref{R:even-rosary}, $\ell$ must be odd.
By Proposition \ref{P:generic-tacnode}, we obtain
\[
\left(\frac{39}{4}\lambda-\delta+\psi\right)_{\C/B}
=\left(\frac{39}{4}\lambda-\delta+\psi\right)_{\mathcal{T}/B}.
\]
Since $(\mathcal{T}, \sigman, \chi_1,\chi_2)$ is $(\secondme)$-stable and 
$\mathcal{R}/B$ is isotrivial, we reduce
to proving Theorem  \ref{T:main-positivity}(b) for $(\mathcal{T}, \sigman, \chi_1,\chi_2)$, which has one less generic rosary than
$(\C/B, \sigman)$.

\emph{Case 2: $R$ is $A_1/A_3$-attached rosary.} Suppose $\chi_1$ is a nodal section and $\chi_2$ is a tacnodal section.
By the maximality assumption on $R$, the irreducible component of $T$ meeting $R_\ell$ is not a $2$-pointed
smooth rational curve. It follows
by Lemma \ref{L:psi-class} that $(\psi_{\chi_2})_{\mathcal{T}}\geq 0$. 
By Proposition \ref{P:generic-tacnode}, we have
\[
\left(\frac{39}{4}\lambda-\delta+\psi\right)_{\C/B}=
\left(\frac{39}{4}\lambda-\delta+\psi\right)_{\mathcal{T}/B}+
(\psi_{\chi_1})_{\mathcal{R}_1/B}+\frac{5}{4}(\psi_{\chi_2})_{\mathcal{R}_\ell/B}
+\frac{9}{4}\sum_{i=1}^{\ell-1}(\psi_{\tau_i})_{\mathcal{R}_i}.
\]
If $\ell$ is odd, then $\sum_{i=1}^{\ell-1}\left(\psi_{\tau_i}\right)_{\mathcal{R}_i}=0$ and $\psi_{\chi_1}=-\psi_{\chi_2}$. 
We thus obtain:
\[
\left(\frac{39}{4}\lambda-\delta+\psi\right)_{\C/B}=\left(\frac{39}{4}\lambda-\delta+\psi\right)_{\mathcal{T}/B}
+\frac{1}{4}(\psi_{\chi_2})_{\mathcal{T}/B} 
\geq \left(\frac{39}{4}\lambda-\delta+\psi\right)_{\mathcal{T}/B}.
\]
Noting that $\psi_{\chi_2}=0$ only if $\mathcal{R}/B$ is isotrivial, we reduce to proving Theorem \ref{T:main-positivity}(b)
for $(\mathcal{T}, \sigman, \chi_1,\chi_2)$.

If $\ell$ is even, then $\psi_{\chi_1}=\psi_{\chi_2}$ and $\sum_{i=1}^{\ell-1}(\psi_{\tau_i})_{\mathcal{R}_i}+\psi_{\chi_2}=0$, so that
\[
\left(\frac{39}{4}\lambda-\delta+\psi\right)_{\C/B}
=\left(\frac{39}{4}\lambda-\delta+\psi\right)_{\mathcal{T}/B}.
\]
Furthermore, we observe that $\mathcal{R}/B$ is isotrivial 
and we reduce to proving Theorem  \ref{T:main-positivity}(b) for $(\mathcal{T}, \sigman, \chi_1,\chi_2)$.

\emph{Case 3: $R$ is $A_3/A_3$-attached rosary.}
By the maximality assumption on $R$, neither $T\cap R_1$ nor $T\cap R_2$ lies on a $2$-pointed rational component of $T$. 
It follows
by Lemma \ref{L:psi-class} that $(\psi_{\chi_1})_{\mathcal{T}}, (\psi_{\chi_2})_{\mathcal{T}}\geq 0$. However,
$\psi_{\chi_1}=(-1)^\ell\psi_{\chi_2}$.
Therefore, either $\psi_{\chi_1}=\psi_{\chi_2}=0$, in which case $\mathcal{R}/B$ is an isotrivial family, or 
$\ell$ is even and $\psi_{\chi_1}=\psi_{\chi_2}>0$. In either case, Proposition \ref{P:generic-tacnode} gives
\[
\left(\frac{39}{4}\lambda-\delta+\psi\right)_{\C/B}
=\left(\frac{39}{4}\lambda-\delta+\psi\right)_{\mathcal{T}/B}+\frac{1}{4}(\psi_{\chi_2})_{\mathcal{R}/B} 
\geq \left(\frac{39}{4}\lambda-\delta+\psi\right)_{\mathcal{T}/B},
\]
and the inequality is strict if $\mathcal{R}$ is not isotrivial. 
Thus we reduce to proving Theorem  \ref{T:main-positivity}(b) for $(\mathcal{T}, \sigman, \chi_1,\chi_2)$.

\subsubsection{Reduction 2: The case of generic $A_1/A_3$ or $A_3/A_3$-attached elliptic bridges}
\label{reduction-2}
Suppose the geometric generic fiber of $\C/B$ can be written as 
$C=T_1 \cup E \cup T_2$, where $E$ is an $A_1/A_3$-attached elliptic bridge.
Let  $q_1=T_1\cap E$ be a node and $q_2=T_2\cap E$ be a tacnode. 
By Proposition \ref{P:limits-outer}, the limit of $q_1$ (resp., $q_2$) remains a node (resp., a tacnode)
in every fiber. Thus we can write
 $\C=(\cT_1, \tau_0) \cup (\E, \tau_1, \tau_2)  \cup (\cT_2, \tau_3)$,
where $\tau_0 \sim \tau_1$ are glued nodally and $\tau_2 \sim \tau_3$ are glued tacnodally. 
Since $A_1/A_1$-attached elliptic bridges are disallowed, fibers of $\E$ have no separating nodes and so 
$(\E, \tau_1, \tau_2)$ is a family of elliptic bridges.
By Lemma \ref{L:Norm}, $(\cT_1, \tau_0)$ is $(\second-\epsilon)$-stable. Also, $(\cT_2,\tau_3)$ is $(\second-\epsilon)$-stable 
because $\tau_3$ cannot lie on an $A_1$-attached elliptic tail in $\cT_2$.

Set $\C'=(\cT_1, \tau_0) \cup (\cT_2, \tau_3)$, where we glue by $\tau_0 \sim \tau_3$ nodally. 
Then $(\C'/B, \sigman)$ is a $(\second-\epsilon)$-stable family by Lemma \ref{L:Glue}.
By Proposition \ref{P:generic-tacnode}, we have
\begin{multline*}
\left(\frac{39}{4}\lambda-\delta+\psi\right)_{\C/B}=
\left(\frac{39}{4}\lambda-\delta+\psi\right)_{\cT_1/B}+
\left(\frac{39}{4}\lambda-\delta+\psi_{\tau_1}+\frac{5}{4}\psi_{\tau_2}\right)_{\E/B}
+\left(\frac{39}{4}\lambda-\delta+\psi\right)_{\cT_2/B} \\ = \left(\frac{39}{4}\lambda-\delta+\psi\right)_{\cT_1/B}
+\left(\frac{39}{4}\lambda-\delta+\psi\right)_{\cT_2/B}=\left(\frac{39}{4}\lambda-\delta+\psi\right)_{\C'/B},
\end{multline*}
where we have used relations
$(\psi_{\tau_1})_{\E/B}=(\psi_{\tau_2})_{\E/B}=\lambda_{\E/B}$
and $\delta_{\E/B}=12\lambda_{\E/B}$, both of which hold because $(\delta_{\red})_{\E/B}=0$.

Note that $(\E/B, \tau_1, \tau_2)$ is trivial if and only if $\psi_{\tau_2}=\psi_{\tau_3}=0$.
Thus we have reduced to proving the requisite inequalities for the family $\C'/B$ with one less generic $A_1/A_3$-attached
elliptic bridge. Moreover, the equality for $\C'/B$ holds if and only if the equality for $\C/B$ holds and $\C'/B$ 
is obtained by replacing a generic node of $\C'$ by a family of elliptic bridges $A_1/A_3$-attached along
the nodal transforms. 

Similarly, if the generic fiber of $\C/B$ has an $A_3/A_3$-attached elliptic bridge,
then we can remove the bridge and recrimp the two remaining components of $\C$ 
along a generic tacnode. The calculation similar to the above shows that the degree of $\left(\frac{39}{4}\lambda-\delta+\psi\right)$ does not 
change under this operation. 

Replacing an attaching node of a Weierstrass chain of length $\ell$ by 
an $A_1/A_3$-attached elliptic bridge in a way that preserves $(\second-\epsilon)$-stability gives 
a Weierstrass chain of length $\ell+1$. Similarly, 
replacing a tacnode in a Weierstrass chain of length $\ell$ by 
an $A_3/A_3$-attached elliptic bridge gives a Weierstrass chain of length $\ell+1$.
In what follows, we will prove that for a non-isotrivial $(\secondme)$-stable family $(\C/B, \sigman)$ 
\emph{with no generic $A_1/A_3$ or $A_3/A_3$-attached elliptic bridges}, 
we have $\left(\frac{39}{4}\lambda-\delta+\psi\right)_{\C/B}\geq 0$
and equality holds if and only if $\C/B$ is a family of \emph{Weierstrass tails}. This implies that for every non-isotrivial $(\secondme)$-stable family $(\C/B, \sigman)$, 
we have $\left(\frac{39}{4}\lambda-\delta+\psi\right)_{\C/B}\geq 0$
and equality holds if and only if $\C/B$ is a family of \emph{Weierstrass chains}.

\subsubsection{Reduction 3: Normalization along generic tacnodes}
\label{reduction-3}

Suppose now $(\C/B, \sigman)$ is a family of $(\second-\epsilon)$-stable curves with no generic rosaries
and no generic $A_1/A_3$ or $A_3/A_3$-attached elliptic bridges.
Let $\X$ be the normalization of $\C$ along generic tacnodes. Denote
by $\{\tau^{\pm}_i\}_{i=1}^{d}$ the preimages of the generic tacnodes, and call them 
tacnodal transforms. Set $\psi_{tacn}:=\sum_{i=1}^{d} (\psi_{\tau^+_i}+\psi_{\tau^-_i})$ and 
$\psi_{\X/B}:=\psi_{\C/B}+ \psi_{tacn}$. 
Applying Proposition \ref{P:generic-tacnode} we have 
\[
\left(\frac{39}{4}\lambda-\delta+\psi\right)_{\C/B}=\left(\frac{39}{4}\lambda-\delta+\psi+\frac{1}{8}\psi_{tacn}\right)_{\X/B}.
\]
If we now treat each tacnodal transform $\tau_{i}^{\pm}$ as a marked section, then every connected component 
of $\X$ is a generically $(\second-\epsilon)$-stable family (there are no generic $A_1/A_3$ or $A_3/A_3$-attached elliptic bridges).
Blowing-down all rational tails meeting a single tacnodal transform and no other marked sections does not 
change $\left(\frac{39}{4}\lambda-\delta+\psi\right)_{\X/B}$ but makes 
$(\X/B, \sigman, \{\tau^{\pm}_i\}_{i=1}^{d})$ into a $(\second-\epsilon)$-stable family.
We still have $\psi_{tacn}\geq 0$ by Lemma \ref{L:psi-class}, with strict inequality if $d\geq 1$.
Thus, we have reduced to proving Theorem \ref{T:main-positivity}(b) for a $(\second-\epsilon)$-stable family 
with no generic tacnodes.

\subsubsection{Reduction 4: Normalization along generic outer nodes} 
\label{reduction-4}
Suppose $(\X/B, \sigman)$ is a $(\secondme)$-stable family with no generic tacnodes. 
Let $\Y$ be the normalization of $\X$ along the generic outer nodes and 
let $\{\zeta^+_i, \zeta_i^-\}_{i=1}^{b}$ be the transforms of the generic outer nodes. 
Set 
$\delta_{tacn}:=\sum_{i\neq j} (\zeta^\pm_i \cdot \zeta_j^\pm)$ and
$\psi_{\Y/B}:=\psi_{\X/B}+\sum_{i=1}^{b} (\psi_{\zeta^+_i}+\psi_{\zeta^-_i})$.
Then by Proposition \ref{P:generic-outer-node}, we have 
\begin{align*}
\left(\frac{39}{4}\lambda-\delta+\psi\right)_{\X/B}=\left(\frac{39}{4}\lambda-\delta+\psi +\frac{7}{8}\delta_{tacn}\right)_{\Y/B}.
\end{align*}

\subsubsection{Reduction 5: Normalization along generic cusps}
\label{cuspidal-reduction}


Let $\Y$ be as in \ref{reduction-4} and 
let $\Z$ be the normalization of (a connected component of) $\Y$ along generic cusps and let $\{\xi_i\}_{i=1}^{c}$ be the cuspidal 
transforms on $\Z$. 
Then the family $(\Z/B, \sigman, \{\zeta_i\}_{i=1}^{b}, \{\xi_i\}_{i=1}^{c})$ satisfies the following properties:
\begin{enumerate}
\item The generic fiber is irreducible and at-worst nodal.
\item The sections $\sigman$ are smooth, pairwise non-intersecting and disjoint from $\{\zeta_i\}_{i=1}^{b}$.
\item The sections $\{\zeta_i\}_{i=1}^{b}$ are smooth and at most two of them can meet at any given point of $\Z$.
\item The sections $\{\xi_i\}_{i=1}^{c}$ are pairwise non-intersecting and disjoint from $\{\zeta_i\}_{i=1}^{b}$ and $\sigman$.
\end{enumerate}
Set $c(B):=2\sum_{i=1}^{c} \iota(\xi_i)$, where $\iota(\xi_i)$ is the index of the cuspidal transform $\xi_i$,
and $\psi_{cusp}:=\sum_{i=1}^c \psi_{\xi_i}$. 
Then we have by Proposition \ref{P:generic-cusp}
\begin{align}\label{E:goal}
\left(\frac{39}{4}\lambda-\delta+\psi+\frac{7}{8}\delta_{tacn}\right)_{\Y/B}
=\left(\frac{39}{4}\lambda-\delta+\psi+\frac{5}{4}\psi_{cusp}-\frac{1}{4}c(B)+\frac{7}{8}\delta_{tacn}\right)_{\Z/B}.
\end{align}

Our goal for the rest of the section is to prove that the expression on the right-hand side of \eqref{E:goal} is non-negative
and equals $0$ if and only if the only non-isotrivial components of the family $\X/B$ from \ref{reduction-4} are 
$A_1$-attached Weierstrass tails. 

Let $h$ be the geometric genus of the generic fiber of $\Z$ and let $a$ be the number of generic inner nodes of $\Z$.
Our further analysis breaks down according to the following possibilities:
\begin{enumerate}
\item[(A)] $h \geq 3$; see \S\ref{genus-3}.
\item[(B)] $h=2$, or $(h,a)=(1,1)$, or $(h,a)=(0,2)$; see \S\ref{geometric-genus-2}.
\item[(C)] $h=1$ and $a\neq 1$, or $(h,a)=(0,1)$; see \S\ref{geometric-genus-1}.
\item[(D)] $h=0$ and $a\geq 3$, or $(h,a)=(0,0)$; see \S\ref{geometric-genus-0}. 
\end{enumerate}

\subsubsection{Case A: Relative geometric genus $h\geq 3$}
\label{genus-3}
Suppose $\Z/B$ is a family as in \ref{cuspidal-reduction}. 
Let $\W$ be the normalization of $\Z$ along the generic inner nodes.
Let $\{\eta_i^+, \eta_i^{-}\}_{i=1}^{a}$ be the inner nodal 
transforms on $\W$. 
Then $(\W/B, \sigman, \{\eta_i^{\pm}\}_{i=1}^{a}, \{\zeta_i\}_{i=1}^{b}, \{\xi_i\}_{i=1}^{c})$ satisfies the following properties:
\begin{enumerate}
\item The generic fiber is a smooth curve of genus $h\geq 3$.
\item Sections $\sigman$ are smooth, non-intersecting, and disjoint from $\{\eta_i^{\pm}\}_{i=1}^{a}$, $\{\zeta_i\}_{i=1}^{b}$, and 
$\{\xi_i\}_{i=1}^{c}$.
\item Inner nodal transforms $\{\eta_i^{\pm}\}_{i=1}^{a}$ are disjoint from $\{\zeta_i\}_{i=1}^{b}$ and 
$\{\xi_i\}_{i=1}^{c}$. Their properties are described by Proposition \ref{P:generic-inner-node}.
\item Outer nodal transforms $\{\zeta_i\}_{i=1}^{b}$ are disjoint from $\{\xi_i\}_{i=1}^{c}$. Their properties are described by Proposition \ref{P:generic-outer-node}.
\item Cuspidal transforms $\{\xi_i\}_{i=1}^{c}$ have properties described by  Proposition \ref{P:generic-cusp}.
\end{enumerate}
We let $\psi_{\W/B}:=\psi_{\Z/B}+\sum_{i=1}^a (\psi_{\eta_i^+}+\psi_{\eta_i^-})$ and 
$(\delta_{tacn})_{\W/B}:=(\delta_{tacn})_{\Z/B}+\sum_{i\neq j} (\eta_i^{\pm} \cdot \eta_j^{\pm})$.
We set $\delta_{inner}:=\sum_{i=1}^{a} (\eta_i^{+}\cdot \eta_i^{-})$ and $n(B):=\sum_{i=1}^{a} \iota(\eta^{+}_i)$, where $\iota(\eta^{+}_i)$ is the index of the inner nodal transform $\eta^{+}_i$.
Then by Proposition \ref{P:generic-inner-node}:
\begin{multline}\label{E:D-g-3}
\left(\frac{39}{4}\lambda-\delta+\psi+\frac{5}{4}\psi_{cusp}-\frac{1}{4}c(B)+\frac{7}{8}\delta_{tacn}\right)_{\Z/B}
\\=\left(\frac{39}{4}\lambda-\delta+\psi-\frac{1}{4}\delta_{inner}-\frac{1}{4}n(B)-\frac{1}{4}c(B)+\frac{5}{4}\psi_{cusp}+\frac{7}{8}\delta_{tacn}\right)_{\W/B}.
\end{multline}

Passing to the relative minimal model of $\W/B$ does not increase the degree of 
the divisor on the right-hand side of \eqref{E:D-g-3}.
Hence we will assume that $\omega_{\W/B}$ is relatively nef. 
Then by Proposition \ref{P:CH-inequality}, we have $\left(8+4/h\right)\lambda-\delta \geq 0$. Since $h\geq 3$ and $\delta\geq 0$,
we obtain $
\frac{39}{4}\lambda-\delta>0
$ (when $\delta=0$, we have $\lambda>0$ by the existence of the Torelli morphism).
We proceed to estimate the remaining terms of \eqref{E:D-g-3}. Clearly, $\delta_{tacn}\geq 0$.
Since $h\geq 3$ and $\kappa=12\lambda-\delta> 0$, the inequalities of Lemmas \ref{L:cuspidal-section-high} and 
\ref{L:inner-sections-high} give 
\begin{align*}
\psi_{cusp} &=\sum_{i=1}^c\psi_{\xi_i} \geq \frac{(h-1)}{h}\sum_{i=1}^c \iota(\xi_i)+c\frac{\kappa}{4h(h-1)}=\frac{h-1}{2h}c(B)
+c\frac{\kappa}{4h(h-1)}>\frac{1}{3}c(B), \\
\sum_{i=1}^a (\psi_{\eta_i^+}+\psi_{\eta_i^-})&\geq 
\frac{2(h-1)}{h+1}\sum_{i=1}^{a}((\eta^+_i\cdot \eta^-_i)+\iota(\eta_i^+))+a\frac{\kappa}{h^2-1}>\delta_{inner}+n(B).
\end{align*}
 Summarizing, we conclude that the right hand side of \eqref{E:D-g-3} is strictly positive. 

\subsubsection{Case B: Relative genus $2$}
\label{geometric-genus-2}
Suppose $\Z/B$ is a family as in \ref{cuspidal-reduction} with relative geometric genus $h=2$.
Let $\W$ be the normalization of $\Z$ along the generic inner nodes.
As in \ref{genus-3}, we reduce to proving that
\begin{equation}\label{E:D-g-2}
\left(\frac{39}{4}\lambda-\delta+\psi-\frac{1}{4}\delta_{inner}-\frac{1}{4}n(B)-\frac{1}{4}c(B)+\frac{5}{4}\psi_{cusp}+\frac{7}{8}\delta_{tacn}\right)_{\W/B} \geq 0,
\end{equation}
under the assumption that $\omega_{\W/B}$ is relatively nef.

For any family $\W/B$ of arithmetic genus $2$ curves with a relatively nef $\omega_{\W/B}$, we have 
\begin{equation}\label{E:genus-2}
10\lambda=\delta_{\irr}+2\delta_{\red},
\end{equation}
This relation implies that $\delta\leq 10\lambda$ for any generically irreducible family
and, consequently, $\kappa =12\lambda-\delta \geq 2\lambda$, with the equality 
achieved only if $\delta_{\red}=0$, i.e., if there are no fibers where two genus $1$ components meet at a node.
It follows that $\frac{39}{4}\lambda-\delta\geq -\lambda/4$, with the equality only if $\delta_{\red}=0$.

By Lemma \ref{L:inner-sections-high}, we have 
\[
\sum_{i=1}^a (\psi_{\eta_i^+}+\psi_{\eta_i^-})  \geq \frac{2}{3} (\delta_{inner}+n(B))+a\frac{\kappa}{3}.
\]
By Lemma \ref{L:cuspidal-section-high}, we have
\[
\psi_{cusp} \geq \frac{1}{4}c(B)+c\frac{\kappa}{8}.
\]
Putting these inequalities together and using $\kappa\geq 2\lambda$, we obtain
\[
\frac{9}{4}\psi_{cusp}+\sum_{i=1}^a (\psi_{\eta_i^+}+\psi_{\eta_i^-}) \geq \frac{1}{4}\delta_{inner}+\frac{1}{4}n(B)+\frac{1}{4}c(B)+\left(\frac{2a}{3}+\frac{9c}{16}\right)\lambda.
\]
If $a+c\geq 1$, we obtain a strict inequality in \eqref{E:D-g-2} at once. Suppose $a=c=0$.
So far, we have that
\[ 
\left(\frac{39}{4}\lambda-\delta+\psi\right)_{\W/B}
\geq \sum_{i=1}^{n}\psi_{\sigma_i}+\sum_{i=1}^{b}\psi_{\zeta_i} - \frac{1}{4}\lambda.
\]
We now invoke Lemma \ref{L:genus-2-equality} that gives
\[
\sum_{i=1}^{n}\psi_{\sigma_i}+\sum_{i=1}^{b}\psi_{\zeta_i} \geq \frac{(n+b)}{8}\kappa \geq \frac{(n+b)}{4}\lambda.
\]
Since $n+b \geq 1$ (otherwise, $\W/B$ is an unpointed family of genus $2$ curves, which is impossible),
we conclude that $\sum_{i=1}^{n}\psi_{\sigma_i}+\sum_{i=1}^{b}\psi_{\zeta_i} - \lambda/4\geq 0$
and that equality is achieved if and only if $n+b=1$, $\delta_{\red}=0$, and equality is achieved in Lemma \ref{L:genus-2-equality}.
This is precisely the situation when $\Y/B=\W/B$ is a family of $A_1$-attached Weierstrass genus $2$ tails. 

Finally, if  $(h,a)=(1,1)$ or $(h,a)=(0,2)$, we proceed exactly as above
but without normalizing the inner nodes: 
For a family $(\Z/B, \sigman, \{\zeta_i\}_{i=1}^{b}, \{\xi_i\}_{i=1}^{c})$ as in \ref{cuspidal-reduction},
where the relative arithmetic genus of $\Z/B$ is $2$, we need to prove
\[
\left(\frac{39}{4}\lambda-\delta+\psi+\frac{5}{4}\psi_{cusp}-\frac{1}{4}c(B)+\frac{7}{8}\delta_{tacn}\right)_{\Z/B}\geq 0.
\]
Applying \eqref{E:genus-2} to estimate $\delta$, Lemma \ref{L:cuspidal-section-high} to estimate $\psi_{cusp}$,
and Lemma \ref{L:genus-2-equality} to estimate $\sum_{i=1}^{n}\psi_{\sigma_i}+\sum_{i=1}^{b}\psi_{\zeta_i}$
(all of which apply even if the total space $\Z$ is not normal), we obtain
\begin{multline*}
\frac{39}{4}\lambda-\delta+\psi+\frac{5}{4}\psi_{cusp}-\frac{1}{4}c(B)+\frac{7}{8}\delta_{tacn}
\geq -\frac{1}{4}\lambda+ \frac{4n+4b+9c}{16}\lambda+\frac{5}{16}c(B)+\frac{7}{8}\delta_{tacn} \geq 0.
\end{multline*}
Moreover, equality is achieved if and only if $\delta_{\red}=0$, $c=0$, and $n+b=1$, which is precisely 
the situation when $\Y/B=\Z/B$ is a family of $A_1$-attached (generically nodal) Weierstrass genus $2$ tails.

\subsubsection{Case C: Relative genus $1$}
\label{geometric-genus-1}

Suppose $\Z/B$ is a family as in \ref{cuspidal-reduction} of relative genus $1$ and
with $a$ generic inner nodes, where $a\neq 1$. We consider the case $a\geq 2$ first.
Let $\W$ be the family obtained from $\Z$ by the following operations:
\begin{enumerate}
\item Normalize $\Z$ along all generic inner nodes to obtain inner nodal pairs
$\{\eta_i^{+},\eta_i^{-}\}_{i=1}^{a}$.
\item Blow-up all cuspidal and inner nodal transforms to make them Cartier divisors.
\item Blow-up points of $\eta_i^{\pm}\cap \eta_j^{\pm}$ for all $i\neq j$. 
\item Blow-up points of $\zeta_i \cap \zeta_j$ for all $i\neq j$. 
\end{enumerate}
As a result, the sections of $\W/B$ 
do not intersect pairwise with the only possible exception
that $\eta_i^{+}$ is allowed to meet $\eta_i^{-}$. A node of $\Z$ through which $\xi_i$ passes is 
replaced in $\W$ by a balanced rational bridge meeting the strict transform of $\xi_i$, which we continue 
to denote by $\xi_i$.
We say that such a bridge is a \emph{cuspidal bridge 
associated to $\xi_i$}.
Moreover, if we let $c(\xi_i)$ be the sum of the indices of all bridges associated to $\xi_i$, then 
\[
2\iota(\xi_i)_{\Z/B}=c(\xi_i)_{\W/B}.
\]
Suppose $\{\eta_i^{+},\eta_i^{-}\}$ is an inner nodal pair of $\Z/B$. Then a node of $\Z$ through which $\eta_i^{+}$ and $\eta_i^{-}$
both pass is replaced in $\W$ by a balanced rational bridge meeting the strict transforms of $\eta_i^{+}$ and $\eta_i^{-}$,
which we continue to denote by $\eta_i^{+}$ and $\eta_i^{-}$.
We say that such a bridge is an \emph{inner nodal bridge associated to $\{\eta_i^{+},\eta_i^{-}\}$}.
Moreover, if we let $n(\eta_i)$ be the sum of the indices of all bridges associated to $\{\eta_i^{+},\eta_i^{-}\}$, then 
\[
((\eta_i^+\cdot \eta_i^-)+\iota(\eta_i^{+}))_{\Z/B}=((\eta_i^+ \cdot \eta_i^-)+n(\eta_i))_{\W/B}.
\]
On $\W/B$, we define
\begin{align*}
\delta_{inner}:=\sum_{i=1}^a (\eta^+_i \cdot \eta^-_i), \quad
\delta_{tacn}:= \sum_{i\neq j} \delta_{0, \{\eta_i^{\pm}, \eta_j^{\pm}\}}+\sum_{i\neq j} \delta_{0, \{\zeta_i, \zeta_j\}},
\end{align*}
and let $n(B)$ (resp., $c(B)$) be the sum of the indices of all inner nodal (resp., cuspidal) bridges.
We reduce to proving that
\[
\left(\frac{39}{4}\lambda-\delta+\psi-\frac{1}{4}\delta_{inner}-\frac{1}{4}n(B)-\frac{1}{4}c(B)+\frac{5}{4}\psi_{cusp}-\frac{1}{8}\delta_{tacn}\right)_{\W/B} \geq 0.
\]

We will make use of the standard relations 
for pointed families of genus $1$ curves and Lemmas \ref{L:inner-sections-1} and \ref{L:cuspidal-section-1}. 
Let $N=n+2a+b+c$ be the total number of marked sections of $\W/B$. Clearly, $N\geq 2$. We consider
first the case when $N\geq 3$. Then by Lemma \ref{L:inner-sections-1}, we have
\[
\delta_{inner}+n(B) \leq \frac{N}{2(N-2)}\sum_{i=1}^a (\psi_{\eta_i^+}+\psi_{\eta_i^-})+\frac{a}{2(N-2)^2}\delta_{\red}.
\]
Applying Lemma \ref{L:cuspidal-section-1}, we obtain 
\[
c(B)\leq \frac{2N}{N-1} \psi_{cusp}+\frac{c}{4(N-1)^2}\delta_{\red}.
\]
Using the above two inequalities and rewriting $\delta=12\lambda+\delta_{\red}$,
we see that
\begin{multline}
\label{genus-1}
\frac{39}{4}\lambda-\delta+\psi-\frac{1}{4}\left(\delta_{inner}+n(B)+c(B)\right)+\frac{5}{4}\psi_{cusp}-\frac{1}{8}\delta_{tacn} \\
\geq -\frac{9}{4}\lambda+\left(\frac{5}{4}-\frac{N}{2(N-1)}\right)\psi_{cusp}
+\psi-\frac{N}{8(N-2)}\sum_{i=1}^a (\psi_{\eta_i^+}+\psi_{\eta_i^-}) \\
-\left(1+\frac{a}{8(N-2)^2}+\frac{c}{16(N-1)^2}\right)\delta_{\red}-\frac{1}{8}\delta_{tacn}.
\end{multline}
We rewrite each $\psi$-class on the right-hand side of \eqref{genus-1} using the standard relation on families of arithmetic genus $1$ curves: 
\[
\psi_{\sigma}=\lambda+\sum_{\sigma\in S} \delta_{0, S}.
\]
The coefficient of $\lambda$ in the resulting expression for the right-hand side of \eqref{genus-1} is 
\begin{equation}
\label{coeff-lambda}
-\frac{9}{4}+c\left(\frac{5}{4}-\frac{N}{2(N-1)}\right)+N-\frac{aN}{4(N-2)}.
\end{equation}
Using $N\geq 2a+c$ and the assumption $N\geq 3$, 
it is easy to check that \eqref{coeff-lambda} is always positive.

A similarly straightforward but tedious calculation shows that each boundary divisor $\delta_{0,S}$ appears 
in the resulting expression for the right-hand side of \eqref{genus-1}
with a positive coefficient. Thus we have shown that the right-hand side of \eqref{genus-1} is positive for every non-isotrivial family with $N\geq 3$.

We consider now the case of $N=2$.
Since $\C/B$ in \ref{reduction-3} has no generic elliptic bridges (nodally or tacnodally attached),
we must have $c=1$ and $n+b=1$. Let $\xi$ be the corresponding cuspidal transform 
and $\tau$ be either a marked smooth section (if $n=1$) or an outer nodal transform (if $b=1$).
We trivially have $\delta_{inner}=n(B)=\delta_{tacn}=\delta_{\red}=0$.
Using $\delta_{\irr}=12\lambda$ and the inequality $c(B)\leq 4\psi_{cusp}$ from Lemma \ref{L:cuspidal-section-1}, we obtain:
\begin{multline*}
\frac{39}{4}\lambda-\delta+\psi-\frac{1}{4}\delta_{inner}-\frac{1}{4}n(B)-\frac{1}{4}c(B)+\frac{5}{4}\psi_{cusp}-\frac{1}{8}\delta_{tacn} \\
=\frac{39}{4}\lambda-\delta+\psi-\frac{1}{4}c(B)+\frac{5}{4}\psi_{cusp}
\geq \frac{39}{4}\lambda-12\lambda+\psi+\frac{1}{4}\psi_{cusp} \\
=\frac{39}{4}\lambda-12\lambda+2\lambda+\frac{1}{4}\lambda=0.
\end{multline*}
Moreover, equality holds only if equality holds in Lemma \ref{L:cuspidal-section-1}. 
This happens if and only if $2\xi \sim 2\tau$ and implies that $\Y/B$ in \ref{reduction-4} is a generically cuspidal family 
of $A_1$-attached Weierstrass genus $2$ tails. We are done with the analysis in the case $g=1$ and $a\neq 1$.

If $(g,a)=(0,1)$, we proceed exactly as
above, but without normalizing the inner node.

\subsubsection{Case D: Relative geometric genus $0$}
\label{geometric-genus-0}

Suppose $\Z/B$ is a family as in \ref{cuspidal-reduction} of relative geometric genus $0$ and
with $a$ generic inner nodes, where either $a\geq 3$ or $a=0$. We consider the case $a\geq 3$ first.
Let $\W$ be the family obtained from $\Z$ by the following operations:
\begin{enumerate}
\item Normalize $\Z$ along all generic inner nodes to obtain inner nodal pairs
$\{\eta_i^{+},\eta_i^{-}\}_{i=1}^{a}$.
\item Blow-up all cuspidal and inner nodal transforms to make them Cartier divisors.
This operation introduces cuspidal or nodal bridges as in \ref{geometric-genus-1}.
\item Blow-up points of $\eta_i^{\pm}\cap \eta_j^{\pm}$ for all $1\leq i < j \leq a$. 
\item Blow-up points of $\zeta_i \cap \zeta_j$ for all $1\leq i < j \leq b$. 
\item Blow-up points of $\eta_i^{+} \cap \eta_i^{-}$ for all $1\leq i \leq a$.
\item Blow-down all rational tails marked by a single section (such tails are necessarily adjacent 
either to cuspidal or inner nodal bridges). 
\end{enumerate}
As a result, $\W/B$ is a family in $\Mg{0,N}$, where $N=n+2a+b+c$ and $a\geq 3$.  
On $\W/B$, we define 
\begin{align*}
\delta_{inner}&:=\sum_{i=1}^a \delta_{\{\eta^+_i, \eta^-_i\}}, \qquad
\delta_{tacn}:= \sum_{i\neq j} \delta_{\{\eta_i^{\pm}, \eta_j^{\pm}\}}+\sum_{i\neq j} \delta_{\{\zeta_i, \zeta_j\}}, \\
\delta_{3}^{NB}&:=\sum_{i=1}^{a} \sum_{\beta \neq \eta_i^{+}, \eta_i^{-}}  \delta_{\{\eta_i^{+}, \eta_i^{-},\beta\}},  \qquad
\delta_{2}^{CB} :=\sum_{i=1}^{c}\sum_{\beta\neq \xi_i} \delta_{\{\xi_i, \beta\}},
\end{align*}
and let $n(B)$ (resp., $c(B)$) be the sum of the indices of all inner nodal (resp., cuspidal) bridges.
Then 
\begin{multline*}
\left(\frac{39}{4}\lambda-\delta+\psi+\frac{5}{4}\psi_{cusp}-\frac{1}{4}c(B)+\frac{7}{8}\delta_{tacn}\right)_{\Z/B} \\
=\left(\psi+\frac{5}{4} \psi_{cusp}-\delta-\frac{5}{4}\delta_{inner}-\frac{1}{4}(n(B)+c(B)+\delta_{3}^{NB}+\delta_2^{CB}) 
- \frac{1}{8} \delta_{tacn}\right)_{\W/B}.
\end{multline*}
We are going to prove that a (strict!) inequality
\begin{align*}
\psi+\frac{5}{4} \psi_{cusp}-\delta-\frac{5}{4}\delta_{inner}-\frac{1}{4}(n(B)+c(B)+\delta_{3}^{NB}+\delta_2^{CB}) - \frac{1}{8} \delta_{tacn} > 0
\end{align*}
always holds on $\W/B$. In doing so, we will use the following standard relation on $\Mg{0,N}$:
\begin{equation}\label{E:standard-relation}
\psi=\sum_{r\geq 2} \frac{r(N-r)}{N-1}\delta_{r}.
\end{equation}

First we deal with the case of a family with $3$ inner nodal pairs and no other marked sections, 
i.e., $a=3$ and $N=6$. The desired inequality in this case simplifies to
\begin{align*}
\psi-\delta-\frac{5}{4}\delta_{inner}-\frac{1}{4}(n(B)+\delta_{3}^{NB})-\frac{1}{8}(\delta_2-\delta_{inner}) > 0.
\end{align*}
We have an obvious inequality $2n(B) \leq \delta_{2}$. 
Thus we reduce to proving 
\begin{equation}
\label{E:06}
\psi > \frac{5}{4}\delta_{2}+\frac{9}{8}\delta_{inner}+\delta_{3}+\frac{1}{4}\delta_{3}^{NB}.
\end{equation}
For $a\geq 3$, Lemma \ref{L:inner-sections} gives 
\[
\psi \geq 4\delta_{inner}+ \delta_{tacn}+3\delta_3^{NB}=3\delta_{inner}+\delta_{2}+3\delta_3^{NB}.
\]
Combining this with the standard relation $5\psi=8\delta_{2}+9\delta_{3}$ gives
\[
8\psi\geq 9\delta_{inner}+11\delta_{2}+9\delta_{3}+9\delta_3^{NB}.
\]
This clearly implies \eqref{E:06} as desired.

Next, we consider the case of $N\geq 7$. 
In this case, every inner nodal or cuspidal bridge is adjacent to a node from $\sum_{r\geq 3} \delta_r$. As 
a result, we have $n(B)+c(B) \leq 2\sum_{r\geq 3} \delta_r$.
Furthermore, $\frac{1}{4}\delta_{2}^{CB}+\frac{1}{8}\delta_{tacn} +\frac{1}{4}\delta_{inner} \leq \frac{1}{4}\delta_2$
(because a node from $\delta_2$ can contribute only to one of the $\delta_{inner}$, $\delta_{tacn}$, or $\delta_{2}^{CB}$).
Hence we reduce to proving
\begin{equation}\label{E:07}
\psi+\frac{5}{4} \psi_{cusp}-\frac{5}{4}\delta_2 - \delta_{inner} - \frac{3}{2}\sum_{r\geq 3} \delta_r -\frac{1}{4}\delta_{3}^{NB} > 0
\end{equation}

We combine the inequality of Lemma \ref{L:inner-sections} with 
the standard relation \eqref{E:standard-relation}, and the obvious $\psi \geq \psi_{inner}$
to obtain 
\[
3\left(\psi-\sum_{r\geq 2} \frac{r(N-r)}{N-1}\delta_r \right)+\bigl(\psi_{inner}-4 \delta_{inner} - 3\delta_{3}^{NB}\bigr)+\bigl(\psi-\psi_{inner}\bigr) \geq 0.
\]
This gives the estimate
\[
4\psi\geq 4\delta_{inner}+\frac{6(N-2)}{N-1}\delta_{2}+\frac{9(N-3)}{N-1}\sum_{r\geq 3} \delta_{r}+ 3\delta_{3}^{NB}.
\]

Using $N\geq 7$ and $\psi_{cusp}\geq 0$, we finally get
\[
\psi+\frac{5}{4} \psi_{cusp} \geq \delta_{inner}+\frac{5}{4}\delta_2+\frac{3}{2}\sum_{r\geq 3} \delta_{r}+ \frac{3}{4}\delta_{3}^{NB}.
\]
Moreover, the equality could be achieved only if $N=7$ and $\psi-\psi_{inner}=0$
which is impossible because $\psi=\psi_{inner}$ implies that all sections are inner nodal transforms and so $N$ must be even.
Hence we have established \eqref{E:07} as desired. 

At last, we consider the case of $a=0$. Because the family $\W/B$ is non-isotrivial, we must have $N\geq 4$. 
In addition, if $N=4$, then there exists a unique family of $4$-pointed Deligne-Mumford stable rational curves.
The requisite inequality is easily verified for this family by hand. 
If $N\geq 5$, then using the inequality $2c(B)\leq \delta$, we reduce
to proving 
\begin{align*}
\psi+\frac{5}{4}\psi_{cusp}-\frac{9}{8} \delta-\frac{1}{4}\delta_{2}^{CB}-\frac{1}{8}\delta_{tacn} > 0.
\end{align*}

The standard relation \eqref{E:standard-relation} gives
\[
\psi\geq \frac{3}{2}\sum_{r\geq 2}\delta_r > \frac{11}{8}\delta_2+\frac{9}{8}\sum_{r\geq 3}\delta_r 
>  \frac{9}{8}\delta + \frac{1}{4}\delta_2.
\]
Finally, the inequality $\delta_{2} \geq \delta_{2}^{CB}+\delta_{tacn}$ gives the desired result.

This completes the proof of Theorem \ref{T:main-positivity} (b).

%% file: 2nd-flip-Appendix-submit-v9.tex
\appendix
\section{} \label{appendix}
In this appendix, we give examples of algebraic stacks including moduli stacks of curves which fail to have a good moduli space owing to a failure of conditions (1a), (1b), and (2) in Theorem \ref{T:general-existence}.  Note that there is an obviously necessary topological condition for a stack to admit a good moduli space, namely that every $\CC$-point has a unique isotrivial specialization to a closed point, and each of our examples satisfies this condition. The purpose of these examples is to illustrate the more subtle kinds of stacky behavior that can obstruct the existence of good moduli spaces.

\subsubsection*{Failure of condition (1a) in Theorem \ref{T:general-existence}}

\begin{example}  Let $\cX = [X/\ZZ_2]$ be the quotient stack where $X$ is the non-separated affine line and $\ZZ_2$ acts on $X$ by swapping the origins and fixing all other points.  The algebraic stack clearly satisfies condition (1b) and (2).  Then there is an \'etale, affine morphism $\AA^1 \to \cX$ which is stabilizer preserving at the origin but is not stabilizer preserving in an open neighborhood.  The algebraic stack $\cX$ does not admit a good moduli space.
\end{example}

While the above example may appear entirely pathological, we now provide two natural moduli stacks similar to this example.

\begin{example}
Consider the Deligne-Mumford locus $\cX \subseteq [\Sym^4 \PP^1 / \PGL_2]$ of unordered tuples $(p_1, p_2, p_3, p_4)$ where at least three points are distinct.  Consider the family  $(0, 1, \lambda, \infty)$ with $\lambda \in \PP^1$.  When $\lambda \notin \{ 0,1, \infty\}$, $\Aut(0,1, \lambda, \infty) \cong \ZZ/2\ZZ \times \ZZ/2\ZZ$; indeed, if $\sigma \in \PGL_2$ is the unique element such that $\sigma(0)=\infty$, $\sigma(\infty)=0$ and $\sigma(1) = \lambda$, then $\sigma([x,y]) = [y,\lambda x]$ so that $\sigma(\lambda)=1$ and therefore $\sigma \in \Aut(0,1, \lambda, \infty)$.  Similarly, there is an element $\tau$ which acts via $0 \stackrel{\tau}{\leftrightarrow} 1$, $\lambda \stackrel{\tau}{\leftrightarrow} \infty$ and an element $\alpha$ which acts via $0 \stackrel{\alpha}{\leftrightarrow} \lambda$, $1 \stackrel{\tau}{\leftrightarrow} \infty$.  However, if $\lambda \in \{0,1,\infty\}$, $\Aut(0,1, \lambda, \infty) \cong \ZZ/2\ZZ$.

Therefore if $x=(0,1,\infty, \infty)$, any \'etale morphism $f\co  [\Spec A/ \ZZ_2] \to \cX$, where $\Spec A$ is a $\ZZ_2$-equivariant algebraization of the deformation space of $x$, will be stabilizer preserving at $x$ but not in any open neighborhood.  This failure of condition (1a) here is due to the fact that automorphisms of the generic fiber do not extend to the special fiber.  The algebraic stack $\cX$ does not admit a good moduli space but we note that if one enlarges the stack $\cX$ to  $[(\Sym^4 \PP^1)^{\ss} / \PGL_2]$ by including the point $(0,0,\infty,\infty)$, there does exist a good moduli space.

\end{example}

\begin{example}
Let $\cV_{2}$ be the stack of all reduced, connected curves of genus 2, and let $[C] \in \cV_{2}$ denote a cuspidal curve whose pointed normalization is a generic 1-pointed smooth elliptic curve $(E,p)$. We will show that any Deligne-Mumford open neighborhood $\cM \subset \cV_{2}$ of $[C]$ is non-separated and fails to satisfy condition (1a).

Note that $\Aut(C)=\Aut(E,p)=\ZZ/2\ZZ$. Thus, to show that no \'etale neighborhood 
$$[\Def(C)/\Aut(C)] \rightarrow \cM$$ can be stabilizer preserving where $\Def(C) = \Spec A$ is an $\Aut(C)$-equivariant algebraized miniversal deformation space, it is sufficient to exhibit a family $\cC \rightarrow \Delta$ whose special fiber is $C$, and whose generic fiber has automorphism group $\ZZ/2\ZZ \times \ZZ/2\ZZ$. To do this, let $C'$ be the curve obtained by nodally gluing two identical copies of $(E,p)$ along their respective marked points. Then $C'$ admits an involution swapping the two components, and a corresponding degree 2 map $C' \rightarrow E$ ramified over the single point $p$. We may smooth $C'$ to a family $\cC' \rightarrow \Delta$ of smooth double covers of $E$, simply by separating the ramification points. By \cite[Lemma 2.12]{smyth_elliptic1}, there exists a birational contraction 
$\cC' \rightarrow \cC$ contracting one of the two copies of $E$ in the central fiber to a cusp. The family $\cC \rightarrow \Delta$ now has the desired properties; the generic fiber has both a hyperelliptic and bielliptic involution while the central fiber is $C$.
\end{example}

\subsubsection*{Failure of condition (1b) in Theorem \ref{T:general-existence} }

\begin{example}
Let $\cX = [\AA^2 \setminus 0 / \GG_m]$ where $\GG_m$ acts via $t \cdot (x,y) = (x,ty)$.  Let $\cU = \{y \neq 0 \}  = [\Spec \CC[x,y]_y / \GG_m] \subseteq \cX$.   Observe that the point $(0,1)$ is closed in $\cU$ and $\cX$.  Then the open immersion $f\co  \cU \to \cX$ has the property that $f(0,1) \in \cX$ is closed while for $x \neq 0$, $(x,1) \in \cU$ is closed but $f(x,1) \in \cX$ is not closed.  In other words, $f \co \cU \to \cX$ does not send closed points to closed points and, in fact, there is no \'etale neighborhood $\cW \to \cX$ of $(0,1)$ which sends closed points to closed points.  The algebraic stack $\cX$ does not admit a good moduli space.
\end{example}

\begin{example}
Let $\cM=\Mg{g} \cup \cM^1 \cup \cM^2$, where $\cM^1$ consists of all curves of arithmetic genus $g$ with a single cusp and smooth normalization, and 
$\cM^2$ consist of all curves of the form $D \cup E_0$, where $D$ is a smooth curve of genus $g-1$ and $E_0$ is a rational cuspidal curve attached to $C$ nodally. 

We observe that $\cM$ has the following property: If $C=D \cup E$, where $D$ is a curve of genus $g-1$ and $E$ is an elliptic tail, then $[C] \in \cM$ is a closed point if and only if $D$ is singular. Indeed, if $D$ is smooth, then $C$ admits an isotrivial specialization to $D \cup E_0$, where $E_0$ is a rational cuspidal tail.
Now consider any curve of the form $C=D \cup E$ where $D$ is a singular curve of genus $g-1$ and $E$ is a smooth elliptic tail, and, for simplicity, assume that $D$ has no automorphisms. We claim that there is no \'etale neighborhood of the form $[\Def(C)/\Aut(C)] \rightarrow \cM$, which sends closed points to closed points. Indeed, curves of the form $D' \cup E$ where $D'$ is smooth will appear in any such neighborhood and will obviously be closed in $[\Def(C)/\Aut(C)]$ (since this is a Deligne-Mumford stack), but are not closed in $\cM$.
\end{example}

\subsubsection*{Failure of condition (2) in Theorem \ref{T:general-existence}}

\begin{example} \label{example-nodal-cubic}
Let  $\cX = [X/\GG_m]$ where $X$ is the nodal cubic curve with the $\GG_m$-action given by multiplication.  Observe that $\cX$ is an algebraic stack with two points -- one open and one closed.  But $\cX$ does not admit a good moduli space; if it did, $\cX$ would necessarily be cohomologically affine and consequently $X$ would be affine, a contradiction.  However, there is an \'etale and affine (but not finite) morphism $\cW=[\Spec (\CC[x,y]/xy)  / \GG_m] \to \cX$ where $\GG_m = \Spec \CC[t,t^{-1}]$ acts on $\Spec \CC[x,y]/xy$ via $t \cdot (x,y) = (tx, t^{-1}y)$ which is stabilizer preserving and sends closed points to closed points;  however, the two projections $\cW \times_{\cX} \cW \rrarrows \cW$ do not send closed points to closed points.

To realize this \'etale local presentation concretely, we may express $X=Y/\ZZ_2$ where $Y$ is the union of two $\PP^1$'s with coordinates $[x_1,y_1]$ and $[x_2,y_2]$ glued via nodes at $p=0_1=0_2$ and $q=\infty_1 = \infty_2$ by the action of $\ZZ/2\ZZ$ where $-1$ acts via $[x_1,y_1] \leftrightarrow [y_2, x_2]$.  There is a  $\GG_m$-action on $Y$ given by $t \cdot [x_1,y_1] = [tx_1,y_1]$ and $t \cdot [x_2, y_2] = [x_1, ty_1]$ which descends to the $\GG_m$-action on $X$.  There is a finite \'etale morphism $[Y/\GG_m] \to \cX$, but $[Y/\GG_m]$ is not cohomologically affine. 
 If we instead, consider the open substack  $\cW = [ (Y \setminus \{p\} ) / \GG_m]$, then $\cW \cong [\Spec (\CC[x,y]/xy)  / \GG_m]$ is cohomologically affine and there is an \'etale representable morphism $f\co  \cW \to \cX$.  It is easy to see that
 $$\cW \times_{\cX} \cW \cong 
 [(Y \setminus \{p\})/\GG_m] \coprod [(Y \setminus \{p,q\}) / \GG_m] 
 $$
But  $[(Y \setminus \{p,q\}) / \GG_m] \cong \Spec \CC \coprod \Spec \CC$ and the projections $p_1, p_2\co  \cW \times_{\cX} \cW \to \cW$ correspond to the inclusion of the two open points into $\cW$ which clearly don't send closed points to closed points.
\end{example}

\begin{example}
Consider the algebraic stack $\cM_{g}^{\ss,1}$ of Deligne-Mumford semistable curves $C$ where any rational subcurve connected to $C$ at only two points is smooth.
Let $D_0$ be the Deligne-Mumford semistable curve $D' \cup \PP^1$, obtained by gluing a $\PP^1$ to a smooth genus $g-1$ curve $D'$ at two points $p,q$. For simplicity, let us assume that $\Aut(D',p,q)=0$, so $\Aut(D_0)=\GG_m$.   There is a unique isomorphism class of curves which isotrivially specializes to $D_0$, namely the nodal curve $D_1$ obtained by gluing $D$ at $p$ and $q$. Thus, $\overline{ \{ [D_1] \} }$ has two points -- one open and one closed.  In fact, 
$\overline{ \{ [D_1] \} }$ is isomorphic to the quotient stack $[X/\GG_m]$ 
considered in Example \ref{example-nodal-cubic}.
\end{example}